\documentclass[a4paper,11pt,twoside]{article}
\author{
	William Salkeld \\[8pt]
	\small School of Mathematical Sciences\\
	\small University of Nottingham \\
	\small University Park, Nottingham, NG7 2RD \\
	\small william.salkeld@nottingham.ac.uk
}
\usepackage[nodayofweek]{datetime}

\usepackage[utf8]{inputenc}
\usepackage[T1]{fontenc}
\usepackage[english]{babel}

\usepackage{charter}

\usepackage{xcolor}
\usepackage{verbatim}

\usepackage{hyperref}

\usepackage[top=3.cm, bottom=4.0cm, left=2.2cm, right=2.2cm]{geometry}
\usepackage{amsmath}
\usepackage{amsthm}	
\usepackage{amsfonts}	
\usepackage{amssymb,bbm}
\usepackage{shuffle}
\usepackage{mathrsfs}
\usepackage{pifont}

\usepackage{enumerate}
\usepackage{enumitem}
\usepackage[nobysame, alphabetic,initials]{amsrefs}

\usepackage{stmaryrd}
\usepackage{appendix}
\usepackage{etoolbox}		
\usepackage{graphicx}
\usepackage{
						ulem		
} 

\numberwithin{equation}{section}
\theoremstyle{plain}
\newtheorem{theorem}{Theorem}[section]
\newtheorem{lemma}[theorem]{Lemma}
\newtheorem{proposition}[theorem]{Proposition}
\newtheorem{corollary}[theorem]{Corollary}

\newtheorem{definition}[theorem]{Definition}

\newtheorem{remark}[theorem]{Remark}
\newtheorem{example}[theorem]{Example}
\newtheorem{assumption}[theorem]{Assumption}

\newtheorem{hypothesis}[theorem]{Hypothesis}

\allowdisplaybreaks
 
\newcommand{\bE}{\mathbb{E}}
\newcommand{\bF}{\mathbb{F}}

\newcommand{\bN}{\mathbb{N}}
\newcommand{\bP}{\mathbb{P}}

\newcommand{\bR}{\mathbb{R}}

\newcommand{\bU}{\mathbb{U}_I}
\newcommand{\bV}{\mathbb{V}}
\newcommand{\bW}{\mathbb{W}}
\newcommand{\bX}{\mathbb{X}}

\newcommand{\cA}{\mathcal{A}}
\newcommand{\cB}{\mathcal{B}}

\newcommand{\cE}{\mathcal{E}}
\newcommand{\cF}{\mathcal{F}}

\newcommand{\cH}{\mathcal{H}}

\newcommand{\cL}{\mathcal{L}}

\newcommand{\cP}{\mathcal{P}}

\newcommand{\cR}{\mathcal{R}}

\newcommand{\cT}{\mathcal{T}}

\newcommand{\fF}{\mathfrak{F}}

\newcommand{\fH}{\mathfrak{H}}
\newcommand{\fI}{\mathfrak{I}}

\newcommand{\fR}{\mathfrak{R}}

\newcommand{\fh}{\mathfrak{h}}

\newcommand{\fm}{\mathfrak{m}}

\newcommand{\fr}{\mathfrak{r}}

\newcommand{\scE}{\mathscr{E}}
\newcommand{\scF}{\mathscr{F}}
\newcommand{\scG}{\mathscr{G}}
\newcommand{\scH}{\mathscr{H}}

\newcommand{\scL}{\mathscr{L}}
\newcommand{\scM}{\mathscr{M}}
\newcommand{\scN}{\mathscr{N}}
\newcommand{\scP}{\mathscr{P}}
\newcommand{\scQ}{\mathscr{Q}}

\newcommand{\scT}{\mathscr{T}}
\newcommand{\scW}{\mathscr{W}}

\newcommand{\rD}{\mathbf{D}}

\newcommand{\rw}{\mathbf{W}}

\newcommand{\rId}{\mathbf{1}}

\newcommand{\rPi}{\mathbf{\Pi}}

\newcommand{\vertiii}{{\vert\kern-0.25ex \vert\kern-0.25ex \vert}}

\DeclareMathOperator{\Shuf}{Shuf}

\DeclareMathOperator{\spn}{span}
\DeclareMathOperator{\lin}{Lin}

\usepackage{forest}
\forestset{
	decor/.style = {
		label/.expanded = {[inner sep = 0.1ex, font=\unexpanded{\tiny}]right:{$#1$}}  
	},
	root/.style = {minimum size = 0.5ex},   
	decorated/.style = {
		for tree = {
			fill, 
			circle,
			inner sep = 0ex, 
			minimum size = 0.5ex,   
			grow' = north, 
			l = 0, 
			l sep = 0.8ex, 
			s sep = 1.3em,
			fit = tight, 
			parent anchor = center, 
			child anchor = center,
			delay = {decor/.option = content, content =}
		} 
	},
	x/.style = {
		circle, draw, fill=white
	},
	default preamble = {decorated, root},
	begin draw/.code={\begin{tikzpicture}[baseline={([yshift=-0.5ex]current bounding box.center)}]},
	}
	
	\usepackage{pgf,tikz}
	\usetikzlibrary{arrows}
	
	\tikzstyle{vertex} = [fill, shape=circle,inner sep=2pt,]
	\tikzstyle{edge} = [fill, line width = 0.5pt]
	\tikzstyle{zhyedge1} = [opacity=.5,fill opacity=.5, line cap=round, line join=round, line width=27pt,color=black]
	\tikzstyle{zhyedge2} = [opacity=.5,fill opacity=.5, line cap=round, line join=round, line width=25pt,color=white]
	\tikzstyle{hyedge1} = [opacity=.5,fill opacity=.5, line cap=round, line join=round, line width=27pt]
	\tikzstyle{hyedge2} = [opacity=.5,fill opacity=.5, line cap=round, line join=round, line width=25pt, color=white]
	
	\tikzstyle{vertexS} = [fill, shape=circle,inner sep=1pt,]
	\tikzstyle{edgeS} = [fill, line width = 0.25pt]
	\tikzstyle{zhyedge1S} = [opacity=.5,fill opacity=.5, line cap=round, line join=round, line width=12pt,color=black]
	\tikzstyle{zhyedge2S} = [opacity=.5,fill opacity=.5, line cap=round, line join=round, line width=10pt,color=white]
	\tikzstyle{hyedge1S} = [opacity=.5,fill opacity=.5, line cap=round, line join=round, line width=12pt]
	\tikzstyle{hyedge2S} = [opacity=.5,fill opacity=.5, line cap=round, line join=round, line width=10pt, color=white]
	
	\pgfdeclarelayer{background}
	\pgfsetlayers{background,main}
	
\providetoggle{figure}
\settoggle{figure}{true}
\providetoggle{Plus}
\settoggle{Plus}{true}

\title{Coupled bialgebras and Lions trees}

\begin{document}
	
	\maketitle
	
	\begin{abstract} 
		A bialgebra is a module over a ring that is both an associative algebra and a co-associative coalgebra with the product and coproduct additionally satisfying an appropriate commutative relationship. One application of bialgebras is in the study of Taylor expansions where the product describes how for any increment, two terms of the Taylor expansion can be combined together to obtain a higher order term, while the coproduct describes how a term can be decomposed into lower order terms over two adjacent increments. 
		
		Our motivation is the study of higher order Lions-Taylor expansions of functionals on the Wasserstein space of measures and more generally to probabilistic rough paths. The application of iterative Lions derivatives generates a collection of free variables for each term of the Lions-Taylor expansion. When these terms are themselves Lions-Taylor expanded, one needs to additionally differentiate in the free variables which greatly increases the complexity of the expressions. This additional structure typically takes the form of couplings between terms that arise through the application of the operator that corresponds to the coproduct. Hence, we explore the properties of this operation which we call the \emph{coupled coproduct}. 
		
		The purpose of this work is to document the combinatorial properties associated to Lions-Taylor expansions (which are central to many key interpretations of probabilistic rough paths), explore examples of coupled coproducts and more formally describe the \emph{coupled bialgebra}. 
	\end{abstract} 
	
	
	
	\vspace{0.3cm}
	
	
	
	\noindent
	{\bf Acknowledgements}: William Salkeld wishes to thank the London Mathematical Society for the award of an \emph{Early Career Fellowship} (ECF-1920-29) which facilitated this research.
	
	Further, William Salkeld was supported by MATH+ project AA4-2 and by the US Office of Naval Research under the Vannevar Bush Faculty Fellowship N0014-21-1-2887.  
	
	\setcounter{tocdepth}{2}
	\tableofcontents
	
	\section{Introduction}
	
	This paper is the result of the development of several of the ideas found in \cite{2021Probabilistic} and \cite{salkeld2021Probabilistic2} describing the algebraic properties of probabilistic rough paths, the first paper on which is \cite{2019arXiv180205882.2B}. To improve the readability of the subject matter, these unpublished papers were shortened to the subsequent works \cite{salkeld2022ExamplePRP} with several results updated and moved into this work. Additional results were moved into \cite{salkeld2022Lions}. 
	
	This paper acts as a sequel to \cite{salkeld2022ExamplePRP} and \cite{salkeld2022Lions} and provides a more detailed and rigorous study of the coupled Hopf algebras introduced in \cite{2021Probabilistic}. In this work, we will only refer to coupled bialgebras because we do not address the existence of an antipode in the coupled setting and we refer the reader to future work. For a classical reference on bialgebras and Hopf algebras we refer the reader to \cite{cartier2021hopf}. 
	
	The purpose of this paper is to abstract and contextualise the coupled bialgebra structures that arise in the papers \cite{salkeld2022ExamplePRP} and \cite{salkeld2022Lions}. Further, we include several technical results that allow us to interweave the concepts of tagged partitions, Lions forests and coupled pairs that are both necessary and provide insight into the study of probabilistic rough paths and yet we felt also limited the accessibility of this material. 
	
	\subsection{Motivation}
	This work is part of a wider program of research at the intersection between two active fields in mathematics. On the one hand, rough path theory, as first developed in \cite{lyons1998differential}, has proved to be a robust tool in the study of stochastic differential systems, particularly those that are driven by highly oscillatory signals and are thus out of the scope of standard integration theories such as Young integration for regular enough paths or stochastic integration for semimartingales. On the other hand, mean-field theory, which was originally dedicated to the analysis of large weakly interacting particle systems in statistical physics and in fluid mechanics, see \cites{kac1956foundations,McKean1966}, but has grown recently due to a surge of interest in line with the development of a calculus of variations on the space of probability measures and the study of related optimisation problems from transport theory, mean-field control or mean-field games, see for instance \cite{villani2008optimal} for the former and \cites{LionsVideo,CarmonaDelarue2017book1, CarmonaDelarue2017book2} for the latter two. 
	
	The adaptation of the rough path theory to the mean-field setting is a very natural and meaningful question which has already stimulated a series of works in the recent years. The analysis of such \textit{rough mean-field models} goes back to the article \cite{CassLyonsEvolving}. In the latter and subsequent works \cites{Bailleul2015Flows, deuschel2017enhanced}, the interaction between the particles appears solely in the drift and not in the volatility. Whilst this may appear to be an innocent simplification, this restriction reveals a substantial difficulty that is intrinsic to the mean-field setting: the provision of a convenient expansion of the coefficients with respect to the distributional argument.
	
	The appropriate strategy to address this challenge was clarified in a recent contribution \cite{2019arXiv180205882.2B}. At its simplest, the idea is to consider a Taylor expansion of the coefficients using the advances achieved over more than twenty years in the construction of a differential calculus on the Wasserstein space of probability measures. These ideas originate in the work of \cite{Jordan1998variation} on the connection between Fokker-Planck equations and gradient flows on the Wasserstein space. We refer the reader to the monograph \cite{Ambrosio2008Gradient} for a complete overview of the subject. In \cite{2019arXiv180205882.2B}, differential calculus on the Wasserstein space is implemented along Lions' approach, as originally introduced in \cite{LionsVideo} (see also \cite{CarmonaDelarue2017book1}*{Chapter 5}).  The main idea is to lift a differentiable function defined on the Wasserstein space to a function on the Hilbert space of square-integrable random variables and take a Fr\'echet derivative. The lifting operation consists of associating with any probability measure $\mu$ a random variable $X$ which has exactly $\mu$ as its distribution, $X$ being called a representative of $\mu$. Although there may be many choices for $X$, there are in practice canonical representatives for the probability measures under study. Since Lions' derivative is computed as a Fr\'echet derivative on the space of random variables, these canonical representatives appear explicitly in all the related derivatives. 
	
	One strength of the theory of rough paths is that provided enough information about the driving signal is known (taking the form of iterated integrals of the signal with respect to itself that are not known exogenously), integrals driven by the rough signal are known to exist and have an explicit form. The sense of `enough information' is directly captured via an abstract Taylor expansion, also referred to as a \emph{regularity structure}. Thus the rough path of some highly oscillatory signal has a higher order expansion of terms inversely proportional to the regularity. As we see in \cite{salkeld2022Lions}, higher order Lions-Taylor expansions contain a new level of complexity because the application of each Lions derivative generates a new free variable and subsequent differentiation must necessarily capture the regularity within all previously generated free variables. 
	
	As observed in \cites{2021Probabilistic, salkeld2021Probabilistic2} and \cite{salkeld2022Lions}, there are two quantities associated to each term within a Lions-Taylor expansion: the first is the order and corresponds to the number of derivatives (total) taken, while the second is the number of Lions derivatives taken within this iterated derivative. In \cite{salkeld2022Lions}, for a partition sequence $a\in A[0]$ corresponding to a iterated Lions derivative we say $m[a]$ is the number of Lions derivatives with all other derivatives either being in the spatial variable or in the free spatial variables generated by the application of those Lions derivatives. At a more abstract level, in \cites{2021Probabilistic, salkeld2021Probabilistic2, salkeld2022ExamplePRP} each term of the Lions-Taylor expansion is indexed by a Lions forest $T = (\scN, \scE, h_0, H^T) \in \scF_{0, d}[0]$ where the order of the derivative corresponds to $|\scN|$ and the number of free variables corresponds to $|H^T|$. 
	
	This work serves to explore the algebraic and combinatorial properties that arise from these higher order Lions-Taylor expansions. While this work is central to the program of probabilistic rough paths, no rough path theory is needed. 
	
	\subsection{Purpose of this paper}
	Let us take a step back for the moment and address the question \emph{What is a bialgebra?}. Given a ring $(\cR, +, \centerdot)$, a bialgebra $(\cH, \odot, \rId, \triangle, \epsilon)$ is a quintuple containing $\cH$ a module over $\cR$ such that $(\cH, \odot, \rId)$ is an associative, unital algebra;
	\begin{equation*}
		\odot: \cH \otimes \cH \to \cH, 
		\quad
		\rId: \cR \to \cH 
		\quad \mbox{and}\quad
		\odot \circ (\fI \otimes \odot) = \odot \circ (\odot \otimes \fI),
		\quad
		\odot \circ (\fI \otimes \rId) = \centerdot = \odot \circ (\rId \otimes \fI),
	\end{equation*}
	$(\cH, \triangle, \epsilon)$ is a co-associative, co-unital coalgebra;
	\begin{equation*}
		\triangle: \cH \to \cH \otimes \cH, 
		\quad
		\epsilon: \cH \to \cR 
		\quad \mbox{and}\quad
		(\fI \otimes \triangle) \circ \triangle = (\triangle \otimes \fI) \circ \triangle,
		\quad
		\centerdot \circ (\fI \otimes \epsilon) \circ \triangle = \fI = \centerdot \circ (\epsilon \otimes \fI) \circ \triangle,
	\end{equation*}
	and finally the bialgebra identities hold;
	\begin{equation*}
		\triangle \circ \odot = \odot \otimes \odot \circ \mbox{Twist} \circ \triangle \otimes \triangle, 
		\qquad
		\centerdot \circ \epsilon \otimes \epsilon = \epsilon \circ \odot,
		\qquad
		\rId \otimes \rId = \triangle \circ \rId \circ \centerdot,
		\quad\mbox{and}\quad
		\epsilon \circ \rId = \fI
	\end{equation*}
	where $\mbox{Twist}$ is the twist operation and $\fI$ is the identity operation. Notice that the tensor product $\otimes$ is central to each of these identities and the product/coproduct should be thought of as linear operators on the $\cR$-modules
	\begin{equation*}
		\odot: \cH \otimes \cH \to \cH
		\qquad
		\triangle: \cH \to \cH \otimes \cH. 
	\end{equation*}
	The focus of this paper is the scenario where $\triangle$ is replaced by $\Delta: \cH \to \cH \tilde{\otimes} \cH$ where the codomain $\cH \tilde{\otimes} \cH$ is a richer $\cR$-module than the canonical tensor module $\cH \otimes \cH$. At its absolute most basic, a \emph{coupled bialgebra} $(\cH, \odot, \rId, \Delta, \epsilon)$ is a quintuple containing $\cH$ a module over $\cR$ such that $(\cH, \odot, \rId)$ is an associative unital algebra; $(\cH, \Delta, \epsilon)$ satisfies the commutative relationship;
	\begin{equation}
		\label{eq:CommutativeDiagram1}
		(\Delta \tilde{\otimes} \fI) \circ \Delta = (\fI \tilde{\otimes} \Delta) \circ \Delta
		\quad\mbox{and}\quad
		\centerdot \circ (\fI \tilde{\otimes} \epsilon) \circ \Delta = \fI = \centerdot \circ (\epsilon \tilde{\otimes} \fI) \circ \Delta
	\end{equation}
	which can be written as commutative diagrams
	\begin{center}
		\begin{tikzpicture}
			\node at (0,-1) {$\cH \tilde{\otimes} \cH$};
			\node at (4,2) {$\cH \tilde{\otimes} \cH$};
			\node at (0,2) {$\cH$};
			\node at (4,-1) {$(\cH \tilde{\otimes} \cH) \tilde{\otimes} \cH$};
			\node at (4,0) {$\cH \tilde{\otimes} (\cH \tilde{\otimes} \cH)$};
			\node at (4,-0.5) {\rotatebox{90}{$\,=$}};
			\draw[-to](0.75,-1) to (3,-1);
			\draw[-to](4,1.5) to (4,0.5);
			\draw[-to](0,1.5) to (0,-0.5);
			\draw[-to](0.75,2) to (3,2);
			\node at (-0.5,0.5) {$\Delta$};
			\node at (2,2.5) {$\Delta$};
			\node at (2,-1.5) {$\Delta \tilde{\otimes} \fI$};
			\node at (4.5,1) {$\fI \tilde{\otimes} \Delta$};
		\end{tikzpicture}
		\\
		\quad
		\\
		\begin{tikzpicture}
			\node at (0,0) {$\cH$};
			\node at (2,1) {$\cH \tilde{\otimes} \cH$};
			\node at (2,-1) {$\cH \tilde{\otimes} \cH$};
			\node at (6,1) {$\cR \tilde{\otimes} \cH = \cR \otimes \cH$};
			\node at (6,-1) {$\cH \tilde{\otimes} \cR = \cH \otimes \cR$};
			\node at (9,0) {$\cH$};
			\draw[-to](0.25,0.25) to (1.5,0.75);
			\draw[-to](0.25,-0.25) to (1.5,-0.75);
			\draw[-to](2.75,-1) to (4.5,-1);
			\draw[-to](2.75,1) to (4.5,1);
			\draw[-to](7.5,-1) to (8.75,-0.25);
			\draw[-to](7.5,1) to (8.75,0.25);
			\draw[-to](0.5,0) to (8.5,0);
			\node at (4.5,0.25) {$\fI$};
			\node at (0.75,0.75) {$\Delta$};
			\node at (0.75,-0.75) {$\Delta$};
			\node at (3.5,1.5) {$\epsilon \tilde{\otimes} \fI$};
			\node at (3.5,-1.5) {$\fI \tilde{\otimes} \epsilon$};
			\node at (8.25,1) {$\centerdot$};
			\node at (8.25,-1) {$\centerdot$};
		\end{tikzpicture}
	\end{center}
	and finally;
	\begin{equation}
		\label{eq:CommutativeDiagram2}
		\Delta \circ \odot = \odot \tilde{\otimes} \odot \circ \overline{\mbox{Twist}} \circ \Delta \otimes \Delta, 
		\qquad
		\centerdot \circ \epsilon \otimes \epsilon = \epsilon \circ \odot,
		\qquad
		\rId \tilde{\otimes} \rId = \Delta \circ \rId \circ \centerdot,
		\quad\mbox{and}\quad
		\epsilon \circ \rId = \fI
	\end{equation}
	which can be written as commutative diagrams
	\begin{center}
		\begin{tikzpicture}
			\node at (0,0) {$\cH \otimes \cH$};
			\node at (0,-2) {$(\cH \tilde{\otimes} \cH) \otimes (\cH \tilde{\otimes} \cH)$};
			\node at (8,-2) {$(\cH \otimes \cH) \tilde{\otimes} (\cH \otimes \cH)$};
			\node at (8,0) {$\cH \tilde{\otimes} \cH$};
			\node at (4,0) {$\cH$};
			\draw[-to](0.75,0) to (3.5,0);
			\draw[-to](4.5,0) to (7.25,0);
			\draw[-to](0,-0.5) to (0,-1.5);
			\draw[-to](8,-1.5) to (8,-0.5);
			\draw[-to](2,-2) to (6,-2);
			\node at (2,0.5) {$\odot$};
			\node at (6,0.5) {$\Delta$};
			\node at (4,-2.5) {$\overline{\mbox{Twist}}$};
			\node at (-1,-1) {$\Delta \otimes \Delta$};
			\node at (9,-1) {$\odot \tilde{\otimes} \odot$};
		\end{tikzpicture}
		\\
		\quad
		\\
		\begin{tikzpicture}
			\node at (0,0) {$\cH \otimes \cH$};
			\node at (2,0) {$\cH$};
			\node at (0,-2) {$\cR \otimes \cR$};
			\node at (2,-2) {$\cR$};
			\draw[-to](0.75,0) to (1.5,0);
			\draw[-to](0.75,-2) to (1.5,-2);
			\draw[-to](0,-0.5) to (0,-1.5);
			\draw[-to](2,-0.5) to (2,-1.5);
			\node at (1.125,0.5) {$\odot$};
			\node at (1.125,-2.5) {$\centerdot$};
			\node at (-0.5,-1) {$\epsilon \otimes \epsilon$};
			\node at (2.5,-1) {$\epsilon$};
		\end{tikzpicture}
		\qquad
		\begin{tikzpicture}
			\node at (0,0) {$\cR \otimes \cR$};
			\node at (2,-0.5) {$\cR$};
			\node at (2,-2.5) {$\cH$};
			\node at (0,-3) {$\cH \tilde{\otimes} \cH$};
			\node at (0,-0.5) {\rotatebox{90}{$\,=$}};
			\node at (0,-1) {$\cR \tilde{\otimes} \cR$};
			\draw[-to] (0.75,0) to (1.75,-0.5);
			\draw[-to] (2,-1) to (2,-2);
			\draw[-to] (1.75,-2.5) to (0.75,-3);
			\draw[-to] (0,-1.5) to (0,-2.5);
			\node at (1.25,0) {$\centerdot$};
			\node at (1.25,-3) {$\Delta$};
			\node at (2.5,-1.5) {$\rId$};
			\node at (1.25,0) {$\centerdot$};
			\node at (-0.5,-2) {$\rId \tilde{\otimes} \rId$};
		\end{tikzpicture}
		\qquad
		\begin{tikzpicture}
			\node at (0,0) {$\cR$};
			\node at (2,0) {$\cR$};
			\node at (1,2) {$\cH$};
			\draw[-to](0.5,0) to (1.5,0);
			\draw[-to](0.25,0.25) to (0.75,1.75);
			\draw[-to] (1.25,1.75) to (1.75,0.25);
			\node at (1,-0.5) {$\fI$};
			\node at (0,1) {$\rId$};
			\node at (2,1) {$\epsilon$};
		\end{tikzpicture}
	\end{center}
	When formulating these properties in the lexicon of higher order Lions-Taylor expansions as discussed above, the nature of the coupled tensor product is clear given context and care. However, framed as above there is still a lot of work to be done in formalising what is described here. 
	
	Therefore, the question \emph{What is a coupled bialgebra?} can be answered if we can answer the question \emph{What is a coupled tensor product?}. To answer this question, we expound three known examples of coupled bialgebras: the module on which the coefficients of a Lions-Taylor expansion of some smooth function exist (see Section \ref{section:TaylorExpansions}), the module from the regularity structure associated with \emph{weakly geometric probabilistic rough paths} (see Section \ref{section:ProbabilisticRoughPaths}) and the module from the regularity structure associated with \emph{branched probabilistic rough paths} (see Section \ref{section:Lions-Trees}). While it will not be necessary for the reader to understand the distinction between \emph{weakly geometric} and \emph{branched} probabilistic rough paths, in a single sentence: weakly geometric probabilistic rough paths correspond to the \emph{Stratonovich} definition of a stochastic integral while branched probabilistic rough paths correspond to the \emph{It\^o} definition of a stochastic integral. These are distinct definitions of stochastic calculus and both have important applications. 
	
	\subsubsection*{Categorical perspective}
	
	Recall that the \emph{category} $\cR-\mathbf{Mod}$ of all modules of the ring $\cR$ is a \emph{monoidal category} with the tensor product of modules serving as the monoidal functor and the ring $\cR$ serving as the monoidal unit. Given that in this paper the coupled tensor product is only defined for specific modules (where basis elements have an additional poset structure), the author conjectures that there exists a sub-category of $\cR$-modules for which the coupled tensor product forms \emph{monoidal category}, but such a category is not defined here. We can already observe that in this scenario, the monoidal unit will be $\cR$ since we can identify $\cR \tilde{\otimes} \cH$ with $\cR \otimes \cH$ and $\cR \tilde{\otimes} \cR$ with $\cR \otimes \cR$.  
	
	Notice that for \eqref{eq:CommutativeDiagram2} to be satisfied, we require the existence of the morphism
	\begin{equation*}
		\overline{\mbox{Twist}}: (\cH \tilde{\otimes} \cH) \otimes (\cH \tilde{\otimes} \cH) \to (\cH \otimes \cH) \tilde{\otimes} (\cH \otimes \cH)
	\end{equation*}
	which in simple words tells me that ``$\otimes$ should distribute over $\tilde{\otimes}$''. This would suggest that for any sub-category of $\cR$-modules for which $\tilde{\otimes}$ forms a monoidal category, we would additionally expect that the classical tensor product would also form a monoidal category and we would obtain a \emph{duoidal category} since $(\tilde{\otimes}, \cR)$ will be a lax monoidal functor with respect to $(\otimes, \cR)$. 
	
	\subsection{Contribution}
	The central contribution of this work is formalising the concept of the  \emph{coupled tensor product} with the goal of precisely describing coupled bialgebras that arise in the context of higher order Lions-Taylor expansions. This concept first arises in \cite{2021Probabilistic}, but this work greatly expands on this. For the most basic intuition with regard to this, we direct the reader to Section \ref{subsec:Couplings-LionsDeriv} where a coupled tensor product arises in the context of the Lions-Taylor expansions proved in \cite{salkeld2022Lions}
	
	Another contribution of this work is greatly expanding the details of the \emph{coupled bialgebra of Lions words}. First included in \cite{2021Probabilistic} as an example to counter-balance the theory of Lions forests and motivated by the study of \emph{weakly geometric probabilistic rough paths}, many of these results were not explored in detail. These concepts are analogous to the \emph{shuffle}/\emph{deconcatenation} bialgebra on Lyndon words and work addresses these concepts in a lot of detail. This involves a lengthy discussion of \emph{tagged partitions}, which it transpires provide all the same combinatorial properties as partition couplings (the language of \cite{2021Probabilistic}) while also allowing for a more accessible introduction of the coupled tensor product.  
	
	Building on this progress, we consider the module spanned by Lions forests (first introduced in \cite{2021Probabilistic}) as a directed forest paired with a tagged partition that is described as \emph{Lions admissible} (see Definition \ref{definition:Quotient-Partition}). This \emph{coupled bialgebra of Lions forests} is analogous to the \emph{Connes-Kreimer} bialgebra. A key difference between the results of this work and the previous works \cites{2021Probabilistic, salkeld2021Probabilistic2} is the definition of coupled pairs (see Definition \ref{definition:coupled_pair-W} and \ref{definition:coupled_pair-F}) which is equivalent to previous results that used coupled partitions. It is worth highlighting that the set of couplings between two Lions forests is \emph{not} taken to be the set of couplings between the two collections of hyperedges, but rather a subset of couplings which we refer to as \emph{Lions couplings}. This isn't particularly important for calculations relating to \emph{random controlled rough paths} (which is why this detail was not spotted earlier) but is becomes critical when considering \emph{Lions elementary differentials}. 
	
	This work also includes several results relating to the combinatorial properties of Lions forests, see Theorem \ref{theorem:EquivalenceTrees}, Proposition \ref{proposition:CompletionOfTrees} and Proposition \ref{proposition:Coproduct-Equivalence}. Theorem \ref{theorem:EquivalenceTrees} serves to justify the hypergraphic structure of Lions forests since it shows that this is equivalent to allocating a partition sequence to each vertex of a directed graph that has length equal to the number of children of that vertex. This is highly natural since if we were to consider the \emph{Lions elementary differentials} that correspond to the solution of some rough McKean-Vlasov equation, we would see exactly this same structure. Proposition \ref{proposition:CompletionOfTrees} shows us that the collection of Lions forests is closed under completion of the operators $\big( \circledast, \cE, \lfloor \cdot \rfloor \big)$. This facilitates the use of inductive arguments with regard to any statement to do with Lions forests (either included here or related to rough paths). Finally, Proposition \ref{proposition:Coproduct-Equivalence} shows that the coupled coproduct can be defined in terms of its relationship with $\big( \circledast, \cE, \lfloor \cdot \rfloor \big)$. 
	
	\subsection{Partitions instead of labels?}
	
	The well-informed reader might ask ``We already have a fully expounded solution theory for $n$ interacting rough differential equations: in this setting, the solution is uniquely defined by a path on the characters of a Hopf algebra. However, the perspective of probabilistic rough paths is that the solution of a McKean-Vlasov rough differential equation is uniquely defined by a path on the characters of a coupled Hopf algebra. \emph{Where does this coupled tensor product arise in the mean-field limit?}''. 
	
	For collections of $n$ interacting equations, the basis for the associated Hopf algebra carry a labelling which takes values on the set $\{1, ..., n\}$. Therefore, one could naively construct another Hopf algebra for which the basis elements carry a labelling taking values on $\bN$. This approach would allow us to keep track of every possible interaction between every individual within the ensemble. However, the mean-field perspective is that this is the incorrect approach and instead we should utilise the exchangeability of the system. More specifically, we do not want to uniquely identify the basis elements corresponding to an interaction between two specific equations. Instead, we need a way to represent a basis element that captures the way in which a \emph{specific equation interacts with itself} and a distinct basis element that captures the way in which a \emph{specific equation interacts with another unidentified equation} following a symmetry argument that arises through exchangeability. Therefore, we \emph{don't want to label} components of our Hopf algebra basis but we \emph{do want to identify or separate} these components. Our solution was to \emph{remove labels} and \emph{replace them with partitions}. However, labels lead to a very different kind of algebraic structure for the bialgebra which gives rise to the \emph{coupled tensor product}. 
	
	For a more detailed explanation as to the justification of using partitions instead of labels, we direct the reader to \cite{salkeld2022ExamplePRP}. 
	
	\subsubsection*{Commonly used notation}
	
	Let $\bN$ be the set of positive integers and $\bN_{0}=\bN \cup \{0\}$. Let $\bR$ be the field of real numbers and for $d \in \bN$, let $\bR^d$ be the $d$-dimensional vector space over the field $\bR$. 
	
	For modules $U$ and $V$ over a ring $\cR$, we define $U \oplus V$ and $U \otimes V$ be the direct sum and tensor product of two modules. For a topological module $U$, let $\cB(U)$ be the Borel $\sigma$-algebra. Let $(\Omega, \cF, \bP)$ be a probability space and let $L^0(\Omega, \bP; U)$ be the space of measurable mappings $(\Omega, \cF) \mapsto (U, \cB(U))$. 
	
	For a set $\scN$, we write $\scP(\scN)$ for the set of all partitions of the set $\scN$. The set of partitions $\scP(\scN)$ has a partial ordering $\subseteq$ where $P \subseteq Q$ iff $Q$ is finer than $P$:
	\begin{equation*}
		P \subseteq Q \quad \iff \quad \forall q\in Q, \quad \exists p \in P: q\subseteq p. 
	\end{equation*}
	
	\section{Lions-Taylor expansions}
	\label{section:TaylorExpansions}
	
	In this first section, our goal is to take some of the results first proved in \cite{salkeld2022Lions} and develop them with the purpose of introducing the origin of the coupled tensor product. This ``algebraic perspective'' on Lions-Taylor expansions is not present in \cite{salkeld2022Lions} where the focus is instead put on the ``analytic perspective'' (that is, accurate general formulas for Lions-Taylor remainder terms). 
	
	\subsection{Lions-Taylor expansions and coupled structures}
	\label{subsection:TaylorExpansions}
	
	We start by considering how we should represent a Taylor expansion for a function of the form
	\begin{equation*}
		f: (\bR^e)^{\times |I|} \times \cP_2(\bR^e) \to \bR^d. 
	\end{equation*}
	where $I$ is an index set. To answer this, we consider a derivative constructed on the so-called `Wasserstein space' of probability measures with finite second moment. For any $p\geq 1$, let $\cP_p(\bR^e)$ be the set of all probability measures on $\big( \bR^e, \cB(\bR^e) \big)$. This is equipped with the $\bW^{(p)}$-Wasserstein distance defined by:
	\begin{equation}
		\label{eq:WassersteinDistance}
		\bW^{(p)}(\mu,\nu) = \inf_{\Pi \in \cP_{p}(\bR^e \times \bR^e)}
		\biggl( \int_{\bR^e \times \bR^e}
		|x-y|^p
		d \Pi(x,y) \biggr)^{1/p},
	\end{equation}
	the infimum being taken with respect to all the probability measures $\Pi$ on the product space $\bR^e \times \bR^e$ with $\mu$ and $\nu$ as respective $e$-dimensional marginal laws.
	
	\subsubsection*{The Lions derivative}
	\label{subsection:1-Lip_sup_envelope}
	
	For a function $f:\cP_2(\bR^e) \to \bR^d$, we consider the canonical lift $F: L^2(\Omega, \cF, \bP; \bR^e) \to \bR^d$ defined by $F(X) = f( \bP\circ X^{-1})$. We say that $f$ is $L$-differentiable at $\mu$ if $F$ is Fr\'echet differentiable at some point $X$ such that $\mu = \bP\circ X^{-1}$. Denoting the Fr\'echet derivative by $DF$, it is now well known (see for instance \cite{GangboDifferentiability2019} that $DF$ is a $\sigma(X)$-measurable random variable of the form $DF(\mu, \cdot):\bR^e \to \lin(\bR^e, \bR^d)$ depending on the law of $X$ and satisfying $DF(\mu, \cdot) \in L^2\big( \bR^e, \cB(\bR^e), \mu; \lin(\bR^e, \bR^d) \big)$. We denote the $L$-derivative of $f$ at $\mu$ by the mapping $\partial_\mu f(\mu)(\cdot): \bR^e \ni x \to \partial_\mu f(\mu, x) \in \lin(\bR^e, \bR^d)$ satisfying $DF(\mu, \cdot) = \partial_\mu f(\mu, X(\cdot))$. This derivative is known to coincide with the so-called Wasserstein derivative, as defined in for instance \cite{Ambrosio2008Gradient}, \cite{CarmonaDelarue2017book1} and \cite{GangboDifferentiability2019}. As we explained in the introduction, Lions' approach is well-fitted to probabilistic approaches for mean-field models since, very frequently, we have  a `canonical' random variable $X$ for representing the law of a given probability measure $\mu$. 
	
	The second order derivatives are obtained by differentiating $\partial_{\mu} f$ with respect to $x$ (in the standard Euclidean sense) and $\mu$ (in the same Lions' sense). The two derivatives $\nabla_{x} \partial_{\mu} f$ and $\partial_{\mu} \partial_{\mu} f$ are thus very different functions: The first one is defined on ${\cP}_{2}({\bR}^e) \times {\bR}^e$ and writes $(\mu,x) \mapsto \nabla_{x} \partial_{\mu} f(\mu,x)$ whilst the second one is defined on ${\cP}_{2}({\bR}^e) \times {\bR}^e \times {\bR}^e$ and writes $(\mu,x,x') \mapsto \partial_{\mu} \partial_{\mu} f(\mu,x,x')$. The $e$-dimensional entries of $\nabla_{x} \partial_{\mu} f$ and $\partial_{\mu} \partial_{\mu} f$ are called here the \textit{free} variables, since they are integrated with respect to the measure $\mu$ itself. 
	
	In practice, we will be working with measure functionals that are dependent on a multi-variable $(x_{I}, \mu):=\big( (x_{\iota})_{\iota \in I}, \mu) \in (\bR^e)^{\times |I|} \times \cP_2(\bR^e)$ so need to consider a Taylor expansion involving both Lions and spatial derivatives. Unlike with the free variables generated by the application of iterative Lions derivatives where the $m^{th}$ variable $x_m$ can only appear if $\partial_a f$ contains at least $m$ derivatives with respect to $\mu$, the variable $x_I$ can appear in any derivatives of $f$ whenever $f$ is a function of the form $f(x_I, \mu)$. 
	
	This leads us to Definition \ref{def:a} below, the principle of which can be stated as follows for the first and second order Lions derivatives: The derivative symbol $\partial_{\mu}$ can be denoted by $\partial_{1}$ and then the two derivative symbols $\nabla_{x_1} \partial_{\mu}$ and $\partial_{\mu} \partial_{\mu}$ can be respectively denoted by $\partial_{(1,1)}$ and $\partial_{(1,2)}$. Derivatives with respect to the $x_\iota$-component (for $\iota \in I$) are encoded through the inclusion of a ‘$\iota$’. Repeated $\iota$’s thus account for repeated derivatives in the direction of $x_\iota$.
	
	\begin{definition}
		\label{def:a}
		The sup-envelope of an integer-valued sequence $(a_{i})_{i=1,...,n}$ of length $n$ is the non-decreasing sequence $(\max_{i=1,...,k} a_{i})_{k=1,...,n}$. The sup-envelope is said to be $1$-Lipschitz (or just $1$-Lip) if, for any $k \in \{2,...,n\}$, 
		\begin{equation*}
			\max_{i=1,...,k} a_{i} \leq 1+ \max_{i=1,...,k-1} a_{i}.
		\end{equation*}
		We call $A_{n}$ the collection of all ${\bN}$-valued sequences of length $n$, with $a_{1}=1$ as initial value and with a 1-Lip sup-envelope. Thus $A_n$ is the collection of all sequences $(a_k)_{k=1, ..., n} \in A_n$ taking values on $\{1, ..., n\}$ such that 
		\begin{equation*}
			a_1 = 1, \quad a_k \in \Big\{1, ..., 1+ \max_{i=1, ..., k-1} a_i \Big\}. 
		\end{equation*}
		We refer to $A_n$ as the collection of \emph{partition sequences} of length $n$. 
		
		Let $I$ be an index and let $k, n\in \bN_0$. Let $A_{k,n}[I]$ be the collection of all sequences $a' = (a_i')_{i=1, ..., k+n} = \sigma( a \cdot b)$ where $a\in A_n$, $(b) = (b_i)_{i=1, ..., k}$ is the sequence of length $k$ with entries in $I$ and $\sigma$ is a $(k,n)$-shuffle, i.e., a permutation of $\{1, ..., n+k\}$ such that $\sigma(1) < ... < \sigma(k)$ and $\sigma(k+1)<... < \sigma(n+k)$. 
		
		For a given $n \in \bN$, we let 
		\begin{equation*}
			A_{n}[I]= \bigcup_{k=0}^n A_{k,n-k}[I], \quad A^n[I] = \bigcup_{k=0}^n A_k[I] \quad \mbox{and}\quad A[I] = \bigcup_{k\in \bN_0} A_{k}[I].
		\end{equation*}
		We refer to $A[I]$ as the set of \emph{tagged partition sequences}. Given $a\in A_{n}[I]$, we denote 
		\begin{equation*}
			|a| = n \quad \mbox{and} \quad m[a] = \max_{\substack{i=1, ..., n \\ a_i \notin I}} a_i.
		\end{equation*}
		For $a, a'\in A_n[I]$, we say that $a \subseteq a'$ if and only if
		\begin{equation}
    			\label{eq:partial-ordering}
	    		\begin{aligned}
    				\forall \iota \in I \quad &a^{-1}[\iota] \subseteq (a')^{-1}[\iota] \quad \mbox{and}
    				\\
    				&\forall j=1, ..., m[a]\quad \exists j' \in \{1, ..., m[a']\}\cup I \quad \mbox{such that} \quad  a^{-1}[j] \subseteq (a')^{-1}[j'].  
    			\end{aligned}
	    	\end{equation}
	\end{definition}
	We denote $\llbracket a\rrbracket_I$ to be the equivalence class of all sequences such that
	\begin{align*}
		(b_i)_{i=1, ..., n} \in \llbracket a\rrbracket_I \quad \iff \quad &\Big\{ b^{-1} \Big\} = \Big\{ a^{-1}[j]: j \in \big\{ 1, ..., m[a] \big\} \cup I \Big\} \in \scP\big( \{1, ..., n\} \big)
		\\ 
		& \mbox{and}\quad \forall \iota \in I, \quad b^{-1}[\iota] = a^{-1}[\iota]. 
	\end{align*}
	For a sequence $b = (b_i)_{i=1, ..., n}$ and $a\in A_n[I]$ such that $a \subseteq \llbracket b\rrbracket_I$ we define
	\begin{equation}
		\label{eq:K-sequence2}
		\boldsymbol{b}\circ(a) = \Big( b_{a^{-1}[j]} \Big)_{j=1, ..., m[a]}. 
	\end{equation}	
	We define $\scG^I: A[I] \to \bN_0^{\times 2}$ by $\scG^I[a] = \big( |a|, m[a] \big)$. 

	Let $g: \bN_0^{\times 2}  \to \bR$ be a monotone increasing function such that $g(0, 0)\leq 0$. Then we denote 
	\begin{equation*}
		A^{g,-}[I]:= \Big\{ a\in A[I]: g \big( \scG^I[a] \big) \leq 0 \Big\}. 
	\end{equation*}
	
	\begin{proposition}
		\label{proposition:taggedpartition*}
		Let $\scN$ be a finite set and let $I$ be a finite index set (which for clarity should not be confused with $\scN$). For each $\iota\in I$, let $p_\iota \subseteq \scN$ such that the collection of subsets $p_I = (p_\iota)_{\iota\in I}$ are mutually disjoint. Let $P\in \scP\Big( \scN \backslash \big( \bigcup_{\iota\in I} p_\iota \big) \Big)$. Then we say that $\big( p_I, P\big)$ is a tagged partition and we denote the set of all tagged partitions by $\scP(\scN)[I]$.  
		
		The set $\scP(\scN)[I]$ is a poset with partial ordering
		\begin{equation}
			\label{definition:taggedpartition*}
			\big( p_I, P) \subseteq \big( q_I, Q) \quad \iff \quad 
			\begin{aligned}
				& \forall \iota\in I\quad p_\iota \subseteq q_\iota \quad \mbox{and}
				\\
				&\forall p \in P, \exists q \in \Big( Q \cup \bigcup_{\iota\in I} \{q_\iota\} \Big) \quad \mbox{such that} \quad p \subseteq q.
			\end{aligned}
		\end{equation}
		We denote $\fm :\scP(\scN)[I]\to \bN_0$ to be the function such that for $\fm\big( (p_I, P) \big) = |P|$. 
		
		Then there is an isomorphism between
		\begin{equation*}
			\big( A_{|\scN|}[I] , \subseteq, m \big) 
			\quad \mbox{and}\quad
			\big( \scP(\scN)[I], \subseteq, \fm \big). 
		\end{equation*}
	\end{proposition}
	
	\subsubsection*{A Lions-Taylor expansion}
	
	For $n\in \bN$ and $a\in A_n[I]$, we define $\partial_a$, according to the following induction:
	\begin{align*}
		\partial_{(\iota)} =& \nabla_{x_\iota}, \quad \partial_{(1)} = \partial_\mu \quad \mbox{and}
		\\
		\partial_{(a_1, ..., a_{k-1}, a_k)} =&
		\begin{cases} 
			\nabla_{x_\iota} \cdot \partial_{(a_1, ..., a_{k-1})} & \quad a_k =\iota \in I, 
			\\
			\nabla_{x_{a_k}} \cdot \partial_{(a_1, ..., a_{k-1})} & \quad 0 < a_k \leq \max \{ a_1, ..., a_{k-1}\}, 
			\\
			\partial_\mu \cdot \partial_{(a_1, ..., a_{k-1})} & \quad a_k > \max \{a_1, ..., a_{k-1}\}.  
		\end{cases}
	\end{align*}
	
	Let $k, n \in \bN_0$ and let $a\in A_{k, n}[I]$. Let $x_I, y_I \in (\bR^e)^{\times |I|}$ and let $\Pi^{\mu, \nu}$ be a measure on $(\bR^e)^{\oplus 2}$ with marginal distribution $\mu, \nu \in \cP_{n+1}(\bR^e)$. Then we define the operator
	\begin{align}
		\nonumber
		\rD^a& f(x_I, \mu)[ y_I-x_I, \Pi^{\mu, \nu}]
		\\
		\label{eq:rDa}
		&= \underbrace{\int_{(\bR^e)^{\oplus 2}} ... \int_{(\bR^e)^{\oplus 2}} }_{\times m[a]} \partial_a f \Big( x_I, \mu, \boldsymbol{x}_{m\{a\}} \Big) \cdot \bigotimes_{i=1}^{|a|} ( y_{a_i} - x_{a_i})  \cdot d\big( \Pi^{\mu, \nu}\big)^{\times m[a]} \Big( (\boldsymbol{x}, \boldsymbol{y})_{m\{a\}} \Big). 
	\end{align}
	Here, for compact notation we have denoted
	\begin{align*}
		&d\big( \Pi^{\mu, \nu}\big)^{\times m[a]} \Big( (\boldsymbol{x}, \boldsymbol{y})_{m\{a\}}\Big) = d\Pi^{\mu, \nu}(x_1, y_1) \times ... \times d\Pi^{\mu, \nu}(x_{m[a]}, y_{m[a]} ) \quad \mbox{and}
		\\
		& \boldsymbol{x}_{m\{a\}} = (x_1, ..., x_{m[a]}). 
	\end{align*}
	Following \cite{salkeld2022Lions}*{Definition 2.23}, for any $A' \subseteq A[I]$ and function $f:(\bR^e)^{\times |I|} \times \cP_2(\bR^e) \to \bR^d$ we say
	\begin{equation*}
		f \in C_b\big[ A' \big]\big( (\bR^e)^{\times |I|} \times \cP_2(\bR^e); \bR^d\big)
	\end{equation*}
	if for every $a\in A'$, the derivative $\partial_a f: (\bR^e)^{\times |I|} \times \cP_2(\bR^e) \times (\bR^e)^{\times m[a]} \to \lin\big((\bR^e)^{\otimes |a|}, \bR^d \big)$ exist and is bounded and continuous. 
	
	\begin{proposition}
		\label{proposition:classicTay<=>LionsTay*}
		Let $I:=\{0\}$ which we denote by $0$ and let $n, N\in \bN$. Let $f: \bR^e \times \cP_2(\bR^e) \to \bR^d$ and suppose that $f\in C_b\big[A^{n}[0]\big]\big( \bR^e \times \cP_2(\bR^e); \bR^d\big)$. Let $x^{1, N}, ..., x^{N, N} \in \bR^e$ and denote $\boldsymbol{x} = (x^{1, N}, ..., x^{N, N})$. 
		
		We define $\boldsymbol{f}: (\bR^e)^{\oplus N} \to (\bR^d)^{\oplus N}$ by 
		\begin{equation}
			\label{eq:bf-field}
			\boldsymbol{f}\big( \boldsymbol{x} \big) = \boldsymbol{f}\big( (x^{1, N}, ..., x^{N, N}) \big) = \bigoplus_{k=1}^N \overline{f}_k \big( \boldsymbol{x} \big)
		\end{equation}
		where $\overline{f}_k: (\bR^e)^{\oplus N} \to \bR^d$ is defined by
		\begin{equation*}
			\overline{f}_k \big( \boldsymbol{x} \big) = \overline{f}_k \big( (x^{1, N}, ..., x^{N, N}) \big) = f\Big( x^{k, N}, \bar{\mu}_N\big[ \boldsymbol{x}^N \big] \Big), \quad \bar{\mu}_N\big[ \boldsymbol{x} \big] = \sum_{j=1}^N \delta_{x^{j, N}}. 
		\end{equation*}
		Let $\boldsymbol{i} = (i_1, ..., i_n)$ where each $i_j \in \{1, ..., e\}$. Then 
		\begin{equation}
			\label{eq:example:Sums-finer-partitions}
			\nabla_{\boldsymbol{i}} \overline{f}_i \big( \boldsymbol{x} \big) = \sum_{\substack{a\in A_n[0] \\ a\subseteq \llbracket \boldsymbol{i} \rrbracket_i }} \partial_a f\Big( x^{i, N}, \bar{\mu}_N\big[\boldsymbol{x}\big] , \boldsymbol{x}^{\boldsymbol{i}\circ (a), N} \Big)  
		\end{equation}
		where we used the notation 
		\begin{equation*}
			\boldsymbol{x}^{\boldsymbol{i} \circ a, N} = \Big( x^{i_{a^{-1}[1]}, N}, ..., x^{i_{a^{-1}[m[a]]}, N} \Big). 
		\end{equation*}
		
		Additionally for $\boldsymbol{x}, \boldsymbol{y} \in (\bR^e)^{\oplus N}$ we define $\Pi^{\boldsymbol{x}, \boldsymbol{y}} \in \cP_2(\bR^e \oplus \bR^e)$ by 
		\begin{equation}
			\label{eq:proposition:classicTay<=>LionsTay}
			\Pi^{\boldsymbol{x}, \boldsymbol{y}} = \frac{1}{N} \sum_{j=1}^N \delta_{(x^{j, N}, y^{j, N})}. 
		\end{equation}
		
		Then
		\begin{equation}
			\label{eq:proposition:classicTay<=>LionsTay*}
			\boldsymbol{f}\Big( \boldsymbol{y} \Big) = \sum_{a\in A^{n}[0]} \frac{1}{|a|!} \bigoplus_{i=1}^N \rD^a f \Big( x^{i, N}, \bar{\mu}_N\big[ \boldsymbol{x} \big] \Big)\Big[ y^{i, N} - x^{i. N}, \Pi^{\boldsymbol{x}, \boldsymbol{y}} \Big] + \boldsymbol{R}_n^{\boldsymbol{x}, \boldsymbol{y}}\big( \boldsymbol{f} \big). 
		\end{equation}
	\end{proposition}
	
	\begin{theorem}
		\label{theorem:LionsTaylor2}
		Let $f: (\bR^e)^{\times |I|} \times \cP_2(\bR^e) \to \bR^d$ and suppose that $f$ has Lions and spatial derivatives of all orders. Let $g: \bN_0^{\times 2}  \to \bR$ be a monotone increasing function such that $g(0, 0) \leq 0$ and let $n \in \bN$ be the greatest integer such that $g(0, n)\leq 0$. 
		
		For any $\mu, \nu \in \cP_{n+1} (\bR^e)$ with joint distribution $\Pi^{\mu, \nu}$ and any $x_I, y_I \in \bR^e$, we have that
		\begin{align}
			\label{eq:FullTaylorExpansion}
			f(y_I, \nu) 
			= \sum_{a\in A[I]}^{g,-} \frac{1}{|a|!} \cdot \rD^a f\Big( x_I, \mu \Big)\Big[ y_I - x_I, \Pi^{\mu, \nu} \Big]  
			+ R_{g,-}^{(x_I, y_I), \Pi^{\mu, \nu}} \big( f \big). 
		\end{align}
	\end{theorem}
	
	For a full proof along with an explicit statement for $R_{g,-}^{(x_I, y_I), \Pi^{\mu, \nu}} \big( f \big)$ which is beyond the scope of this work, we direct the reader to \cite{salkeld2022Lions}. This result is included to give context to the ideas that we will be exploring in great detail in the coming paper. 
	
	Another natural observation is that the Fr\'echet derivatives are symmetric multi-linear operators. This leads us to the following
	
	\begin{theorem}[Schwarz Theorem for Lions Derivatives]
		\label{thm:Schwarz-Lions}
		Let $a \in A_n[I]$ and let $f: (\bR^e)^{\times |I|} \times \cP_2(\bR^e) \to \bR^d$ such that $\partial_a f$ exists. Let $\Shuf(|a|)$ be the set of permutations on the set $\{1, ..., |a| \}$. 
		
		Then $\forall \sigma\in \Shuf(|a|)$, $v_1, ..., v_{|a|} \in \bR^e$, $\mu\in \cP_2(\bR^e)$ and for every $\iota \in I$ we have $x_\iota, x_1, ..., x_{|a|} \in \bR^e$, 
		\begin{equation}
			\label{eq:proposition:Schwartz-2}
			\partial_{a} f\Big( x_I, \mu, \boldsymbol{x}_{m\{a\}} \Big) \cdot \bigotimes_{i=1}^{|a|} v_i 
			=
			\partial_{\llbracket \sigma[a]\rrbracket_I} f\Big( x_I, \mu, \boldsymbol{x}_{\llbracket \sigma(a)\rrbracket \circ (a)} \Big) \cdot \bigotimes_{i=1}^{|a|} v_{\sigma(i)} . 
		\end{equation}
		where
		\begin{equation*}
			\boldsymbol{x}_{m\{a\}} = \Big( x_1, ..., x_{m[a]} \Big)
			\quad \mbox{and}\quad 
			\boldsymbol{x}_{\llbracket \sigma(a)\rrbracket \circ (a)} = \Big( x_{\sigma(a)_{a^{-1}[1]}}, ..., x_{\sigma(a)_{a^{-1}[m[a]]}} \Big). 
		\end{equation*}
	\end{theorem}

	\subsubsection*{An iterative Lions-Taylor expansion}
	
	\cite{salkeld2022Lions} also includes results on how one takes a Lions-Taylor expansion of the Lions derivative of some smooth function of on the Wasserstein space of measures: 
	
	Throughout this work, for $a\in A_n[I]$ we will often interchangeably write
	\begin{equation}
		\label{eq:Tag-Part-Seq[a]}
		\begin{aligned}
			&A_n[a] = A_n\Big[ \big\{a^{-1}[j]: j \in I \cup \{1, ..., m[a]\} \big\} \Big]. 
		\end{aligned}
	\end{equation}
	
	Theorem \ref{theorem:LionsTaylor2} admits the following Corollary:
	\begin{corollary}
		\label{corollary:LionsTaylor2}
		Let $g: \bN_0^{\times 2}  \to \bR$ be a monotone increasing function such that $g(0, 0)\leq 0$ and let $n\in \bN_0$ be the greatest integer such that $g(0,n) \leq 0$. Let
		\begin{equation*}
			f: (\bR^e)^{\times |I|} \times \cP_2(\bR^e) \to \bR^d
			\quad\mbox{such that}\quad
			f\in C_b\big[ A^{g,-}[I] \big]\big( (\bR^e)^{\times |I|} \times \cP_2(\bR^e); \bR^d \big)
		\end{equation*}
		For every $\iota \in I$, let $x_\iota, y_\iota \in \bR^e$, let $x_1, y_1, ..., x_{m[a]}, y_{m[a]} \in \bR^e$ and let $\mu,\nu\in \cP_{n+1}(\bR^e)$ with joint distribution $\Pi^{\mu, \nu}$. 
		
		Let $a\in A^{g, -}[I]$ and denote $\scG^a = \scG^{I \cup \{a^{-1}[j]: j=1, ..., m[a]\}}$. Suppose for every $\overline{a}\in A[a]$, we have that $\scG^a(\overline{a}):= \scG\big[ \llbracket a \cdot \overline{a} \rrbracket_{I} \big]$. 
		
		Then we have that
		\begin{align*}
			\partial_a f\Big( y_I, \nu, y_{m\{a\}}\Big)
			=& 
			\sum_{\overline{a}\in A[a]}^{g, -} \frac{ \rD^{\overline{a}} \partial_a f(x_I, \mu, \boldsymbol{x}_{m\{a\}})}{|\overline{a}|!} \big[ y_I - x_I, \Pi^{\mu, \nu}, (\boldsymbol{y} - \boldsymbol{x})_{m\{a\}} \big]
			\\
			&+
			R_{g, -}\Big( \partial_a f\Big)\Big[ (x_I, y_I), \Pi^{\mu, \nu}, (\boldsymbol{x}_{m\{a\}}, \boldsymbol{y}_{m\{a\}}) \Big]. 
		\end{align*}
	\end{corollary}
	For full details of this remainder term along with the proof, we refer the reader to \cite{salkeld2022Lions}. 
	
	\subsection{Couplings between Lions derivatives}
	\label{subsec:Couplings-LionsDeriv}
	
	Returning to the classical theory for a moment, recall that for $x, y\in \bR^e$ and a function $f:\bR^e \to \bR^d$ that is $n$-times differentiable, we can uniquely express the value of $f(y)$ in terms of $f(x)$, $\nabla^i f(x)$ for each $i \in \{1, ..., n\}$ and the remainder term:
	\begin{equation}
		\label{eq:TaylorExpansion1}
		\begin{aligned}
			f(y) =& \sum_{i=0}^n \frac{\nabla^i f(x)}{i!} \cdot (y-x)^{\otimes i} + R_n^{x, y}(f),
			\\
			\nabla^j f(y) =& \sum_{i =0}^{n-j} \frac{\nabla^{i+j} f(x)}{i!} \cdot (y-x)^{\otimes i} + R_{n-j}^{x, y}(\nabla^j f)
		\end{aligned}
	\end{equation}
	The central precept of the theory of regularity structures \cite{hairer2014theory} is that the values of the function $f$ can be expressed in terms of a jet
	\begin{equation*}
		\begin{aligned}
			\Big( f(x), \nabla f(x), ..., \nabla^n f(x) \Big)& 
			\\
			\Big( R_n^{x, y}(f), R_{n-1}^{x, y}(\nabla f), ..., R_0^{x, y}(\nabla^n f) \Big)&
		\end{aligned}
		\in \bigoplus_{i=0}^n \lin\big( (\bR^e)^{\otimes i}, \bR^d \big) = T^n\big( \bR^e, \bR^d \big) 
	\end{equation*}
	where $T^n(\bR^e, \bR^d)$ is the quotient algebra of the tensor algebra $\big( T(\bR^e, \bR^d), \otimes \big)$ 
	\begin{equation*}
		T(\bR^e, \bR^d) = \bigoplus_{i=0}^{\infty} \lin\big( (\bR^e)^{\otimes i}, \bR^d \big) 
		\quad\mbox{by the algebra ideal}\quad 
		I^n(\bR^e, \bR^d) = \bigoplus_{i=n+1}^{\infty} \lin\big( (\bR^e)^{\otimes i}, \bR^d \big). 
	\end{equation*}
	We use the notational convention that
	\begin{equation*}
		\lin\big( (\bR^e)^{\otimes 0}, \bR^d \big) = \bR^d. 
	\end{equation*}
	
	Further, if we Taylor expand each of these functions (again), we get the lattice of vectors
	\begin{equation}
		\label{eq:diagonal-structure}
		\begin{pmatrix}
			f(x), & \nabla f(x), &..., &\nabla^{n-1} f(x), &\nabla^{n}f(x) \\
			\nabla f(x), & \nabla^2 f(x), &..., &\nabla^n f(x), & 0 \\
			\vdots, & \vdots, & \reflectbox{$\ddots$} & & \\
			\nabla^{n-1}f(x), & \nabla^n f(x) &&& \\
			\nabla^n f(x), &0,&...,&&0
		\end{pmatrix}
		\in \bigoplus_{i=0}^n \bigoplus_{j=0}^{n} \lin\big( (\bR^e)^{\otimes (i+j)},\bR^d \big). 
	\end{equation}
	Next, ignoring the truncation for the moment we observe that this tensor module can equivalently be written as 
	\begin{align}
		\nonumber
		T\Big(\bR^e, \bR^d\Big) \otimes T\Big(\bR^e, \bR^d\Big)
		=& 
		\bigoplus_{i=0}^\infty \bigoplus_{j=0}^{\infty} \lin\big( (\bR^e)^{\otimes (i+j)}, \bR^d \big) 
		=
		\bigoplus_{i=0}^\infty \bigoplus_{j=0}^{\infty} \lin\Big( (\bR^e)^{\otimes i}, \lin\big( (\bR^e)^{\otimes j}, \bR^d \big) \Big)
		\\
		\label{eq:TensorProd-construction}
		=& \bigoplus_{i=0}^{\infty} \lin\Big( (\bR^e)^{\otimes i}, \bigoplus_{j=0}^{\infty} \lin\big( (\bR^e)^{\otimes j}, \bR^d \big) \Big)
		= 
		T\Big( \bR^e, T\big( \bR^e, \bR^d \big) \Big)
	\end{align}
	By additionally including the truncation, 
	\begin{align}
		\nonumber
		&T^n\Big(\bR^e, \bR^d\Big) \otimes T^n\Big(\bR^e, \bR^d\Big)
		=
		\bigg( \bigoplus_{i=0}^n \lin\big( (\bR^e)^{\otimes i}, \bR^d \big) \bigg) \otimes \bigg(\bigoplus_{j=0}^n \lin\big( (\bR^e) ^{\otimes j}, \bR^d \big) \bigg)
		\\
		\label{eq:TensorProd-construction2}
		&=\bigoplus_{i=0}^n \bigoplus_{j=0}^{n} \lin\big( (\bR^e)^{\otimes (i+j)}, \bR^d \big)= \bigoplus_{i=0}^{n} \lin\Big( (\bR^e)^{\otimes i}, T^{n}\big( \bR^e, \bR^d\big) \Big) 
		= 
		T^n\Big( \bR^e, T^n\big( \bR^e, \bR^d \big) \Big). 
	\end{align}
	
	\begin{remark}
		The takeaway from Equations \eqref{eq:TensorProd-construction} and \eqref{eq:TensorProd-construction2} is that tensor products of tensor algebras with themselves can be defined in terms of the ``tensor algebra of the tensor algebra''. This is the perspective that we will use when defining the ``coupled tensor product'' below. 
	\end{remark}
	
	In order to reflect our Taylor expansion of the function $f$ in terms of the tensor algebra, we define the function $\bF: \bR^e \to T^n(\bR^e, \bR^d)$ and $\fR: \bR^e \times \bR^e \to T^n(\bR^e, \bR^d)$ by
	\begin{align*}
		\bF(x) &= \Big( f(x), \nabla f(x), ..., \nabla^n f(x) \Big)
		\quad\mbox{so that for $j=0, ..., n$,}\quad
		\Big\langle \bF(x), j \Big\rangle = \nabla^j f(x),
		\\
		\fR^{x,y} &= \Big( R_n^{x, y}(f), R_{n-1}^{x, y}(\nabla f), ..., R_{0}^{x, y}(\nabla^n f) \Big)
		\quad\mbox{so that for $j=0, ..., n$,}\quad
		\Big\langle \fR^{x, y}, j \Big\rangle = R_{n-j}^{x, y}(\nabla^j f). 
	\end{align*}
	Next define the coproduct operation $\triangle: T^n(\bR^e, \bR^d) \to T^n(\bR^e, \bR^d) \otimes T^n(\bR^e, \bR^d)$ for
	\begin{equation*}
		X\in T^n(\bR^e, \bR^d)
		\quad \mbox{by}\quad
		\Big\langle \triangle\big[ X\big], (i_1, i_2)\Big\rangle = \Big\langle X, i_1\Big\rangle \otimes \Big\langle X, i_2 \Big\rangle
	\end{equation*}
	Then Equation \eqref{eq:diagonal-structure} can be interpreted as $\triangle\big[ \bF(x) \big]$. Secondly, the dual space of the finite dimensional space
	\begin{equation*}
		T^n\Big( \bR^e, \bR^d \Big)^* = \bigoplus_{i=0}^n (\bR^e)^{\otimes i}
	\end{equation*}
	since we are treating $T^n(\bR^e, \bR^d)$ as a module over the ring $\bR^d$ endowed with coordinatewise product for a ring multiplication. For $x, y \in \bR^e$ we define $\rPi: \bR^e \times \bR^e \to T(\bR^e, \bR^d)^*$ by
	\begin{equation*}
		\rPi^{x, y} = \sum_{i=0}^n \frac{(y-x)^{\otimes i}}{i!} \in \bigoplus_{i=0}^n (\bR^e)^{\otimes i}. 
	\end{equation*}
	In order to ensure that $\rPi^{x, y}$ remains on the truncated tensor algebra, we also require that
	\begin{equation*}
		\rPi^{x, y}\otimes j = \sum_{i=0}^{n-j} \frac{(y-x)^{\otimes i}}{i!} \in \bigoplus_{i=0}^{n-j} (\bR^e)^{\otimes i}. 
	\end{equation*}
	Then Equation \eqref{eq:TaylorExpansion1} can be equivalently rewritten as
	\begin{equation}
		\label{eq:TaylorExpansion2}
		\Big\langle \bF(y), j \Big\rangle = \Big\langle \triangle \big[ \bF(x) \big], \rPi^{x, y} \otimes j \Big\rangle + \Big\langle \fR^{x, y}, j \Big\rangle. 
	\end{equation}
	
	\subsubsection*{Lions-Taylor expansion}
	Firstly, to keep things simple let us consider when the index set $I = \emptyset$ for the moment. Recalling Theorem \ref{theorem:LionsTaylor2}, for any two measures $\mu, \nu \in \cP_2(\bR^e)$, we can express an $A^{g,-}$-differentiable function in terms of the collection of functions
	\begin{equation*}
		\Big( \partial_a f\big( \mu, (\cdot)_{m\{a\}} \big) \Big)_{a\in A^{g,-}} 
		\in 
		\bigoplus_{a \in A}^{g,-} L^0\Big( (\bR^e)^{\times m[a]}; \lin\big( (\bR^e)^{\otimes |a|}, \bR^d \big) \Big). 
	\end{equation*}
	where $L^0\Big( (\bR^e)^{\times m[a]}; \lin\big( (\bR^e)^{\otimes |a|}, \bR^d \big) \Big)$ is the module of measurable functions from $(\bR^e)^{\times m[a]}$ to the module $\lin\big((\bR^e)^{\times m[a]}, \bR^d\big)$ over the ring $(\bR^d, +, \centerdot)$
	\begin{equation*}
		\Big( (\bR^e)^{\times m[a]}, \mbox{Leb}\big( (\bR^e)^{\times m[a]} \big) \Big)
		\mapsto
		\bigg( \lin\big( (\bR^e)^{\otimes |a|}, \bR^d \big), \cB\Big( \lin\big( (\bR^e)^{\otimes |a|}, \bR^d \big)\Big) \bigg)
	\end{equation*}
	and we use the convention that
	\begin{equation*}
		L^0\Big( (\bR^e)^{\times 0}; \lin\big( (\bR^e)^{\otimes |a|}, \bR^d \big) \Big) = \lin\big( (\bR^e)^{\otimes |a|}, \bR^d \big). 
	\end{equation*}
	By Taylor expanding each of these functions (as above), we get a lattice of measurable functions
	\begin{equation}
		\label{eq:LionsDerivMatrix}
		\begin{pmatrix}
			f, 
			&
			\partial_{(1)}f(x_1), 
			&
			\begin{pmatrix} \partial_{(11)}f(x_1)\\ \partial_{(12)} f(x_{\{1,2\}}) \end{pmatrix}
			&
			\dots
			\\
			\partial_{(1)}f(x_1), 
			& 
			\begin{pmatrix} \partial_{(11)}f(x_1)\\ \partial_{(12)} f(x_{\{1,2\}}) \end{pmatrix} 
			&
			\begin{pmatrix}
			\begin{pmatrix} \partial_{(111)}f(x_1) \\ \partial_{(112)}f(x_{\{1,2\}}) \end{pmatrix}
			\\
			\begin{pmatrix} \partial_{(121)}f(x_{\{1,2\}}) \\ \partial_{(122)}f(x_{\{1,2\}}) \\ \partial_{(123)} f(x_{\{1,2,3\}})\end{pmatrix}
			\end{pmatrix}
			&
			\dots
			\\
			\begin{pmatrix} \partial_{(11)}f(x_1) \\ \partial_{(12)}f(x_{1,2})\end{pmatrix}, 
			& 
			\begin{pmatrix} 
			\begin{pmatrix}
			\partial_{(111)}f(x_1) \\ \partial_{(112)}f(x_{\{1,2\}})
			\end{pmatrix}
			\\
			\begin{pmatrix}
			\partial_{(121)} f(x_{\{1,2\}}) \\ \partial_{(122)} f(x_{\{1,2\}}) \\ \partial_{(123)} f(x_{\{1,2,3\}})
			\end{pmatrix}
			\end{pmatrix} 
			&
			\ddots
			&
			\\
			\vdots  & \vdots  & &\ddots
		\end{pmatrix}. 
	\end{equation}
	The challenge with this lattice is that when we Lions-Taylor expand one term from this series, all the free variables that we differentiate with respect to exist on a single probability space. This is in contrast to two collections of free variables derived from distinct terms from the Lions-Taylor series which exist on distinct probability spaces. Thus, we represent this lattice of measurable functions as
	\begin{equation*}
		\bigoplus_{a\in A}^{g,-} L^0\Bigg( (\bR^e)^{\times m[a]}; \bigoplus_{\overline{a} \in A[a]}^{g, -} L^0\bigg( (\bR^e)^{\times m[\overline{a}]} ; \lin\Big( (\bR^e)^{\otimes |\overline{a}|}, \lin\big( (\bR^e)^{\otimes |a|}, \bR^d \big) \Big) \bigg) \Bigg). 
	\end{equation*}
	
	With Equation \eqref{eq:TensorProd-construction2} in mind, our running hypothesis will be as follows:
	\begin{hypothesis}
		\label{hypothesis:coupledTP}
		We define the \underline{coupled tensor product} to be the operation such that 
		\begin{align*}
			&\bigg( \bigoplus_{a\in A[I]}^{g,-} L^0\Big( \Omega^{\times m[a]}, \bP^{\times m[a]}; \lin\big( (\bR^e)^{\otimes |a|}, \bR^d \big) \Big) \bigg) 
			\tilde{\otimes} 
			\bigg( \bigoplus_{a\in A[I]}^{g,-} L^0\Big( \Omega^{\times m[a]}, \bP^{\times m[a]}; \lin\big( (\bR^e)^{\otimes |a|}, \bR^d \big) \Big) \bigg)
			\\
			&=\bigoplus_{a\in A[I]}^{g,-} L^0\Bigg( \Omega^{\times m[a]}, \bP^{\times m[a]}; \bigoplus_{\overline{a} \in A[a]}^{g, - } L^0\bigg( \Omega^{\times m[\overline{a}]}, \bP^{\times m[\overline{a}]}; \lin\Big( (\bR^e)^{\otimes |\overline{a}|}, \lin\big( (\bR^e)^{\otimes |a|}, \bR^d \big) \Big) \bigg) \Bigg). 
		\end{align*}
	\end{hypothesis}
	
	\begin{remark}
		\label{remark:couplings1}
		Through Equation \eqref{eq:diagonal-structure}, we saw that the basis of the tensor product of the module $T^n(\bR^e, \bR^d)$ with itself is isomorphic to the span of the cartesian product of the basis of $T^n(\bR^e, \bR^d)$ with itself. Indeed, this is clear for any module with an explicit basis. 
		
		By contrast, in \ref{eq:LionsDerivMatrix} we see that the coupled tensor product corresponds to a span of the set of elements
		\begin{equation}
			A^n\tilde{\times} A^n[I]:= \bigsqcup_{a \in A^n[I]} A^n[a]. 
		\end{equation}
		This is a richer basis set than the cartesian product, which is natural because we would expect the Lions-Taylor expansion of the iterated Lions derivative of some smooth measure valued function to have more variables and therefore a more involved Lions-Taylor expansion. 
	\end{remark}
		
	Next, we highlight the link between coupled partition sequences and iterative partition sequences. 
	\begin{proposition}
		\label{proposition:coupling(1)}
		Let $I$ be an index set. Then for any $j, k \in \bN$, there is an isomorphism between the sets
		\begin{equation*}
			A_{j+k}\big[ I \big] \quad \mbox{and}\quad
			A_j \tilde{\times} A_k[I]. 
		\end{equation*}
	\end{proposition}
	This relates the link between an iterative Lions derivative of order $j+k$ (with $|I|$ tagged variables) and all the ways in which one can take an iterative Lions derivative $a$ of length $j$ (with $|I|$ tagged variables) followed by an iterative Lions derivative of length $k$ with $|I| + m[a]$ tagged variables.
	\begin{proof}
		Let $j, k \in\bN$ and let $I$ be an index set. 
		
		\textit{Step 1.} Suppose that $a\in A_{j+k}[I]$. Then we can write $a$ as the concatenation of two sequences
		\begin{equation*}
			a = \big( (a_1, ..., a_{j}), (a_{j+1}, ..., a_{j+k}) \big). 
		\end{equation*}
		By Definition \ref{def:a}, we get that $a' = (a_1, ..., a_j) \in A_j[I]$. The values of the sequence $(a_{j+1}, ..., a_{j+k})$ are contained in the set
		\begin{align*}
			a_{j+1} \in& I \cup \Big\{ 1, ..., \big( \max_{l=1, ..., j} a_l \big) \Big\} \cup \Big\{ \big( \max_{l=1, ..., j} a_l \big) + 1 \Big\}
			\\
			a_{j+i} \in& I \cup \Big\{ 1, ..., \big( \max_{l=1, ..., j} a_l \big) \Big\} \cup \Big\{ \big( \max_{l=1, ..., j} a_l \big) + 1, ..., \big( \max_{l=1, ..., j} a_l \big)\vee\big( \max_{l=j, ..., j+i-1} a_l \big) + 1 \Big\}. 
		\end{align*}
		As such, we can rewrite this sequence as
		\begin{equation*}
			\overline{a}_{i} = 
			\begin{cases}
				(a')^{-1}[p] \quad&\quad \mbox{if $a_{j+i} = p$ for any $p \in I \cup \{1, ..., m[a']\}$, }
				\\
				a_{j+i}-m[a'] \quad&\quad \mbox{if $a_{j+i}>m[a']$.}
			\end{cases}
		\end{equation*}
		The sequence $(a_{j+1}, ..., a_{j+k})$ is contained in the equivalence class $\llbracket \overline{a} \rrbracket_{I\cup \{(a')^{-1}[p] \}}$ so that 
		\begin{equation*}
			\overline{a} =(\overline{a}_i)_{i = 1, ..., k} \in A_k[a'].
		\end{equation*}
		Thus we identify the sequence $a$ with the pair
		\begin{equation*}
			\big( a', \overline{a} \big) \in A_j[I] \times A_k[a']. 
		\end{equation*}
		
		\textit{Step 2.} Any element 
		\begin{equation*}
			(a', \overline{a}) \in \bigsqcup_{a\in A_{j}[I]} A_{k}[a]
		\end{equation*}
		is a pair of partition sequences, with $a' \in A_{j}[I]$ and $\overline{a} \in A_{k}[a']$. We define the sequence $a = (a_i)_{i=1, ..., j+k}$ by
		\begin{equation*}
			a_i = 
			\begin{cases}
				a'_i \quad&\quad \mbox{if $i \in \{1, ..., j\}$}
				\\
				p \quad &\quad \mbox{if $i\in \{j+1, ..., j+k\}$ and $\overline{a}_{i-j} = (a')^{-1}[p]$}
				\\
				\overline{a}_{i-j}+m[a'] \quad&\quad \mbox{if $i\in \{j+1, ..., j+k\}$ and $\overline{a}_{i-j} \in \{1, ..., m[\overline{a}]\}$}
			\end{cases}
		\end{equation*}
		and we verify that the sequence $a = (a_i)_{i=1, ..., j+k} \in A_{j+k}[I]$. However, 
		\begin{align*}
			a_1 =& a_1' \in I \cup \Big\{ 1 \Big\}, 
			\\
			\mbox{For $i\in \{2, ..., j\}$}\quad a_i =& a_i' \in I \cup \Big\{1, ..., \big( \max_{l=1, ..., i-1} a_l \big) + 1 \Big\} 
		\end{align*}
		and for $i \in \{j+1, ..., j+k\}$
		\begin{align*}
			a_{j+1} =& 
			\begin{cases}
				p \quad&\quad \in \{1, ..., m[a'] \}
				\\
				m[a'] + 1 \quad&\quad \in \{m[a'] + 1\}
			\end{cases}
			\quad \mbox{so that}\quad
			a_{j+1} \in I \cup \Big\{ 1, ..., \big( \max_{l=1, ..., j} a_l \big) + 1 \Big\}, 
			\\
			a_{i} =& 
			\begin{cases}
				p \quad&\quad \in \{1, ..., m[a'] \}
				\\
				m[a'] + \overline{a}_{i-j} \quad&\quad \in \Big\{ m[a'] + (\underset{l=1, ..., i-j}{\max} \overline{a}_l) \Big\}
			\end{cases}
			\\
			\mbox{so that}\quad
			a_{i} \in& I \cup \Big\{ 1, ..., \big( \underset{l=1, ..., i-1}{\max} a_l \big) + 1 \Big\}
		\end{align*}
		and we conclude that $a$ satisfies Definition \ref{def:a}. 
	\end{proof}
	
	\subsubsection*{Abstraction}
	
	We have seen that the application of a Lions derivative yields a function of a free variable and the application of iterated Lions derivatives identified with a partition yields a collection of free variables that are identified with the elements of that partition. When considering the Lions-Taylor expansions of some ``smooth measure functional'' $f$, a natural aspiration would be to embed the Lions derivative $\partial_a f$ into a vector space of appropriate functions but the choice of such a space should be specific to the setting. 
	
	To explain this a little, let $\cP_2(\bR)$ be the set of all measures on $(\bR, \cB)$ that have finite second moments. Suppose that the function $f:\cP_2(\bR) \to \bR$ has a Lions derivative at $\nu \in \cP_2(\bR)$ so that $\partial_\mu f(\nu):\bR \to \bR$. The only thing that we know about $\partial_\mu f(\nu)$ is that it is a measurable function so that
	$$
	\partial_\mu f(\nu) \in L^0\Big( \bR, \cB, \nu; \bR \Big). 
	$$
	Since the ring multiplication over the field $\bR$ is continuous with respect to the standard topology on $\bR$, the space $L^0(\Omega, \cB, \bP; \bR)$ endowed with the topology of weak convergence forms a topological algebra. 
	
	We know that the Fr\'echet derivative of the canonical lift can be seen as square integrable random variable so that
	\begin{equation*}
		\int_{\bR} \Big| \partial_\mu f(\nu) \big( x\big) \Big|^2 d\nu(x)< \infty,
		\quad \mbox{or equivalently}\quad
		\partial_\mu f(\nu) \in L^2\Big( \bR, \cB, \nu; \bR\Big). 
	\end{equation*}
	If we suppose (as is natural when considering Taylor expansions) that the function $\partial_\mu f(\nu):\bR \to \bR$ is bounded then
	$$
	\partial_\mu f(\nu) \in L^\infty\Big( \bR, \cB, \nu; \bR\Big). 
	$$
	If we are interested in considering the second order Lions derivatives of $f$, then we know that we need the function $\partial_\mu f:\cP_2(\bR) \times \bR \to \bR$ to be differentiable in both the measure and free variable. Hence, it would make sense to consider the function $\partial_\mu f(\nu):\bR \to \bR$ as an element of the space of bounded and differentiable function
	$$
	\partial_\mu f(\nu) \in C_b^1\Big( \bR; \bR\Big).	
	$$
	Hence, there is no correct ``collection'' of modules to choose to work with. However, we will require that any modules we work with satisfy certain nesting properties.  
	
	For any $n \in \bN$, we denote
	\begin{equation*}
		L^{\boldsymbol{0}}\Big( \Omega^{\times n}; \cR \Big)= L^{(0, ..., 0)}\Big( \Omega^{\times n}; \cR \Big) 
		\quad\mbox{where}\quad
		\boldsymbol{0} = (\underbrace{0, ..., 0}_{\times n}). 
	\end{equation*}	
	
	With this in mind, we introduce the following notation: 
	\begin{assumption}
		\label{notation:U-(1)}
		Let $(\Omega, \cF)$ be a measurable space and let $(\cR, +, \centerdot)$ be a separable topological ring. We denote $\cB(\cR)$ to be the Borel $\sigma$-algebra, giving us that $(\cR, \cB(\cR))$ is also a measurable space. Let $I$ be an index set. 
		
		We write
		\begin{equation*}
			\bU\big( \Omega; \cR \big) \subseteq L^{\boldsymbol{0}}\Big( \Omega^{\times |I|}; \cR \Big)
		\end{equation*}
		to refer to a module of measurable functions $f:\Omega \to \cR$ with binary operation $\centerdot: \bU(\Omega; \cR) \times \bU(\Omega; \cR) \to \bU(\Omega; \cR)$ that satisfies 
		\begin{equation*}
			(f \centerdot g):\Omega^{\times |I|} \to \cR,
			\quad 
			\big( f\centerdot g \big)(\omega_I) = f(\omega_I) \centerdot g(\omega_I). 
		\end{equation*}
	\end{assumption}
	
	\begin{assumption}
		\label{notation:V-(1)}
		Let $(\Omega, \cF, \bP)$ be a probability space. Let $d \in \bN$ and let $(\cR, +, \centerdot)$ be a separable normed ring which we associate with the Borel $\sigma$-algebra $\cB(\cR)$. 
		
		For any index set $I$, let
		\begin{equation*}
			\Big( \bV^a\big( \Omega; \cR \big) \Big)_{a\in A[I]}, 
		\end{equation*}
		to be a collection of unital $\cR$-modules that satisfy:
		\begin{enumerate}[label=(\ref*{notation:V-(1)}.\roman*)]
			\item 
			\label{enum:notation:V-(1)-1}
			For every $a \in A[I]$, we have 
			\begin{equation*}
				\bV^a\big( \Omega; \cR \big) \subseteq L^{\boldsymbol{0}}\Big( \Omega^{\times m[a]}; \cR \Big); 
			\end{equation*}
			\item 
			\label{enum:notation:V-(1)-2}
			For any $a\in A[I]$ such that $m[a] = 0$, we have that
			\begin{equation*}
				\bV^{a}\big( \Omega; \cR \big) = \cR; 
			\end{equation*}
			\item 
			\label{enum:notation:V-(1)-3}
			For any $a_1, a_2 \in A[I]$, we have that
			\begin{equation*}
				\bV^{a_1}\big( \Omega; \cR \big) \otimes \bV^{a_2}\big( \Omega; \cR \big) \subseteq \bV^{\llbracket (a_1 \cdot a_2) \rrbracket_I }\big(\Omega; \cR \big);
			\end{equation*}
			\item 
			\label{enum:notation:V-(1)-4}
			For every $a\in A[I]$, and for any
			\begin{equation*}
				\hat{a} \in A[I] \quad \mbox{and} \quad \bar{a} \in A[\hat{a}] \quad \mbox{s.t.} \quad a = \big\llbracket (\hat{a} \cdot \bar{a}) \big\rrbracket_I
			\end{equation*}
			we have that
			\begin{equation*}
				\bV^{a}\Big( \Omega; \cR \Big) \subseteq \bV^{\hat{a}}\Big( \Omega; \bV^{\bar{a}}\big( \Omega; \cR \big) \Big). 
			\end{equation*}
		\end{enumerate}
	\end{assumption}
	
	The perspective should be that $\bU(\Omega; \cR)$ may be the collection of measurable functions paired with coordinate multiplication but we are not restricting ourselves to this. 
	
	Similarly, the collection $\big( \bV^a(\Omega, \cR) \big)_{a\in A[I]}$ should be thought of as the collection of modules in which each element of our Taylor expansions exist within. 
	\begin{example}
		Suppose that the function $f: \cP_2(\bR^e) \to \bR^d$ satisfies that for all $\mu \in \cP_2(\bR^e)$, 
		\begin{equation*}
			\forall a \in A, \quad \partial_a f \in C_b^{\infty} \Big( (\bR^e)^{\times m[a]}; \lin\big( (\bR^e)^{\otimes |a|},\bR^d \big) \Big). 
		\end{equation*}
		Then we choose
		\begin{equation*}
			\bV^a(\bR^e, \bR^d) = C_b^{\infty} \Big( (\bR^e)^{\times m[a]}; \lin\big( (\bR^e)^{\otimes |a|},\bR^d \big) \Big)
		\end{equation*}
		and obtain
		\begin{equation*}
			\big( \partial_a f(\mu) \big)_{a\in A} \in \bigoplus_{a\in A} \bV^a\big( \bR^e, \bR^d \big). 
		\end{equation*}
		
		On the other hand, suppose that the function $f:\bR^e \times \cP_2(\bR^e) \to \bR^d$ satisfies that for all $\mu \in \cP_2(\bR^e)$, 
		\begin{align*}
			\mathbb{U}_0(\bR^e, \bR^d) =& C_b^{\infty}\big( \bR^e; \bR^d \big), 
			\\
			\Big( \partial_a f(\cdot, \mu) \Big)_{a\in A[0]} \in& C_b^{\infty}\bigg( \bR^e; \bigoplus_{a\in A[0]} C_b^{\infty}\Big( (\bR^e)^{\times m[a]}, \lin\big( (\bR^e)^{\otimes |a|}, \bR^d \big) \Big) \bigg)
			\\
			=& \mathbb{U}_0\bigg( \bR^e; \bigoplus_{a\in A[0]} \bV^a\Big( \bR^e, \bR^d \Big) \bigg). 
		\end{align*}
		First of all, note that smooth functions are measurable so that \ref{enum:notation:V-(1)-1} is satisfied. Similarly, we use the notational convention that 
		\begin{equation*}
			C_b^{\infty}\Big( (\bR^e)^{\otimes 0}; \bR^d \Big) = \bR^d
		\end{equation*}
		so that \ref{enum:notation:V-(1)-2} is satisfied. Critically, it is not the case that
		\begin{equation*}
			C_b^{\infty}\Big( \bR; \bR \Big) \otimes C_b^{\infty}\Big( \bR; \bR \Big) \neq C_b^{\infty}\Big( \bR \otimes \bR; \bR \Big)
		\end{equation*}
		since there are many smooth functions of two variables that cannot be written in product form. Rather, we do have that
		\begin{equation*}
			C_b^{\infty}\Big( \bR; \bR \Big) \otimes C_b^{\infty}\Big( \bR; \bR \Big) \subseteq C_b^{\infty}\Big( \bR \otimes \bR; \bR \Big)
		\end{equation*}
		which means \ref{enum:notation:V-(1)-3} is satisfied. In the same fashion, 
		\begin{equation*}
			C_b^{\infty}\Big( \bR \otimes \bR; \bR \Big) \subseteq C_b^{\infty}\Big( \bR; C_b^{\infty}\big( \bR ; \bR \big) \Big)
		\end{equation*}
		so that \ref{enum:notation:V-(1)-4} is satisfied. 
	\end{example}
	We use Assumption \ref{notation:V-(1)} to construct a module (analogous to a tensor module) for which we can define a coupled tensor product of:
	\begin{definition}
		For any index set $I$, let $\big( \bV^a(\bR^e; \bR^d) \big)_{a\in A[I]}$ be a collection of $\bR^d$-modules that satisfy Assumption \ref{notation:V-(1)}. We define
		\begin{equation*}
			\cT_I^{g, -}(\bR^e, \bR^d) := \bigoplus_{a\in A[I]}^{g, -} \bV^a\big( \bR^e, \bR^d \big). 
		\end{equation*}
		Further, we define the \emph{coupled tensor product}
		\begin{align*}
			\cT_I^{g, -}(\bR^e, \bR^d) \tilde{\otimes} \cT_I^{g, -}(\bR^e, \bR^d) 
			:&=\bigoplus_{a\in A[I]}^{g, -} \bV^a\Big( \bR^e, \bigoplus_{\overline{a}\in A[a]}^{g, -} \bV^{\overline{a}} \big( \bR^e, \bR^d \big) \Big). 
		\end{align*}
	\end{definition}
	Then $\cT_I^{g, -}(\bR^e, \bR^d)$ is a $\bR^d$-module so that for $X\in \cT_0^{g, -}(\bR^e, \bR^d)$, we write
	\begin{equation*}
		X = \sum_{a\in A[I]}^{g, -} \big\langle X, a\big\rangle
        		\quad\mbox{where}\quad
        		\big\langle X, a\big\rangle \in \bV^a\big( \bR^e, \bR^d \big)
        	\end{equation*}
	while for $X \in \cT_I^{g, -}(\bR^e, \bR^d) \tilde{\otimes} \cT_I^{g, -}(\bR^e, \bR^d)$ we write
	\begin{equation*}
		X = \sum_{a\in A[I]}^{g,-} \sum_{\overline{a} \in A[a]}^{g, -} \Big\langle X, (\overline{a}, a) \Big\rangle
        		\quad\mbox{where}\quad
        		\Big\langle X, (\overline{a},a) \Big\rangle \in \bV^a\Big( \bR^e, \bV^{\overline{a}}\big( \bR^e, \bR^d \big) \Big)
	\end{equation*}

	\begin{definition}
		For any index set $I$, let $\big( \bV^a(\bR^e; \bR^d) \big)_{a\in A[I]}$ be a collection of $\bR^d$-modules that satisfy Assumption \ref{notation:V-(1)}. We define $\Delta: \cT_I^{g,-}(\bR^e, \bR^d) \to \cT_I^{g,-}(\bR^e, \bR^d) \tilde{\otimes} \cT_I^{g,-}(\bR^e, \bR^d)$ by
		\begin{equation*}
			\Big\langle \Delta\big[ X \big], (\overline{a}, a) \Big\rangle = \Big\langle X, \big\llbracket (a \cdot \overline{a}) \big\rrbracket_{I} \Big\rangle. 
		\end{equation*}
		Note, $\Delta$ is well defined due to \ref{enum:notation:V-(1)-4}. 
	\end{definition}
	The operation $\Delta$ is the coupled coproduct that serves the same purpose as the coproduct in Equation \eqref{eq:TaylorExpansion2}. We conclude this section by developing a reformulation of Theorem \ref{theorem:LionsTaylor2} and Corollary \ref{corollary:LionsTaylor2}:
	\begin{theorem}
		\label{Theorem:LionsTaylor2*}
		Let $I$ be an index set, let $f: (\bR^e)^{\times |I|}  \times \cP_2(\bR^e) \to \bR^d$ and suppose that $f$ has Lions and spatial derivatives of all orders. Let $g: (\bN_0)^{\times |I|} \times \bN_0 \to \bR$ be a monotone increasing function such that $g(0_I, 0) \leq 0$ and let $n\in \bN$ be the greatest integer such that $g(0_I, n) \leq 0$. 
		
		Let $\bU(\bR^e; \bR^d)$ satisfy Assumption \ref{notation:U-(1)} and let $\big( \bV^a(\bR^e; \bR^d) \big)_{a\in A[I]}$ be a collection of $\bR^d$-modules that satisfy Assumption \ref{notation:V-(1)}. Suppose that
		\begin{align*}
			&\forall x_I, y_I \in (\bR^e)^{\times |I|} 
			\quad
			\Big( \partial_a f(x_I, \mu)(\cdot) \Big)_{a\in A^{g, -}[I]}, \Big( R_{g,-}^{(x_I, y_I), \Pi^{\mu, \nu}} \big( \partial_a f\big) \Big)_{a \in A^{g, -}[I]} \in \bigoplus_{a\in A[I]}^{g, -} \bV^a\big( \bR^e; \bR^d \big)
			\\
			&\mbox{and}
			\quad
			\Big( \partial_a f(\cdot, \mu)(\cdot) \Big)_{a\in A^{g, -}[I]}, \Big( R_{g,-}^{(\cdot, \cdot), \Pi^{\mu, \nu}} \big( \partial_a f\big) \Big)_{a \in A^{g, -}[I]} \in \bU\bigg( \bR^e; \bigoplus_{a\in A[I]}^{g, -} \bV^a\big( \bR^e; \bR^d \big) \bigg)
		\end{align*}
		Then we define
		\begin{align*}
			\bF:& (\bR^e)^{\times |I|} \times \cP_2(\bR^e) \to \cT_I^{g, -}(\bR^e, \bR^d)
			\\
			\fR:& \big( (\bR^e)^{\times |I|} \times (\bR^e)^{\times |I|} \big) \times \cP_2\big( \bR^e \times \bR^e \big) \to \cT_I^{g, -}(\bR^e, \bR^d)
		\end{align*}
		by
		\begin{align*}
			\bF(x_I, \mu) = \Big( \partial_a f(x_I, \mu)(\cdot) \Big)_{a\in A^{g, -}[I]}
			\quad \mbox{and}\quad
			\fR^{(x_I, y_I), \Pi^{\mu, \nu}} = \Big( R_{g,-}^{(x_I, y_I), \Pi^{\mu, \nu}} \big( \partial_a f\big) \Big)_{a \in A^{g, -}[I]}. 
		\end{align*}
		We define
		\begin{equation*}
			\rPi_I^{(y_0, x_0), \Pi^{\mu, \nu}} = \sum_{a\in A[I]}^{g, -} \frac{1}{|a|!} \bigotimes_{i=1}^{|a|} (y_{a_i} - x_{a_i}) d\big( \Pi^{\mu, \nu} \big)^{\times m[a]}\Big( (\boldsymbol{x}, \boldsymbol{y})_{m\{a\}} \Big). 
		\end{equation*}
		 Then
		 \begin{equation*}
		 	\Big\langle \bF(y_I, \nu), a \Big\rangle = \Big\langle \Delta\big[ \bF(x_I, \mu) \big], \rPi_a^{(x_I, y_I), \Pi^{\mu, \nu}} \otimes a \Big\rangle + \Big\langle \fR^{(x_I, y_I), \Pi^{\mu, \nu}}, a \Big\rangle
		 \end{equation*}
	\end{theorem}
	
	\begin{remark}[A regularity structure perspective]
		The key distinction between Theorem \ref{Theorem:LionsTaylor2*} and Equation \eqref{eq:TaylorExpansion2} is that $\rPi_I^{(y_I, x_I), \Pi^{\mu, \nu}}$ (which here serves as the \emph{model} in the language of regularity structures) is defined \emph{for any choice} of index set $I$. The index set identifies the number of tagged variables, which is natural because different choices of partition sequence $a$ correspond to the element $\langle \bF(y, \nu), a \rangle$ which take their values in very different modules. Again, in the language of regularity structures this is captured within the \emph{structure group} of linear operators that, in the regularity structure of classical Taylor expansions, induces a coproduct on the model space. For Lions-Taylor expansions, the structure group does not induce a coproduct, but rather a coupled coproduct. 
	\end{remark}
	
	\subsection{A motivation from probabilistic rough paths}
	
	To motivate the coupled tensor product that we consider in this section, we briefly describe a heuristic derivation of probabilistic rough paths and rough McKean-Vlasov equations. 
	
	\subsubsection*{Introduction}
	
	First, let us consider the rough differential equation driven by a weakly geometric rough path
	\begin{equation*}
		dX_t = f\big( X_t \big) \cdot dW_t
	\end{equation*}
	where $f: \bR^e \to \lin(\bR^e, \bR^d)$ is differentiable function and $W:[0,1] \to \bR^d$ is an $\alpha$-H\"older continuous path for $\alpha \in (\tfrac{1}{4}, \tfrac{1}{3})$. This is an extension of the work \cite{2019arXiv180205882.2B} which addresses the case $\alpha \in (\tfrac{1}{3}, \tfrac{1}{2})$. Under the hypothesis that the weakly geometric rough exists, the iterated integrals
	\begin{align*}
		\int_s^t W_{s, r} \otimes dW_r, 
		\quad&
		\int_s^t \int_s^r W_{s, u} \otimes dW_u \otimes dW_r
		\quad \mbox{are well defined and}
		\\
		\sup_{s, t\in [0,1]} \frac{\Big| \int_s^t W_{s, r} \otimes dW_r \Big|}{|t-s|^{2\alpha}}< \infty
		\quad&
		\sup_{s, t\in [0,1]} \frac{\Big| \int_s^t \int_s^r W_{s, u} \otimes dW_u \otimes dW_r \Big|}{|t-s|^{3\alpha}} < \infty. 
	\end{align*}
	Thus the increments of the solution can be expressed as
	\begin{align*}
		X_{s, t} =& f\big( X_s \big) \cdot W_{s, t} + \nabla_x f\big( X_s \big) \otimes f\big( X_s \big)\cdot \int_s^t W_{s, r} \otimes dW_r 
		\\
		+& \bigg( \tfrac{1}{2} \nabla_x^2 f\big( X_s \big) \otimes \Big( f\big( X_s \big) \otimes f \big( X_s \big) + f\big( X_s \big) \otimes f \big( X_s \big)^T \Big) + \nabla_x f\big( X_s \big) \otimes \nabla_x f \big( X_s \big) \otimes f\big( X_s \big) \bigg)
		\\
		& \cdot \int_s^t \int_s^r \int_s^u dW_{v} \otimes dW_u \otimes dW_r + O\big( |t-s|^{4 \alpha} \big). 
	\end{align*}
	Further, by Taylor expanding the functions $f$ and $\nabla_x f \otimes f$, we obtain that
	\begin{align*}
		f\big( X_{\cdot} \big)_{s, t} =& \nabla_x f\big( X_s \big) \otimes f \big( X_s \big) \cdot W_{s, t}
		\\
		+& 
		\bigg( \nabla_x f \big(X_s \big) \otimes \nabla_x f\big( X_s \big) \otimes f \big( X_s \big) + \tfrac{1}{2} \nabla_x^2 f \big( X_s \big) \otimes \Big( f\big( X_s \big) \otimes f\big( X_s \big) + f\big( X_s \big) \otimes f\big( X_s \big)^T \Big) \bigg)
		\\
		& \cdot \int_s^t \int_s^r dW_u \otimes dW_r 
		+ 
		O\Big( |t-s|^{3 \alpha} \Big)
		\\
		\nabla_x f\big( X_{\cdot} \big) \otimes &f\big( X_{\cdot} \big)_{s, t} 
		\\
		=&\bigg( \nabla_x f \big(X_s \big) \otimes \nabla_x f\big( X_s \big) \otimes f \big( X_s \big) + \tfrac{1}{2} \nabla_x^2 f \big( X_s \big) \otimes \Big( f\big( X_s \big) \otimes f\big( X_s \big) + f\big( X_s \big) \otimes f\big( X_s \big)^T \Big) \bigg)
		\\
		&\cdot W_{s, t} + O\big( |t-s|^{2\alpha} \big) 
	\end{align*}
	Therefore, we describe the solution in terms of a jet of the form
	\begin{align*}
		&\bX: [0,1] \to \bigoplus_{i=0}^3 \lin\Big( (\bR^d)^{\otimes i} , \bR^e \Big)
		\\
		&\bX_t = \bigg( X_t, f\big( X_t \big), \nabla_x f\big( X_t \big) \otimes f\big( X_t \big), 
		\\
		&\qquad\qquad \nabla_x f\big( X_t \big) \otimes \nabla_x f\big( X_t \big) \otimes f \big( X_t \big) + \tfrac{1}{2} \nabla_x^2 f\big( X_t \big) \otimes \Big( f\big( X_t \big) \otimes f\big( X_t \big) + f\big( X_t \big) \otimes f\big( X_t \big)^T \Big) \bigg). 
	\end{align*}
	On the other hand, the rough path
	\begin{align*}
		\rw:[0,1]_{\leq}^{\times 2} \to \bigoplus_{i=0}^3 (\bR^d)^{\otimes i}, 
		\quad
		\rw_{s, t}= \bigg( 1, W_{s, t} , \int_s^t W_{s, r} \otimes dW_r, \int_s^t \int_s^r W_{s, u} \otimes dW_u \otimes dW_r \bigg)
	\end{align*}
	acts on this jet to give
	\begin{equation*}
		\Big\langle \bX_{t}, i \Big\rangle = \Big\langle \triangle\big[ \bX_s \big], \rw_{s,t} \otimes i \Big\rangle + O \Big( |t-s|^{\alpha(4-|i| )} \Big). 
	\end{equation*}
	and $\triangle$ is the deconcatenation coproduct. 
	
	\subsubsection*{Probabilistic rough paths}
	
	Motivated by applying the same techniques to derive a similar expression for rough McKean-Vlasov equation (as is explored in more detail in \cite{2019arXiv180205882.2B} and \cite{salkeld2022ExamplePRP}), we obtain that the solution of the rough McKean-Vlasov equation
	\begin{equation*}
		dX_t = f\big( X_t, \cL_t^X \big) \cdot dW_t
	\end{equation*}
	has increments that can be approximated by
	\begin{align}
		\nonumber
		X_{s, t} &= f\big( X_s, \cL_s^X \big) \cdot W_{s, t} 
		\\
		\nonumber
		&+ \partial_{0} f\big( X_s, \cL_s^X \big) \otimes f\big( X_s, \cL_s^X \big) \cdot \int_s^t W_{s, r} \otimes dW_r 
		\\
		\nonumber
		&+ \tilde{\bE}\bigg[ \partial_{1} f\big( X_s, \cL_s^X \big) \big( \tilde{X}_s \big) \otimes f\big( \tilde{X}_s, \cL_s^X \big) \cdot \int_s^t \tilde{W}_{s, r} \otimes dW_r \bigg]
		\\
		\nonumber
		&+ \tfrac{1}{2} \partial_{00} f\big( X_s, \cL_s^X \big) \otimes \Big( f\big( X_s, \cL_s^X \big) \otimes f \big( X_s, \cL_s^X \big) \Big) \cdot \int_s^t W_{s, r} \otimes W_{s, r} \otimes dW_r 
		\\
		\nonumber
		&+ \tfrac{1}{2}\tilde{\bE}\bigg[ \partial_{01} f\big( X_s, \cL_s^X \big)\big(\tilde{X}_s \big) \otimes \Big( f\big( X_s, \cL_s^X \big) \otimes f \big( \tilde{X}_s, \cL_s^X \big) \Big) \cdot \int_s^t W_{s, r} \otimes \tilde{W}_{s,r} \otimes dW_r \bigg]
		\\
		\nonumber
		&+ \tfrac{1}{2} \tilde{\bE} \bigg[ \partial_{10} f\big( X_s, \cL_s^X \big)\big(\tilde{X}_s \big) \otimes \Big( f\big( \tilde{X}_s, \cL_s^X \big) \otimes f \big( X_s, \cL_s^X \big) \Big) \cdot \int_s^t \tilde{W}_{s, r} \otimes W_{s,r} \otimes dW_r \bigg]
		\\
		\nonumber
		&+ \tfrac{1}{2} \tilde{\bE}\bigg[ \partial_{11} f\big( X_s, \cL_s^X \big)\big(\tilde{X}_s \big) \otimes \Big( f\big( \tilde{X}_s, \cL_s^X \big) \otimes f \big( \tilde{X}_s, \cL_s^X \big) \Big) \cdot \int_s^t \tilde{W}_{s,r} \otimes \tilde{W}_{s,r} \otimes dW_r \bigg]
		\\
		\nonumber
		&+ \tfrac{1}{2} \tilde{\bE}\hat{\bE} \bigg[ \partial_{12} f\big( X_s, \cL_s^X \big)\big(\tilde{X}_s, \hat{X}_s \big) \otimes \Big( f\big( \tilde{X}_s, \cL_s^X \big) \otimes f \big( \hat{X}_s, \cL_s^X \big) \Big)  \cdot \int_s^t \tilde{W}_{s,r} \otimes \hat{W}_{s,r} \otimes dW_r \bigg]
		\\
		\nonumber
		&+ \partial_0 f\big( X_s, \cL_s^X \big) \otimes \partial_0 f \big( X_s, \cL_s^X \big) \otimes f\big( X_s, \cL_s^X \big) \cdot \int_s^t \int_s^r \int_s^u dW_{v} \otimes dW_u \otimes dW_r 
		\\
		\nonumber
		&+ \tilde{\bE}\bigg[ \partial_0 f\big( X_s, \cL_s^X \big) \otimes \partial_1 f \big( X_s, \cL_s^X \big) \big( \tilde{X}_s \big) \otimes f\big( \tilde{X}_s, \cL_s^X \big) \cdot \int_s^t \int_s^r \int_s^u d \tilde{W}_{v} \otimes dW_u \otimes dW_r \bigg]
		\\
		\nonumber
		&+ \tilde{\bE}\bigg[ \partial_1 f\big( X_s, \cL_s^X \big)\big( \tilde{X}_s \big) \otimes \partial_0 f \big( \tilde{X}_s, \cL_s^X \big) \otimes f\big( \tilde{X}_s, \cL_s^X \big) \cdot \int_s^t \int_s^r \int_s^u d\tilde{W}_{v} \otimes d\tilde{W}_u \otimes dW_r \bigg]
		\\
		\nonumber
		&+ \tilde{\bE}\hat{\bE}\bigg[ \partial_1 f\big( X_s, \cL_s^X \big)\big( \tilde{X}_s \big) \otimes \partial_1 f \big( \tilde{X}_s, \cL_s^X \big) \big( \hat{X}_s \big) \otimes f\big( \hat{X}_s, \cL_s^X \big) \cdot \int_s^t \int_s^r \int_s^u d \hat{W}_{v} \otimes d\tilde{W}_u \otimes dW_r \bigg]
		\\
		\label{eq:PRP-expans1}
		&+ O\big( |t-s|^{4 \alpha} \big). 
	\end{align}
	As before, Taylor expanding the functions $(x, \mu) \mapsto \partial_0 f( x, \mu ) \otimes f(x, \mu)$ and $(x, \mu, \tilde{x}) \mapsto \partial_1 f(x, \mu)(\tilde{x}) \otimes f(\tilde{x}, \mu )$, we obtain that
	\begin{align}
		\nonumber
		\partial_0 &f\big( X_{\cdot}, \cL_{\cdot}^X \big) \otimes f\big( X_{\cdot}, \cL_{\cdot}^X \big)_{s, t} 
		\\
		\nonumber
		=& 
		\Big( \partial_0 f\big( X_{s}, \cL_{s}^X \big) \otimes \partial_0 f\big( X_{s}, \cL_{s}^X \big) \otimes f\big( X_{s}, \cL_{s}^X \big) + \partial_{00} f\big( X_{s}, \cL_{s}^X \big) \otimes f\big( X_{s}, \cL_{s}^X \big)^{\otimes 2} \Big) \cdot W_{s, t}
		\\
		\nonumber
		&+ \tilde{\bE}\bigg[ \Big( \partial_0 f\big( X_{s}, \cL_{s}^X \big) \otimes \partial_1 f\big( X_{s}, \cL_{s}^X \big)\big( \tilde{X}_s\big) \otimes f\big( \tilde{X}_{s}, \cL_{s}^X \big) 
		\\
		\label{eq:PRP-expans2}
		&\qquad + \partial_{01} f\big( X_{s}, \cL_{s}^X \big)\big( \tilde{X}_s\big) \otimes f\big( X_{s}, \cL_{s}^X \big) \otimes f\big( \tilde{X}_{s}, \cL_{s}^X \big) \Big) \cdot \tilde{W}_{s, t} \bigg]
		+
		O\Big( |t-s|^{2\alpha} \Big)
		\\
		\nonumber
		\partial_1 &f\big( X_{\cdot}, \cL_{\cdot}^X \big)\big( \tilde{X}_{\cdot} \big) \otimes f\big( \tilde{X}_{\cdot}, \cL_{\cdot}^X \big)_{s, t} 
		\\
		\nonumber
		&= \partial_{10} f\big( X_s, \cL_s^X \big)\big( \tilde{X}_s \big) \otimes f\big( \tilde{X}_s, \cL_s^X \big) \otimes f\big( X_s, \cL_s^X \big) \cdot W_{s, t}
		\\
		\nonumber
		&+\Big( \partial_1 f\big( X_{s}, \cL_{s}^X \big)\big( \tilde{X}_s\big) \otimes \partial_0 f\big( \tilde{X}_{s}, \cL_{s}^X \big) \otimes f\big( \tilde{X}_{s}, \cL_{s}^X \big) + \partial_{11} f\big( X_{s}, \cL_{s}^X \big) \big( \tilde{X}_s\big) \otimes f\big( \tilde{X}_{s}, \cL_{s}^X \big)^{\otimes 2} \Big) \cdot \tilde{W}_{s, t}
		\\
		\nonumber
		&+ \hat{\bE}\bigg[ \Big( \partial_1 f\big( X_{s}, \cL_{s}^X \big)\big( \tilde{X}_s\big) \otimes \partial_1 f\big( \tilde{X}_{s}, \cL_{s}^X \big)\big( \hat{X}_s\big) \otimes f\big( \hat{X}_{s}, \cL_{s}^X \big) 
		\\
		\label{eq:PRP-expans3}
		&\qquad + \partial_{12} f\big( X_{s}, \cL_{s}^X \big)\big( \tilde{X}_s, \hat{X}_s \big) \otimes f\big( \tilde{X}_{s}, \cL_{s}^X \big) \otimes f\big( \hat{X}_{s}, \cL_{s}^X \big) \Big) \cdot \hat{W}_{s, t} \bigg]
		+
		O\Big( |t-s|^{2\alpha} \Big)
	\end{align}
	\begin{remark}
		By studying Equation \eqref{eq:PRP-expans1}, we can draw an impression of what the increments of the probabilistic rough path should look like. Further, we can see in Equation \eqref{eq:PRP-expans2} a sub-jet with the same expansion. The challenge which this paper addresses is that the jet expansion for \eqref{eq:PRP-expans3} is described by the increments of
		\begin{equation*}
			\Big( W_{s, t}, \tilde{W}_{s, t}, \hat{W}_{s, t} \Big)
		\end{equation*}
		The contribution of this paper is exploring how to index expansions of the form \eqref{eq:PRP-expans1}, describing the modules on which expansions of the form
		\begin{equation*}
			\bX_t = \bigg( X_t, f\big( X_t, \cL_t^X \big), \partial_0 f\big( X_t, \cL_t^X \big) \otimes f\big( X_t, \cL_t^X \big), \partial_1 f\big( X_t, \cL_t^X \big) \otimes f\big( \tilde{X}_t, \cL_t^X \big), ... \bigg)
		\end{equation*}
		take their values, defining the \emph{coupled coproduct} $\Delta$ on these modules that describes expressions of the form
		\begin{equation*}
			\Big\langle \bX_t, a \Big\rangle = \Big\langle \Delta\big[ \bX_s \big], \Pi_a^{s, t} \Big\rangle + O\Big( |t-s|^{\alpha(4-|a|)} \Big)
		\end{equation*}
		and concluding with the algebraic properties of $\Delta$. 
	\end{remark}
	
	\section{The coupled tensor algebra and Lions words}
	\label{section:ProbabilisticRoughPaths}
	
	Partition sequences and the way in which they couple provide a different kind of polynomial expansion that is more closely aligned with the Lions-Taylor expansions developed in Section \ref{section:TaylorExpansions}. The goal of this section is to document how these couplings create a new theory on which we will build our rough path for mean-field equations. 

	\subsection{Lions words and the shuffle product}
	\label{subsec:Lionswords}
	
	Our first step is to introduce a generalisation of the concept of Lyndon words (for a reference, we refer the reader to \cite{Melancon1989Lyndon}) where we additionally assign a partition structure:
	\begin{definition}
		\label{definition:Lions-words}
		Let $d\in \bN$, let $\cA = \{1, ..., d\}$ and let $M(\cA)$ be the free monoid generated by $\cA$;
		\begin{equation*}
			M(\cA)= \Big\{ w:=(w_1, ..., w_n): w_i \in \cA, n \in \bN_0 \Big\}. 
		\end{equation*}
		Let $|\cdot|: M(\cA) \to \bN_0$ denote length, $\big| (w_1, ..., w_n) \big| = n$.		
		
		Let $I$ be an index set. We denote the set of \emph{tagged Lions word} by
		\begin{equation*}
			\scW_{0, d}[I]:= \Big\{ (w, a): w \in M(\cA), a \in A_{|w|}[I] \Big\} 
			\quad\mbox{and the empty Lions word by}\quad
			\rId=(\emptyset, \emptyset). 
		\end{equation*}
		The set $\scW_{0, d}[I]$ is a poset with partial ordering
		\begin{align*}
			(w, a) \subseteq (w', a') \quad \iff w = w' \quad \mbox{and}\quad a \subseteq a'. 
		\end{align*}
		Finally, we define $\fm: \scW_{0, d}[I] \to \bN_0$ by $\fm\big[ (w, a) \big] = m[a]$. 
	\end{definition}
	
	Let $(\cR, +, \centerdot)$ be a ring. Motivated by the shuffle algebra (see \cite{reutenauer2003free}) and the way it captures ``integration by parts'' in Taylor expansions, we consider a shuffle product on the free $\cR$-module generated by $\scW_{0, d}[I]$ which we denote by $\spn_\cR\big( \scW_{0, d}[I] \big)$:
	\begin{definition}
		\label{definition:shuffleproduct1}
		Let $d\in \bN$ and let $I$ be an index set. We define the operator 
		\begin{equation*}
			\tilde{\shuffle}: \spn_\cR\big( \scW_{0, d}[I] \big) \times \spn_\cR\big( \scW_{0, d}[I] \big) \to \spn_\cR\big( \scW_{0, d}[I] \big)
		\end{equation*}
		to be the bilinear operator that satisfes that for any $W^1=(w^1, a^1)$ and $W^2=(w^2, a^2)$ we have
		\begin{equation*}
			W^1 \tilde{\shuffle} W^2 = \sum_{\sigma \in \Shuf\big( |W^1|, |W^2| \big) } \sigma\big(W^1, W^2 \big)
		\end{equation*}
		where $\Shuf(k, n)$ is the collection of $(k, n)$-riffle shuffles, we define  
		\begin{align}
			\label{eq:definition:shuffleproduct1}
			&\sigma\big(W^1, W^2 \big) := \Big( \sigma( w^1 \cdot w^2), \big\llbracket \sigma(a^1 \sqcup_I a^2) \big\rrbracket_I \Big)
		\end{align}
		and for $(a^1, a^2) \in A_k[I] \times A_n[I]$,
		\begin{align*}
			&a^1 \sqcup_I a^2 : = \big( b_i \big)_{i=1, ..., k+n} 
			\\
			&\mbox{where}\quad
			b_i := 
			\begin{cases}
				a^1_i \quad&\quad \mbox{when $i \in \{1, ..., k\}$;}
				\\
				a_i^2 \quad&\quad \mbox{when $i \in \{k+1, ..., k+n\}$ and if $a_{i-k}^2 \in I$;}
				\\
				a_i^2 + m[a^1] \quad&\quad \mbox{when $i \in \{k+1, ..., k+n\}$ and if $a_{i-k}^2 \in \{1, ..., m[a^2]\}$.}
			\end{cases}
		\end{align*} 
	\end{definition}

	\begin{proposition}
		\label{proposition:Shuffle=Assoc}
		Let $d\in \bN$, let $I$ be an index set and let $(\cR, +, \centerdot)$ be a ring. Then the $\cR$-module
		\begin{equation*}
			\Big( \spn_{\cR}(\scW_{0, d}[I], \tilde{\shuffle} , \rId \Big) 
			\quad \mbox{is \emph{associative} and \emph{unitary}.}
		\end{equation*}
		Further, if $\cR$ is a commutative ring then $\big( \spn_{\cR}(\scW_{0, d}[I], \tilde{\shuffle} , \rId \big) $ is \emph{commutative}.
	\end{proposition}

	\begin{proof}
		Using the associativity of the classical $\shuffle$ product, we have that for any $w^1, w^2, w^3 \in M(\cA)$ that
		\begin{equation*}
			\sum_{\substack{\sigma_1 \in \Shuf(|w^1|+|w^2|, |w^3|) \\ \sigma_2 \in \Shuf(|w^1|, |w^2|)}} \sigma_1\Big( \sigma_2\big( w^1 \cdot w^2 \big) \cdot w^3 \Big) = \sum_{\substack{\tilde{\sigma}_1 \in \Shuf(|w^1|, |w^2|+|w^3|) \\ \tilde{\sigma}_2 \in \Shuf(|w^2|, |w^3|)}} \tilde{\sigma}_1\Big( w^1 \cdot \tilde{\sigma}_2\big( w^2 \cdot w^3 \big) \Big). 
		\end{equation*}
		Therefore, for any $(w^1, a^1), (w^2,  a^2), (w^3, a^3) \in \scW_{0, d}[I]$, we have
		\begin{align*}
			\Big( (w^1&, a^1) \tilde{\shuffle} (w^2, a^2) \Big) \tilde{\shuffle} (w^3, a^3) 
			\\
			&= \sum_{\substack{\sigma_1 \in \Shuf(|w^1|+|w^2|, |w^3|) \\ \sigma_2 \in \Shuf(|w^1|, |w^2|)}}  \Big( \sigma_1\big( \sigma_2(w^1\cdot w^2) \cdot w^3 \big), \big\llbracket \sigma_1\big( \llbracket \sigma_2(a^1 \sqcup_I a^2) \rrbracket_I \sqcup_I a^3 \big) \big\rrbracket_I \Big)
			\\
			&= \sum_{\substack{\sigma_1 \in \Shuf(|w^1|+|w^2|, |w^3|) \\ \sigma_2 \in \Shuf(|w^1|, |w^2|)}}  \Big( \sigma_1\big( \sigma_2(w^1\cdot w^2) \cdot w^3 \big), \big\llbracket \sigma_1\big( \sigma_2(a^1 \sqcup_I a^2) \sqcup_I a^3  \big) \big\rrbracket_I \Big)
			\\
			&=\sum_{\substack{\tilde{\sigma}_1 \in \Shuf(|w^1|, |w^2|+|w^3|) \\ \tilde{\sigma}_2 \in \Shuf(|w^2|, |w^3|)}} \Big( \tilde{\sigma}_1\big( w^1 \cdot \tilde{\sigma}_2(w^2\cdot w^3) \big), \big\llbracket \tilde{\sigma}_1 \big( a^1 \sqcup_I \tilde{\sigma}_2(a^2 \sqcup_I a^3) \big) \big\rrbracket_I \Big)
			\\
			&=\sum_{\substack{\tilde{\sigma}_1 \in \Shuf(|w^1|, |w^2|+|w^3|) \\ \tilde{\sigma}_2 \in \Shuf(|w^2|, |w^3|)}} \Big( \tilde{\sigma}_1\big( w^1 \cdot \tilde{\sigma}_2(w^2\cdot w^3) \big), \big\llbracket \tilde{\sigma}_1 \big( a^1 \sqcup_I \llbracket \tilde{\sigma}_2(a^2 \sqcup_I a^3) \rrbracket \big) \big\rrbracket_I \Big)
			\\
			&=(w^1, a^1) \tilde{\shuffle} \Big( (w^2, a^2) \tilde{\shuffle} (w^3, a^3) \Big). 
		\end{align*}
		Similarly, 
		\begin{align*}
			(w, a) \tilde{\shuffle} \rId = \Big( w \cdot \emptyset, \llbracket a \sqcup_I \emptyset \rrbracket \Big) = \Big( \emptyset \cdot w, \llbracket \emptyset \sqcup_I a \rrbracket \Big) = (w, a)
		\end{align*}
		so that $\rId$ is unital. 
		
		Finally, using the commutativity of the classical $\shuffle$ product, we have that for any $w^1, w^2 \in M(\cA)$ that
		\begin{equation*}
			\sum_{\sigma \in \Shuf(|w^1|,|w^2|)} \sigma(w^1 \cdot w^2) = \sum_{\tilde{\sigma} \in \Shuf(|w^2|,|w^1|)} \tilde{\sigma}(w^2 \cdot w^1). 
		\end{equation*}
		so that when $\cR$ is commutative we obtain that
		\begin{align*}
			(w^1, a^1) \tilde{\shuffle} (w^2, a^2) 
			=& \sum_{\sigma \in \Shuf(|w^1|,|w^2|)}  \Big( \sigma( w^1\cdot w^2) , \big\llbracket \sigma ( a^1 \sqcup_I a^2) \big\rrbracket_I \Big)
			\\
			=& \sum_{\tilde{\sigma} \in \Shuf(|w^2|,|w^1|)}  \Big( \tilde{\sigma}( w^2\cdot w^1) , \big\llbracket \sigma ( a^2 \sqcup_I a^1) \big\rrbracket_I \Big) = (w^2, a^2)\tilde{\shuffle} (w^1, a^1). 
		\end{align*}
	\end{proof}
	
	\subsection{Couplings for Lions words and coupled coproduct}
	\label{subsubsec:CoupledTensorWords}
	
	We need a sense of product between two Lions words that is richer than the cartesian product in that it additionally needs to capture all of the ways in which the free variables can be coupled together. Building on Equation \eqref{eq:Tag-Part-Seq[a]}, for any $W=(w, a) \in \scW_{0, d}[I]$ we define
	\begin{equation}
		\label{eq:Tag-Part-Seq[a]2}
		\scW_{0, d}[W]: = \scW_{0, d}\Big[ I \cup \big\{ a^{-1}[j]: j \in \{1, ..., m[a] \} \big\} \Big].
	\end{equation}
	Following on from Remark \ref{remark:couplings1}, we introduce the next definition:
	\begin{definition}
		\label{definition:coupled_pair-W}
		Let $d\in \bN$ and let $I$ be an index set. We define the set
		\begin{align*}
			\scW_{0, d} \tilde{\times} \scW_{0, d}[I]
			:=& 
			\bigsqcup_{\hat{W} \in \scW_{0, d}[I]} \scW_{0, d}[\hat{W}]
			=
			\Big\{ (\bar{W}, \hat{W}): \hat{W} \in \scW_{0, d}[I] \quad \mbox{and}\quad \bar{W} \in \scW_{0, d}[\hat{W}] \Big\}. 
		\end{align*}
		The set $\scW_{0, d} \tilde{\times} \scW_{0, d}[I]$ is a poset with partial ordering $\subseteq$
		\begin{equation*}
			\big( (\bar{w}, \bar{a}), (\hat{w}, \hat{a}) \big) \subseteq \big( (\bar{w}', \bar{a}'), (\hat{w}', \hat{a}') \big)
			\quad \iff\quad
			(\bar{w}, \hat{w}) = (\bar{w}', \hat{w}') 
			\quad \mbox{and both}\quad
			\bar{a} \subseteq \bar{a}' , \hat{a} \subseteq \hat{a}'. 
		\end{equation*}
		We define $\fm: \scW_{0, d} \tilde{\times}\scW_{0, d}[I] \to \bN_0$ by $\fm\big[ \big( (\bar{w}, \bar{a}), (\hat{w}, \hat{a}) \big) \big] = m[\hat{a}] + m[\bar{a}]$.  
		
		Further, we inductively define
		\begin{align}
			\nonumber
			&\Big( \scalebox{1.5}{$\tilde{\times}$}_{1}^n \scW_{0, d}\Big)[I] := \scW_{0, d} \tilde{\times} \Big( \scalebox{1.5}{$\tilde{\times}$}_{1}^{n-1} \scW_{0, d} \Big)[I]
			\\
			\label{eq:iterative-coupling}
			&=\Big\{ \big( \hat{W}^n, ..., \hat{W}^1\big): \hat{W}^n \in \scW_{0, d}[\hat{W}^{n-1}], \hat{W}^{n-1} \in \scW_{0, d}[\hat{W}^{n-2}], ..., \hat{W}^{2} \in \scW_{0, d}[\hat{W}^1],  \hat{W}^1 \in \scW_{0, d}[I] \Big\}
		\end{align}
		and extend the partial ordering $\subseteq$ and $\fm$ appropriately. 
	\end{definition}
	
	\begin{lemma}
		\label{lemma:associativity-coupledproduct}
		Let $d \in \bN$ and let $I$ be an index set. Then for any choice of $k, n \in \bN$, there is an isomorphism between the posets
		\begin{equation*}
			\Big(\scalebox{1.5}{$\tilde{\times}$}_{1}^n \scW_{0, d} \Big) \tilde{\times} \Big( \scalebox{1.5}{$\tilde{\times}$}_{1}^k \scW_{0, d} \Big)[I]:= \bigsqcup_{\big( \hat{W}^k, ..., \hat{W}^1 \big) \in \scalebox{1.5}{$\tilde{\times}$}_{1}^k \scW_{0, d}[I] } \Big( \scalebox{1.5}{$\tilde{\times}$}_{1}^n \scW_{0, d} \Big)[\hat{W}^k]
			\quad \mbox{and}\quad
			\Big( \scalebox{1.5}{$\tilde{\times}$}_{1}^{k+n} \scW_{0, d} \Big) [I]. 
		\end{equation*}
		What is more, the isomorphism respects the partial ordering and $\fm$. 
	\end{lemma}
	
	\begin{proof}
		The set
		\begin{align*}
			&\Big(\scalebox{1.5}{$\tilde{\times}$}_{1}^n \scW_{0, d} \Big) \tilde{\times} \Big( \scalebox{1.5}{$\tilde{\times}$}_{1}^k \scW_{0, d} \Big)[I]
			\\
			&:= \bigg\{ \big( (\bar{W}^n, ..., \bar{W}^1), (\hat{W}^k, ..., \hat{W}^1) \big): (\hat{W}^k, ..., \hat{W}^1) \in \Big( \scalebox{1.5}{$\tilde{\times}$}_{1}^k \scW_{0, d} \Big)[I], (\bar{W}^n, ..., \bar{W}^1) \in \Big( \scalebox{1.5}{$\tilde{\times}$}_{1}^k \scW_{0, d} \Big)[\hat{W}^k] \bigg\}
			\\
			&= \bigg\{ \big( (\bar{W}^n, ..., \bar{W}^1), (\hat{W}^k, ..., \hat{W}^1) \big): \bar{W}^n \in \scW_{0,d}[\bar{W}^{n-1}], ..., \bar{W}^1 \in \scW_{0, d}[\hat{W}^k], 
			\\
			&\hspace{175pt} \hat{W}^k\in \scW_{0, d}[\hat{W}^{k-1}], ..., \hat{W}^1 \in \scW_{0, d}[I] \bigg\} = \Big(\scalebox{1.5}{$\tilde{\times}$}_{1}^{k+n} \scW_{0, d} \Big)[I]
		\end{align*}
		so the set isomorphism is clear. The partial ordering $\subseteq$ and equivalence of $\fm$ also follow. 
	\end{proof}
	
	\begin{remark}
		\label{remark:Whatis-CoupledTensor}
		Given a ring $(\cR, +, \centerdot)$ and two $\cR$-modules $U$ and $V$ with bases $B_U$ and $B_V$ respectively, we can equivalently write
		\begin{equation*}
			U = \spn_{\cR}\big( B_U \big) = \bigoplus_{u\in B_U} \cR 
			\quad \mbox{and}\quad
			V = \spn_{\cR}\big( B_V \big) = \bigoplus_{v\in B_V} \cR 
		\end{equation*}
		One interpretation of the tensor $\cR$-module $U \otimes V$ is as
		\begin{equation*}
			U \otimes V = \spn_{\cR}\big( B_U \times B_V \big) = \bigoplus_{(u, v) \in B_U \times B_V} \cR = \bigoplus_{u \in B_U} \bigoplus_{v \in B_V} \cR
		\end{equation*}
		where $B_U \times B_V$ is the Cartesian product of the two basis sets. 
		
		In contrast to this, we \emph{intuitively} define the coupled tensor product as
		\begin{equation}
			\label{eq:remark:Whatis-CoupledTensor}
			\spn_{\cR} \big( \scW_{0, d}[I] \big) \tilde{\otimes} \spn_{\cR} \big( \scW_{0, d}[I] \big) := \spn_{\cR} \big( \scW_{0, d} \tilde{\times} \scW_{0, d}[I] \big) = \bigoplus_{\hat{W} \in \scW_{0, d}[I]} \bigoplus_{\bar{W} \in \scW_{0, d}[\hat{W}]} \cR. 
		\end{equation}
	\end{remark}
	
	\begin{remark}
		\label{remark:couplings=partitions}
		Recalling Proposition \ref{proposition:taggedpartition*} and Proposition \ref{proposition:coupling(1)} for the moment, we can see that there is an interpretation of a tagged Lions word $(w, a) \in \scW_{0, d}[I]$ as an element of the free monoid $w \in M(\cA)$ paired with tagged partition $\scP(\scN)[I]$ (where $\scN$ is a set containing $|w|$ elements). 
		
		Similarly, a coupled pair of tagged Lions words $\big( (\bar{w},\bar{a}), (\hat{w}, \hat{a}) \big)$ can be identified with as a pair of elements of the free monoid $\bar{w}, \hat{w} \in M(\cA)$, a tagged partition from both $\scP(\bar{\scN})[I]$ and $\scP(\hat{\scN})[I]$ (where $\bar{\scN}$ and $\hat{\scN}$ are disjoint sets containing $|\bar{w}|$ and $|\hat{w}|$ elements respectively) and a tagged partition $\scP(\bar{\scN} \cup \hat{\scN})[I]$ that agrees with the two partitions when restricted to the subsets $\hat{\scN}$ and $\hat{\scN}$ respectively. 
	\end{remark}
	
	\subsubsection{Coupled coproduct}
	
	The purpose of introducing Lions words and couplings between them is to create index sets for a module on which we will introduce a \emph{coupled coproduct}. However, first we consider the module on which this operator will be defined. 
	
	\begin{definition}
		\label{definition:deconcatenation-coproduct}
		Let $I$ be an index set, let $n\in \bN$ and let $W=(w,a) \in \scW_{0, d}[I]$ such that $|w| = n$. 
			
		We define $\Delta: \spn_{\cR}\big( \scW_{0, d}[I] \big) \to \spn_{\cR}\big( \scW_{0, d} \tilde{\times} \scW_{0, d}[I] \big)$ to be the linear operator that satisfies
		\begin{align}
			\nonumber
			\Delta\Big[ (w, a) \Big] =& \big( (w, a), \rId) + \big( \rId, (w, a) \big)
			\\
			\label{eq:definition:deconcatenation-coproduct}
			&+
			 \sum_{k=1}^{n-1} \bigg( \Big( (w_{k+1}, ..., w_{n}), \big\llbracket (a_{k+1}, ..., a_n) \big\rrbracket_{I \cup m_k\{a\}} \Big), \big( (w_1, ..., w_k), (a_1, ..., a_k) \big) \bigg)
		\end{align}
		where we denote
		\begin{equation*}
			m_k\{a\}:=\Big\{ a^{-1}[j]: j=1, ..., m\big[(a_1, ...., a_k) \big] \Big\}. 
		\end{equation*}
		Further, let
		\begin{equation*}
			\epsilon: \spn_{\cR}\big( \scW_{0, d}[I]\big) \to \cR 
			\quad \mbox{by}\quad
			\epsilon \big[ X\big] = \big\langle X, \rId \big\rangle. 
		\end{equation*}
	\end{definition}
	The operator defined in Definition \ref{definition:deconcatenation-coproduct} shares a resemblance with the deconcatenation coproduct (see for instance \cite{reutenauer2003free}) which we shall be exploiting at length here. Our first goal is to prove that the operator $\Delta$ is co-associative and co-unital. However, to do this we first need to give a meaning to the coupled tensor product of operators. 
	
	With this in mind, first note that for any choice of index set $I$ the operation $\fm$ allows us to partition the elements of $\scW_{0, d}[I]$ into the disjoint sets
	\begin{equation*}
		\scW_{0, d}^{\{k\}}[I]:= \Big\{ W \in \scW_{0, d}[I]: \fm[W] = k \Big\}
		\quad \mbox{so that}\quad
		\scW_{0, d}[I] = \bigcup_{k \in \bN_0} \scW_{0, d}^{\{k\}}[I]. 
	\end{equation*}
	Similarly, the operator $\fm$ allows us to partition the set of coupled $n$-tuples into the disjoint sets
	\begin{align*}
		&\Big(\scalebox{1.5}{$\tilde{\times}$}_{1}^n \scW_{0, d}\Big)^{\{k\}}[I]:= \bigg\{ (\hat{W}^n,..., \hat{W}^1) \in \Big(\scalebox{1.5}{$\tilde{\times}$}_{1}^n \scW_{0, d} \Big)[I]: \fm\big[ (\hat{W}^n, ..., \hat{W}^1) \big] = k \bigg\}. 
	\end{align*}
	For any $n_1, n_2 \in \bN$, we say that a linear operators 
	\begin{align}
		\nonumber
		&A: \spn_{\cR} \Big(\scalebox{1.5}{$\tilde{\times}$}_{1}^{n_1} \scW_{0, d} \Big)[I] \to \spn_{\cR} \Big(\scalebox{1.5}{$\tilde{\times}$}_{1}^{n_2} \scW_{0, d} \Big)[I]
		\quad \mbox{s.t.}
		\\
		\label{eq:lemma:Delta-diagonal}
		\forall k \in \bN_0, \quad &A\Big|_{\spn_{\cR}\big(\scalebox{1.5}{$\tilde{\times}$}_{1}^{n_1} \scW_{0, d} \big)^{\{k\}}[I] } \subseteq \spn_{\cR}\Big(\scalebox{1.5}{$\tilde{\times}$}_{1}^{n_2} \scW_{0, d} \Big)^{\{k\}}[I]
	\end{align}
	are \emph{$\fm$-diagonal}. 
	
	\begin{lemma}
		\label{lemma:Delta-diagonal}
		For every index set $I$, the operator $\Delta$ is $\fm$-diagonal. 
	\end{lemma}
	
	\begin{proof}
		For any $W=(w, a)  \in \scW_{0, d}[I]$, we have $a \in A_{|w|}[I]$ and $\fm[W] = m[a]$. As $\fm[\rId] = 0$, we conclude that $\fm\big[ (W, \rId) \big] = \fm\big[ (\rId, W) \big] = \fm[W]$. Similarly, for any $k \in \{1, ..., |w|\}$
		\begin{equation*}
			(a_1, ..., a_k) \in A_k[I]
			\quad \mbox{and}\quad
			\big\llbracket (a_{k+1}, ..., a_{|w|}) \big\rrbracket_{I \cup m_k\{a\}} \in A_{|w|-k}\big[ I \cup m_l\{a\} \big]
		\end{equation*}
		so that
		\begin{align*}
			\fm \Big[& \Big( \big( (w_{k+1}, ..., w_n), \big\llbracket (a_{k+1}, ..., a_n) \big\rrbracket_{I \cup m_k\{a\}} \big), \big((w_1, ..., w_k), (a_1, ..., a_k) \big) \Big) \Big] 
			\\
			&= m\Big[ \big\llbracket (a_{k+1}, ..., a_n) \big\rrbracket_{I \cup m_k\{a\}} \Big] + m\big[ (a_1, ..., a_k) \big]=m[a]. 
		\end{align*}
		Therefore $\Delta$ as defined in Equation \eqref{eq:definition:deconcatenation-coproduct} satisfies Equation \eqref{eq:lemma:Delta-diagonal} with $n_1 = 1$ and $n_2 = 2$. 
	\end{proof}
	This $\fm$-diagonality will give us a sense in which we can define the coupled product of linear operators on modules indexed by the tagged Lions words:
	\begin{definition}
		\label{definition:coupledTP-operators}
		Let $I$ be an index set and let
		\begin{align*}
			A: \spn_{\cR}\Big(\scalebox{1.5}{$\tilde{\times}$}_{1}^{k_1} \scW_{0, d} \Big)[I] \to \spn_{\cR}\Big(\scalebox{1.5}{$\tilde{\times}$}_{1}^{k_2} \scW_{0, d} \Big)[I]
		\end{align*}
		be an $\fm$-diagonal linear operator. Next, for every index set $J$ such that $I \subseteq J$, let
		\begin{equation*}
			B_J: \spn_{\cR}\Big(\scalebox{1.5}{$\tilde{\times}$}_{1}^{n_1} \scW_{0, d} \Big)[J] \to \spn_{\cR}\Big(\scalebox{1.5}{$\tilde{\times}$}_{1}^{n_2} \scW_{0, d} \Big)[J]
		\end{equation*}
		be an $\fm$-diagonal linear operator. Then we define the linear operator
		\begin{equation*}
			B_I \tilde{\otimes} A: \spn_{\cR}\Big( \scalebox{1.5}{$\tilde{\times}$}_{1}^{k_1+n_1} \scW_{0, d} \Big)[I]
			 \to 
			 \spn_{\cR}\Big(\scalebox{1.5}{$\tilde{\times}$}_{1}^{k_2+n_2} \scW_{0, d} \Big)[I]
		\end{equation*}
		for every $(\bar{W}^{\{n_1\}}, \hat{W}^{\{k_1\}}) \in \Big( \scalebox{1.5}{$\tilde{\times}$}_{1}^{n_1} \scW_{0, d} \Big) \tilde{\times} \Big( \scalebox{1.5}{$\tilde{\times}$}_{1}^{k_1} \scW_{0, d} \Big)[I]$ by
		\begin{align*}
			&B_I \tilde{\otimes} A \Big[ (\bar{W}^{\{n_1\}}, \hat{W}^{\{k_1\}} ) \Big] 
			\\
			&= \sum_{\substack{\bar{W}^{\{n_2\}} \in \big(\scalebox{1.5}{$\tilde{\times}$}_{1}^{n_2} \scW_{0, d} \big)[\hat{W}^{\{k_2\}}]: \\ \fm[\bar{W}^{\{n_2\}}] = \fm[\bar{W}^{\{n_1\}}] ; \\ \bar{W}^{\{k_2\}} \in \big(\scalebox{1.5}{$\tilde{\times}$}_{1}^{k_2} \scW_{0, d} \big)[I]: \\ \fm[\bar{W}^{\{k_2\}}] = \fm[\bar{W}^{\{k_1\}}] }} \Big\langle B_{\hat{W}'} \big[ \bar{W}^{\{n_1\}} \big], \bar{W}^{\{n_2\}} \Big\rangle \centerdot \Big\langle A\big[ \hat{W}^{\{k_1\}} \big], \hat{W}^{\{k_2\}} \Big\rangle (\bar{W}^{\{n_2\}}, \hat{W}^{\{k_2\}}). 
		\end{align*}
	\end{definition}
	With this definition under our belt, we are now able to define the coupled tensor product between the identity operator (which is clearly $\fm$-diagonal) and $\Delta$ (which is $\fm$-diagonal by Lemma \ref{lemma:Delta-diagonal}). 
	
	\begin{proposition}
		\label{proposition:M-coassociativity*}
		Let $d\in \bN$, let $I$ be an index set and let $(\cR, +, \centerdot)$ be a ring. Let 
		\begin{align*}
			\fI:& \spn_{\cR}\big( \scW_{0, d}[I] \big) \to \spn_{\cR}\big( \scW_{0, d}[I] \big)
		\end{align*}
		be the identity operator and let $\Delta$ be the linear operator defined in Definition \ref{definition:deconcatenation-coproduct}. Then 
		\begin{equation}
			\label{eq:proposition:M-coassociativity*-1}
			\Big( \Delta \tilde{\otimes} \fI \Big) \circ \Delta = \Big( \fI \tilde{\otimes} \Delta \Big) \circ \Delta
			\quad \mbox{and}\quad
			\centerdot \circ \Big( \epsilon \tilde{\otimes} \fI \Big) \circ \Delta = \centerdot \circ \Big( \fI \tilde{\otimes} \epsilon \Big) \circ \Delta = \fI. 
		\end{equation}
	\end{proposition}
	
	\iftoggle{figure}{In particular, this tells us that the coupled coproduct $\Delta$ is \emph{coupled-coassociative}, or equivalently satisfies the two commutative diagrams described in Figure \ref{fig:coupled-associativity}
	\begin{figure}[htb]
		\centering
		\begin{tikzpicture}
			\node at (0,0) {$\spn_{\cR}(\scW_{0,d} \tilde{\times} \scW_{0,d}[I])$};
			\node at (6,3) {$\spn_{\cR}(\scW_{0,d} \tilde{\times} \scW_{0,d}[I])$};
			\node at (0,3) {$\spn_{\cR}(\scW_{0,d}[I])$};
			\node at (6,0) {$\spn_{\cR}\Big(\scalebox{1.5}{$\tilde{\times}$}_{1}^3 \scW_{0, d}[I] \Big)$};
			\draw[-to](2.25,0) to (4,0);
			\draw[-to](6,2.5) to (6,0.5);
			\draw[-to](0,2.5) to (0,0.5);
			\draw[-to](1.5,3) to (3.75,3);
			\node at (0.5,1.5) {$\Delta$};
			\node at (3,2.5) {$\Delta$};
			\node at (3,0.5) {$\Delta \tilde{\otimes} \fI$};
			\node at (5.5,1.5) {$\fI \tilde{\otimes} \Delta$};
		\end{tikzpicture}
		\begin{tikzpicture}
			\node at (0,0) {$\spn_{\cR}(\scW_{0,d}[I])$};
			\node at (3,1) {$\spn_{\cR}(\scW_{0,d} \tilde{\times} \scW_{0,d}[I])$};
			\node at (9,1) {$\spn_{\cR}(\scW_{0,d}[I]) \otimes \cR$};
			\node at (3,-1) {$\spn_{\cR}(\scW_{0,d} \tilde{\times} \scW_{0,d}[I])$};
			\node at (9,-1) {$\cR \otimes \spn_{\cR}( \scW_{0,d}[I])$};
			\node at (12,0) {$\spn_{\cR}(\scW_{0,d}[I])$};
			\draw[-to](0,-0.25) to (0.9,-1);
			\draw[-to](0, 0.25) to (0.9, 1);
			\draw[-to](5.25, -1) to (7, -1);
			\draw[-to](5.25, 1) to (7, 1);
			\draw[-to] (11.1,-1) to (12,-0.25);
			\draw[-to] (11.1,1) to (12, 0.25);
			\node at (0.25,0.75) {$\Delta$};
			\node at (0.25,-0.75) {$\Delta$};
			\node at (6,-1.5) {$\epsilon \tilde{\otimes} \fI$};
			\node at (6,1.5) {$\fI \tilde{\otimes} \epsilon$};
			\node at (11.75,0.75) {$\centerdot$};
			\node at (11.75,-0.75) {$\centerdot$};
		\end{tikzpicture}
		\caption{Coupled coassociativity}
		\label{fig:coupled-associativity}
	\end{figure}
	}
	
	\begin{proof}
		The identity operator is $\fm$-diagonal, so that we can define
		\begin{align*}
			\Delta \tilde{\otimes} \fI : \spn_{\cR}\Big( \scalebox{1.5}{$\tilde{\times}$}_{1}^{2} \scW_{0, d}[I] \Big) \to \spn_{\cR} \Big( \scalebox{1.5}{$\tilde{\times}$}_{1}^{3} \scW_{0, d}[I] \Big)
			\\
			\fI \tilde{\otimes} \Delta : \spn_{\cR}\Big( \scalebox{1.5}{$\tilde{\times}$}_{1}^{2} \scW_{0, d}[I] \Big) \to \spn_{\cR} \Big( \scalebox{1.5}{$\tilde{\times}$}_{1}^{3} \scW_{0, d}[I] \Big)
		\end{align*}
		according to Definition \ref{definition:coupledTP-operators}. 
		
		For any $W = (w, a)\in \scW_{0, d}[I]$ such that $|W| = n$ we get that
		\begin{align*}
			\big( \fI &\tilde{\otimes} \Delta \big) \circ \Delta \Big[ W \Big] 
			\\
			=& \sum_{k=0}^n \big( \fI \tilde{\otimes} \Delta \big)\bigg[ \Big(\Big( (w_{k+1}, ..., w_n), \big\llbracket(a_{k+1}, ..., a_n) \big\rrbracket_{I \cup m_k\{a\}} \Big), \big( (w_1, ..., w_k), (a_1, ..., a_k) \big) \Big) \bigg]
			\\
			=& \sum_{k=0}^n \sum_{i=0}^k \bigg( \Big( (w_{k+1}, ..., w_n), \big\llbracket(a_{k+1}, ..., a_n) \big\rrbracket_{I \cup m_k\{a\}} \Big), \Big( (w_{i+1}, ..., w_k), \big\llbracket(a_{i+1}, ..., a_k)\big\rrbracket_{I \cup a_i\{a\}} \Big) , 
			\\
			&\qquad \qquad \big( (w_1, ..., w_i), (a_1, ..., a_i) \big) \bigg)
			\\
			=& \sum_{i=0}^n \sum_{k=i}^{n} \bigg( \Big( (w_{k+1}, ..., w_n), \big\llbracket(a_{k+1}, ..., a_n) \big\rrbracket_{I \cup m_k\{a\}} \Big), \Big( (w_{i+1}, ..., w_k), \big\llbracket(a_{i+1}, ..., a_k)\big\rrbracket_{I \cup a_i\{a\}} \Big) , 
			\\
			&\qquad \qquad \big( (w_1, ..., w_i), (a_1, ..., a_i) \big) \bigg) 
			\\
			=& \sum_{i=0}^n \big( \Delta \tilde{\otimes} \fI \big)\bigg[ \Big(\Big( (w_{i+1}, ..., w_n), \big\llbracket(a_{i+1}, ..., a_n) \big\rrbracket_{I \cup m_i\{a\}} \Big), \big( (w_1, ..., w_i), (a_1, ..., a_i) \big) \Big) \bigg]
			\\
			&=\big( \Delta \tilde{\otimes} \fI \big) \circ \Delta \Big[ W \Big] 
		\end{align*}
		and we conclude. Verifying the counit identity is similar. 
	\end{proof}

	\subsubsection{Bialgebra identity on Lions words}
	
	We have seen how the shuffle product (defined as a linear operator)
	\begin{equation*}
		\tilde{\shuffle}: \spn_{\cR}\big( \scW_{0,d}[I] \big) \otimes \spn_{\cR}\big( \scW_{0,d}[I] \big) \to \spn_{\cR}\big( \scW_{0,d}[I] \big)
	\end{equation*}
	and coupled coproduct (also a linear operator) 
	\begin{equation*}
		\Delta: \spn_{\cR}\big( \scW_{0,d}[I] \big) \to \spn_{\cR}\big( \scW_{0,d}[I] \big) \tilde{\otimes} \spn_{\cR}\big( \scW_{0,d}[I] \big)
	\end{equation*}
	are defined individually, but now we want to see how these two operators interact with one another in order to demonstrate a coupled bialgebra structure. 
	
	As a $\cR$-module, we can classically define the tensor product
	\begin{align*}
		\spn_{\cR}&\big( \scW_{0, d} \tilde{\times} \scW_{0, d}[I] \big) \otimes \spn_{\cR}\big( \scW_{0, d} \tilde{\times} \scW_{0, d}[I] \big)
		\\
		&=\spn_{\cR}\Big( \big( \scW_{0, d} \tilde{\times} \scW_{0, d}[I] \big) \times \big( \scW_{0, d} \tilde{\times} \scW_{0, d}[I] \big) \Big). 
	\end{align*}
	In contrast to this, we want to give a meaning to the $\cR$-module
	\begin{align*}
		\spn_{\cR}&\big( \scW_{0, d}[I] \times \scW_{0, d}[I] \big) \tilde{\otimes} \spn_{\cR}\big( \scW_{0, d}[I] \times \scW_{0, d}[I] \big)
		\\
		&\overset{?}{=} \spn_{\cR}\Big( \big( \scW_{0, d}[I] \times \scW_{0, d}[I] \big)\tilde{\times} \big( \scW_{0,d}[I] \times \scW_{0, d}[I] \big) \Big). 
	\end{align*}

	\begin{definition}
		\label{definition:product-couplingsM}
		Let $I$ be index sets and write
		\begin{equation*}
			\big(\scW_{0, d} \times \scW_{0, d}\big) [I]:= \Big\{ (W^1, W^2): W^1 \in \scW_{0, d}[I], W^2 \in \scW_{0, d}[I] \Big\}
		\end{equation*}
		and we define $\fm: (\scW_{0, d}\times \scW_{0, d})[I] \to (\bN_0)^{\times 2}$ by $\fm[W^1, W^2] = \big( \fm[W^1], \fm[W^2] \big)$. 
		
		Next, we define
		\begin{align}
			\label{eq:definition:coupling_pairs}
			&\big( \scW_{0, d} \times \scW_{0, d} \big) \tilde{\times} \big( \scW_{0, d} \times \scW_{0, d} \big)[I]
			:=\bigsqcup_{\substack{ (W^1, W^2) \in \\ (\scW_{0, d} \times \scW_{0, d})[I] }} \big(\scW_{0, d} \times \scW_{0, d} \big)\big[ (W^1, W^2) \big]
		\end{align}
		where
		\begin{align*}
			\big(\scW_{0, d}& \times \scW_{0, d} \big)\big[ \big((w^1, a^1), (w^2, a^2) \big) \big]
			\\
			&:= \big(\scW_{0, d} \times \scW_{0, d}\big)\Big[ I \cup \big\{ (a^1)^{-1}[j]: j \in \{1, ..., m[a^1]\} \big\} \cup \big\{ (a^2)^{-1}[j]: j \in \{1, ..., m[a^2] \} \big\} \Big]. 
		\end{align*}
		We pair this with the operation
		\begin{align*}
			&\fm: \big( \scW_{0, d} \times \scW_{0, d} \big) \tilde{\times} \big( \scW_{0, d} \times \scW_{0, d} \big)[I] \to (\bN_0)^{\times 2}
			\\
			&\fm\big[ \big( (\bar{W}^1, \bar{W}^2), (\hat{W}^1, \hat{W}^2) \big) \big] = \big( \fm[ \bar{W}^1, \hat{W}^1], \fm[\bar{W}^2, \hat{W}^2] \big). 
		\end{align*}
	\end{definition}
	We can write
	\begin{align*}
		\Big(& \scW_{0, d} \tilde{\times} \scW_{0, d}[I] \Big) \times \Big( \scW_{0, d} \tilde{\times} \scW_{0, d}[I] \Big) 
		\\
		&= \Big\{ \big( (\bar{W}^1, \hat{W}^1), (\bar{W}^2, \hat{W}^2) \big): \hat{W}^1 \in \scW_{0, d}[I], \hat{W}^2 \in \scW_{0, d}[I], \quad \bar{W}^1 \in \scW_{0, d}[\hat{W}^1], \bar{W}^2 \in \scW_{0, d}[\hat{W}^2] \Big\}
	\end{align*}
	whereas
	\begin{align*}
		\Big(& \scW_{0, d} \times \scW_{0, d} \Big) \tilde{\times} \Big( \scW_{0, d} \times \scW_{0, d} \Big)[I]
		\\
		&= \Big\{ \big( (\bar{W}^1, \bar{W}^2), (\hat{W}^1, \hat{W}^2) \big) : \hat{W}^1, \hat{W}^2 \in \scW_{0, d}[I], \quad \bar{W}^1, \bar{W}^2 \in \scW_{0, d}\big[ (\hat{W}^1,\bar{W}^2) \big] \Big\}. 
	\end{align*}	
	In order to demonstrate the appropriate bialgebra properties, we need to introduce our own notion of $\mbox{Twist}$ operation that additionally keeps track of the couplings:
	\begin{definition}
		\label{definition:Twist-}
		Let $d\in \bN$ and let $I$ be an index sets. We define 
		\begin{align*}
			\overline{\mbox{Twist}}:& \spn_{\cR}\Big( \big(\scW_{0, d} \tilde{\times} \scW_{0, d} \big)[I] \times \big( \scW_{0,d} \tilde{\times} \scW_{0,d} \big)[I] \Big) 
			\to \spn_{\cR} \Big( \big( \scW_{0,d} \times \scW_{0,d} \big) \tilde{\times} \big( \scW_{0,d} \times \scW_{0, d} \big)[I] \Big)
		\end{align*}
		to be the linear operator that satisfies
		\begin{align*}
			&\overline{\mbox{Twist}}\Big[ \big( (\bar{W}^1, \hat{W}^1), (\bar{W}^2, \hat{W}^2) \big) \Big] = \big( (\bar{W}^1, \bar{W}^2), (\hat{W}^1, \hat{W}^2) \big). 
		\end{align*} 
	\end{definition}
	The twist operation swaps the two middle Lions words and transforms a couplings between pairs into a pair of couplings which relies on the fact that
	\begin{equation*}
		\bar{W}^1 \in \scW_{0, d}[\hat{W}^1] 
		\quad \mbox{and}\quad
		\bar{W}^2 \in \scW_{0, d}[\hat{W}^2]
		\quad \implies \quad
		(\bar{W}^1, \bar{W}^2) \in \scW_{0, d}\big[ (\hat{W}^1, \hat{W}^2)\big]. 
	\end{equation*}
	What is more, $\overline{\mbox{Twist}}$ is a linear $\cR$-module isomorphism and $\fm$-diagonal. 
	
	\begin{lemma}
		\label{lemma:Shuffle_diagonal}
		For any choice of index set $I$ the linear map
		\begin{equation*}
			\tilde{\shuffle}: \spn_{\cR}\big( \scW_{0, d}[I] \times \scW_{0, d}[I] \big) \to \spn_{\cR}\big( \scW_{0, d}[I] \big)
			\qquad
			\mbox{is $\fm$-diagonal. }
		\end{equation*}
	\end{lemma}
	
	\begin{proof}
		Let $n \in \bN$, let $(W^1, W^2) = \big( (w^1, a^1), (w^2, a^2) \big) \in (\scW_{0, d}\times \scW_{0, d})[I]$ such that $\fm\big[ (W^1, W^2) \big] = \fm[a^1] + \fm[a^2]= n$. Then for any $\sigma \in \Shuf( |W^1|, |W^2|)$, 
		\begin{equation*}
			\fm\Big[ \big( \sigma(w^1 \cdot w^2), \big\llbracket \sigma(a^1 \sqcup_I a^2) \big\rrbracket_I \big) \Big] = m[a^1] + m[a^2] = n
		\end{equation*}
		so that
		\begin{equation*}
			\tilde{\shuffle}\big[ (W^1, W^2) \big] \subseteq \spn_{\cR} \big( \scW_{0, d}^{\{n\}} [I] \big)
		\end{equation*}
		and we conclude. 
	\end{proof}
	
	Given $\tilde{\shuffle}$ is $\fm$-diagonal, we can define the linear map
	\begin{equation*}
		\tilde{\shuffle} \tilde{\otimes} \tilde{\shuffle}: \spn_{\cR}\Big( \big( \scW_{0, d} \times \scW_{0, d}[I] \big) \tilde{\times} \big( \scW_{0, d} \times \scW_{0, d}[I] \big) \to \scW_{0, d} \tilde{\times} \scW_{0, d}[I]
	\end{equation*}
	which allows us to conclude the following:
	\begin{theorem}
		\label{theorem:M-coupledBialgebra-}
		Let $d\in \bN$, let $I$ be an index set and let $(\cR, +, \centerdot)$ be a ring. Then the shuffle product $\tilde{\shuffle}$ defined in Definition \ref{definition:shuffleproduct1} and the coupled coproduct $\Delta$ defined in Definition \ref{definition:deconcatenation-coproduct} satisfy 
		\begin{equation}
			\label{eq:theorem:M-coupledBialgebra-}
			\begin{aligned}
				\Delta \circ \tilde{\shuffle} =& \Big( \tilde{\shuffle} \tilde{\otimes} \tilde{\shuffle} \Big) \circ \overline{\mbox{Twist}} \circ \Big( \Delta \otimes \Delta \Big),
				\\
				\epsilon \circ \tilde{\shuffle} =& \centerdot \circ \epsilon \otimes \epsilon, \quad \Delta \circ \rId \circ \centerdot = \rId \otimes \rId, \quad \epsilon \otimes \rId = \fI_{\cR}. 
			\end{aligned}
		\end{equation}
	\end{theorem}
	
	\iftoggle{figure}{Theorem \ref{theorem:M-coupledBialgebra-} is equivalent to the operators $(\tilde{\shuffle}, \rId, \Delta, \epsilon)$ satisfying the commutative relationship described in Figure \ref{fig:coupled-bialgebra}. 
	\begin{figure}[htb]
		\centering
		\begin{tikzpicture}
			\node at (0,2) {$\spn_{\cR}\big( (\scW_{0,d}\times \scW_{0,d})[I] \big)$};
			\node at (0,0) {$\spn_{\cR}\big( \scW_{0,d}[I] \big)$};
			\node at (0,-2) {$\spn_{\cR} \big( (\scW_{0,d}\tilde{\times}\scW_{0,d})[I] \big)$};
			\node at (9,2) {$\spn_{\cR}\Big( \big(\scW_{0,d} \tilde{\times}\scW_{0,d}\big)[I] \times \big(\scW_{0,d} \tilde{\times} \scW_{0,d}\big)[I] \Big)$};
			\node at (9,-2) {$\spn_{\cR}\Big( \big(\scW_{0,d}\times\scW_{0,d}\big) \tilde{\times} \big(\scW_{0,d} \times \scW_{0,d}\big)[I] \Big)$};
			\draw[-to](0,1.6) to (0,0.4);
			\draw[-to](0,-0.4) to (0,-1.6);
			\draw[-to](9.5,1.6) to (9.5,-1.6);
			\draw[-to](2.4,2) to (4.6,2);
			\draw[-to](4.6,-2) to (2.4,-2);
			\node at (0.5,1) {$\tilde{\shuffle}$};
			\node at (0.5,-1) {$\Delta$};
			\node at (3.5,1.5) {$\Delta\otimes \Delta$};
			\node at (3.5,-1.5) {$\tilde{\shuffle} \tilde{\otimes} \tilde{\shuffle}$};
			\node at (8.5,0) {$\overline{\mbox{Twist}}$};
		\end{tikzpicture}	
		\\
		\begin{tikzpicture}
			\node at (0,0) {$\spn_{\cR}\big( (\scW_{0,d} \times \scW_{0,d})[I] \big)$};
			\node at (6,0) {$\spn_{\cR}\big( \scW_{0,d}[I] \big)$};
			\node at (3,-2) {$\cR \otimes \cR \equiv \cR$};
			\draw[-to](1,-0.4) to (2,-1.6);
			\draw[-to](5,-0.4) to (4,-1.6);
			\draw[-to](2.25,0) to (4.25,0);
			\node at (3.25,0.5) {$\tilde{\shuffle}$};
			\node at (1,-1) {$\epsilon \otimes \epsilon$};
			\node at (5,-1) {$\epsilon$};
		\end{tikzpicture}
		\\
		\begin{tikzpicture}
			\node at (0,0) {$\spn_{\cR}\big( (\scW_{0,d} \tilde{\times} \scW_{0,d})[I] \big)$};
			\node at (6,0) {$\spn_{\cR}\big( \scW_{0,d}[I] \big)$};
			\node at (3,2) {$\cR \otimes \cR \equiv \cR$};
			\draw[-to](2,1.6) to (1,0.4);
			\draw[-to](4,1.6) to (5,0.4);
			\draw[-to](4.25,0) to (2.25,0);
			\node at (3.25,0.5) {$\Delta$};
			\node at (1,1) {$\rId \tilde{\otimes} \rId$};
			\node at (5,1) {$\rId$};
		\end{tikzpicture}
		\\
		\begin{tikzpicture}
			\node at (0,0) {$\spn_{\cR}\big( \scW_{0,d}[I] \big)$};
			\node at (4,0) {$\cR$};
			\draw[->] (1,0.4) to[bend left] (3.75,0.4);
			\draw[->] (3.75,-0.4) to[bend left] (1,-0.4);
			\node at (2.375,1) {$\epsilon$};
			\node at (2.375,-1) {$\rId$};
		\end{tikzpicture}
		\caption{Coupled bialgebra}
		\label{fig:coupled-bialgebra}
	\end{figure}
	}
	
	\begin{proof}
		Let $I$ be an index set and let $W^1=(w^1, a^1), W^2=(w^2, a^2) \in \scW_{0, d}[I]$ such that $|w^1| = n_1$ and $|w^2| = n_2$. Then
		\begin{align}
			\nonumber
			\Delta \circ& \tilde{\shuffle} \Big[ W^1 \times W^2 \Big] =\Delta \Big[ W^1 \tilde{\shuffle} W^2 \Big]
			\\
			\nonumber
			=& \sum_{\sigma \in \Shuf(n_1, n_2)} \Delta \bigg[ \Big( \sigma(w^1 \cdot w^2), \big\llbracket \sigma(a^1 \sqcup_I a^2) \big\rrbracket_I \Big)  \bigg]
			\\
			\nonumber
			=& \sum_{\sigma \in \Shuf(n_1, n_2)} \sum_{k=0}^{n_1+n_2} \Bigg( \bigg( \big( \sigma(w^1 \cdot w^2)_{k+1}, ..., \sigma(w^1\cdot w^2)_{n_1+n_2} \big), 
			\\
			\nonumber
			&\qquad\qquad\qquad\qquad \Big\llbracket \Big( \big\llbracket \sigma(a^1 \sqcup_I a^2) \big\rrbracket_{I, i} \Big)_{i=k+1, ..., n_1+n_2} \Big\rrbracket_{I \cup m_k\big[ \llbracket \sigma(a^1 \sqcup_I a^2) \rrbracket_I \big]} \bigg), 
			\\
			\label{eq:proposition:M-coupledBialgebra-1}
			&\quad \bigg( \big( \sigma(w^1 \cdot w^2)_1, ..., \sigma(w^1 \cdot w^2)_k \big), \Big( \big\llbracket \sigma( a^1 \sqcup_I a^2 ) \big\rrbracket_{I, i} \Big)_{i=1, ..., k} \bigg) \Bigg). 
		\end{align}
		On the other hand, 
		\begin{align*}
			\Delta\big[ W^1 \big] =& \sum_{k_1=0}^{n_1} \Big( \big( (w_{k_1+1}^1, ..., w_{n_1}^1), \big\llbracket (a_{k_1+1}^1, ..., a_{n_1}^1) \big\rrbracket_{I \cup m_k\{a^1\}} \big), \big( (w_1^1, ..., w_k^1), (a_1^1, ..., a_k^1) \big) \Big)
			\quad\mbox{and}
			\\
			\Delta\big[ W^2 \big] =& \sum_{k_2=0}^{n_2} \Big( \big( (w_{k_2+1}^2, ..., w_{n_2}^2), \big\llbracket (a_{k_2+1}^2, ..., a_{n_2}^2) \big\rrbracket_{I \cup m_k\{a^2\}} \big), \big( (w_1^2, ..., w_k^2), (a_1^2, ..., a_k^2) \big) \Big)
		\end{align*}
		so that
		\begin{align*}
			&\overline{\mbox{Twist}} \circ \big( \Delta \otimes \Delta \big) \Big[ \big( W^1, W^2 \big) \Big] 
			\\
			&= \sum_{k_1=0}^{n_1} \sum_{k_2=0}^{n_2}\bigg( \Big( \big( (w_{k_1+1}^1, ..., w_{n_1}^1), \big\llbracket (a_{k_1+1}^1, ..., a_{n_1}^1) \big\rrbracket_{J_{k_1, k_2}} \big), 
			\big( (w_{k_2+1}^2, ..., w_{n_2}^2), \big\llbracket (a_{k_2+1}^2, ..., a_{n_2}^2) \big\rrbracket_{J_{k_1,k_2}} \big) \Big),
			\\
			&\qquad \Big( \big( (w_1^1, ..., w_{k_1}^1), (a_1^1, ..., a_{k_1}^1) \big), \big( (w_1^2, ..., w_{k_2}^2), (a_1^2, ..., a_{k_2}^2) \big) \Big) \bigg)
		\end{align*}
		where
		\begin{equation*}
			J_{k_1, k_2}:= I \cup m_{k_1}\{a^1\} \cup m_{k_2}\{a^2\}. 
		\end{equation*}
		Therefore
		\begin{align}
			\nonumber
			&\big( \tilde{\shuffle} \tilde{\otimes} \tilde{\shuffle} \big) \circ \overline{\mbox{Twist}} \circ \big( \Delta \otimes \Delta \big) \Big[ \big( W^1, W^2 \big) \Big]
			\\
			\nonumber
			&= \sum_{k_1=0}^{n_1} \sum_{k_2=0}^{n_2}  \sum_{\substack{\bar{\sigma} \in \Shuf( n_1 - k_1, n_2-k_2) \\ \hat{\sigma} \in \Shuf( k_1, k_2) }} \bigg( \Big( \bar{\sigma}\big( (w_{k_1+1}^1, ..., w_{n_1}^1) \cdot (w_{k_2+1}^2, ..., w_{n_2}^2) \big), 
			\\
			\nonumber
			&\qquad \qquad \Big\llbracket \bar{\sigma}\Big( \big\llbracket (a_{k_1+1}^1, ..., a_{n_1}^1) \big\rrbracket_{J_{k_1, k_2}} \sqcup_{J_{k_1, k_2}} \big\llbracket (a_{k_2+1}^2, ..., a_{n_2}^2) \big\rrbracket_{J_{k_1,k_2}}\Big) \Big\rrbracket_{J_{k_1, k_2}} \Big), 
			\\
			\label{eq:proposition:M-coupledBialgebra-2}
			&\qquad \Big( \hat{\sigma}\big( (w_1^1, ..., w_{k_1}^1) \cdot (w_1^2, ..., w_{k_2}^2) \big), \big\llbracket \hat{\sigma}\big( (a_1^1, ..., a_{k_1}^1) \sqcup_I (a_1^2, ..., a_{k_2}^2) \big) \big\rrbracket_I \Big) \bigg). 
		\end{align}
		Next, we can verify that
		\begin{equation*}
			J_{k_1, k_2} = m_{k_1+k_2}\Big\{ \big\llbracket \hat{\sigma}\big( (a_1^1, ..., a_{k_1}^1) \sqcup_I (a_1^2, ..., a_{k_2}^2) \big) \big\rrbracket_I \Big\}
		\end{equation*}
		and by applying the identity
		\begin{align}
			\nonumber
			&\sum_{k_1=0}^{n_1} \sum_{k_2=0}^{n_2} \sum_{\substack{\bar{\sigma} \in \Shuf(n_1 - k_1, n_2-k_2) \\ \hat{\sigma} \in \Shuf(k_1, k_2)}} \hspace{-20pt}\Big( \bar{\sigma}\big( (w_{k_1+1}^1, ..., w_{n_1}^1) \cdot (w_{k_2+1}^2, ..., w_{n_2}^2) \big) , \hat{\sigma}\big( (w_1^1, ..., w_{k_1}^1) \cdot (w_1^2, ..., w_{k_2}^2) \big) \Big)
			\\
			&= \sum_{\sigma \in \Shuf(n_1, n_2)} \sum_{k=0}^{n_1+n_2} \Big( \big( \sigma(w^1 \cdot w^2)_{k+1}, ..., \sigma(w^1 \cdot w^2)_{n} \big), \big( \sigma(w^1 \cdot w^2)_{1}, ..., \sigma(w^1 \cdot w^2)_{k} \big) \Big)
		\end{align}
		we can change the order of summation in Equation \eqref{eq:proposition:M-coupledBialgebra-2} to obtain Equation \eqref{eq:proposition:M-coupledBialgebra-1}. 
		
		The unit and counit identities are easy to verify and we skip here. 
	\end{proof}

	\subsubsection{Grading on Lions words}
	\label{subsubsection:Grading-Lionswords}
	
	The $\cR$-module spanned by Lyndon words has a grading determined by the length of Lyndon words. For this work, we need a grading that additionally captures the partition structure of a Lions word:	
	\begin{definition}
		\label{lemma:grading(1)}
		For any index set $I$, we define $\scG^I:\scW_{0,d}[I] \to \bN_0^{\times 2}$ by 
		\begin{equation}
			\label{eq:lemma:grading(1)}
			\scG^I \big[(w,a) \big]:= \big( |w|, m[a] \big)
		\end{equation}
		For $(k, n) \in \bN_0^{\times 2}$, we define
		\begin{equation*}
			\scW_{0, d}^{(k, n)}[I] := \Big\{ (w,a)\in \scW_{0, d}[I]: \scG^I[W] = (k, n ) \Big\}
		\end{equation*}
		and denote
		\begin{equation*}
			\big(\scW_{0, d}^{(k_2, n_2)} \tilde{\times} \scW_{0, d}^{(k_1, n_1)}\big)[I] := \bigsqcup_{W \in \scW_{0, d}^{(k_1, n_1)}[I]} \scW_{0, d}^{(k_2, n_2)}[W]
		\end{equation*}	
	\end{definition}
	In the next result, we demonstrate that this decomposition is natural to the underlying structure of the algebra operation $\tilde{\shuffle}$ and the coupled coproduct operation $\Delta$:
	\begin{proposition}
		\label{proposition:M-Grading-}
		Let $d\in \bN$, let $I$ be an index set and let $(\cR, +, \centerdot)$ be a ring. Then 
		\begin{equation}
			\label{eq:proposition:M-Grading-}
			\left.
			\begin{aligned}
				&\spn_{\cR}\Big( \scW_{0,d}^{(0, 0)}[I] \Big) = \cR, 
				\\
				&\spn_{\cR}\Big( \scW_{0, d}^{(k_1, n_1)}[I] \Big) \tilde{\shuffle} \spn_{\cR}\Big( \scW_{0, d}^{(k_2, n_2)}[I] \Big) \subseteq \spn_{\cR}\Big( \scW_{0, d}^{(k_1+k_2, n_1+n_2)} \Big),  
				\\
				&\Delta\Big[ \spn_{\cR}\big( \scW_{0, d}^{(k, n)}[I] \big) \Big] \subseteq  \bigoplus_{k' = 0}^k \bigoplus_{n'=0}^n \spn_{\cR}\Big( \scW_{0, d}^{(k-k', n-n')} \tilde{\times} \scW_{0, d}^{(k', n')}[I] \Big). 
			\end{aligned}
			\right\}
		\end{equation}
		Finally, for any choice of index set $I$ and $(k, n) \in \bN_0^{\times 2}$ we have that the set $\big| \scW_{0, d}^{(k, n)}[I] \big|$ is finite. 
	\end{proposition}

	\begin{remark}
		\label{remark:Link-Graded.Connected.Finite}
		Drawing on classical ideas from Hopf algebra theory (see \cite{cartier2021hopf}), we want to ultimately introduce the concepts of \emph{gradings}, \emph{connectedness} and \emph{local finiteness}. 
		
		For the moment, we say that $\scG^I$ describes a $\bN_0^{\times 2}$-grading on the locally finite connected coupled bialgebra
		\begin{equation*}
			\Big( \spn_{\cR}\big( \scW_{0, d}[I] \big), \tilde{\shuffle}, \rId, \Delta, \epsilon \Big)
		\end{equation*} 
		if Equation \eqref{eq:proposition:M-Grading-} is satisfied and $\big| \scW_{0, d}^{(k, n)}[I] \big|< \infty$. 
	\end{remark}

	\begin{proof}
		We proceed to address each of the identities in \eqref{eq:proposition:M-Grading-} individually:
		
		\textit{Step 1.} Firstly, since for any $W\in \scW_{0, d}[I]$,
		\begin{equation*}
			\scG^I(W) = (0, 0) \quad \iff \quad W = \rId. 
		\end{equation*}
		As such, 
		\begin{equation*}
			\scW_{0, d}^{(0, 0)}[I] = \big\{\rId \big\}
			\quad \mbox{so that} \quad 
			\spn_{\cR}\Big( \scW_{0, d}^{(0, 0)}[I] \Big) = \cR. 
		\end{equation*}
		
		\textit{Step 2.} Let $k_1, k_2, n_1, n_2  \in \bN_0$ and suppose that 
		\begin{align*}
			W^1 = (w^1, a^1) \in \scW_{0,d}[I]& \quad \mbox{such that} \quad \scG^I[W^1] = (k_1, n_1)
			\quad \mbox{and}
			\\
			W^2 = (w^2, a^2) \in \scW_{0,d}[I]& \quad \mbox{such that} \quad \scG^I[W^2] = (k_2, n_2). 
		\end{align*}
		Then
		\begin{equation*}
			W^1 \tilde{\shuffle} W^2 = \sum_{\sigma \in \Shuf(|W^1|, |W^2|)} \sigma(W^1, W^2) 
		\end{equation*}	
		and for any $\sigma \in \Shuf\big( |W^1|, |W^2| \big)$
		\begin{align*}
			\scG^I\Big[ \sigma(W^1, W^2) \Big] =& \Big( \big| \sigma(w^1\cdot w^2) \big| , m\big[ \llbracket \sigma(a^1 \sqcup_I a^2) \rrbracket \big] \Big)
			\\
			=& \big( |w^1| + |w^2|, m[a^1] + m[a^2] \big) = (k_1+k_2, n_1+n_2)
		\end{align*}
		and we conclude. 
		
		\textit{Step 3.} For any $W = (w, a) \in \scW_{0,d}[I]$ and any coupled pair $(\bar{W}, \hat{W}) \in \scW_{0,d} \tilde{\times} \scW_{0,d}[I]$ such that
		\begin{equation*}
			\Big\langle \Delta[W], (\bar{W}, \hat{W}) \Big\rangle > 0 \quad \implies \quad m[\bar{a}] + m[\hat{a}] = m[a] 
			\quad \mbox{and}\quad
			|\bar{W}| + |\hat{W}| = |W|. 
		\end{equation*}
		Hence for any $W \in \scW_{0, d}[I]$ such that $\scG^I[W] = (k, n)$ we obtain that
		\begin{equation*}
			\Delta \Big[ W \Big] \in \bigoplus_{k'=0}^k \bigoplus_{n'=0}^n \spn_{\cR}\Big( \scW_{0, d}^{(k - k', n - n')} \tilde{\times} \scW_{0, d}^{(k',n')}[I] \Big). 
		\end{equation*}
		and we conclude that Equation \eqref{eq:proposition:M-Grading-} holds. The finiteness of the set $|\scW_{0, d}^{(k_I, n)}[I]|$ is immediate. 
	\end{proof}
	
	While these infinite expansions are more favourable to work with in terms of their symmetry relations, in practice we want to work with truncated expansions. Let $g: \bN_0^{\times 2} \to \bR$ be a monotone increasing function such that $g(0, 0) \leq 0$. We denote
	\begin{equation}
		\label{eq:truncatedWord}
		\begin{aligned}
			\scW_{0, d}^{g,+}[I]:=&\Big\{ W \in \scW_{0, d}[I]: g\big( \scG^I[W] \big) > 0 \Big\}, 
			\quad \mbox{and}
			\\
			\scW_{0, d}^{g,-}[I]:=&\Big\{ W \in \scW_{0, d}[I]: g\big( \scG^I[W] \big) \leq 0 \Big\}. 
		\end{aligned}
	\end{equation}
	
	\begin{corollary}
		\label{lemma:M-Finite-grading-}
		Let $d\in \bN$, let $I$ be an index set and let $(\cR, +, \centerdot)$ be a ring. Let $g: \bN_0^{\times 2}\to \bR$ be a monotone increasing function such that $g(0, 0) \leq 0$. Then 
		\begin{enumerate}[label=(\ref*{lemma:M-Finite-grading-}.\roman*)]
			\item
			\label{enum:lemma:M-Finite-grading-1}
			The $\cR$-module
			\begin{equation*}
				\spn_{\cR}\Big( \scW_{0, d}^{g,+} \Big)
				\quad \mbox{is an algebra ideals of } \quad
				\Big( \spn_{\cR}\big( \scW_{0, d}[I]\big), \tilde{\shuffle}, \rId \Big)
			\end{equation*}
			and we can identify the algebra over the $\cR$-module quotient by
			\begin{align*}
				\spn_{\cR}\Big( \scW_{0, d}^{g,-}[I] \Big) =& \spn_{\cR} \Big( \scW_{0, d}[I] \Big) / \spn_{\cR}\Big( \scW_{0, d}^{g,+}[I] \Big) 
			\end{align*}
			\item 
			\label{enum:lemma:M-Finite-grading-2}
			The coupled coproduct and counit $(\Delta, \epsilon)$ restricted to the sub-module
			\begin{equation*}
				\Big( \spn_{\cR}\big( \scW_{0, d}^{g,-}[I] \big), \Delta, \epsilon \Big)
			\end{equation*}
			is a co-associative sub-coupled coalgebras of $\big( \spn_{\cR}( \scW_{0, d}[I] ), \Delta, \epsilon \big)$. 
			\item 
			\label{enum:lemma:M-Finite-grading-3}
			By pairing the quotient algebra and unit $(\tilde{\shuffle}, \rId)$ to the restriction of the coupled coproduct and counit $(\Delta, \rId)$, we obtain
			\begin{equation*}
				\Big( \spn_{\cR}\big( \scW_{0, d}^{g,-}[I] \big), \tilde{\shuffle}, \rId, \Delta, \epsilon \Big)
			\end{equation*}
			satisfies the commutative identities of Equation \eqref{eq:theorem:M-coupledBialgebra-}. 
		\end{enumerate}
	\end{corollary}
	
	\begin{proof}
		\ref{enum:lemma:M-Finite-grading-1}: Thanks to Proposition \ref{proposition:M-Grading-} and the monotoncity of $g$, we have for any $W^1\in \scW_{0, d}^{g,+}[I]$ and $W^2 \in \scW_{0, d}[I]$ that
		\begin{equation*}
			W^1 \tilde{\shuffle} W^2 \in \spn_{\cR}\Big( \scW_{0, d}^{g,+} \Big). 
		\end{equation*}
		Thus the sub $\cR$-module $\spn_{\cR}\big( \scW_{0, d}^{g,+}[I] \big)$ is an algebra ideal over $\cR$ and the conclusion follows from the disjointedness of the sets $\scW_{0, d}^{g, -}[I]$ and $\scW_{0, d}^{g, +}[I]$. 
		
		\ref{enum:lemma:M-Finite-grading-2}: In the same fashion, Proposition \ref{proposition:M-Grading-} and the monotinicity of $g$ implies that for any $W\in \scW_{0, d}^{g,-}[I]$ and any $k \in \{1, ..., |W| \}$ we have that
		\begin{align*}
			g\Big(& \scG^I\big[ \big( (w_1, ...., w_k), (a_1, ..., a_k) \big) \big] \Big) \vee g\Big( \scG^I\big[ \big( (w_{k+1}, ..., w_{|w|}), \big\llbracket (a_{k+1}, ..., a_{|w|}) \big\rrbracket_{I\cup m_k\{a\}} \big) \big] \Big) 
			\\
			&\leq g\Big( \scG^I\big[ W \big] \Big)
		\end{align*}
		so that 
		\begin{equation*}
			\Delta\Big[ \spn_{\cR}\big( \scW_{0, d}^{g, -}[I] \big) \Big] \subseteq \spn_{\cR}\big( \scW_{0, d}^{g, -}[I] \big) \tilde{\otimes} \spn_{\cR}\big( \scW_{0, d}^{g, -}[I] \big)
		\end{equation*}
		and thus $\big( \spn_{\cR}( \scW_{0, d}^{g, -}[I] ), \Delta, \epsilon \big)$ is a sub-coupled coalgebra. 
		
		\ref{enum:lemma:M-Finite-grading-3}: This follows from Proposition \ref{proposition:M-Grading-} and Theorem \ref{theorem:M-coupledBialgebra-}. 
	\end{proof}
	
	\iftoggle{Plus}{
	\subsection{Modules of measurable functions indexed by Lions words}
	
	To streamline notation, we define the coproduct counting function 
	\begin{equation}
		\label{eq:definition:deconcatenation-coproduct-counting}
		c_I: \scW_{0, d}[I] \times \big( \scW_{0, d}[I] \tilde{\times} \scW_{0, d}[I] \big) \to \bN_0
		\quad \mbox{by}\quad
		c_I\Big( W, \bar{W}, \hat{W} \Big) = \Big\langle \Delta\big[ W \big] , (\bar{W}, \hat{W}) \Big\rangle. 
	\end{equation}
			
	\begin{assumption}
		\label{notation:V-(2)}
		Let $(\Omega, \cF, \bP)$ and $(\Omega', \cF', \bP')$ be probability spaces. Let $d \in \bN$ and let $(\cR, +, \centerdot)$ be a separable normed ring which we associate with the Borel $\sigma$-algebra $\cB(\cR)$. 
		
		For any index set $I$, let
		\begin{equation*}
			\Big( \bV^W\big( \Omega; \cR \big) \Big)_{W\in \scW_{0, d}[I]}
		\end{equation*}
		be a collection of unital $\cR$-modules that satisfy:
		\begin{enumerate}[label=(\ref*{notation:V-(2)}.\roman*)]
			\item \label{enum:notation:V-(2)-1}
			For every $W\in \scW_{0,d}[I]$, we have
			\begin{equation*}
				\bV^W\big( \Omega; \cR \big) \subseteq L^{\boldsymbol{0}}\Big( \Omega^{\times \fm[W]}; \cR \Big); 
			\end{equation*}
			\item \label{enum:notation:V-(2)-2}
			For any $W\in \scW_{0, d}[I]$ such that $\fm[W] = 0$, we have that
			\begin{equation*}
				\bV^{W}\big( \Omega; \cR \big) = \cR; 
			\end{equation*}
			\item \label{enum:notation:V-(2)-3}
			For any $W^1, W^2 \in \scW_{0, d}[I]$ and any $\sigma \in \Shuf\big( |W^1|, |W^2| \big)$, we have that
			\begin{equation*}
				\bV^{W^1}\big( \Omega; \cR \big) \otimes \bV^{W^2}\big( \Omega; \cR \big) \subseteq \bV^{\sigma(W^1, W^2)} \big(\Omega; \cR \big); 
			\end{equation*}
			\item \label{enum:notation:V-(2)-4}
			For every $W\in \scW_{0, d}[I]$ and for any
			\begin{equation*}
				(\bar{W}, \hat{W}) \in \scW_{0, d} \tilde{\times} \scW_{0, d}[I] 
				\quad\mbox{such that} \quad 
				c_I\Big( W, \bar{W}, \hat{W} \Big)>0,
			\end{equation*}
			we have that
			\begin{equation*}
				\bV^{W}\Big( \Omega; \cR \Big) \subseteq \bV^{\hat{W}}\Big( \Omega; \bV^{\bar{W}}\big( \Omega; \cR \big) \Big). 
			\end{equation*}
		\end{enumerate}
	\end{assumption}
	
	\begin{example}
		\label{example:L1-Module1}
		Let $(\Omega, \cF, \bP)$ be a probability space and consider the normed ring $(\bR, +, \centerdot)$. Suppose for $W = (w, a) \in \scW_{0, d}[I]$ that we choose
		\begin{equation*}
			\bV^{W}\Big(\Omega; \bR \Big) = L^{\boldsymbol{0}}\Big( \Omega^{\times m[a]}, \bP^{\times m[a]};  \bR \Big). 
		\end{equation*}
		Measurable (rather than integrable) functions serve as a simple example since we can take products of measurable functions to obtain another measurable function whereas we cannot take the product of two integrable functions and obtain another integrable function without having to resort to H\"older type inequalities which reduce the amount of integrability of the product. 
	
		Firstly, we immediately have that for any $W\in \scW_{0, d}[I]$, 
		\begin{equation*}
			L^{\boldsymbol{0}}\Big( \Omega^{\times \fm[W]}, \bP^{\times \fm[W]}; \cR \Big)
			\subseteq
			L^{\boldsymbol{0}}\Big( \Omega^{\times \fm[W]}, \bP^{\times \fm[W]}; \cR \Big)
		\end{equation*}
		so that \ref{enum:notation:V-(2)-1} is satisfied. Similarly, using the convention that
		\begin{equation*}
			L^{(\emptyset)}\Big( \Omega^{\times 0}, \bP^{\times 0}; \cR \Big) = \cR
		\end{equation*}
		gives \ref{enum:notation:V-(2)-2}. Next, for any $W^1, W^2 \in \scW_{0,d}[I]$ and $\sigma \in \Shuf\big( |W^1|, |W^2|\big)$ we have thanks to Equation \eqref{eq:definition:shuffleproduct1} that $\fm[W^1] + \fm[W^2] = \fm[\sigma(W^1, W^2)]$ so that
		\begin{align*}
			\bV^{W^1}\big( \Omega; \cR \big) \otimes \bV^{W^2}\big( \Omega; \cR \big) =& L^{\boldsymbol{0}} \Big(\Omega^{\times \fm[W^1]}, \bP^{\times \fm[W^1]}; \cR \Big) \otimes L^{\boldsymbol{0}} \Big(\Omega^{\times \fm[W^2]}, \bP^{\times \fm[W^2]}; \cR \Big)
			\\
			\subseteq& L^{\boldsymbol{0}} \Big(\Omega^{\times (\fm[W^1]+\fm[W^2])}, \bP^{\times (\fm[W^1]+\fm[W^2])}; \cR \Big)
			\\
			\subseteq& L^{\boldsymbol{0}} \Big(\Omega^{\times \fm[\sigma(W^1,W^2)]}, \bP^{\times \fm[\sigma(W^1,W^2)]}; \cR \Big) = \bV^{\sigma(W^1, W^2)} \big(\Omega; \cR \big)
		\end{align*}
		and \ref{enum:notation:V-(2)-3} holds. 
		
		Finally, let $W = \big( (w_1, ..., w_n), (a_1, ..., a_n) \big) \in \scW_{0, d}[I]$ and suppose that $c_I(W, \bar{W}, \hat{W} )>0$ for some $(\bar{W}, \hat{W}) \in \scW_{0, d} \tilde{\times} \scW_{0, d}[I]$. Then $\exists k \in \{0, 1, ..., n\}$ such that
		\begin{align*}
			&(\bar{W}, \hat{W}) = \Big( \big( (w_{k+1}, ..., w_n), \big\llbracket (a_{k+1}, ..., a_n) \big\rrbracket_{I \cup m_k\{a\}} \big) , \big( (w_{1}, ..., w_k), (a_1, ..., a_k) \big) \Big)
			\\
			&\mbox{and} \quad \fm\big[ (\bar{W}, \hat{W}) \big] = \fm[\hat{W}] + \fm[\bar{W}]. 
		\end{align*}
		Then we can use the $\cR$-module identity that
		\begin{equation*}
			L^{\boldsymbol{0}}\Big( \Omega^{\times (m+n)}, \bP^{\times (m+n)}; \cR \Big) = L^{\boldsymbol{0}}\bigg( \Omega^{\times m}, \bP^{\times m}; L^{\boldsymbol{0}}\Big( \Omega^{\times n}, \bP^{\times n}; \cR \Big) \bigg)
		\end{equation*}
		to conclude that 
		\begin{equation*}
			\bV^{W}\Big( \Omega; \cR \Big) \subseteq \bV^{\hat{W}}\Big( \Omega; \bV^{\bar{W}}\big( \Omega; \cR \big) \Big). 
		\end{equation*}
		Thus \ref{enum:notation:V-(2)-4} follows. 
	\end{example}

	The purpose of Assumption \ref{notation:V-(2)} is to provide a collection of $\cR$-modules that will be used to construct the following module:
	\begin{definition}
		\label{definition:coupled-shuffle}
		Let $(\Omega, \cF, \bP)$ and $(\Omega', \cF, \bP')$ be probability spaces. Let $(\cR, +, \centerdot)$ be a normed ring and let $I$ be an index set. 
		
		Suppose that $\bU\big( \Omega; \cR \big)$ satisfies Assumption \ref{notation:U-(1)} and that $\big( \bV^W \big)_{W\in \scW_{0, d}[I]}$ are a collection of modules that satisfies Assumption \ref{notation:V-(2)}.  
		
		We define the $\bU(\Omega; \cR)$-module
		\begin{equation}
			\label{eq:definition:coupled-shuffle-1}
			\scM_I(\Omega, \Omega') := \bU\bigg( \Omega; \bigoplus_{W \in \scW_{0, d}[I]} \bV^{W}\Big( \Omega'; \cR \Big) \bigg). 
		\end{equation}
	
		We use the representation that for $X \in \scM_I(\Omega, \Omega')$, we write
		\begin{equation*}
			X(\omega_I) = \sum_{W \in \scW_{0, d}[I]} \Big\langle X, W \Big\rangle(\omega_I, \cdot)
			\quad \mbox{for}\quad
			\omega_I:=(\underbrace{\omega_{\iota}, ...}_{\iota \in I}) \in (\Omega)^{\times |I|}
		\end{equation*}
		where
		\begin{equation}
			\begin{aligned}
				&\Big\langle X, W \Big\rangle \in \bU\bigg( \Omega; \bV^{W}\Big( \Omega'; \cR \Big) \bigg) 
				\quad \mbox{or} \quad
				\Big\langle X, W \Big\rangle(\omega_I, \cdot) \in \bV^W\Big( \Omega'; \cR \Big). 
			\end{aligned}
		\end{equation}
		
		The ring operators $+: \scM_I(\Omega, \Omega') \times \scM_I(\Omega, \Omega') \to \scM_I(\Omega, \Omega')$ is defined for $X, Y \in \scM_I(\Omega, \Omega')$ by
		\begin{equation}
			\label{eq:definition:coupled-shuffle+}
			(X+Y)(\omega_I) 
			= \sum_{W \in \scW_{0, d}[I]} \Big\langle (X+Y), W \Big\rangle(\omega_I, \cdot)
			= \sum_{W \in \scW_{0, d}[I]} \bigg( \Big\langle X, W \Big\rangle(\omega_I, \cdot) + \Big\langle Y, W \Big\rangle(\omega_I, \cdot) \bigg)
		\end{equation}
		and $\centerdot: \bU(\Omega; \cR) \times \scM_I(\Omega, \Omega') \to \scM_I(\Omega, \Omega')$ is defined for $R \in \bU(\Omega; \cR)$ and $X \in \scM_I(\Omega, \Omega')$ by
		\begin{equation}
			\label{eq:definition:coupled-shufflex}
			(R \centerdot X)(\omega_I) 
			= \sum_{W \in \scW_{0, d}[I]} \Big\langle (R \centerdot X), W \Big\rangle(\omega_I, \cdot)
			=\sum_{W \in \scW_{0, d}[I]} R(\omega_I) \centerdot \Big\langle X, W \Big\rangle(\omega_I, \cdot). 
		\end{equation}
	\end{definition}
	We should view $(\Omega, \cF, \bP)$ as the \emph{tagged probability space} in contrast to $(\Omega', \cF', \bP')$ which is viewed as the \emph{free probability space}. 
	
	Returning to Equation \eqref{eq:definition:coupled-shuffle+} again, we emphasise that
	\begin{align*}
		&\Big\langle X, W \Big\rangle, \Big\langle Y, W \Big\rangle\in \bU\bigg( \Omega; \bV^{W}\Big( \Omega'; \cR \Big) \bigg) \quad \mbox{so that}
		\\
		\mbox{for $\omega_I \in (\Omega)^{\times |I|}$} \quad &\Big\langle X, W \Big\rangle(\omega_I, \cdot), \Big\langle Y, W \Big\rangle(\omega_I, \cdot) \in \bV^{W}\Big( \Omega'; \cR \Big), 
	\end{align*}
	and similarly for Equation \eqref{eq:definition:coupled-shufflex}
	\begin{align*}
		&R \in \bU\big( \Omega; \cR \big), \quad \Big\langle X, W \Big\rangle \in \bU\bigg( \Omega; \bV^W\Big( \Omega'; \cR \Big) \bigg) \quad \mbox{so that}
		\\
		\mbox{for $\omega_I \in (\Omega)^{\times |I|}$} \quad &R(\omega_I) \in \cR , \quad \Big\langle X, W \Big\rangle(\omega_I, \cdot) \in \bV^W\Big( \Omega'; \cR \Big). 
	\end{align*}
	
	\subsubsection{Algebra over $\scM_I$}
	Following on from Definition \ref{definition:shuffleproduct1}, we extend $\scM_I(\Omega, \Omega')$ to be an algebra over the ring $\bU(\Omega; \cR)$:
	\begin{definition}
		\label{definition:shuffleproduct2}
		Let $(\Omega, \cF, \bP)$ and $(\Omega', \cF, \bP')$ be probability spaces. Let $(\cR, +, \centerdot)$ be a normed ring and let $I$ be an index set. 
		
		Suppose that $\bU\big( \Omega; \cR \big)$ satisfies Assumption \ref{notation:U-(1)} and $\big( \bV^W \big)_{W\in \scW_{0, d}[I]}$ is a collection of modules that satisfies Assumption \ref{notation:V-(2)}. 
		
		We define $\tilde{\shuffle}: \scM_I(\Omega, \Omega') \times \scM_I(\Omega, \Omega') \to \scM_I(\Omega, \Omega')$ be the bilinear mapping that satisfies the identity
		\begin{align*}
			&X \tilde{\shuffle} Y(\omega_I) = \sum_{W\in \scW_{0, d}[I]} \Big\langle X \tilde{\shuffle} Y, W \Big\rangle(\omega_I, \cdot)
			\quad \mbox{where}
			\\
			&\Big\langle X \tilde{\shuffle} Y, W \Big\rangle(\omega_I, \cdot) = \sum_{\substack{W^1, W^2 \in \scW_{0, d}[I] \\ \exists \sigma \in \Shuf(|W^1|, |W^2|) \\ \sigma(W^1, W^2) = W}} \Big\langle X, W^1 \Big\rangle(\omega_I, \cdot) \otimes \Big\langle Y, W^2 \Big\rangle(\omega_I, \cdot). 
		\end{align*}
	\end{definition}
	We use the convention that for any $W = (w, a) \in \scW_{0, d}[I]$ the sample space element 
	\begin{equation*}
		(\omega'_{m\{a\}}) = \big( \omega'_{1}, ..., \omega'_{m[a]} \big) \in (\Omega')^{\times m[a]}. 
	\end{equation*}
	
	In particular, because of Assumption  \ref{enum:notation:V-(2)-3}
	\begin{align*}
		\Big\langle X_1&, (w^1, a^1) \Big\rangle(\omega_I, \omega'_{m\{a^1\}} ) \otimes \Big\langle X_2, (w^2, a^2) \Big\rangle(\omega_I, \omega'_{m\{a^2\}})
		\\
		\in& \bU\bigg( \Omega;  \bV^{W^1}\Big( \Omega'; \cR \Big) \otimes \bV^{W^2}\Big( \Omega'; \cR \Big)  \bigg)
		\subseteq
		\bU\bigg( \Omega; \bigoplus_{\substack{\sigma(W^1, W^2) \in \scW_{0, d}[I] \\  \sigma \in \Shuf(W^1, W^2)}} \bV^{\sigma(W^1, W^2)}\Big( \Omega'; \cR \Big) \bigg). 
	\end{align*}
	
	In this work, it will be important to distinguish free variables but not to order them as is common convention. For this reason, we index variables (both free and tagged) in terms of a set (unordered collection) rather than a sequence (ordered collection). 
	\begin{proposition}
		\label{proposition:associativity(1)}
		Let $(\Omega, \cF, \bP)$ and $(\Omega', \cF, \bP')$ be probability spaces. Let $(\cR, +, \centerdot)$ be a normed ring and let $I$ be an index set. 
		
		Suppose that $\bU\big( \Omega; \cR \big)$ satisfies Assumption \ref{notation:U-(1)} and $\big( \bV^W \big)_{W\in \scW_{0, d}[I]}$ is a collection of modules that satisfies Assumption \ref{notation:V-(2)}. Then
		\begin{enumerate}
			\item $\scM_I(\Omega, \Omega')$ as defined in Equation \eqref{eq:definition:coupled-shuffle-1} is a $\Big( \bU\big( \Omega; \cR \big), +, \centerdot \Big)$ module. 
			\item $\big( \scM_I(\Omega, \Omega'), \tilde{\shuffle} , \rId \big)$ is a \emph{associative} \emph{unital} algebra over the ring $\bU(\Omega; \cR)$. 
			\item If $\bU(\Omega; \cR)$ is commutative, then $\big( \scM_I(\Omega, \Omega'), \tilde{\shuffle} , \rId \big)$ is a commutative algebra over the ring. 
		\end{enumerate}
	\end{proposition}
	
	\begin{proof}
		For any $X\in \scM_I(\Omega, \Omega')$ and $R\in \bU(\Omega; \cR)$, we note that for any $W\in \scW_{0, d}[I]$ that
		\begin{equation*}
			\Big\langle X, W \Big\rangle \in \bU\Big( \Omega; \bV^{W}\big( \Omega'; \cR \big) \Big) 
			\quad \mbox{and}\quad
			R\in \bU\big( \Omega; \cR \big)
		\end{equation*}
		so that
		\begin{equation*}
			R \centerdot \Big\langle X, W \Big\rangle \in \bU\Big( \Omega; \bV^W\big( \Omega'; \cR \big) \Big). 
		\end{equation*}
		Hence by defining
		\begin{align*}
			\Big\langle R \centerdot X, W \Big\rangle = R \centerdot \Big\langle X, W \Big\rangle
			\quad \mbox{and} \quad
			R \centerdot X = \sum_{W \in \scW_{0, d}[I]} \Big\langle R \centerdot X, W \Big\rangle, 
		\end{align*}
		we are able to extend the ring multiplication to $\centerdot : \bU(\Omega; \cR) \times \scM_I(\Omega, \Omega') \to \scM_I(\Omega, \Omega')$ and it is easy to verify that $\scM_I(\Omega, \Omega')$ is indeed a module. 
		
		For associativity, for any $X, Y, Z \in \scM_I(\Omega, \Omega')$ and $W \in \scW_{0, d}[I]$, we have that
		\begin{align*}
			\Big\langle \big( X \tilde{\shuffle} Y \big) \tilde{\shuffle} Z, W \Big\rangle 
			&= 
			\sum_{\substack{W^1, W^2 \in \scW_{0, d}[I] \\ \exists \sigma \in \Shuf(|W^1|, |W^2|) \\ \sigma(W^1, W^2) = W}} \Big\langle \big( X \tilde{\shuffle} Y \big), W^1 \Big\rangle \otimes \Big\langle Z, W^2 \Big\rangle
			\\
			&=
			\sum_{\substack{W', W^1\in \scW_{0, d}[I] \\ \exists \sigma \in \Shuf(|W'|, |W^1) \\ \sigma(W', W^1) = W}}
			\sum_{\substack{W^2, W^3\in \scW_{0, d}[I] \\ \exists \sigma' \in \Shuf(|W^2|, |W^3|) \\ \sigma'(W^2, W^3) = W'}} \Big\langle X, W^1 \Big\rangle \otimes \Big\langle Y, W^2 \Big\rangle \otimes \Big\langle Z, W^3 \Big\rangle 
			\\
			&= 
			\sum_{\substack{W^1, W^2 \in \scW_{0, d}[I] \\ \exists \sigma \in \Shuf(|W^1|, |W^2|) \\ \sigma(W^1, W^2) = W}} \Big\langle X , W^1 \Big\rangle \otimes \Big\langle Y \tilde{\shuffle} Z, W^2 \Big\rangle
			=
			\Big\langle X \tilde{\shuffle} \big( Y  \tilde{\shuffle} Z\big), W \Big\rangle 
		\end{align*}
		thanks to Proposition \ref{proposition:Shuffle=Assoc} and \ref{enum:notation:V-(2)-3}. Thus $\tilde{\shuffle}$ is associative and further the empty word $\rId \in \scW_{0, d}[I]$ satisfies that for any $W\in \scW_{0, d}[I]$
		\begin{equation*}
			\rId \tilde{\shuffle} W = W \tilde{\shuffle} \rId = W 
		\end{equation*}
		Let $Y \in \scM_I(\Omega, \Omega')$ defined by
		\begin{equation*}
			\Big\langle Y, \rId \Big\rangle(\omega_I) = 1_\cR \quad \bP\mbox{-a.s}
			\quad \mbox{and}\quad
			\Big\langle Y, W \Big\rangle(\omega_I, \cdot ) = 0. 
		\end{equation*}
		Then for any $X\in \scM_I(\Omega, \Omega')$
		\begin{equation*}
			X \tilde{\shuffle} Y = Y \tilde{\shuffle} X = X
		\end{equation*}
		so that $\big( \scM_I(\Omega, \Omega'), \tilde{\shuffle}, \rId \big)$ is unital. 
		
		For commutativity, suppose that $\bU(\Omega; \cR)$ is commutative. Then for any $X, Y \in \scM_I(\Omega, \Omega')$ and $W\in \scW_{0, d}[I]$, 
		\begin{align*}
			\Big\langle X \tilde{\shuffle} Y, W \Big\rangle 
			&= 
			\sum_{\substack{W^1, W^2 \in \scW_{0, d}[I] \\ \exists \sigma' \in \Shuf(|W^2|, |W^1|) \\ \sigma'(W^2, W^1) = W}} \Big\langle X, W^2 \Big\rangle \otimes \Big\langle Y, W^1 \Big\rangle
			\\
			&= 
			\sum_{\substack{W^1, W^2 \in \scW_{0, d}[I] \\ \exists \sigma' \in \Shuf(|W^2|, |W^1|) \\ \sigma'(W^2, W^1) = W}} \Big\langle Y, W^1 \Big\rangle \otimes \Big\langle X, W^2 \Big\rangle 
			= 
			\Big\langle Y \tilde{\shuffle} X, W \Big\rangle
		\end{align*}
		thanks to Proposition \ref{proposition:Shuffle=Assoc} and \ref{enum:notation:V-(2)-3}. Thus, $X\tilde{\shuffle} Y = Y \tilde{\shuffle} X$. 
	\end{proof}

	\subsubsection{The coupled tensor product of $\scM_I$}
	
	We have seen that the module of Lions words has a meaningful algebra operation and that the product of measurable functions is defined on some product measure space. For the next research component, we are interested in additionally allowing for couplings between the collection of free variables indexed by partitions. 
	
	\begin{remark}
		\label{remark:taggedTensorproduct}
		Before we talk more about coupled tensor products, we briefly remind the reader the non-trivial fact about tensor products of measurable functions: Let $(\cR, +, \centerdot)$ be a commutative unital normed ring. Firstly, recall that by treating $\cR$ as the (trivial) $\cR$-module we obtain that the $\cR$-module tensor product $\cR \otimes \cR \equiv \cR$. 
		
		Let $(\Omega, \cF, \bP)$, $(\Omega', \cF', \bP')$ and $(\hat{\Omega}, \hat{\cF}, \hat{\bP})$ be three probability spaces and let
		\begin{equation*}
			L^0\Big( \Omega', \bP'; \cR \Big) 
			\quad \mbox{and}\quad
			L^0\Big( \hat{\Omega}, \hat{\bP}; \cR \Big)
		\end{equation*}
		be two $\cR$-modules of measurable functions with codomain equal to the (trivial) $\cR$-module $\cR$. Then the tensor product of these two $\cR$-modules is the $\cR$-module
		\begin{equation*}
			L^0\Big( \Omega', \bP'; \cR \Big) \otimes L^0\Big( \hat{\Omega}, \hat{\bP}; \cR \Big) = L^0\Big(\Omega' \times \hat{\Omega}, \bP' \times \hat{\bP}; \cR \Big)
		\end{equation*}
		On the other hand, the two $L^0(\Omega; \cR)$-modules have tensor product
		\begin{equation*}
			L^0\bigg( \Omega, \bP; L^0\Big( \Omega', \bP'; \cR \Big) \bigg) \otimes L^0\bigg( \Omega, \bP; L^0\Big( \hat{\Omega}, \hat{\bP}; \cR \Big) \bigg) = L^0\bigg( \Omega, \bP; L^0\Big( \Omega' \times \hat{\Omega}, \bP' \times \hat{\bP}; \cR \Big) \bigg). 
		\end{equation*}
		We include this to illustrate that considering a larger underlying ring leads to a smaller tensor product between two modules. 
	\end{remark}
	
	\begin{definition}
		\label{definition:M-coupledTensorModule}
		Let $(\Omega, \cF, \bP)$ and $(\Omega', \cF, \bP')$ be probability spaces. Let $(\cR, +, \centerdot)$ be a unital normed ring and let $I$ be an index set. 
		
		Suppose that $\bU\big( \Omega; \cR \big)$ satisfies Assumption \ref{notation:U-(1)} and $\big( \bV^W \big)_{W\in \scW_{0, d}[I]}$ is a collection of modules that satisfies Assumption \ref{notation:V-(2)}. We define
		\begin{equation}
			\label{eq:definition:M-coupledTensorModule}
			\scM_I \tilde{\otimes} \scM_I (\Omega, \Omega') = \bU\Bigg( \Omega; \bigoplus_{\hat{W} \in \scW_{0, d}[I]} \bV^{\hat{W}}\bigg( \Omega'; \bigoplus_{\bar{W} \in \scW_{0, d}[\hat{W}] } \bV^{\bar{W}}\Big( \Omega' ; \cR \Big) \bigg) \Bigg). 
		\end{equation}
		We use the representation that for $X \in \scM_I\tilde{\otimes}\scM_I(\Omega, \Omega')$, we write
		\begin{align*}
			X(\omega_I) =& \sum_{\substack{\bar{W} \in \scW_{0, d}[\hat{W}] \\ \hat{W}\in \scW_{0, d}[I]}} \Big\langle X, (\bar{W}, \hat{W}) \Big\rangle(\omega_I, \cdot)
		\end{align*}
		where for $\hat{W}=(\hat{w},\hat{a}) \in \scW_{0, d}[I]$ and $\bar{W}=(\bar{w},\bar{a}) \in \scW_{0, d}[\hat{W}]$
		\begin{align*}
			&\Big\langle X, (\bar{W}, \hat{W}) \Big\rangle(\omega_I, \omega_{m\{\hat{a}\}}', \cdot ) 
			\in 
			\bV^{\bar{W}}\Big( \Omega'; \cR \Big),
			\\
			\sum_{\bar{W} \in \scW_{0,d}[\hat{W}]} &\Big\langle X, (\bar{W}, \hat{W}) \Big\rangle(\omega_I, \omega_{m\{\hat{a}\}}', \cdot ) 
			\in 
			\bigoplus_{\bar{W} \in \scW_{0,d}[\hat{W}]} \bV^{\bar{W}}\Big( \Omega'; \cR \Big),
		\end{align*}
		and
		\begin{align*}
			\sum_{\substack{\bar{W} \in \scW_{0,d}[\hat{W}] \\ \hat{W} \in \scW_{0, d}[I]}} & \Big\langle X, (\bar{W}, \hat{W}) \Big\rangle (\omega_I, \cdot, \cdot)
			\in 
			\bigoplus_{\hat{W}\in \scW_{0,d}[I]} \bV^{\hat{W}} \bigg( \Omega'; \bigoplus_{\bar{W} \in \scW_{0,d}[\hat{W}]} \bV^{\bar{W}} \Big( \Omega'; \cR \Big) \bigg),
			\\
			\sum_{\substack{\bar{W} \in \scW_{0,d}[\hat{W}] \\ \hat{W} \in \scW_{0, d}[I]}} & \Big\langle X, (\bar{W}, \hat{W}) \Big\rangle \in \bU\Bigg( \Omega; \bigoplus_{\hat{W}\in \scW_{0,d}[I]} \bV^{\hat{W}} \bigg( \Omega'; \bigoplus_{\bar{W} \in \scW_{0,d}[\hat{W}]} \bV^{\bar{W}} \Big( \Omega'; \cR \Big) \bigg) \Bigg). 
		\end{align*}
	\end{definition}
	
	Further, following the ideas of Equation \eqref{eq:iterative-coupling} we also define for any $n\in \bN$ 
	\begin{equation*}
		\scM_I^{\tilde{\otimes}n}(\Omega, \Omega') = \bU\Bigg( \Omega; \bigoplus_{\hat{W}^1 \in \scW_{0, d}[I]} \bV^{\hat{W}^1}\bigg( \Omega'; \bigoplus_{\hat{W}^{2} \in \scW_{0, d}[\hat{W}^1]} \bV^{\hat{W}^{2}}\Big( \Omega';... \bigoplus_{\hat{W}^n \in \scW_{0, d}[\hat{W}^{n-1}]} \bV^{\hat{W}^n}\big( \Omega'; \cR \big) \Big) \bigg) \Bigg)
	\end{equation*}

	\subsubsection{The coupled coproduct on $\scM_I$}
	
	The purpose of introducing the coupled tensor product in the previous section was to provide a space into which our coupled coproduct that arises from the theory of probabilistic rough paths can map. The next step on our agenda is to introduce the coupled deconcatenation coproduct:
	\begin{definition}
		\label{definition:deconcatenation-coproduct(2)}
		Let $(\Omega, \cF, \bP)$ and $(\Omega', \cF, \bP')$ be probability spaces. Let $(\cR, +, \centerdot)$ be a commutative unital normed ring and let $I$ be an index set. 
		
		Suppose that $\bU\big( \Omega; \cR \big)$ satisfies Assumption \ref{notation:U-(1)} and $\big( \bV^W \big)_{W\in \scW_{0, d}[I]}$ is a collection of modules that satisfies Assumption \ref{notation:V-(2)}. We define
		\begin{equation*}
			\Delta: \scM_I(\Omega, \Omega') \to \scM_I \tilde{\otimes} \scM_I(\Omega,\Omega') 
			\quad \mbox{and} \quad 
			\epsilon:\scM_I(\Omega, \Omega') \to \bU\big( \Omega; \cR \big)
		\end{equation*}
		for $X\in \scM_I(\Omega, \Omega')$ and $(\bar{W},\hat{W}) \in \scW_{0, d}[I] \tilde{\times} \scW_{0, d}[I]$ by
		\begin{equation*}
			\Big\langle \Delta\big[ X \big] , (\bar{W},\hat{W}) \Big\rangle = \sum_{W\in \scW_{0, d}[I]} c_I\Big(W, \bar{W}, \hat{W} \Big) \cdot \Big\langle X, W \Big\rangle
			\quad\mbox{and} \quad
			\epsilon[X] = \Big\langle X, \rId \Big\rangle. 
		\end{equation*}
	\end{definition}

	\begin{proposition}
		\label{proposition:M-coassociativity}
		Let $(\Omega, \cF, \bP)$ and $(\Omega', \cF', \bP')$ be probability spaces. Let $(\cR, +, \centerdot)$ be a commutative unital normed ring and let $I$ be an index set. 
		
		Suppose that $\bU\big( \Omega; \cR \big)$ satisfies Assumption \ref{notation:U-(1)} and $\big( \bV^W \big)_{W\in \scW_{0, d}[I]}$ is a collection of modules that satisfies Assumption \ref{notation:V-(2)}. Let
		\begin{align*}
			\Delta:& \scM_I(\Omega, \Omega') \to \scM_I \tilde{\otimes}\scM_I(\Omega, \Omega')
			\quad\mbox{and}\quad
			\fI: \scM_I(\Omega, \Omega') \to \scM_I(\Omega, \Omega')
		\end{align*}
		be the linear operator defined in Definition \ref{definition:deconcatenation-coproduct(2)} and the identity operator. Additionally, let
		\begin{align*}
			\fI \tilde{\otimes} \Delta : \scM_I \tilde{\otimes} \scM_I(\Omega, \Omega') \to \scM_I^{\tilde{\otimes}3}(\Omega, \Omega')
			\\
			\Delta \tilde{\otimes} \fI : \scM_I \tilde{\otimes} \scM_I(\Omega, \Omega') \to \scM_I^{\tilde{\otimes}3}(\Omega, \Omega')
		\end{align*}
		Then
		\begin{equation}
			\label{eq:proposition:M-coassociativity}
			\fI \tilde{\otimes} \Delta \circ \Delta = \Delta \tilde{\otimes} \fI \circ \Delta
			\quad \mbox{and} \quad 
			\centerdot \circ \epsilon \tilde{\otimes} \fI \circ \Delta = \centerdot \circ \fI \tilde{\otimes} \epsilon \circ \Delta = \fI. 
		\end{equation}
		More specifically, $\big( \scM_I(\Omega, \Omega'), \Delta, \epsilon \big)$ is a coassociative coupled coalgebras over the ring $\big( \bU(\Omega; \cR), +, \centerdot \big)$.
	\end{proposition}
	
	\begin{proof}
		Firstly, note that for any $X \in \scM_I(\Omega, \Omega')$, we have that
		\begin{equation*}
			\Delta[X] = \sum_{(\bar{W}, \hat{W}) \in \scW_{0,d}\tilde{\times} \scW_{0, d}[I] }  \bigg( \sum_{W \in \scW_{0, d}[I]} c_I\Big( W, \bar{W}, \hat{W} \Big) \cdot \Big\langle X, W \Big\rangle \bigg) 
		\end{equation*}
		so that by applying \ref{enum:notation:V-(2)-4} we conclude that
		\begin{align*}
			\bigg( \sum_{W \in \scW_{0, d}[I]} c_I\Big( W, \bar{W}, \hat{W}\Big) \cdot \Big\langle X, W \Big\rangle \bigg) \in& \bU\bigg( \Omega; \bV^{\hat{W}}\Big( \Omega'; \bV^{\bar{W}}\big( \Omega; \cR \big) \Big) \bigg)
			\\
			&\subseteq \bU\bigg( \Omega; \bV^{\hat{W}}\Big( \Omega'; \bigoplus_{\bar{W} \in \scW_{0, d}[\hat{W}]} \bV^{\bar{W}}\big( \Omega; \cR \big) \Big) \bigg)
		\end{align*}
		and as such
		\begin{align*}
			\Delta[X] \in \bU\bigg( \Omega; \bigoplus_{\hat{W} \in \scW_{0, d}[I]} \bV^{\hat{W}}\Big( \Omega'; \bigoplus_{\bar{W} \in \scW_{0, d}[\hat{W}]} \bV^{\bar{W}}\big( \Omega'; \cR \big) \Big) \bigg)
			=\scM_I^{\otimes 2}(\Omega, \Omega'). 
		\end{align*}
		In the same way, by applying Proposition \ref{proposition:M-coassociativity*} we can verify that for every $X \in \scM_I(\Omega, \Omega')$,
		\begin{align*}
			\Big\langle \fI \tilde{\otimes} \Delta \circ \Delta[X], (\check{W}, \bar{W}, \hat{W}) \Big\rangle =& \sum_{W \in \scW_{0, d}[I]} \sum_{\tilde{W} \in \scW_{0, d}[I]} \Big\langle X, W \Big\rangle \cdot c_I\Big( W, \check{W}, \tilde{W} \Big) \cdot c_I \Big( \tilde{W}, \bar{W}, \hat{W} \Big)
			\\
			\Big\langle \Delta \tilde{\otimes} \fI \circ \Delta[X], (\check{W}, \bar{W}, \hat{W}) \Big\rangle =& \sum_{W \in \scW_{0, d}[I]} \sum_{\tilde{W} \in \scW_{0, d}[\hat{W}]} \Big\langle X, W \Big\rangle \cdot c_I\Big( W, \tilde{W}, \hat{W} \Big) \cdot c_I \Big( \tilde{W}, \check{W}, \bar{W} \Big)
		\end{align*}
		so that by applying \ref{enum:notation:V-(2)-4} twice we conclude that
		\begin{align*}
			\bigg( \sum_{W \in \scW_{0, d}[I]}& \sum_{\tilde{W} \in \scW_{0, d}[I]} \Big\langle X, W \Big\rangle \cdot c_I\Big( W, \check{W}, \tilde{W} \Big) \cdot c_I \Big( \tilde{W}, \bar{W}, \hat{W} \Big) \bigg) 
			\\
			\in& \bU\Bigg( \Omega; \bV^{\hat{W}}\bigg( \Omega'; \bV^{\bar{W}}\Big( \Omega'; \bV^{\check{W}}\big( \Omega'; \cR \big) \Big) \bigg) \Bigg)
			\\
			&\subseteq \bU\Bigg( \Omega; \bigoplus_{\hat{W} \in \scW_{0, d}[I]} \bV^{\hat{W}}\bigg( \Omega'; \bigoplus_{\bar{W} \in \scW_{0, d}[\hat{W}]} \bV^{\bar{W}}\Big( \Omega'; \bigoplus_{\check{W} \in \scW_{0, d}[\bar{W}]} \bV^{\check{W}} \big( \Omega'; \cR \big) \Big) \bigg) \Bigg)
			\\
			\bigg( \sum_{W \in \scW_{0, d}[I]}& \sum_{\tilde{W} \in \scW_{0, d}[\hat{W}]} \Big\langle X, W \Big\rangle \cdot c_I\Big( W, \tilde{W}, \hat{W} \Big) \cdot c_I \Big( \tilde{W}, \check{W}, \bar{W} \Big) \bigg) 
			\\
			\in& \bU\Bigg( \Omega; \bV^{\hat{W}}\bigg( \Omega'; \bV^{\bar{W}}\Big( \Omega'; \bV^{\check{W}}\big( \Omega'; \cR \big) \Big) \bigg) \Bigg)
			\\
			&\subseteq \bU\Bigg( \Omega; \bigoplus_{\hat{W} \in \scW_{0, d}[I]} \bV^{\hat{W}}\bigg( \Omega'; \bigoplus_{\bar{W} \in \scW_{0, d}[\hat{W}]} \bV^{\bar{W}}\Big( \Omega'; \bigoplus_{\check{W} \in \scW_{0, d}[\bar{W}]} \bV^{\check{W}} \big( \Omega'; \cR \big) \Big) \bigg) \Bigg)
		\end{align*}
		and as such
		\begin{equation*}
			\fI \tilde{\otimes} \Delta \circ \Delta\big[X \big] \in \scM_I^{\tilde{\otimes}3}(\Omega, \Omega')
			\quad \mbox{and}\quad
			\Delta \tilde{\otimes} \fI \circ \Delta\big[X \big] \in \scM_I^{\tilde{\otimes}3}(\Omega, \Omega'). 
		\end{equation*}
		Next, courtesy of Equation \eqref{eq:proposition:M-coassociativity*-1} we have that
		\begin{align*}
			\bigg( \sum_{W \in \scW_{0, d}[I]}& \sum_{\tilde{W} \in \scW_{0, d}[I]} \Big\langle X, W \Big\rangle \cdot c_I\Big( W, \check{W}, \tilde{W} \Big) \cdot c_I \Big( \tilde{W}, \bar{W}, \hat{W} \Big) \bigg)
			\\
			=&\bigg( \sum_{W \in \scW_{0, d}[I]} \sum_{\tilde{W} \in \scW_{0, d}[\hat{W}]} \Big\langle X, W \Big\rangle \cdot c_I\Big( W, \tilde{W}, \hat{W} \Big) \cdot c_I \Big( \tilde{W}, \check{W}, \bar{W} \Big) \bigg). 
		\end{align*}
		so that we conclude that for every $X \in \scM_I(\Omega, \Omega')$,
		\begin{equation*}
			\fI \tilde{\otimes} \Delta \circ \Delta [X] = \Delta \tilde{\otimes} \fI \circ \Delta[X]. 
		\end{equation*}
		Similarly, thanks to \ref{enum:notation:V-(2)-2}, we have that for every $X \in \scM_I(\Omega, \Omega')$
		\begin{equation*}
			\epsilon[X] \in \bU(\Omega; \cR)
		\end{equation*}
		so that for any $(\bar{W}, \hat{W}) \in \scW_{0, d} \tilde{\times} \scW_{0, d}[I]$
		\begin{align*}
			&\Big\langle \centerdot \circ \epsilon \tilde{\otimes} \fI \circ \Delta[X], (\bar{W}, \hat{W}) \Big\rangle = \sum_{W \in \scW_{0, d}[I]} c_I\Big( W, \bar{W}, \hat{W} \Big) \Big\langle X, W \Big\rangle \delta_{\bar{W} = \rId} = \Big\langle X, \hat{W} \Big\rangle, 
			\\
			&\Big\langle \centerdot \circ \fI \tilde{\otimes} \epsilon \circ \Delta[X], (\bar{W}, \hat{W}) \Big\rangle = \sum_{W \in \scW_{0, d}[I]} c_I\Big( W, \bar{W}, \hat{W} \Big) \Big\langle X, W \Big\rangle \delta_{\hat{W} = \rId} = \Big\langle X, \bar{W} \Big\rangle. 
		\end{align*}
		Summing these terms up, we conclude that Equation \eqref{eq:proposition:M-coassociativity} holds. 
	\end{proof}
	
	\subsubsection{The coupled bialgebra over $\scM_I$}
		
	The purpose of this Section is to establish that the shuffle product and coupled deconcatenation coproduct interact with one another in a specific way similar to the properties of a bialgebra. In order to do this, let us start by considering the difference between the two $\bU(\Omega; \cR)$-modules
	\begin{align*}
		&\scM_I(\Omega, \Omega') \otimes \scM_I(\Omega, \Omega')
		= \bU\bigg( \Omega; \bigoplus_{\substack{(W^1, W^2) \in \\ (\scW_{0, d}\times \scW_{0, d})[I] }} \bV^{W^1}\Big( \Omega'; \cR \Big)\otimes \bV^{W^2}\Big( \Omega'; \cR \Big) \bigg)
		\\
		&\scM_I \tilde{\otimes} \scM_I(\Omega, \Omega')
		= \bU\bigg( \Omega; \bigoplus_{\hat{W}\in \scW_{0, d}[I]} \bV^{\hat{W}}\Big( \Omega'; \bigoplus_{\bar{W} \in \scW_{0, d}[\hat{W}]} \bV^{\bar{W}}\big( \Omega'; \cR \big) \Big) \bigg)
	\end{align*}
	As the collection of modules $\big( \bV^{W}(\Omega'; \cR) \big)_{W \in \scW_{0, d}[I]}$ satisfies Assumption \ref{notation:V-(2)}, we define
	\begin{equation*}
		\Big( \bV^{(W^1, W^2)}(\Omega'; \cR) \Big)_{(W^1, W^2) \in \scW_{0, d} \times \scW_{0, d}[I]},
		\qquad
		\bV^{(W^1, W^2)}(\Omega'; \cR) = \bV^{W^1}(\Omega'; \cR) \otimes \bV^{W^2}(\Omega'; \cR). 
	\end{equation*}
	Then this collection of modules satisfies:
	\begin{enumerate}[label=(\roman*)]
		\item 
		For every $(W^1, W^2) \in (\scW_{0, d} \times \scW_{0, d})[I]$,  we have
		\begin{equation*}
			\bV^{(W^1, W^2)}(\Omega'; \cR) \subseteq L^{\boldsymbol{0}}\Big( \Omega^{\times |\fm[ (W^1, W^2) ]|_1}; \cR \Big); 
		\end{equation*}
		\item 
		For any $(W^1, W^2) \in (\scW_{0, d} \times \scW_{0, d})[I]$ such that $\fm\big[ (W^1, W^2) \big] = (0,0)$, we have that
		\begin{equation*}
			\bV^{(W^1, W^2)}(\Omega'; \cR) = \cR; 
		\end{equation*}
	\end{enumerate}
	
	In particular, following similar ideas to those of Definition \ref{definition:product-couplingsM} gives us that
	\begin{align}
		\nonumber
		&\Big( \scM_I(\Omega, \Omega') \otimes \scM_I(\Omega, \Omega') \Big) \tilde{\otimes} \Big( \scM_I(\Omega, \Omega') \otimes \scM_I(\Omega, \Omega') \Big) 
		\\
		\label{eq:couptes-of-tesor}
		&= \bU\Bigg( \Omega; \bigoplus_{\substack{(\hat{W}^1, \hat{W}^2) \in \\ (\scW_{0,d} \times \scW_{0, d})[I]}} \bV^{(\hat{W}^1, \hat{W}^2)} \bigg( \Omega'; \bigoplus_{\substack{(\bar{W}^1, \bar{W}^2) \in \\ (\scW_{0,d} \times \scW_{0, d})[\hat{W}^1,\hat{W}^2] }} \bV^{(\hat{W}^1, \hat{W}^2)} \Big( \Omega'; \cR \Big) \bigg) \Bigg) 
	\end{align}
	In contrast to Equation \eqref{eq:couptes-of-tesor}, we can canonically define the tensor product of the two $\bU(\Omega; \cR)$-modules 
	\begin{align*}
		&\Big( \scM_I \tilde{\otimes} \scM_I(\Omega, \Omega') \Big) \otimes \Big( \scM_I \tilde{\otimes} \scM_I(\Omega, \Omega') \Big) 
		\\
		&= \bU\Bigg( \Omega; \bigoplus_{\substack{(\hat{W}^1, \hat{W}^2) \in \\ (\scW_{0, d} \times \scW_{0, d})[I] }} \bV^{\hat{W}^1}\bigg( \Omega'; \bigoplus_{\bar{W}^1 \in \scW_{0, d}[\hat{W}^1]} \bV^{\bar{W}^1}\Big( \Omega'; \cR \Big)\bigg) \otimes \bV^{\hat{W}^2}\bigg( \Omega'; \bigoplus_{\bar{W}^2 \in \scW_{0, d}[\hat{W}^2]} \bV^{\bar{W}^2}\Big( \Omega'; \cR \Big)\bigg) \Bigg).  
	\end{align*}
	In order to demonstrate the appropriate bialgebra properties of $\scM_I(\Omega, \Omega')$, we need to introduce our own notion of $\mbox{Twist}$ operation in the same fashion as Definition \ref{definition:Twist-}:
	\begin{definition}
		\label{definition:Twist}
		Let $(\Omega, \cF, \bP)$ and $(\Omega', \cF', \bP')$ be probability spaces. Let $(\cR, +, \centerdot)$ be a commutative unital normed ring and let $I$ be an index set. 
		
		Suppose that $\bU\big( \Omega; \cR \big)$ satisfies Assumption \ref{notation:U-(1)} and $\big( \bV^W \big)_{W\in \scW_{0, d}[I]}$ is a collection of modules that satisfies Assumption \ref{notation:V-(2)}. 
		
		We define $\overline{\mbox{Twist}}$ to be the operator
		\begin{align*}
			\overline{\mbox{Twist}}:& \Big( \scM_I \tilde{\otimes} \scM_I(\Omega, \Omega') \Big) \otimes \Big( \scM_I \tilde{\otimes} \scM_I(\Omega, \Omega') \Big) 
			\\
			&\to \Big( \scM_I(\Omega, \Omega') \otimes \scM_I(\Omega, \Omega')\Big) \tilde{\otimes} \Big( \scM_I(\Omega, \Omega') \otimes \scM_I(\Omega, \Omega') \Big)
		\end{align*}
		that satisfies
		\begin{align*}
			&\overline{\mbox{Twist}}\Bigg[ \bigg( \sum_{\substack{(\bar{W}^1, \hat{W}^1) \in \\ (\scW_{0,d}\tilde{\times} \scW_{0,d})[I] }} \Big\langle X^1, \bar{W}^1, \hat{W}^1 \Big\rangle(\omega_I, \omega'_{\fm\{\hat{W}^1\}}, \omega'_{\fm\{\bar{W}^1\}} ) \bigg)
			\\
			&\qquad \otimes \bigg( \sum_{\substack{(\bar{W}^2, \hat{W}^2) \in \\ (\scW_{0,d} \tilde{\times} \scW_{0,d})[I] }} \Big\langle X^2, (\bar{W}^2, \hat{W}^2) \Big\rangle(\omega_I, \omega'_{\fm\{\hat{W}^2\}}, \omega'_{\fm\{\bar{W}^2\}} ) \bigg) \Bigg](\omega_I)
			\\
			&= \sum_{\substack{(\hat{W}^1, \hat{W}^2) \in \\ (\scW_{0,d} \times \scW_{0,d})[I] }} \sum_{\substack{(\bar{W}^1, \bar{W}^2 ) \in \\ (\scW_{0,d} \times \scW_{0,d})[(\hat{W}^1,\hat{W}^2)] }} \hspace{-10pt} \mbox{Twist}\bigg[ \Big\langle X^1, (\bar{W}^1, \hat{W}^1) \Big\rangle(\omega_I, \omega'_{\fm\{\hat{W}^1\}}, \omega'_{\fm\{\bar{W}^1\}} ) 
			\\
			&\hspace{120pt}\otimes \Big\langle X^2, (\bar{W}^2, \hat{W}^2) \Big\rangle(\omega_I, \omega'_{\fm\{\hat{W}^2\}}, \omega'_{\fm\{\bar{W}^2\}} ) \bigg]. 
		\end{align*}
		where $\mbox{Twist}$ is the linear map that for any $(\bar{W}^1, \hat{W}^1), (\bar{W}^2, \hat{W}^2) \in (\scW_{0, d} \tilde{\times} \scW_{0,d})[I]$, 
		\begin{equation*}
			\mbox{Twist} : \bV^{\hat{W}^1}\Big( \Omega'; \bV^{\bar{W}^1}\big( \Omega'; \cR\big) \Big) \otimes \bV^{\hat{W}^2}\Big( \Omega'; \bV^{\bar{W}^2}\big( \Omega'; \cR \big) \Big) \to \bV^{(\hat{W}^1, \hat{W}^2)} \Big( \Omega'; \bV^{(\bar{W}^1, \bar{W}^2)}\big( \Omega'; \cR \big) \Big)
		\end{equation*}
		via the canonical embeddings 
		\begin{equation}
			\label{eq:definition:Twist-canon}
			\bV^{\hat{W}^i}\Big( \Omega'; \bV^{\bar{W}^i}\big( \Omega'; \cR \big) \Big) \hookrightarrow \bV^{\hat{W}^i}\Big( \Omega'; \bV^{\bar{W}^i}\big( \Omega'; \cR \big) \otimes \bV^{\bar{W}^j}\big( \Omega'; \cR \big) \Big).
		\end{equation}
	\end{definition}
		
	\begin{theorem}
		\label{theorem:M-coupledBialgebra}
		Let $(\Omega, \cF, \bP)$ and $(\Omega', \cF', \bP')$ be probability spaces. Let $(\cR, +, \centerdot)$ be a commutative unital normed ring and let $I$ be an index set. 
		
		Suppose that $\bU\big( \Omega; \cR \big)$ satisfies Assumption \ref{notation:U-(1)} and $\big( \bV^W \big)_{W\in \scW_{0, d}[I]}$ is a collection of modules that satisfies Assumption \ref{notation:V-(2)}.
		
		Then 
		\begin{equation}
			\label{eq:theorem:M-coupledBialgebra}
			\begin{aligned}
				\Delta \circ \tilde{\shuffle} =& \Big( \tilde{\shuffle} \tilde{\otimes} \tilde{\shuffle} \Big) \circ \overline{\mbox{Twist}} \circ \Big( \Delta \otimes \Delta\Big),
				\\
				\epsilon \circ \tilde{\shuffle} =& \centerdot \circ \epsilon \otimes \epsilon,
				\quad 
				\Delta \circ \rId \circ \centerdot = \rId \otimes \rId,
				\quad
				\epsilon \circ \rId = \fI_{\cR}. 
			\end{aligned}
		\end{equation}
		We say that $\big( \scM_I(\Omega, \Omega'), \tilde{\shuffle}, \rId, \Delta, \epsilon \big)$ is a coupled bialgebra over the ring $\big( \bU(\Omega; \cR), +, \centerdot \big)$. More specifically, $\big( \scM_I(\Omega, \Omega'), \tilde{\shuffle}, \rId \big)$ are associative algebras and $\big( \scM_I(\Omega, \Omega'), \Delta, \epsilon \big)$ is a co-associative coupled coalgebras.
	\end{theorem}
	
	\begin{proof}
		Firstly, since $\tilde{\shuffle}$ is bilinear, we can extend it to a linear operator
		\begin{equation*}
			\tilde{\shuffle}: \scM_I(\Omega, \Omega') \otimes \scM_I(\Omega,\Omega') \to \scM_I(\Omega, \Omega')
		\end{equation*}
		which implies that
		\begin{equation*}
			\Delta \circ \tilde{\shuffle}: \scM_I(\Omega, \Omega') \otimes \scM_I(\Omega,\Omega') \to \scM_I\tilde{\otimes} \scM_I(\Omega, \Omega'). 
		\end{equation*}
		In the same fashion, we also obtain that
		\begin{equation*}
			\Big( \tilde{\shuffle} \tilde{\otimes} \tilde{\shuffle} \Big) \circ \overline{\mbox{Twist}} \circ \Big( \Delta \otimes \Delta \Big): \scM_I(\Omega, \Omega') \otimes \scM_I(\Omega,\Omega') \to \scM_I\tilde{\otimes} \scM_I(\Omega, \Omega'). 
		\end{equation*}
		Let $X, Y \in \scM_I(\Omega, \Omega')$ so that 
		\begin{equation*}
			X \otimes Y \in \scM_I(\Omega, \Omega') \otimes \scM_I(\Omega, \Omega'). 
		\end{equation*} 
		Then for any coupled pair $(\bar{W}, \hat{W}) \in (\scW_{0, d} \tilde{\times} \scW_{0, d})[I]$, we have that 
		\begin{align}
			\nonumber
			\Big\langle& \Delta \circ \tilde{\shuffle}\big[ X \otimes Y \big], (\bar{W}, \hat{W}) \Big\rangle
			=\sum_{W \in \scW_{0, d}[I]} \Big\langle \tilde{\shuffle}[X \otimes Y], W \Big\rangle c_I \Big( W, \bar{W}, \hat{W} \Big)
			\\
			\label{eq:pf:theorem:M-coupledBialgebra}
			&=\sum_{W \in \scW_{0, d}[I]} \bigg( \sum_{\substack{W^1, W^2 \in \scW_{0, d}[I] \\ \exists \sigma \in \Shuf(W^1, W^2) \\ \sigma(W^1, W^2) = W}} \Big\langle X, W^1\Big\rangle \otimes \Big\langle Y, W^2 \Big\rangle \bigg) \cdot c_I\Big( W, \bar{W}, \hat{W} \Big). 
		\end{align}
		On the other hand for any $(\bar{W}^1, \hat{W}^1), (\bar{W}^2, \hat{W}^2) \in (\scW_{0, d} \tilde{\times} \scW_{0, d})[I]$, 
		\begin{align}
			\nonumber
			\Big\langle \Delta& \otimes \Delta \big[ X \otimes Y \big], \big( (\bar{W}^1, \hat{W}^1), (\bar{W}^2, \hat{W}^2) \big) \Big\rangle 
			\\
			\nonumber
			&= 
			\bigg( \sum_{W^1 \in \scW_{0, d}[I]} c_I\Big( W^1, \bar{W}^1, \hat{W}^1\Big) \cdot \Big\langle X, W^1 \Big\rangle \bigg) \otimes \bigg( \sum_{W^2 \in \scW_{0, d}[I]} c_I\Big( W^2, \bar{W}^2, \hat{W}^2\Big) \cdot \Big\langle Y, W^2 \Big\rangle \bigg)
			\\
			\label{eq:theorem:M-coupledBialgebra:pf1}
			&= \Big\langle \overline{\mbox{Twist}} \circ \Delta \otimes \Delta \big[ X \otimes Y \big], \big( (\bar{W}^1, \bar{W}^2), (\hat{W}^1, \hat{W}^2) \big) \Big\rangle
		\end{align}
		from the embedding Equation \eqref{eq:definition:Twist-canon}. 
		
		Next, the linear operator 
		\begin{equation*}
			\tilde{\shuffle} \tilde{\otimes} \tilde{\shuffle} : \Big( \scM_I(\Omega, \Omega') \otimes \scM_I(\Omega, \Omega') \Big) \tilde{\otimes} \Big( \scM_I(\Omega, \Omega') \otimes \scM_I(\Omega, \Omega') \Big) \to  \scM_I \tilde{\otimes} \scM_I(\Omega, \Omega') 
		\end{equation*}
		is defined for 
		\begin{equation*}
			Z \in \Big( \scM_I(\Omega, \Omega') \otimes \scM_I(\Omega, \Omega') \Big) \tilde{\otimes} \Big( \scM_I(\Omega, \Omega') \otimes \scM_I(\Omega, \Omega') \Big) 
		\end{equation*}
		and $(\bar{W}, \hat{W}) \in (\scW_{0, d} \tilde{\times} \scW_{0, d})[I]$ by
		\begin{equation}
			\label{eq:theorem:M-coupledBialgebra:pf2}
			\Big\langle \tilde{\shuffle} \tilde{\otimes} \tilde{\shuffle}\big[ Z \big], (\bar{W}, \hat{W}) \Big\rangle = \sum_{\substack{\hat{W}^1, \hat{W}^2 \in \scW_{0, d}[I] \\ \exists \sigma' \in \Shuf(|\hat{W}^1|, |\hat{W}^1|) \\ \sigma'(\hat{W}^1, \hat{W}^2) = \hat{W}}} \sum_{\substack{\bar{W}^1, \bar{W}^2 \in \scW_{0, d}[\hat{W}] \\ \exists \sigma \in \Shuf(|\bar{W}^1|, |\bar{W}^1|) \\ \sigma(\bar{W}^1, \bar{W}^2) = \bar{W}}} \Big\langle Z, \big( (\bar{W}^1, \bar{W}^2),(\hat{W}^1, \hat{W}^2) \big) \Big\rangle. 
		\end{equation}
		
		Therefore combining Equations \eqref{eq:theorem:M-coupledBialgebra:pf1} and \eqref{eq:theorem:M-coupledBialgebra:pf2} gives us that
		\begin{align*}
			\bigg\langle &\Big( \tilde{\shuffle} \tilde{\otimes} \tilde{\shuffle} \Big) \circ \overline{\mbox{Twist}} \circ \Big( \Delta \otimes \Delta\Big)\big[ X \otimes Y \big], \big( \bar{W}, \hat{W} \big) \bigg\rangle
			\\
			&=\sum_{\substack{\hat{W}^1, \hat{W}^2 \in \scW_{0, d}[I] \\ \exists \sigma' \in \Shuf(|\hat{W}^1|, |\hat{W}^1|) \\ \sigma'(\hat{W}^1, \hat{W}^2) = \hat{W}}} \sum_{\substack{\bar{W}^1, \bar{W}^2 \in \scW_{0, d}[\hat{W}] \\ \exists \sigma \in \Shuf(|\bar{W}^1|, |\bar{W}^1|) \\ \sigma(\bar{W}^1, \bar{W}^2) = \bar{W}}} \bigg\langle \overline{\mbox{Twist}} \circ \Big( \Delta \otimes \Delta\Big)\big[ X \otimes Y \big], \big( (\bar{W}^1, \bar{W}^2), (\hat{W}^1, \hat{W}^2) \big) \bigg\rangle
			\\
			&= \sum_{\substack{\hat{W}^1, \hat{W}^2 \in \scW_{0, d}[I] \\ \exists \sigma' \in \Shuf(|\hat{W}^1|, |\hat{W}^1|) \\ \sigma'(\hat{W}^1, \hat{W}^2) = \hat{W}}} \sum_{\substack{\bar{W}^1, \bar{W}^2 \in \scW_{0, d}[\hat{W}] \\ \exists \sigma \in \Shuf(|\bar{W}^1|, |\bar{W}^1|) \\ \sigma(\bar{W}^1, \bar{W}^2) = \bar{W}}} \bigg( \sum_{W^1 \in \scW_{0, d}[I]} c_I\Big( W^1, \bar{W}^1, \hat{W}^1\Big) \cdot \Big\langle X, W^1 \Big\rangle \bigg) 
			\\
			&\quad \otimes \bigg( \sum_{W^2 \in \scW_{0, d}[I]} c_I\Big( W^2, \bar{W}^2, \hat{W}^2\Big) \cdot \Big\langle Y, W^2 \Big\rangle \bigg)
		\end{align*}
		By applying Theorem \ref{theorem:M-coupledBialgebra-} to swap the order of summations, we conclude with Equation \eqref{eq:pf:theorem:M-coupledBialgebra} that
		\begin{align*}
			\bigg\langle &\Big( \tilde{\shuffle} \tilde{\otimes} \tilde{\shuffle} \Big) \circ \overline{\mbox{Twist}} \circ \Big( \Delta \otimes \Delta\Big)\big[ X \otimes Y \big], \big( \bar{W}, \hat{W} \big) \bigg\rangle
			\\
			&=\sum_{W \in \scW_{0, d}[I]} \bigg( \sum_{\substack{W^1, W^2 \in \scW_{0, d}[I] \\ \exists \sigma \in \Shuf(W^1, W^2) \\ \sigma(W^1, W^2) = W}} \Big\langle X, W^1\Big\rangle \otimes \Big\langle Y, W^2 \Big\rangle \bigg) \cdot c_I\Big( W, \bar{W}, \hat{W} \Big)
			\\
			&= \Big\langle  \Delta \circ \tilde{\shuffle} \big[ X \otimes Y \big], \big( \bar{W}, \hat{W} \big) \Big\rangle 
		\end{align*}
		In the same fashion, 
		\begin{equation*}
			\Big\langle \epsilon \circ \tilde{\shuffle}\big[ X \otimes Y \big], W \Big\rangle 
			= \Big\langle X, \rId \Big\rangle \centerdot \Big\langle Y, \rId \Big\rangle
			= \Big\langle \centerdot \circ \epsilon \otimes \epsilon\big[ X \otimes Y \big], W \Big\rangle. 
		\end{equation*}
		By considering the linear operator $\rId: \bU(\Omega; \cR) \to \scM_I(\Omega, \Omega')$, we obtain that for any $R^1, R^2 \in \bU(\Omega; \cR)$ and $(\bar{W}, \hat{W}) \in (\scW_{0, d} \tilde{\times} \scW_{0, d})[I]$ that
		\begin{align*}
			\Big\langle \rId \otimes \rId\big[ R^1 \otimes R^2 \big], (\bar{W}, \hat{W}) \Big\rangle =& 
			\begin{cases}
				R^1 \centerdot R^2 \quad&\quad \mbox{if } (\bar{W}, \hat{W})=\rId \otimes \rId
				\\
				0 \quad&\quad \mbox{otherwise. }
			\end{cases}
			\\
			=& \Big\langle \rId \circ \centerdot \big[ R^1 \otimes R^2 \big], (\bar{W}, \hat{W}) \Big\rangle
			= \Big\langle \Delta \circ \rId \circ \centerdot \big[ R^1 \otimes R^2 \big], (\bar{W}, \hat{W}) \Big\rangle
		\end{align*}
		due to Equation \eqref{eq:theorem:M-coupledBialgebra-}. It is immediate that $\epsilon \circ \rId = \fI_{\cR}$ so that Equation \eqref{eq:theorem:M-coupledBialgebra} holds. 
	\end{proof}

	\subsubsection{The grading over $\scM_I$}
	\label{subsubsection:grading-M}
	
	Before we conclude with a truncation argument, we show that the coupled bialgebra we have been working with admits a grading:	
	\begin{definition}		
		Let $(\Omega, \cF, \bP)$ and $(\Omega', \cF', \bP')$ be probability spaces. Let $(\cR, +, \centerdot)$ be a commutative unital normed ring and let $I$ be an index set. 
		
		Suppose that $\bU\big( \Omega; \cR \big)$ satisfies Assumption \ref{notation:U-(1)} and $\big( \bV^W \big)_{W\in \scW_{0, d}[I]}$ is a collection of modules that satisfies Assumption \ref{notation:V-(2)}.
		
		For $(k, n) \in \bN_0^{\times 2}$, we define the $\bU(\Omega; \cR)$-module
		\begin{align*}
			\scM_I^{(k, n)}\big( \Omega, \Omega' \big) 
			:=& 
			\bU\bigg( \Omega; \bigoplus_{\substack{W\in \scW_{0,d}[I] \\ \scG^I[W] = (k,n)}}
			\bV^W\Big( \Omega'; \cR \Big) \bigg). 
		\end{align*}
		while for $(k_1, n_1), (k_2, n_2) \in \bN_0^{\times 2}$, we define the $\bU(\Omega; \cR)$-module
		\begin{align*}
			\scM_I^{(k_1, n_1)} \tilde{\otimes} \scM_I^{(k_2, n_2)}\big( \Omega, \Omega' \big)
			:=
			\bU\bigg( \Omega; \bigoplus_{\hat{W} \in \scW_{0,d}^{(k_2, n_2)}[I]} \bV^{\hat{W}}\Big( \Omega'; \bigoplus_{\bar{W} \in \scW_{0, d}^{(k_1, n_1)}[\hat{W}] } \bV^{\bar{W}}\big(\Omega'; \cR \big) \Big) \bigg).  
		\end{align*}
	\end{definition}
	
	\begin{proposition}
		\label{proposition:M-Grading}
		Let $(\Omega, \cF, \bP)$ and $(\Omega', \cF', \bP')$ be probability spaces. Let $(\cR, +, \centerdot)$ be a unital normed ring and let $I$ be an index set. 
		
		Suppose that $\bU\big( \Omega; \cR \big)$ satisfies Assumption \ref{notation:U-(1)} and $\big( \bV^W \big)_{W\in \scW_{0, d}[I]}$ is a collection of modules that satisfies Assumption \ref{notation:V-(2)}. Then 
		\begin{equation}
			\label{eq:proposition:M-Grading}
			\begin{split}
				&\scM_I^{(0,0)}(\Omega, \Omega') = \bU\big( \Omega; \cR \big), 
				\\
				&\scM_I^{(k^1, n^1)}(\Omega, \Omega') \tilde{\shuffle} \scM_I^{(k^2, n^2)}(\Omega, \Omega') \subseteq  \scM_I^{(k^1+k^2, n^1+n^2)}(\Omega, \Omega'),  
				\\
				&\Delta\Big[ \scM_I^{(k, n)}(\Omega, \Omega') \Big] \subseteq  \bigoplus_{k'=0}^n \bigoplus_{n'=0}^n \Big(  \scM_I^{(k - k', n - n')} \Big) \tilde{\otimes} \scM_I^{(k', n')}(\Omega, \Omega'). 
			\end{split}
		\end{equation}
		That is, the function $\scG^I:\scW_{0,d}[I] \to \bN_0^{\times 2}$ as defined in Equation \eqref{eq:lemma:grading(1)} describes a $\bN_0^{\times 2}$-grading on the connected coupled bialgebras $\big( \scM_I(\Omega, \Omega'), \tilde{\shuffle}, \rId, \Delta, \epsilon \big)$. 
	\end{proposition}	
	
	\begin{proof}
		This follows pretty naturally from Corollary \ref{lemma:M-Finite-grading-} along with similar arguments to the proofs of Proposition \ref{proposition:associativity(1)} and \ref{proposition:M-coassociativity} plus Theorem \ref{theorem:M-coupledBialgebra}. 
	\end{proof}
	
	Following on from Equation \eqref{eq:truncatedWord}, we define
	\begin{equation}
		\label{eq:truncatedWord-Mod}
		\begin{aligned}
			\scM_I^{g, -}(\Omega, \Omega'):=& \bU\bigg( \Omega; \bigoplus_{W \in \scW_{0, d}^{g,-}[I]} \bV^{W}\Big(\Omega'; \cR \Big) \bigg)
			\\
			\scM_I^{g, +}(\Omega, \Omega'):=& \bU\bigg( \Omega; \bigoplus_{W \in \scW_{0, d}^{g,+}[I]} \bV^{W}\Big(\Omega'; \cR \Big) \bigg)
		\end{aligned}
	\end{equation}
	
	\begin{corollary}
		\label{lemma:M-Finite-grading}
		Let $(\Omega, \cF, \bP)$ and $(\Omega', \cF', \bP')$ be probability spaces. Let $(\cR, +, \centerdot)$ be a unital normed ring and let $I$ be an index set. 
		
		Suppose that $\bU\big( \Omega; \cR \big)$ satisfies Assumption \ref{notation:U-(1)} and $\big( \bV^W \big)_{W\in \scW_{0, d}[I]}$ is a collection of modules that satisfies Assumption \ref{notation:V-(2)}.
		
		Let $g: \bN_0^{\times |I|} \times \bN_0 \to \bR$ be a monotone increasing function such that $g(0_I, 0) \leq 0$. Then 
		\begin{enumerate}[label=(\ref*{lemma:M-Finite-grading}.\roman*)]
			\item
			\label{enum:lemma:M-Finitegrading-1}
			The $\bU(\Omega; \cR)$-module
			\begin{equation*}
				\scM_I^{g, +}(\Omega, \Omega')
				\quad \mbox{is an algebra ideals of } \quad
				\Big( \scM_I(\Omega, \Omega'), \tilde{\shuffle}, \rId \Big)
			\end{equation*}
			and we can identify the algebra over the $\bU(\Omega; \cR)$-module quotient by
			\begin{align*}
				\scM_I^{g,-}(\Omega, \Omega') =& \scM_I(\Omega,\Omega') / \scM_I^{g,+}(\Omega,\Omega') 
			\end{align*}
			\item 
			\label{enum:lemma:M-Finitegrading-2}
			The coupled coproduct and counit $(\Delta, \epsilon)$ restricted to the sub-module
			\begin{equation*}
				\Big( \scM_I^{g,-}(\Omega,\Omega') , \Delta, \epsilon \Big)
			\end{equation*}
			is a co-associative sub-coupled coalgebras of $\big( \scM_I(\Omega,\Omega'), \Delta, \epsilon \big)$. 
			\item 
			\label{enum:lemma:M-Finitegrading-3}
			By pairing the quotient algebra and unit $(\tilde{\shuffle}, \rId)$ to the restriction of the coupled coproduct and counit $(\Delta, \rId)$, we obtain
			\begin{equation*}
				\Big( \scM_I^{g,-}(\Omega,\Omega'), \tilde{\shuffle}, \rId, \Delta, \epsilon \Big)
			\end{equation*}
			satisfies the commutative identities of Equation \eqref{eq:theorem:M-coupledBialgebra-}. 
		\end{enumerate}
	\end{corollary}
	
	\begin{proof}
		Follows in the same fashion as Corollary \ref{lemma:M-Finite-grading-}:
		
		\ref{enum:lemma:M-Finitegrading-1}: courtesy of \ref{enum:lemma:M-Finite-grading-1} we conclude that for any 
		\begin{equation*}
			X \in \scM_I(\Omega, \Omega')
			\quad \mbox{and}\quad
			Y \in \scM_I^{g,+}(\Omega, \Omega')
		\end{equation*}
		and for any $W \in \scW_{0, d}^{g, -}$ that
		\begin{equation*}
			\Big\langle X \tilde{\shuffle} Y, W \Big\rangle = \sum_{\substack{W^1, W^2 \in \scW_{0,d}[I] \\ \sigma \in \Shuf(|W^1|, |W^2|) \\ \sigma(W^1, W^2)  =W}} \Big\langle X, W^1\Big\rangle \otimes \Big\langle Y, W^2 \Big\rangle = 0
			\quad \mbox{so that}\quad
			X \tilde{\shuffle} Y \in \scM_I^{g,+}(\Omega,\Omega'). 
		\end{equation*}
		
		\ref{enum:lemma:M-Finitegrading-2}: courtesy of \ref{enum:lemma:M-Finite-grading-2} we conclude that for any $X \in \scM_I^{g,-}(\Omega, \Omega')$, any pair $(\bar{W}, \hat{W}) \in (\scW_{0, d} \tilde{\times} \scW_{0, d})[I]$ such that either
		\begin{equation*}
			g\Big( \scG^I[\bar{W}] \Big)>0
			\quad \mbox{or}\quad
			g\Big( \scG^I[\hat{W}] \Big)>0,
		\end{equation*} 
		we conclude via the monotonicity of $g$ that
		\begin{equation*}
			\Big\langle \Delta\big[ X \big], (\bar{W},\hat{W}) \Big\rangle = 0
			\quad \mbox{so that}\quad
			\Delta \big[ X \big] \in \scM_I^{g, -} \tilde{\otimes} \scM_I^{g, -}(\Omega, \Omega'). 
		\end{equation*}
		
		\ref{enum:lemma:M-Finitegrading-3}: Follows from \ref{enum:lemma:M-Finite-grading-2} and Theorem \ref{theorem:M-coupledBialgebra}. 
	\end{proof}
	}

	\section{Lions forests and the associated coupled bialgebra}
	\label{section:Lions-Trees}
	
	By a directed graph, we mean a pair $(\scN, \scE)$ where $\scN$ is a set of distinct elements which we call nodes (sometimes referred to as vertices in the literature) and $\scE\subseteq \scN\times \scN$ which we call edges. A directed tree is a graph that is connected, acyclic and all directed edges point towards a single node which we call the root. If a directed graph has a finite number of connected components, each of which is a directed tree, we refer to it as a directed forest. We denote the set of directed forests by $\fF$. When referring to more than one directed trees, it will always be assumed that any two sets of nodes are disjoint. 

	For a directed forest $(\scN, \scE)$ with set of roots $\fr(\scN)$, every element $y\in \scN$ has a unique sequence $(y_i)_{i=1, ..., n} \in \scN$ such that $y_1=y$, $y_n \in \fr(\scN)$ and $(y_i, y_{i+1}) \in \scE$ for $i=1, ..., n-1$.
		
	There is a Reflexive, Transitive binary relationship $\leq$ over the nodes of a directed forest $\scN$ determined by the number of edges along the unique path from each node to a root. 
	
	That is $x \leq y \iff$
	\begin{align*}
		&\exists! (x_i)_{i=1, ..., m}: x_1 = x, x_m \in \fr(\scN) \mbox{ and } \forall i=1, ..., m-1, (x_i, x_{i+1}) \in \scE, 
		\\
		&\exists! (y_i)_{i=1, ..., n}: y_1 = y, y_n \in \fr(\scN) \mbox{ and } \forall i=1, ..., n-1, (y_i, y_{i+1}) \in \scE,
		\\
		&\quad \mbox{and} \quad m \leq n.  
	\end{align*}
	When $x\leq y$ and not $y \leq x$, we additionally denote the Transitive binary relationship $<$. Thus $(\scN, \leq)$ is a preorder and the set of all nodes such that
	\begin{equation}
		\label{eq:EquivalenceClassTree}
		x \lesseqqgtr y \quad \mbox{or equivalently} \quad x\leq y \mbox{ and } y \leq x
	\end{equation}
	form equivalence classes of the nodes. 
		
	We start with the following notion which will be used intensively for the rest of the paper:
	\begin{definition}
		Let $\scN$ be a non-empty set containing a finite number of elements and let $H \subseteq 2^{\scN}\backslash \emptyset$ be a collection of subsets of $\scN$. Then we say that the pair $(\scN, H)$ is a hypergraph. The elements $h\in H$ are referred to as hyperedges. 
	\end{definition}
	
	A hypergraph $(\scN, H)$ is a generalisation of a graph in which "edges" $h\in H$ can contain any positive number of nodes, rather than specifically two. Hypergraphs are sometimes referred to as \emph{range spaces} in computational geometry, \emph{simple games} in cooperative game theory and in some literature hyperedges are referred to as \emph{hyperlinks} or \emph{connectors}. 
	
	A hyperedge is said to be \emph{$d$-regular} if every node is contained in exactly $d$ hyperedges. In this work, we make no assumptions about the users background in hypergraph theory. However, the curious reader may choose to refer to \cite{Bretto2014Hypergraph} for a more general introduction to the theory of hypergraphs. 
	
	\begin{example}
		\label{example:3:2}
		The following are all visual representations of 1-regular hypergraphs:
		$$
		
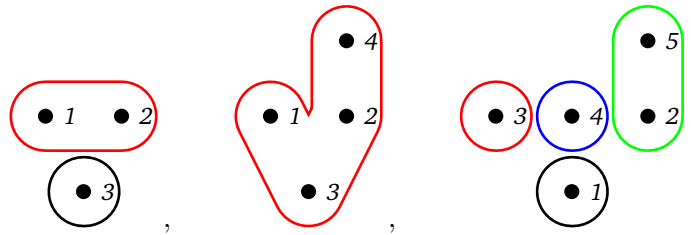
\begin{tikzpicture}
			\node[vertex, label=right:{\footnotesize 3}] at (0,0) {}; 
			\node[vertex, label=right:{\footnotesize 2}] at (0.5,1) {}; 
			\node[vertex, label=right:{\footnotesize 1}] at (-0.5,1) {}; 
			\begin{pgfonlayer}{background}
				\draw[zhyedge1] (0,0) -- (-0,0);
				\draw[zhyedge2] (0,0) -- (-0,0);
				\draw[hyedge1, color=red] (-0.5,1) -- (0.5,1);
				\draw[hyedge2] (-0.5,1) -- (0.5,1);
			\end{pgfonlayer}
		\end{tikzpicture}
		,\qquad
		\begin{tikzpicture}
			\node[vertex, label=right:{\footnotesize 3}] at (0,0) {}; 
			\node[vertex, label=right:{\footnotesize 2}] at (0.5,1) {}; 
			\node[vertex, label=right:{\footnotesize 1}] at (-0.5,1) {}; 
			\node[vertex, label=right:{\footnotesize 4}] at (0.5,2) {}; 
			\begin{pgfonlayer}{background}
				\draw[hyedge1, color=red] (0.5,2) -- (0.5,1) -- (0,0) -- (-0.5,1);
				\draw[hyedge2] (0.5,2) -- (0.5,1) -- (0,0) -- (-0.5,1);
			\end{pgfonlayer}
		\end{tikzpicture}
		, \qquad
		\begin{tikzpicture}
			\node[vertex, label=right:{\footnotesize 1}] at (0,0) {}; 
			\node[vertex, label=right:{\footnotesize 2}] at (1,1) {}; 
			\node[vertex, label=right:{\footnotesize 3}] at (-1,1) {}; 
			\node[vertex, label=right:{\footnotesize 4}] at (0,1) {}; 
			\node[vertex, label=right:{\footnotesize 5}] at (1,2) {}; 
			\begin{pgfonlayer}{background}
				\draw[zhyedge1] (0,0) -- (0,0);
				\draw[zhyedge2] (0,0) -- (0,0);
				\draw[hyedge1, color=red] (-1,1) -- (-1,1);
				\draw[hyedge2] (-1,1) -- (-1,1);
				\draw[hyedge1, color=blue] (0,1) -- (0,1);
				\draw[hyedge2] (0,1) -- (0,1);
				\draw[hyedge1, color=green] (1,1) -- (1,2);
				\draw[hyedge2] (1,1) -- (1,2);
			\end{pgfonlayer}
		\end{tikzpicture}. 
		$$
		These correspond to the Hypergraphs 
		$$
		\Big( \{1,2,3\}, \big\{ \{1,2\}, \{3\}\big\} \Big), \quad
		\Big( \{1,2,3,4\}, \big\{ \{1,2,3,4\}\big\}\Big), \quad
		\Big( \{1,2,3,4,5\}, \big\{ \{1\}, \{2,5\}, \{3\}, \{4\}\big\}\Big).
		$$ 
		We emphasise that despite the aethetics of these examples, these are not trees (or even graphs). 
	\end{example}
	
	\subsection{Lions forests and their properties}
	
	With this standard notation fixed, we start by introducing two classes of forests paired with a hypergraphic structure. These definitions can both be found in \cite{salkeld2022ExamplePRP} and \cite{2021Probabilistic} and we encourage the curious reader to explore how these abstract objects are used in these works as the index set for the regularity structure describing the dynamics of the McKean-Vlasov equation and the associated system of interacting equations:

	\begin{definition}
		\label{definition:Forests}
		Let $I$ be an index set and let $(\scN, \scE) \in \fF$ with preordering $\leq$. Let $(h_I, H) \in \scP(\scN)[I]$ and denote $H' = \big(H \cup \{ h_\iota: \iota \in I\} \big) \backslash\{\emptyset\}$. 
		
		We refer to $T = (\scN, \scE, h_I, H)$ as a \emph{tagged Lions forest} if
		\begin{enumerate}[label=($\ast$.\arabic*)]
			\item 
			\label{definition:Forests:2.2}
			For $h \in H'$, suppose $x, y\in h$ and $x<y$. Then $\exists z \in h$ such that $(y, z)\in \scE$. 
			\item 
			\label{definition:Forests:2.3}
			For $h \in H'$, suppose $x_1, y_1 \in h$, $x_1 \lesseqqgtr y_1$, $(x_1, x_2), (y_1, y_2) \in \scE$ and $x_2 \neq y_2$. Then $x_2, y_2 \in h$. 
			\item 
			\label{definition:Forests:2.1}
			For all $\iota\in I$, $h_\iota \neq \emptyset$ implies that $\exists x\in h_\iota$ such that $\forall y\in \scN$, $x\leq y$. 
		\end{enumerate}
		The collection of all such tagged Lions forests is denoted $\scF[I]$. When a tagged Lions forest $T=( \scN,\scE, h_I, H )$ satisfies that $(\scN, \scE)$ is a directed tree, it is referred to as a tagged Lions tree and the collection of all Lions trees is denoted $\scT[I]$. We define $\scF_0[I]= \scF[I] \cup\{\rId\}$ where $\rId=(\emptyset, \emptyset, \emptyset, \emptyset)$ is the empty tree. 
		
		Let $\scL:\scN \to \{1, ...,d \}$. Then for any tagged Lions forest $(\scN, \scE, h_I, H)$, we say that \\ $(\scN, \scE, h_I, H, \scL)$ is a labelled tagged Lions forest and we denote the set of all such labelled tagged Lions forests by $\scF_{0, d}[I]$. 
		
		The set $\scF_{0, d}[I]$ is a poset with partial ordering
		\begin{align*}
			(\scN, \scE&, h_I, H, \scL) \subseteq (\scN', \scE', h_I', H', \scL')
			\\
			&\iff\quad
			(\scN, \scE, \scL) = (\scN', \scE', \scL') 
			\quad\mbox{and}\quad
			(h_I, H) \subseteq (h_I', H')
		\end{align*}
		as described in \eqref{definition:taggedpartition*}. 
		
		Finally, we define $\fm: \scF_{0, d}[I] \to \bN_0$ by $\fm\big[ (\scN, \scE, h_I, H, \scL) \big] = |H|$. 
	\end{definition}
	We emphasise that this is a slight adjustment of the notation found in \cite{salkeld2022ExamplePRP}. These notational adjustments are critical for this work but were not including in the previous work to facilitate simplification of notation. 
	
	\begin{remark}
		Bullet \ref{definition:Forests:2.2} in Definition \ref{definition:Forests} says that for any two nodes $x$ and $y$ in the same hyperedge $h \in H'$ but at different distances from the roots of the forest, the node $y$ that is further from the root must have a parent in $h$, and then any node in between (along the hence common ancestral line) must also belong to $h$. In particular, there is no way for a node $x$ to have an ancestor and a descendant in the same hyperedge $h$ if $x$ is not in $h$.

		As for bullet \ref{definition:Forests:2.3} in Definition \ref{definition:Forests}, it says that given two nodes $x_{1}$ and $y_{1}$ that are the same distance from the root and belong to the same hyperedge $h$, the parents of $x_{1}$ and $y_{1}$ must belong to $h$ if they are not the same. In particular, by tracing back the genealogy of $x_{1}$ and $y_{1}$, we find that either $x_{1}$ and $y_{1}$ originate from two different roots that belong to $h$ or  $x_{1}$ 
		and $y_{1}$ have a common ancestor in the forest; the latter might not be in $h$, but necessarily its offsprings (at least up until $x_{1}$ and $x_{2}$) belong to $h$. Proceeding forward, we deduce that at any new node $x$ that is revealed, we can either stay in the hyperedge of the parent or start a new hyperedge; this new hyperedge can also contain some of the sister-brother nodes of $x$ (together with their future offsprings) but cannot contain any other ancestral line. 
		
		It should be stressed that for any $\iota \in I$, the set $h_\iota$ can be empty. When $h_{\iota}$ is not empty, bullet \ref{definition:Forests:2.1} in Definition \ref{definition:Forests} says that $h_{\iota}$ must contain at least the root of one of the trees included in the forest. In particular, the $h_{\iota}$-hyperedge of a Lions tree must contain the unique root if it is non-empty so that only $h_{\iota}$ is non-empty for only one $\iota \in I$. From an intuitive point of view, this $h_{\iota}$-hyperedge represent the label carried by a tagged (or a reference) particle in the continuum; in turn, this label is the number of the particle in the corresponding particle system.
		
		We illustrate this with Figure \ref{fig:1}.
	\end{remark}
	
	\begin{figure}[htb]
		\label{fig:1}
		\begin{center}
			\definecolor{ffdxqq}{rgb}{1,0.84,0}
			\definecolor{qqffqq}{rgb}{0,1,0}
			\definecolor{ffqqqq}{rgb}{1,0,0}
			\definecolor{qqqqff}{rgb}{0,0,1}
			\definecolor{cqcqcq}{rgb}{0.75,0.75,0.75}
			\begin{tikzpicture}[line cap=round,line join=round,>=triangle 45,x=1.0cm,y=1.0cm]
				\draw (-2,6)-- (-1,5);
				\draw (0,6)-- (-1,5);
				\draw (-1,4)-- (-1,5);
				\draw (-1,4)-- (0,3);
				\draw (0,3)-- (0,2);
				\draw (0,3)-- (1,4);
				\draw (1,5)-- (1,4);
				\draw (1,6)-- (1,5);
				\draw (5,6)-- (6,5);
				\draw (6,4)-- (6,5);
				\draw (7,6)-- (8,5);
				\draw (8,4)-- (8,5);
				\draw (9,6)-- (8,5);
				\draw (3,4)-- (3,3);
				\draw (3,4)-- (2,5);
				\draw (3,4)-- (4,5);
				\draw (3,2)-- (3,3);
				\draw (7,2)-- (7,3);
				\draw [color=ffqqqq] (2,5)-- (3,4);
				\draw [color=ffqqqq] (3,4)-- (3,2);
				\draw [color=ffqqqq] (0,6)-- (-1,5);
				\draw [color=ffqqqq] (-1,4)-- (-1,5);
				\draw [color=ffqqqq] (-1,4)-- (0,3);
				\draw [color=ffqqqq] (0,2)-- (0,3);
				\draw [line width=2pt,color=ffqqqq] (2,5)-- (3,4);
				\draw [line width=2pt,color=ffqqqq] (3,4)-- (3,3);
				\draw [line width=2pt,color=ffqqqq] (3,2)-- (3,3);
				\draw [line width=2pt,color=ffqqqq] (0,2)-- (0,3);
				\draw [line width=2pt,color=ffqqqq] (0,3)-- (-1,4);
				\draw [line width=2pt,color=ffqqqq] (-1,5)-- (-1,4);
				\draw [line width=2pt,color=ffqqqq] (-1,5)-- (0,6);
				\draw [line width=2pt,color=ffqqqq] (9,6)-- (8,5);
				\draw [line width=2pt,color=ffqqqq] (8,5)-- (8,4);
				\draw [line width=2pt,color=ffqqqq] (6,5)-- (6,4);
				\draw [line width=1.2pt,dotted,color=ffqqqq] (3,2)-- (0,2);
				\draw (6,4)-- (7,3);
				\draw (8,4)-- (7,3);
				\draw (5.82,5.67) node[anchor=north west] {Node 1};
				\draw (8.25,6.69) node[anchor=north west] {Node 2};
				\draw (-0.80,6.68) node[anchor=north west] {Node 1};
				\draw (1.30,5.70) node[anchor=north west] {Node 2};
				\draw (-.20,1.82) node[anchor=north west] {Ancestors are roots};
				\draw (7.28,3.24) node[anchor=north west] {Common ancestor};
				\begin{scriptsize}
					\draw [fill=qqqqff] (-2,6) circle (1.5pt);
					\draw [fill=qqqqff] (-1,5) circle (1.5pt);
					\draw [fill=qqqqff] (0,6) circle (1.5pt);
					\draw [fill=ffqqqq] (-1,4) circle (1.5pt);
					\draw [fill=qqqqff] (0,3) circle (1.5pt);
					\draw [fill=qqqqff] (0,2) circle (1.5pt);
					\draw [fill=qqffqq] (1,4) circle (1.5pt);
					\draw [fill=qqffqq] (1,5) circle (1.5pt);
					\draw [fill=qqffqq] (1,6) circle (1.5pt);
					\draw [fill=qqffqq] (5,6) circle (1.5pt);
					\draw [fill=qqqqff] (6,5) circle (1.5pt);
					\draw [fill=ffqqqq] (6,4) circle (1.5pt);
					\draw [fill=qqqqff] (7,6) circle (1.5pt);
					\draw [fill=ffqqqq] (8,5) circle (1.5pt);
					\draw [fill=ffqqqq] (8,4) circle (1.5pt);
					\draw [fill=ffdxqq] (7,3) circle (3.5pt);
					\draw [fill=qqqqff] (9,6) circle (1.5pt);
					\draw [fill=ffqqqq] (3,4) circle (1.5pt);
					\draw [fill=ffqqqq] (3,3) circle (1.5pt);
					\draw [fill=ffqqqq] (3,2) circle (1.5pt);
					\draw [fill=qqqqff] (2,5) circle (1.5pt);
					\draw [fill=qqqqff] (4,5) circle (1.5pt);
					\draw [fill=ffdxqq] (7,2) circle (1.5pt);
					\draw [fill=ffqqqq] (2,5) circle (3.5pt);
					\draw [fill=ffqqqq] (0,6) circle (3.5pt);
					\draw [fill=ffqqqq] (6,5) circle (3.5pt);
					\draw [fill=ffqqqq] (9,6) circle (3.5pt);
					\draw [fill=ffqqqq] (0,2) circle (1.5pt);
					\draw [fill=ffqqqq] (-1,5) circle (1.5pt);
					\draw [fill=ffqqqq] (0,3) circle (1.5pt);
				\end{scriptsize}
			\end{tikzpicture}
			\caption{Two examples of ancestral lines that are included in the same hyperedge: $(i)$ (Left pane) The two nodes (Node 1 and Node 2) have ancestral lines in the same  hyperedge up to the roots; $(ii)$ The two nodes (Node 1 and Node 2) have ancestral lines in the same hyperedge up to a common ancestor, but the common ancestor may be in another hyperdege (hence a different color on the graph).}
		\end{center}
	\end{figure}
	
	Motivated by Proposition \ref{proposition:taggedpartition*}, we want to consider the collection of partitions of the nodes of some directed forest that satisfy the properties of Definition \ref{definition:Forests}:	
	\begin{definition}
		\label{definition:Quotient-Partition}
		Let $\tau=(\scN, \scE)$ be a directed forest with preordering $\leq$. Let $I$ be an index set. 
		
		We define $\scQ^{\scE}(\scN)[I] \subseteq \scP(\scN)[I]$ to be the set of all tagged partitions $(h_I, H)$ of the set $\scN$ such that
		\begin{enumerate}[label=($\ast$.\arabic*)]
			\item $\forall h \in H'$ such that $x, y \in h$ and  $x<y$, $\exists z \in h$ such that $(y,z) \in \scE$. 
			\item $\forall h \in H'$ such that $x_1, y_1 \in q$ and $x_1 \lesseqqgtr y_1$, $(x_1, x_2), (y_1, y_2) \in \scE$ and $x_2 \neq y_2$, then $x_2, y_2 \in h$. 
			\item For all $\iota \in I$,  $h_\iota \neq \emptyset$ implies that $h_\iota \cap \fr[\tau] \neq \emptyset$. 
		\end{enumerate}
		where we use the convention that
		\begin{equation*}
			H' = \big( H \cup \{h_\iota: \iota \in I\} \big)\backslash \{\emptyset\}. 
		\end{equation*}
		We refer to $\scQ^{\scE}(\scN)[I]$ as the set of all Lions admissible tagged partitions. 
		
		The set $\scQ^{\scE}(\scN)[I]$ is a poset with partial ordering defined by
		\begin{equation}
			\label{eq:definition:PartOrd_T}
			(g_I, G) \subseteq_{\scQ} (h_I, H) \iff 
			\begin{aligned}
				& \forall \iota \in I, g_\iota \subseteq h_\iota \quad \mbox{and}
				\\
				&\forall g\in G, \exists h \in H' \quad \mbox{such that} \quad g \subseteq H. 
			\end{aligned}
		\end{equation}
		We denote $\fm_{\scQ}: \scQ^{\scE}(\scN)[I] \to \bN_0$ to be the function $\fm_{\scQ}\big[(h_I, H) \big] = |H|$. 
	\end{definition}
	We remark that $(\scN, \scE, h_I, H) \in \scF[I]$ if and only if $(\scN, \scE) \in \fF$ and $(h_I, H) \in \scQ^\scE(\scN)[I]$. 
	
	\begin{lemma}
		\label{lemma:Q-poset}
		Let $\tau=(\scN, \scE)$ be a directed forest with preordering $\leq$.  Then for any choice of index set we have
		\begin{equation*}
			\big( \scQ^{\scE}(\scN)[I], \subseteq_{\scQ}, \fm_{\scQ} \big)
			\quad \mbox{is a sub-poset of}\quad
			\big( \scP(\scN)[I], \subseteq, \fm \big)
		\end{equation*} 
	\end{lemma}

	\begin{proof}
		This is immediate since $\scQ^{\scE}(\scN)[I] \subseteq \scP(\scN)[I]$, for every $(g_I, G), (h_I, H) \in \scQ^{\scE}(\scN)[I]$ we have that 
		\begin{equation*}
			(g_I, G) \subseteq (h_I, H)	
			\quad \mbox{under the $\scP(\scN)[I]$ poset} \quad \implies \quad
			(g_I, G) \subseteq_{\scQ} (h_I, H),
		\end{equation*}
		and for every $(h_I, H) \in \scQ^{\scE}(\scN)[I]$, $\fm_{\scQ}\big[ (h_I, H) \big] = \fm \big[ (h_I, H) \big]$. 
	\end{proof}

	\subsubsection{Partition sequences and Lions forests}
	
	As a minor reduction of notation, when the index set $I=\{0\}$ containing a single element, we write
	\begin{equation*}
		A[I] = A[0] 
		\quad \mbox{instead of}\quad
		A[\{0\}], 
		\quad \mbox{and}\quad
		\scP(\scN)[I] = \scP(\scN)[0] 
		\quad \mbox{instead of}\quad
		\scP(\scN)[\{0\}]. 
	\end{equation*}
	Our goal is to show that the hypergraphic structure described in Definition \ref{definition:Forests} arise as a result of the properties of the Lions derivatives. 
	\begin{definition}
		\label{definition:Lions-Partition-Tree}
		Let $(\scN, \scE, \scL)$ is a labelled directed forest with roots $\fr(\scN) \subseteq \scN$. 

		For $x \in \scN$, we denote
		\begin{equation}
			\label{eq:N_x}
			\scN_x := \big\{ y \in \scN: (y, x) \in \scE\big\}
		\end{equation}
		and let 
		\begin{equation*}
			(\hat{h}_I, \hat{H}) \in \scP(\fr(\scN))[I]
			\quad \mbox{and}\quad
			\fH \in \bigsqcup_{x\in \scN} \scP(\scN_x)[0]. 
		\end{equation*}
		
		Then we say that $\big( \scN, \scE, (\hat{h}_I, \hat{H}), \fH, \scL \big)$ is a labelled \emph{Lions partition forest}. We denote $\hat{\scF}_{0,d}[I]$ to be the set of all Lions partition forests. 
		
		The set of Lions partition forests is a poset with partial ordering
		\begin{align*}
			\big( &\scN, \scE, (\hat{h}_I, \hat{H}), \fH, \scL \big) \hat{\subseteq} \big( \scN', \scE', (\hat{h}_I', \hat{H}'), \fH', \scL' \big)
			\\
			&\iff \quad (\scN, \scE, \scL) = (\scN', \scE', \scL') 
			\quad \mbox{and}\quad
			(\hat{h}_I, \hat{H}) \subseteq (\hat{h}_I', \hat{H}')
			\quad \mbox{and}\quad
			\forall x\in \scN\quad
			\fh[x] \subseteq \fh'[x]
		\end{align*}
		(where we used \eqref{definition:taggedpartition*} and  \eqref{eq:partial-ordering}). 
		
		Finally, we define $\fm: \hat{\scF}_{0, d}[I] \to \bN_0$ by
		\begin{equation*}
			\hat{\fm}\big[ \big( \scN, \scE, (\hat{h}_I, \hat{H}), \fH, \scL \big) \big] = \fm[\hat{H}] + \sum_{x\in \scN} \fm\big[ \fh[x] \big]. 
		\end{equation*}
	\end{definition}
	
	\begin{remark}
		\label{remark:Elem-Diff}
		While Definition \ref{definition:Lions-Partition-Tree} may be confusing given we have another definition for Lions forests that matches with the notation of Section \ref{section:ProbabilisticRoughPaths}, it has a highly visual interpretation which helps to illustrate the link between Lions trees and the mean-field elementary differentials that are mentioned in \cite{salkeld2022ExamplePRP}, \cite{salkeld2021Probabilistic2} and \cite{2021Probabilistic} (to name a few). For this remark, we will completely avoid a discussion of higher order Lions calculus and instead treat this definition at a purely abstract level. 
		
		Let us start by considering the directed rooted forest with vertices $\{1, 2, 3,4,5\}$:
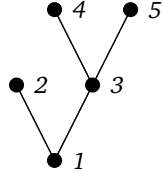
		\begin{equation*}
			\begin{tikzpicture}
				\node[vertex, label=right:{\footnotesize 1}] at (0,0) {}; 
				\node[vertex, label=right:{\footnotesize 2}] at (-0.5,1) {}; 
				\node[vertex, label=right:{\footnotesize 3}] at (0.5,1) {}; 
				\node[vertex, label=right:{\footnotesize 4}] at (0,2) {}; 
				\node[vertex, label=right:{\footnotesize 5}] at (1,2) {}; 
				\draw[edge] (0,0) -- (-0.5, 1);
				\draw[edge] (0,0) -- (0.5, 1);
				\draw[edge] (0.5, 1) -- (0,2);
				\draw[edge] (0.5, 1) -- (1,2);
			\end{tikzpicture}
		\end{equation*}
		Then we have that
		\begin{equation*}
			\scN_1 = \big\{ 2, 3\big\}, \quad 
			\scN_2 = \emptyset, \quad
			\scN_3 = \{4, 5\} \quad \mbox{and} \quad
			\scN_4 = \scN_5 = \emptyset. 
		\end{equation*}
		Therefore, the only values of $x\in \scN$ for which $\fh[x]$ is non-trivial are $1$ and $3$ and $|\scN_1| = |\scN_3| = 2$. There are 5 elements of the set $\scP(\scN_1)[0]$ and $\scP(\scN_3)[0]$, so there are twenty five possible Lions partition forests.
		Recalling the link between partitions and partition sequences (see Proposition \ref{proposition:taggedpartition*}), we could visualise this as 
		\begin{equation}
			\label{eq:Visual-LionsPF}
			\vcenter{
				\hbox{
					\begin{tikzpicture}
						\node[vertex, label=right:{\footnotesize 1}] at (0,0) {}; 
						\node[vertex, label=right:{\footnotesize 2}] at (-0.5,1) {}; 
						\node[vertex, label=right:{\footnotesize 3}] at (0.5,1) {}; 
						\node[vertex, label=right:{\footnotesize 4}] at (0,2) {}; 
						\node[vertex, label=right:{\footnotesize 5}] at (1,2) {}; 
						\draw[edge] (0,0) -- (-0.5, 1);
						\draw[edge] (0,0) -- (0.5, 1);
						\draw[edge] (0.5, 1) -- (0,2);
						\draw[edge] (0.5, 1) -- (1,2);
						\node at (-1,1) {$\cE^a\Big[$};
						\node at (1,1) {$\Big]$};
						\node at (-0.5,2) {$\cE^{a'}\Big[$};
						\node at (1.5,2) {$\Big]$};
					\end{tikzpicture}
				}
			}
		\end{equation}
		where we sum over $a, a' \in A_2[0]$. 
		
		The possible choices for $(a, a') \in A_2[0] \times A_n[0]$ in Equation \eqref{eq:Visual-LionsPF} can be written as
		\begin{equation*}
			\big( a, a' \big) = 
			\begin{pmatrix}
				&&&&
				\\
				\big( (00), (00) \big), & \big( (01), (00) \big), & \big( (10), (00) \big), & \big( (11), (00) \big), & \big( (12), (00) \big)
				\\
				&&&&
				\\
				\big( (00), (01) \big), & \big( (01), (01) \big), & \big( (10), (01) \big), & \big( (11), (01) \big), & \big( (12), (01) \big)
				\\
				&&&&
				\\
				\big( (00), (10) \big), & \big( (01), (10) \big), & \big( (10), (10) \big), & \big( (11), (10) \big), & \big( (12), (10) \big)
				\\
				&&&&
				\\
				\big( (00), (11) \big), & \big( (01), (11) \big), & \big( (10), (11) \big), & \big( (11), (11) \big), & \big( (12), (11) \big)
				\\
				&&&&
				\\
				\big( (00), (12) \big), & \big( (01), (12) \big), & \big( (10), (12) \big), & \big( (11), (12) \big), & \big( (12), (12) \big)
				\\
				&&&&
			\end{pmatrix}
		\end{equation*}
		
		These can be visualised as follows:
		
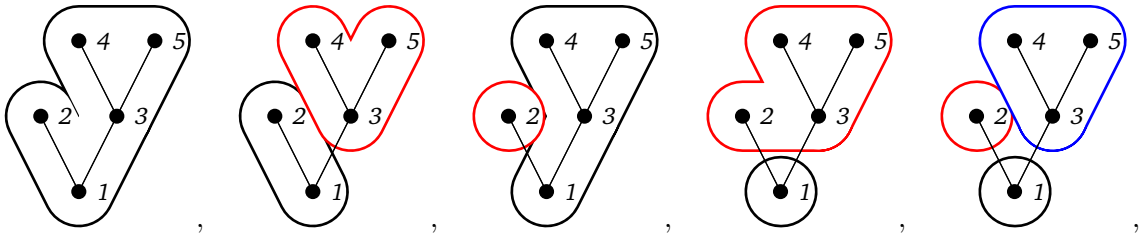
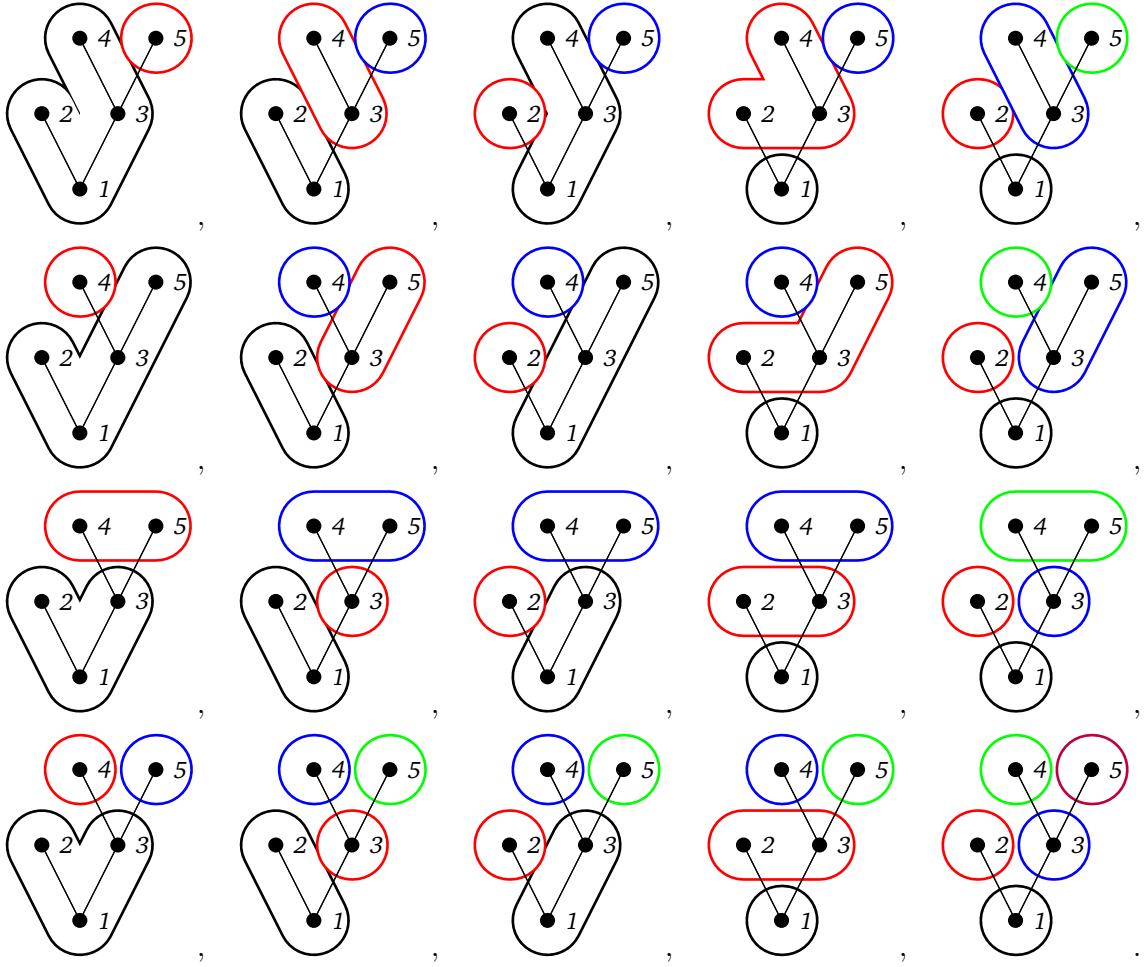
\begin{align*}
			\begin{tikzpicture}
				\node[vertex, label=right:{\footnotesize 1}] at (0,0) {}; 
				\node[vertex, label=right:{\footnotesize 2}] at (-0.5,1) {}; 
				\node[vertex, label=right:{\footnotesize 3}] at (0.5,1) {}; 
				\node[vertex, label=right:{\footnotesize 4}] at (0,2) {}; 
				\node[vertex, label=right:{\footnotesize 5}] at (1,2) {}; 
				\draw[edge] (0,0) -- (-0.5, 1);
				\draw[edge] (0,0) -- (0.5, 1);
				\draw[edge] (0.5, 1) -- (0,2);
				\draw[edge] (0.5, 1) -- (1,2);
				\begin{pgfonlayer}{background}
					\draw[zhyedge1] (-0.5,1) -- (0,0) -- (0.5, 1) -- (0,2) -- (1,2) -- (0.5,1);
					\draw[zhyedge2] (-0.5,1) -- (0,0) -- (0.5, 1) -- (0,2) -- (1,2) -- (0.5,1);
				\end{pgfonlayer}
			\end{tikzpicture}
			,\quad
			\begin{tikzpicture}
				\node[vertex, label=right:{\footnotesize 1}] at (0,0) {}; 
				\node[vertex, label=right:{\footnotesize 2}] at (-0.5,1) {}; 
				\node[vertex, label=right:{\footnotesize 3}] at (0.5,1) {}; 
				\node[vertex, label=right:{\footnotesize 4}] at (0,2) {}; 
				\node[vertex, label=right:{\footnotesize 5}] at (1,2) {}; 
				\draw[edge] (0,0) -- (-0.5, 1);
				\draw[edge] (0,0) -- (0.5, 1);
				\draw[edge] (0.5, 1) -- (0,2);
				\draw[edge] (0.5, 1) -- (1,2);
				\begin{pgfonlayer}{background}
					\draw[zhyedge1] (-0.5,1) -- (0,0);
					\draw[zhyedge2] (-0.5,1) -- (0,0);
					\draw[hyedge1, color=red] (0,2) -- (0.5,1) -- (1,2);
					\draw[hyedge2] (0,2) -- (0.5,1) -- (1,2);
				\end{pgfonlayer}
			\end{tikzpicture}
			,\quad
			\begin{tikzpicture}
				\node[vertex, label=right:{\footnotesize 1}] at (0,0) {}; 
				\node[vertex, label=right:{\footnotesize 2}] at (-0.5,1) {}; 
				\node[vertex, label=right:{\footnotesize 3}] at (0.5,1) {}; 
				\node[vertex, label=right:{\footnotesize 4}] at (0,2) {}; 
				\node[vertex, label=right:{\footnotesize 5}] at (1,2) {}; 
				\draw[edge] (0,0) -- (-0.5, 1);
				\draw[edge] (0,0) -- (0.5, 1);
				\draw[edge] (0.5, 1) -- (0,2);
				\draw[edge] (0.5, 1) -- (1,2);
				\begin{pgfonlayer}{background}
					\draw[zhyedge1] (0,0) -- (0.5,1) -- (0,2) -- (1,2) -- (0.5,1);
					\draw[zhyedge2] (0,0) -- (0.5,1) -- (0,2) -- (1,2) -- (0.5,1);
					\draw[hyedge1, color=red] (-0.5,1) -- (-0.5,1);
					\draw[hyedge2] (-0.5,1) -- (-0.5,1);
				\end{pgfonlayer}
			\end{tikzpicture}
			,\quad
			\begin{tikzpicture}
				\node[vertex, label=right:{\footnotesize 1}] at (0,0) {}; 
				\node[vertex, label=right:{\footnotesize 2}] at (-0.5,1) {}; 
				\node[vertex, label=right:{\footnotesize 3}] at (0.5,1) {}; 
				\node[vertex, label=right:{\footnotesize 4}] at (0,2) {}; 
				\node[vertex, label=right:{\footnotesize 5}] at (1,2) {}; 
				\draw[edge] (0,0) -- (-0.5, 1);
				\draw[edge] (0,0) -- (0.5, 1);
				\draw[edge] (0.5, 1) -- (0,2);
				\draw[edge] (0.5, 1) -- (1,2);
				\begin{pgfonlayer}{background}
					\draw[zhyedge1] (0,0) -- (0,0);
					\draw[zhyedge2] (0,0) -- (0,0);
					\draw[hyedge1, color=red] (-0.5,1) -- (0.5,1) -- (0,2) -- (1,2) -- (0.5,1);
					\draw[hyedge2] (-0.5,1) -- (0.5,1) -- (0,2) -- (1,2) -- (0.5,1);
				\end{pgfonlayer}
			\end{tikzpicture}
			,\quad
			\begin{tikzpicture}
				\node[vertex, label=right:{\footnotesize 1}] at (0,0) {}; 
				\node[vertex, label=right:{\footnotesize 2}] at (-0.5,1) {}; 
				\node[vertex, label=right:{\footnotesize 3}] at (0.5,1) {}; 
				\node[vertex, label=right:{\footnotesize 4}] at (0,2) {}; 
				\node[vertex, label=right:{\footnotesize 5}] at (1,2) {}; 
				\draw[edge] (0,0) -- (-0.5, 1);
				\draw[edge] (0,0) -- (0.5, 1);
				\draw[edge] (0.5, 1) -- (0,2);
				\draw[edge] (0.5, 1) -- (1,2);
				\begin{pgfonlayer}{background}
					\draw[zhyedge1] (0,0) -- (0,0);
					\draw[zhyedge2] (0,0) -- (0,0);
					\draw[hyedge1, color=red] (-0.5,1) -- (-0.5,1);
					\draw[hyedge2] (-0.5,1) -- (-0.5,1);
					\draw[hyedge1, color=blue] (0.5,1) -- (0,2) -- (1,2) -- (0.5,1);
					\draw[hyedge2] (0.5,1) -- (0,2) -- (1,2) -- (0.5,1);
				\end{pgfonlayer}
			\end{tikzpicture}
			,
			\\
			\begin{tikzpicture}
				\node[vertex, label=right:{\footnotesize 1}] at (0,0) {}; 
				\node[vertex, label=right:{\footnotesize 2}] at (-0.5,1) {}; 
				\node[vertex, label=right:{\footnotesize 3}] at (0.5,1) {}; 
				\node[vertex, label=right:{\footnotesize 4}] at (0,2) {}; 
				\node[vertex, label=right:{\footnotesize 5}] at (1,2) {}; 
				\draw[edge] (0,0) -- (-0.5, 1);
				\draw[edge] (0,0) -- (0.5, 1);
				\draw[edge] (0.5, 1) -- (0,2);
				\draw[edge] (0.5, 1) -- (1,2);
				\begin{pgfonlayer}{background}
					\draw[zhyedge1] (-0.5,1) -- (0,0) -- (0.5, 1) -- (0,2);
					\draw[zhyedge2] (-0.5,1) -- (0,0) -- (0.5, 1) -- (0,2);
					\draw[hyedge1, color=red] (1,2) -- (1,2);
					\draw[hyedge2] (1,2) -- (1,2);
				\end{pgfonlayer}
			\end{tikzpicture}
			,\quad
			\begin{tikzpicture}
				\node[vertex, label=right:{\footnotesize 1}] at (0,0) {}; 
				\node[vertex, label=right:{\footnotesize 2}] at (-0.5,1) {}; 
				\node[vertex, label=right:{\footnotesize 3}] at (0.5,1) {}; 
				\node[vertex, label=right:{\footnotesize 4}] at (0,2) {}; 
				\node[vertex, label=right:{\footnotesize 5}] at (1,2) {}; 
				\draw[edge] (0,0) -- (-0.5, 1);
				\draw[edge] (0,0) -- (0.5, 1);
				\draw[edge] (0.5, 1) -- (0,2);
				\draw[edge] (0.5, 1) -- (1,2);
				\begin{pgfonlayer}{background}
					\draw[zhyedge1] (-0.5,1) -- (0,0);
					\draw[zhyedge2] (-0.5,1) -- (0,0);
					\draw[hyedge1, color=red] (0,2) -- (0.5,1);
					\draw[hyedge2] (0,2) -- (0.5,1);
					\draw[hyedge1, color=blue] (1,2) -- (1,2);
					\draw[hyedge2] (1,2) -- (1,2);
				\end{pgfonlayer}
			\end{tikzpicture}
			,\quad
			\begin{tikzpicture}
				\node[vertex, label=right:{\footnotesize 1}] at (0,0) {}; 
				\node[vertex, label=right:{\footnotesize 2}] at (-0.5,1) {}; 
				\node[vertex, label=right:{\footnotesize 3}] at (0.5,1) {}; 
				\node[vertex, label=right:{\footnotesize 4}] at (0,2) {}; 
				\node[vertex, label=right:{\footnotesize 5}] at (1,2) {}; 
				\draw[edge] (0,0) -- (-0.5, 1);
				\draw[edge] (0,0) -- (0.5, 1);
				\draw[edge] (0.5, 1) -- (0,2);
				\draw[edge] (0.5, 1) -- (1,2);
				\begin{pgfonlayer}{background}
					\draw[zhyedge1] (0,0) -- (0.5,1) -- (0,2);
					\draw[zhyedge2] (0,0) -- (0.5,1) -- (0,2);
					\draw[hyedge1, color=red] (-0.5,1) -- (-0.5,1);
					\draw[hyedge2] (-0.5,1) -- (-0.5,1);
					\draw[hyedge1, color=blue] (1,2) -- (1,2);
					\draw[hyedge2] (1,2) -- (1,2);
				\end{pgfonlayer}
			\end{tikzpicture}
			,\quad
			\begin{tikzpicture}
				\node[vertex, label=right:{\footnotesize 1}] at (0,0) {}; 
				\node[vertex, label=right:{\footnotesize 2}] at (-0.5,1) {}; 
				\node[vertex, label=right:{\footnotesize 3}] at (0.5,1) {}; 
				\node[vertex, label=right:{\footnotesize 4}] at (0,2) {}; 
				\node[vertex, label=right:{\footnotesize 5}] at (1,2) {}; 
				\draw[edge] (0,0) -- (-0.5, 1);
				\draw[edge] (0,0) -- (0.5, 1);
				\draw[edge] (0.5, 1) -- (0,2);
				\draw[edge] (0.5, 1) -- (1,2);
				\begin{pgfonlayer}{background}
					\draw[zhyedge1] (0,0) -- (0,0);
					\draw[zhyedge2] (0,0) -- (0,0);
					\draw[hyedge1, color=red] (-0.5,1) -- (0.5,1) -- (0,2);
					\draw[hyedge2] (-0.5,1) -- (0.5,1) -- (0,2);
					\draw[hyedge1, color=blue] (1,2) -- (1,2);
					\draw[hyedge2] (1,2) -- (1,2);
				\end{pgfonlayer}
			\end{tikzpicture}
			,\quad
			\begin{tikzpicture}
				\node[vertex, label=right:{\footnotesize 1}] at (0,0) {}; 
				\node[vertex, label=right:{\footnotesize 2}] at (-0.5,1) {}; 
				\node[vertex, label=right:{\footnotesize 3}] at (0.5,1) {}; 
				\node[vertex, label=right:{\footnotesize 4}] at (0,2) {}; 
				\node[vertex, label=right:{\footnotesize 5}] at (1,2) {}; 
				\draw[edge] (0,0) -- (-0.5, 1);
				\draw[edge] (0,0) -- (0.5, 1);
				\draw[edge] (0.5, 1) -- (0,2);
				\draw[edge] (0.5, 1) -- (1,2);
				\begin{pgfonlayer}{background}
					\draw[zhyedge1] (0,0) -- (0,0);
					\draw[zhyedge2] (0,0) -- (0,0);
					\draw[hyedge1, color=red] (-0.5,1) -- (-0.5,1);
					\draw[hyedge2] (-0.5,1) -- (-0.5,1);
					\draw[hyedge1, color=blue] (0.5,1) -- (0,2);
					\draw[hyedge2] (0.5,1) -- (0,2);
					\draw[hyedge1, color=green] (1,2) -- (1,2);
					\draw[hyedge2] (1,2) -- (1,2);
				\end{pgfonlayer}
			\end{tikzpicture}
			,
			\\
			\begin{tikzpicture}
				\node[vertex, label=right:{\footnotesize 1}] at (0,0) {}; 
				\node[vertex, label=right:{\footnotesize 2}] at (-0.5,1) {}; 
				\node[vertex, label=right:{\footnotesize 3}] at (0.5,1) {}; 
				\node[vertex, label=right:{\footnotesize 4}] at (0,2) {}; 
				\node[vertex, label=right:{\footnotesize 5}] at (1,2) {}; 
				\draw[edge] (0,0) -- (-0.5, 1);
				\draw[edge] (0,0) -- (0.5, 1);
				\draw[edge] (0.5, 1) -- (0,2);
				\draw[edge] (0.5, 1) -- (1,2);
				\begin{pgfonlayer}{background}
					\draw[zhyedge1] (-0.5,1) -- (0,0) -- (0.5, 1) -- (1,2);
					\draw[zhyedge2] (-0.5,1) -- (0,0) -- (0.5, 1) -- (1,2);
					\draw[hyedge1, color=red] (0,2) -- (0,2);
					\draw[hyedge2] (0,2) -- (0,2);
				\end{pgfonlayer}
			\end{tikzpicture}
			,\quad
			\begin{tikzpicture}
				\node[vertex, label=right:{\footnotesize 1}] at (0,0) {}; 
				\node[vertex, label=right:{\footnotesize 2}] at (-0.5,1) {}; 
				\node[vertex, label=right:{\footnotesize 3}] at (0.5,1) {}; 
				\node[vertex, label=right:{\footnotesize 4}] at (0,2) {}; 
				\node[vertex, label=right:{\footnotesize 5}] at (1,2) {}; 
				\draw[edge] (0,0) -- (-0.5, 1);
				\draw[edge] (0,0) -- (0.5, 1);
				\draw[edge] (0.5, 1) -- (0,2);
				\draw[edge] (0.5, 1) -- (1,2);
				\begin{pgfonlayer}{background}
					\draw[zhyedge1] (-0.5,1) -- (0,0);
					\draw[zhyedge2] (-0.5,1) -- (0,0);
					\draw[hyedge1, color=red] (0.5,1) -- (1,2);
					\draw[hyedge2] (0.5,1) -- (1,2);
					\draw[hyedge1, color=blue] (0,2) -- (0,2);
					\draw[hyedge2] (0,2) -- (0,2);
				\end{pgfonlayer}
			\end{tikzpicture}
			,\quad
			\begin{tikzpicture}
				\node[vertex, label=right:{\footnotesize 1}] at (0,0) {}; 
				\node[vertex, label=right:{\footnotesize 2}] at (-0.5,1) {}; 
				\node[vertex, label=right:{\footnotesize 3}] at (0.5,1) {}; 
				\node[vertex, label=right:{\footnotesize 4}] at (0,2) {}; 
				\node[vertex, label=right:{\footnotesize 5}] at (1,2) {}; 
				\draw[edge] (0,0) -- (-0.5, 1);
				\draw[edge] (0,0) -- (0.5, 1);
				\draw[edge] (0.5, 1) -- (0,2);
				\draw[edge] (0.5, 1) -- (1,2);
				\begin{pgfonlayer}{background}
					\draw[zhyedge1] (0,0) -- (0.5,1) -- (1,2);
					\draw[zhyedge2] (0,0) -- (0.5,1) -- (1,2);
					\draw[hyedge1, color=red] (-0.5,1) -- (-0.5,1);
					\draw[hyedge2] (-0.5,1) -- (-0.5,1);
					\draw[hyedge1, color=blue] (0,2) -- (0,2);
					\draw[hyedge2] (0,2) -- (0,2);
				\end{pgfonlayer}
			\end{tikzpicture}
			,\quad
			\begin{tikzpicture}
				\node[vertex, label=right:{\footnotesize 1}] at (0,0) {}; 
				\node[vertex, label=right:{\footnotesize 2}] at (-0.5,1) {}; 
				\node[vertex, label=right:{\footnotesize 3}] at (0.5,1) {}; 
				\node[vertex, label=right:{\footnotesize 4}] at (0,2) {}; 
				\node[vertex, label=right:{\footnotesize 5}] at (1,2) {}; 
				\draw[edge] (0,0) -- (-0.5, 1);
				\draw[edge] (0,0) -- (0.5, 1);
				\draw[edge] (0.5, 1) -- (0,2);
				\draw[edge] (0.5, 1) -- (1,2);
				\begin{pgfonlayer}{background}
					\draw[zhyedge1] (0,0) -- (0,0);
					\draw[zhyedge2] (0,0) -- (0,0);
					\draw[hyedge1, color=red] (-0.5,1) -- (0.5,1) -- (1,2);
					\draw[hyedge2] (-0.5,1) -- (0.5,1) -- (1,2);
					\draw[hyedge1, color=blue] (0,2) -- (0,2);
					\draw[hyedge2] (0,2) -- (0,2);
				\end{pgfonlayer}
			\end{tikzpicture}
			,\quad
			\begin{tikzpicture}
				\node[vertex, label=right:{\footnotesize 1}] at (0,0) {}; 
				\node[vertex, label=right:{\footnotesize 2}] at (-0.5,1) {}; 
				\node[vertex, label=right:{\footnotesize 3}] at (0.5,1) {}; 
				\node[vertex, label=right:{\footnotesize 4}] at (0,2) {}; 
				\node[vertex, label=right:{\footnotesize 5}] at (1,2) {}; 
				\draw[edge] (0,0) -- (-0.5, 1);
				\draw[edge] (0,0) -- (0.5, 1);
				\draw[edge] (0.5, 1) -- (0,2);
				\draw[edge] (0.5, 1) -- (1,2);
				\begin{pgfonlayer}{background}
					\draw[zhyedge1] (0,0) -- (0,0);
					\draw[zhyedge2] (0,0) -- (0,0);
					\draw[hyedge1, color=red] (-0.5,1) -- (-0.5,1);
					\draw[hyedge2] (-0.5,1) -- (-0.5,1);
					\draw[hyedge1, color=blue] (0.5,1) -- (1,2);
					\draw[hyedge2] (0.5,1) -- (1,2);
					\draw[hyedge1, color=green] (0,2) -- (0,2);
					\draw[hyedge2] (0,2) -- (0,2);
				\end{pgfonlayer}
			\end{tikzpicture}
			,
			\\
			\begin{tikzpicture}
				\node[vertex, label=right:{\footnotesize 1}] at (0,0) {}; 
				\node[vertex, label=right:{\footnotesize 2}] at (-0.5,1) {}; 
				\node[vertex, label=right:{\footnotesize 3}] at (0.5,1) {}; 
				\node[vertex, label=right:{\footnotesize 4}] at (0,2) {}; 
				\node[vertex, label=right:{\footnotesize 5}] at (1,2) {}; 
				\draw[edge] (0,0) -- (-0.5, 1);
				\draw[edge] (0,0) -- (0.5, 1);
				\draw[edge] (0.5, 1) -- (0,2);
				\draw[edge] (0.5, 1) -- (1,2);
				\begin{pgfonlayer}{background}
					\draw[zhyedge1] (-0.5,1) -- (0,0) -- (0.5, 1);
					\draw[zhyedge2] (-0.5,1) -- (0,0) -- (0.5, 1);
					\draw[hyedge1, color=red] (0,2) -- (1,2);
					\draw[hyedge2] (0,2) -- (1,2);
				\end{pgfonlayer}
			\end{tikzpicture}
			,\quad
			\begin{tikzpicture}
				\node[vertex, label=right:{\footnotesize 1}] at (0,0) {}; 
				\node[vertex, label=right:{\footnotesize 2}] at (-0.5,1) {}; 
				\node[vertex, label=right:{\footnotesize 3}] at (0.5,1) {}; 
				\node[vertex, label=right:{\footnotesize 4}] at (0,2) {}; 
				\node[vertex, label=right:{\footnotesize 5}] at (1,2) {}; 
				\draw[edge] (0,0) -- (-0.5, 1);
				\draw[edge] (0,0) -- (0.5, 1);
				\draw[edge] (0.5, 1) -- (0,2);
				\draw[edge] (0.5, 1) -- (1,2);
				\begin{pgfonlayer}{background}
					\draw[zhyedge1] (-0.5,1) -- (0,0);
					\draw[zhyedge2] (-0.5,1) -- (0,0);
					\draw[hyedge1, color=red] (0.5,1) -- (0.5,1);
					\draw[hyedge2] (0.5,1) -- (0.5,1);
					\draw[hyedge1, color=blue] (0,2) -- (1,2);
					\draw[hyedge2] (0,2) -- (1,2);
				\end{pgfonlayer}
			\end{tikzpicture}
			,\quad
			\begin{tikzpicture}
				\node[vertex, label=right:{\footnotesize 1}] at (0,0) {}; 
				\node[vertex, label=right:{\footnotesize 2}] at (-0.5,1) {}; 
				\node[vertex, label=right:{\footnotesize 3}] at (0.5,1) {}; 
				\node[vertex, label=right:{\footnotesize 4}] at (0,2) {}; 
				\node[vertex, label=right:{\footnotesize 5}] at (1,2) {}; 
				\draw[edge] (0,0) -- (-0.5, 1);
				\draw[edge] (0,0) -- (0.5, 1);
				\draw[edge] (0.5, 1) -- (0,2);
				\draw[edge] (0.5, 1) -- (1,2);
				\begin{pgfonlayer}{background}
					\draw[zhyedge1] (0,0) -- (0.5,1);
					\draw[zhyedge2] (0,0) -- (0.5,1);
					\draw[hyedge1, color=red] (-0.5,1) -- (-0.5,1);
					\draw[hyedge2] (-0.5,1) -- (-0.5,1);
					\draw[hyedge1, color=blue] (0,2) -- (1,2);
					\draw[hyedge2] (0,2) -- (1,2);
				\end{pgfonlayer}
			\end{tikzpicture}
			,\quad
			\begin{tikzpicture}
				\node[vertex, label=right:{\footnotesize 1}] at (0,0) {}; 
				\node[vertex, label=right:{\footnotesize 2}] at (-0.5,1) {}; 
				\node[vertex, label=right:{\footnotesize 3}] at (0.5,1) {}; 
				\node[vertex, label=right:{\footnotesize 4}] at (0,2) {}; 
				\node[vertex, label=right:{\footnotesize 5}] at (1,2) {}; 
				\draw[edge] (0,0) -- (-0.5, 1);
				\draw[edge] (0,0) -- (0.5, 1);
				\draw[edge] (0.5, 1) -- (0,2);
				\draw[edge] (0.5, 1) -- (1,2);
				\begin{pgfonlayer}{background}
					\draw[zhyedge1] (0,0) -- (0,0);
					\draw[zhyedge2] (0,0) -- (0,0);
					\draw[hyedge1, color=red] (-0.5,1) -- (0.5,1);
					\draw[hyedge2] (-0.5,1) -- (0.5,1);
					\draw[hyedge1, color=blue] (0,2) -- (1,2);
					\draw[hyedge2] (0,2) -- (1,2);
				\end{pgfonlayer}
			\end{tikzpicture}
			,\quad
			\begin{tikzpicture}
				\node[vertex, label=right:{\footnotesize 1}] at (0,0) {}; 
				\node[vertex, label=right:{\footnotesize 2}] at (-0.5,1) {}; 
				\node[vertex, label=right:{\footnotesize 3}] at (0.5,1) {}; 
				\node[vertex, label=right:{\footnotesize 4}] at (0,2) {}; 
				\node[vertex, label=right:{\footnotesize 5}] at (1,2) {}; 
				\draw[edge] (0,0) -- (-0.5, 1);
				\draw[edge] (0,0) -- (0.5, 1);
				\draw[edge] (0.5, 1) -- (0,2);
				\draw[edge] (0.5, 1) -- (1,2);
				\begin{pgfonlayer}{background}
					\draw[zhyedge1] (0,0) -- (0,0);
					\draw[zhyedge2] (0,0) -- (0,0);
					\draw[hyedge1, color=red] (-0.5,1) -- (-0.5,1);
					\draw[hyedge2] (-0.5,1) -- (-0.5,1);
					\draw[hyedge1, color=blue] (0.5,1) -- (0.5,1);
					\draw[hyedge2] (0.5,1) -- (0.5,1);
					\draw[hyedge1, color=green] (0,2) -- (1,2);
					\draw[hyedge2] (0,2) -- (1,2);
				\end{pgfonlayer}
			\end{tikzpicture}
			,
			\\
			\begin{tikzpicture}
				\node[vertex, label=right:{\footnotesize 1}] at (0,0) {}; 
				\node[vertex, label=right:{\footnotesize 2}] at (-0.5,1) {}; 
				\node[vertex, label=right:{\footnotesize 3}] at (0.5,1) {}; 
				\node[vertex, label=right:{\footnotesize 4}] at (0,2) {}; 
				\node[vertex, label=right:{\footnotesize 5}] at (1,2) {}; 
				\draw[edge] (0,0) -- (-0.5, 1);
				\draw[edge] (0,0) -- (0.5, 1);
				\draw[edge] (0.5, 1) -- (0,2);
				\draw[edge] (0.5, 1) -- (1,2);
				\begin{pgfonlayer}{background}
					\draw[zhyedge1] (-0.5,1) -- (0,0) -- (0.5, 1);
					\draw[zhyedge2] (-0.5,1) -- (0,0) -- (0.5, 1);
					\draw[hyedge1, color=red] (0,2) -- (0,2);
					\draw[hyedge2] (0,2) -- (0,2);
					\draw[hyedge1, color=blue] (1,2) -- (1,2);
					\draw[hyedge2] (1,2) -- (1,2);
				\end{pgfonlayer}
			\end{tikzpicture}
			,\quad
			\begin{tikzpicture}
				\node[vertex, label=right:{\footnotesize 1}] at (0,0) {}; 
				\node[vertex, label=right:{\footnotesize 2}] at (-0.5,1) {}; 
				\node[vertex, label=right:{\footnotesize 3}] at (0.5,1) {}; 
				\node[vertex, label=right:{\footnotesize 4}] at (0,2) {}; 
				\node[vertex, label=right:{\footnotesize 5}] at (1,2) {}; 
				\draw[edge] (0,0) -- (-0.5, 1);
				\draw[edge] (0,0) -- (0.5, 1);
				\draw[edge] (0.5, 1) -- (0,2);
				\draw[edge] (0.5, 1) -- (1,2);
				\begin{pgfonlayer}{background}
					\draw[zhyedge1] (-0.5,1) -- (0,0);
					\draw[zhyedge2] (-0.5,1) -- (0,0);
					\draw[hyedge1, color=red] (0.5,1) -- (0.5,1);
					\draw[hyedge2] (0.5,1) -- (0.5,1);
					\draw[hyedge1, color=blue] (0,2) -- (0,2);
					\draw[hyedge2] (0,2) -- (0,2);
					\draw[hyedge1, color=green] (1,2) -- (1,2);
					\draw[hyedge2] (1,2) -- (1,2);
				\end{pgfonlayer}
			\end{tikzpicture}
			,\quad
			\begin{tikzpicture}
				\node[vertex, label=right:{\footnotesize 1}] at (0,0) {}; 
				\node[vertex, label=right:{\footnotesize 2}] at (-0.5,1) {}; 
				\node[vertex, label=right:{\footnotesize 3}] at (0.5,1) {}; 
				\node[vertex, label=right:{\footnotesize 4}] at (0,2) {}; 
				\node[vertex, label=right:{\footnotesize 5}] at (1,2) {}; 
				\draw[edge] (0,0) -- (-0.5, 1);
				\draw[edge] (0,0) -- (0.5, 1);
				\draw[edge] (0.5, 1) -- (0,2);
				\draw[edge] (0.5, 1) -- (1,2);
				\begin{pgfonlayer}{background}
					\draw[zhyedge1] (0,0) -- (0.5,1);
					\draw[zhyedge2] (0,0) -- (0.5,1);
					\draw[hyedge1, color=red] (-0.5,1) -- (-0.5,1);
					\draw[hyedge2] (-0.5,1) -- (-0.5,1);
					\draw[hyedge1, color=blue] (0,2) -- (0,2);
					\draw[hyedge2] (0,2) -- (0,2);
					\draw[hyedge1, color=green] (1,2) -- (1,2);
					\draw[hyedge2] (1,2) -- (1,2);
				\end{pgfonlayer}
			\end{tikzpicture}
			,\quad
			\begin{tikzpicture}
				\node[vertex, label=right:{\footnotesize 1}] at (0,0) {}; 
				\node[vertex, label=right:{\footnotesize 2}] at (-0.5,1) {}; 
				\node[vertex, label=right:{\footnotesize 3}] at (0.5,1) {}; 
				\node[vertex, label=right:{\footnotesize 4}] at (0,2) {}; 
				\node[vertex, label=right:{\footnotesize 5}] at (1,2) {}; 
				\draw[edge] (0,0) -- (-0.5, 1);
				\draw[edge] (0,0) -- (0.5, 1);
				\draw[edge] (0.5, 1) -- (0,2);
				\draw[edge] (0.5, 1) -- (1,2);
				\begin{pgfonlayer}{background}
					\draw[zhyedge1] (0,0) -- (0,0);
					\draw[zhyedge2] (0,0) -- (0,0);
					\draw[hyedge1, color=red] (-0.5,1) -- (0.5,1);
					\draw[hyedge2] (-0.5,1) -- (0.5,1);
					\draw[hyedge1, color=blue] (0,2) -- (0,2);
					\draw[hyedge2] (0,2) -- (0,2);
					\draw[hyedge1, color=green] (1,2) -- (1,2);
					\draw[hyedge2] (1,2) -- (1,2);
				\end{pgfonlayer}
			\end{tikzpicture}
			,\quad
			\begin{tikzpicture}
				\node[vertex, label=right:{\footnotesize 1}] at (0,0) {}; 
				\node[vertex, label=right:{\footnotesize 2}] at (-0.5,1) {}; 
				\node[vertex, label=right:{\footnotesize 3}] at (0.5,1) {}; 
				\node[vertex, label=right:{\footnotesize 4}] at (0,2) {}; 
				\node[vertex, label=right:{\footnotesize 5}] at (1,2) {}; 
				\draw[edge] (0,0) -- (-0.5, 1);
				\draw[edge] (0,0) -- (0.5, 1);
				\draw[edge] (0.5, 1) -- (0,2);
				\draw[edge] (0.5, 1) -- (1,2);
				\begin{pgfonlayer}{background}
					\draw[zhyedge1] (0,0) -- (0,0);
					\draw[zhyedge2] (0,0) -- (0,0);
					\draw[hyedge1, color=red] (-0.5,1) -- (-0.5,1);
					\draw[hyedge2] (-0.5,1) -- (-0.5,1);
					\draw[hyedge1, color=blue] (0.5,1) -- (0.5,1);
					\draw[hyedge2] (0.5,1) -- (0.5,1);
					\draw[hyedge1, color=green] (0,2) -- (0,2);
					\draw[hyedge2] (0,2) -- (0,2);
					\draw[hyedge1, color=purple] (1,2) -- (1,2);
					\draw[hyedge2] (1,2) -- (1,2);
				\end{pgfonlayer}
			\end{tikzpicture}
			.
		\end{align*}
		The grey partition element corresponds to the tagged hyperedge while the coloured hyperedges correspond the detagged hyperedges where different colours indicate distinct sets. 
		
		By examination, we can see that each of these tagged partitions satisfies \ref{definition:Forests:2.2}, \ref{definition:Forests:2.3} and \ref{definition:Forests:2.1} so these are all Lions forests. Further, by manual calculation we can also verify that these are all of the partitions of the nodes of the rooted forest $\big( \{1, 2,3,4,5\}, \{(2,1), (3,1),(4,3),(5,3) \} \big)$ that satisfy Definition \ref{definition:Quotient-Partition} are included in this collection. 
	\end{remark}
	
	\begin{theorem}
		\label{theorem:EquivalenceTrees}
		Let $d\in \bN$ and let $I$ be an index set. There is an isomorphism between 
		\begin{equation*}
			\big( \scF_{0,d}[I], \subseteq, \fm \big)
			\quad \mbox{and}\quad
			\big( \hat{\scF}_{0,d}[I], \hat{\subseteq}, \hat{\fm} \big).
		\end{equation*}
		As such, we view Definitions \ref{definition:Forests} and \ref{definition:Lions-Partition-Tree} as equivalent. 
	\end{theorem}
	
	In order to prove Theorem \ref{theorem:EquivalenceTrees}, we first introduce a sense of union which allows us to stitch together a collection of local partitions. 
	\begin{definition}
		\label{definition:OverlineUnion}
		Let $\scM$ and $\scN$ be finite sets such that each of the sets $\scM\cap \scN$, $\scM\backslash\scN$ and $\scN \backslash\scM$ are all non-empty. Let $P \in \scP(\scM)$ and $Q \in \scP(\scN)$. We denote
		\begin{equation}
			\label{eq:definition:OverlineUnion1}
			P \cap \scN:=\{ p\cap \scN: p\in P\}\backslash\{\emptyset\},
			\quad 
			Q \cap \scM:= \{ q\cap \scM: q\in Q\}\backslash\{\emptyset\}. 
		\end{equation}
		Suppose that 
		\begin{equation}
			\label{eq:definition:OverlineUnion3}
			\forall q\in Q\cap \scM, \exists p \in P \cap \scN \quad \mbox{such that} \quad q \subseteq p.
		\end{equation}
		Then we define $P\overline{\cup} Q \in \scP( \scM\cup \scN )$ by
		\begin{equation}
			\label{eq:definition:OverlineUnion2}
			P\overline{\cup} Q: = \Big\{ p\in P: p\cap \scN = \emptyset \Big\} \cup \Big\{ q\in Q: q\cap \scM = \emptyset \Big\} \cup \Big\{p \cup \bigcup_{\substack{q \in Q\\ q \cap p \neq \emptyset}} q: p \in P, p\cap \scN \neq \emptyset \Big\}. 
		\end{equation}
	\end{definition}
	
	The operation $\overline{\cup}$ allows us to take the union of partitions that nest (in the sense described in Relationship \eqref{eq:definition:OverlineUnion3}) to obtain another partition. Due to the ordered nature of Relationship \eqref{eq:definition:OverlineUnion3}, this union is not commutative. 
	
	\begin{proof}[Proof of Theorem \ref{theorem:EquivalenceTrees}]
		\textit{Step 1.} We show that we can identify a Lions partition forest with a Lions forest. Let
		\begin{equation*}
			\big( \scN, \scE, (\hat{h}_I, \hat{H}), \fH, \scL \big) \in \hat{\scF}_{0, d}[I]. 
		\end{equation*}
		We denote the equivalence classes described by the binary relation $\lesseqqgtr$ on the set of nodes $\scN$ in Equation \eqref{eq:EquivalenceClassTree} by
		\begin{align*}
			\scN_0 =& \big\{x\in \scN: \forall y\in \scN, x\leq y \big\} = \fr(\scN), 
			\quad
			\scN_1 = \big\{y\in \scN: \exists x\in \scN_0 \mbox{ such that } (y, x)\in \scE \big\}, 
			\\
			\scN_{i+1} =& \big\{ y\in \scN: \exists x\in \scN_i \mbox{ such that } (y, x) \in \scE \big\}. 
		\end{align*}
		
		Note that as $\scN$ is finite, there are a finite number of equivalence classes and let $n\in \bN$ be the number of equivalence classes. Also, for $x, x'\in \scN_i$, we have (recalling Equation \eqref{eq:N_x}) that $\scN_x \cap \scN_{x'} = \emptyset$. For every $x \in \scN$, we write $\fH[x]:=(h_0^x, H^x)$ and define $\fH[x]':= H^x \cup \big\{ h_0^x \cup \{x\} \big\}$. Next, we define
		\begin{align*}
			H^0:= \Big( \hat{H} \cup \big\{ \hat{h}_\iota: \iota \in I \big\} \Big) \backslash \{ \emptyset\}
			\quad\mbox{and for every $i\in \bN$,} \quad H^i:=& \bigcup_{x\in \scN_i} \fH[x]' \in \scP\Big( \scN_i \cup \scN_{i+1}\Big). 
		\end{align*}
		The two partitions $H^0$ and $H^1$ satisfy Equation \eqref{eq:definition:OverlineUnion3}, so using Definition \ref{definition:OverlineUnion} we can define
		\begin{equation*}
			\overline{H}^1 = H^0 \overline{\cup} H^1
		\end{equation*}
		Further, the two partitions $H^1$ and $H^2$ satisfy Equation \eqref{eq:definition:OverlineUnion3} and the sets $\scN_0 \cup \scN_1$ and $\scN_2 \cup \scN_3$ are disjoint so that we can further define
		\begin{equation*}
			\overline{H}^2 = (H^0 \overline{\cup} H^1) \overline{\cup} H^2 = \overline{H}^1 \overline{\cup} H^2. 
		\end{equation*}
		By repeating this argument $n$ times, we denote the resulting partition
		\begin{equation}
			\label{eq:definition:Local-Partition}
			H':= \Big( \big( (H^0 \overline{\cup} H^1) \overline{\cup} H^2 \big) \overline{\cup} ... \Big) \overline{\cup} H^n
			\in \scP( \scN ). 
		\end{equation}
		For every $\iota \in I$, when the set $\hat{h}_\iota$ is non-empty $\exists h_\iota \in H'$ such that $\hat{h}_\iota \subseteq h_\iota$. On the other hand, when $\hat{h}_\iota = \emptyset$ we prescribe that the set $h_\iota: = \emptyset$ too. Then
		\begin{equation*}
			H:= H' \backslash \{ h_\iota: \iota \in I \} \in \scP\Big( \scN \backslash \big( \bigcup_{\iota \in I} h_\iota \big) \Big)
			\quad \mbox{and we obtain}\quad
			(h_I, H) \in \scP(\scN)[I]. 
		\end{equation*}
		We want to prove that $(\scN, \scE, h_I, H, \scL)$ is a Lions forest or equivalently that $(h_I, H) \in \scQ^{\scE}(\scN)[I]$. 
		
		Firstly, let $h \in H'= \overline{\bigcup}_{i\in \bN_0} H^i$ and suppose that $\exists x,y \in h$ such that $x <y$. Since $y$ is not a root of $T$, $\exists z\in \scN$ such that $(y,z) \in \scE$. Let us suppose for a contradiction that $z \notin h$. 
		
		$\exists i \in \bN_0$ such that $z \in \scN_i$ and $y\in \scN_{i+1}$. Further, $\exists j\leq i$ such that $x\in \scN_j$. Since $z \notin h$, we have that the partition $\fH[z]$ contains $h^y$ such that $z \notin h^y$ and $y \in h^y$. Therefore $h^y \in H^i$ and $h^y \cap \scN_i = \emptyset$. This implies that $h^y \in H^{i-1} \overline{\cup} H^{i}$. Further, $h^y \cap \scN_{i-1} = \emptyset$ so that $h^y \in H^{i-2} \overline{\cup} H^{i-1} \overline{\cup} H^{i}$ and so forth. Repeating this argument, we get that $h^y \in \overline{\bigcup}_{j\leq i}  H^j $ and $\overline{\bigcup}_{j\leq i} H^j \in \scP\Big( \cup_{j\leq i+1} \scN_j \Big)$. Since $x \in \cup_{j\leq i+1} \scN_j$, $\exists h^x \in \overline{\bigcup}_{j\leq i} H^j $ and $h^y \neq h^x$. Therefore \ref{definition:Forests:2.2} holds. 
		
		As the sets $h^y \neq h^x$, we have that $\exists \tilde{h}^x, \tilde{h}^y \in H$ such that $h^x \subseteq \tilde{h}^x$ and $h^y \subseteq \tilde{h}^y$ and $\tilde{h}^x \cap \tilde{h}^y = \emptyset$. However, this contradicts the assumption that $x$ and $y$ are contained in the same set $h$. Thus $z \in h$. 
		
		Secondly, let $h\in H'$, let $x_1, y_1 \in h$ and suppose $\exists x_2, y_2 \in \scN$ such that $(x_1, x_2), (y_1, y_2) \in \scE$. Without loss of generality, suppose for a contradiction that $x_2 \notin h$. 
		
		Let $i\geq 0$ such that $x_2, y_2 \in \scN_i$. We have $x_1 \in \scN_{x_2}$ and $y_1 \in \scN_{y_2}$, so that $\exists h^x \in \fH[x_2]$ and $\exists h^y \in \fH[y_2]$ and $x_1 \in h^x$ and $y_1 \in h^y$. Thus $h^x, h^y \in H^{i}$ and $h^x \cap h^y = \emptyset$. Since by hypothesis $h^x \cap \scN_i = \emptyset$, we have that $h^x \in H^{i-1} \overline{\cup} H^i$ and $\exists h'^{y} \in H^{i-1} \overline{\cup} H^i$. Repeating the same argument as before, we get that $h^x \in \overline{\bigcup}_{j\leq i} H^j$ and $\exists h'^y \in \overline{\bigcup}_{j\leq i} H^j$ such that $h^y \subseteq h'^y$. 
		
		Then $\exists \tilde{h}^x, \tilde{h}^y \in H'$ such that $h^x \subseteq \tilde{h}^x$ and $h'^y \subseteq \tilde{h}^y$ and $\tilde{h}^x \cap \tilde{h}^y = \emptyset$. However, this is a contradiction since we assumed that $x_1, y_1 \in h \in H'$.  Therefore $x_2 \in h$. By symmetry, we also conclude that $y_2 \in h$.  Therefore \ref{definition:Forests:2.3} holds. 
		
		Finally, for every $\iota \in I$ when the set $h_\iota$ is non-empty we have that $\hat{h}_\iota \subset h_\iota$, and $\hat{h}_\iota \subseteq \fr(\scN)$ so that $h_\iota$ contains a root. Therefore \ref{definition:Forests:2.1} holds. Hence, the partition $\overline{\bigcup}_i H^i$ satisfies the requirements of Definition \ref{definition:Forests} such that $(\scN, \scE, h_\iota, H, \scL) \in \scF_d[I]$.
		
		\textit{Step 2.} Next, we show that we can identify a labelled Lions forest with a labelled Lions partition forest. Let $(\scN, \scE, h_I, H, \scL)$ be a Lions forest. For starters, for every $\iota \in I$ consider the subset
		\begin{equation*}
			\hat{h}_\iota:=h_\iota \cap \fr(\scN)
			\quad \mbox{and}\quad
			\hat{H} = H \cap \fr(\scN) \in \scP( \fr(\scN) \backslash \big( \bigcup_{\iota \in I} \hat{h}_\iota \big)).
		\end{equation*} 
		Denote $H' = \Big( H \cup \bigcup_{\iota \in I} \{h_\iota\}\Big) \backslash \{\emptyset\}$. For each $x\in \scN$, $\exists h_x \in H'$ such that $x \in h_x$. We define 
		\begin{equation}
			\label{eq:FrakH_construction}
			\fH \in \bigsqcup_{x\in \scN} \scP(\scN_x)[0]
			\quad\mbox{by}\quad
			\fH[x]:= \Big( h_x \cap \scN_x, \{ h \cap \scN_x: h \in H' \} \backslash\{ \emptyset, h_x \cap \scN_x\} \Big).
		\end{equation}
		Then $\fH[x] \in \scP(\scN_x )[0]$ so that we have that $\big( \scN, \scE, (\hat{h}_I, \hat{H}), \fH, \scL\big)$ is a Lions partition forest. 
		
		\textit{Step 3.} Finally, it is easy to verify that mappings we have constructed that map $\scF_d[I] \to \hat{\scF}_d[I]$ and $\hat{\scF}_d[I] \to \scF_d[I]$ are the inverse of one another. 
		
		Let $\hat{T} = \big(\scN, \scE, (\hat{h}_I, \hat{H}), \fH, \scL \big) \in \hat{\scF}_{0, d}[I]$ and let $(h_I, H) \in \scQ^{\scE}(\scN)[I]$ be the associated tagged partition constructed in Step 1. Then for any $x\in \scN$ we have that the partition $H'$ restricted to the set $\scN_x$ agrees with the partition $\fH[x]' \in \scP(\scN_x)$ where for $\fH[x] = (h_0^x, H^x)$, we write $\fH[x]':= \big(H^x \cup\{h_0^x\} \big)\backslash \{\emptyset\}$. Therefore, we conclude that $\fH$ as constructed from the partition $H'$ using \eqref{eq:FrakH_construction} is equal to $\fH$ from the hypothesis. 
		
		Similarly, let $T = (\scN, \scE, h_I, H, \scL) \in \scF_{0, d}[I]$ and for every $\iota \in I$ let $\hat{h}_\iota = h_\iota \cap \fr[T]$, $\hat{H} = H \cap \fr(\scN)$ and for each $x\in \scN$ define $\fH[x]$ according to \eqref{eq:FrakH_construction}. By writing 
		\begin{equation*}
			\fH[x]' = H^x \cup \big\{ h_0^x \cup\{x\} \big\}, 
			\quad\mbox{we obtain}\quad
			H^i: = \bigcup_{x\in \scN_i} \fH[x]' = H' \cap \Big( \scN_i \cup \scN_{i+1} \Big)
		\end{equation*}
		so that the partition
		\begin{equation*}
			H^i \overline{\cup} H^{i+1} = H' \cap \Big( \scN_i \cup \scN_{i+1} \cup \scN_{i+2} \Big)
			\quad \mbox{so we conclude}\quad
			\bigcup_{i\in \bN_0} H^i = H'. 
		\end{equation*}
	\end{proof}
	
	\begin{remark}
		Theorem \ref{theorem:EquivalenceTrees} shows us the link between partitions constructed using the method seen in Equation \eqref{eq:definition:Local-Partition} (stitching together local partitions) and partitions that satisfy Definition \ref{definition:Quotient-Partition}. Further, the link between partitions of the set $\scQ^\scE(\scN)[I]$ and tagged Lions forests as defined in Definition \ref{definition:Forests} is clear. 
		
		The use of these partition structures for studying elementary differentials derived from Lions derivatives is physically justified since the process of taking iterative Lions derivatives can be abstracted to correspond to the construction of a partition of a directed forest from its local partitions. Each Lions derivative corresponds to a partition sequence $a \in A[0]$ paired with a node of the forest which uniquely determines how the nodes above that are partitioned. 
		
		Thus, we should see the use of Definition \ref{definition:Forests} in rough path theory as being justified by this result. 
	\end{remark}
	
	\iftoggle{Plus}{
	Our final step is to understand how a partition from $\scP(\scN)[I]$ can be transformed into a partition from $\scQ^{\scE}(\scN)[I]$:
	\begin{proposition}
		\label{proposition:Partitions2}
		Let $\tau=(\scN, \scE)$ be a directed tree and let $I$ be an index set. Suppose that $(h_I, H) \in \scP(\scN)[I]$. Then $\exists (\tilde{h}_I, \tilde{H}) \in \scQ^{\scE}(\scN)[I]$ such that $(\tilde{h}_I, \tilde{H}) \subseteq (h_I, H)$. 
		
		Further, $\exists! (\tilde{h}_I', \tilde{H}') \in \scQ^{\scE}(\scN)[I]$ such that
		\begin{equation}
			\label{eq:proposition:Partitions2}
			\forall (\tilde{h}_I'', \tilde{H}'') \in \Big\{ (\tilde{h}_I, \tilde{H}) \in \scQ^{\scE}(\scN)[I]: (\tilde{h}_I, \tilde{H}) \subseteq (h_I, H) \Big\}, 
			\quad
			(\tilde{h}_I'', \tilde{H}'') \subseteq (\tilde{h}_I', \tilde{H}'). 
		\end{equation}
	\end{proposition}
	Intuitively, this means that tagged partitions $(h_I,H)\in \scP(\scN)[I]$ can always be projected onto the subset $\scQ^{\scE}(\scN)[I]$ (although not in a unique way) and there exists a minimal partition under the partital ordering determined by Definition \ref{definition:Quotient-Partition} (in the sense that it has minimal number of elements). 
	
	\begin{proof}
		The existence of such a tagged partition $(\tilde{h}_I, \tilde{H}) \in \scQ^{\scE}(\scN)[I]$ is immediate by choosing $\tilde{h}_\iota = \emptyset$ for every $\iota \in I$ and $\tilde{H}:=\big\{ \{x\}: x\in \scN\big\}$. Clearly $(\tilde{h}_I, \tilde{H}) \subseteq (h_I, H)$, but this will not be a pertinent choice and we focus on constructing a relevant partition and showing that this partition is a maximal element. 
		
		For $H' = H \cup \{h_{\iota}: \iota\in I\} \in \scP(\scN)$ and $x\in \scN$, (recalling Equation \eqref{eq:N_x}) we define 
		\begin{equation*}
			H^x:=\Big\{ h \cap \big( \{x\} \cup \scN_x\big): h \in H \Big\} \backslash\{ \emptyset \}, 
		\end{equation*}
		we define $H^i:= \bigcup_{x\in \scN_i} H^x$ and $\tilde{H}':= \overline{\bigcup}_{i \in \bN} H^i$ as in Equation \eqref{eq:definition:Local-Partition}. Then by Theorem \ref{theorem:EquivalenceTrees} we have that $\tilde{H}' \in \scQ^\scE(\scN)$ and $\forall \tilde{h} \in \tilde{H}'$ we have that $\exists h\in H$ such that $\tilde{h} \subseteq h$. 
		
		We prove the set $\tilde{H}'$ is the maximal partition in the sense of Equation \eqref{eq:proposition:Partitions2}. Suppose first that the partition $H \in \scQ^\scE(\scN)$. Then by Theorem \ref{theorem:EquivalenceTrees} we have that $\tilde{H}' = H$. Suppose that $\tilde{H}'' \subseteq H$. Then since $H = \tilde{H}'$, we have that $\tilde{H}'' \subseteq \tilde{H}'$. Thus $\tilde{H}'$ is the unique maximal. 
		
		Now suppose that $H \notin \scQ^\scE(\scN)$. Suppose that $\tilde{H}'' \in \scQ^\scE(\scN)$ and that $\tilde{H}'' \subseteq H$. Let $\tilde{h}'' \in \tilde{H}''$ and let $h\in H$ such that $\tilde{h}'' \subseteq h$. Then $\forall x\in \tilde{h}''$ such that $\exists y \in \tilde{h}''$ such that $(y, x) \in \scE$, we have that the sets
		\begin{equation*}
			H^x:=\Big\{ h \cap ( \{x\} \cup \scN_x) \Big\}, \quad (\tilde{H}'')^x:=\Big\{ \tilde{h}'' \cap ( \{x\} \cup \scN_x) \Big\} 
		\end{equation*}
		are equal. Therefore $H^i$ and $(\tilde{H}'')^i$ defined accordingly are also equal, so that $\exists \tilde{h}' \in \tilde{H}'$ such that $\tilde{h}'' \subseteq \tilde{h}'$. Therefore $\tilde{H}'' \subseteq \tilde{H}'$, so that $\tilde{H}'$ is maximal. 
		
		Let $(h_I, H) \in \scP(\scN)[I]$. Firstly, for each $\iota \in I$, we either have that the set $h_{\iota}$ satisfies all three of \ref{definition:Forests:2.2}, \ref{definition:Forests:2.3} and \ref{definition:Forests:2.1} or it does not. For any $\iota \in I$ for which it does, we denote $h_{\iota}' = h_{\iota}$ and for all other $\iota \in I$ we denote $h_{\iota}' = \emptyset$. The partition $\tilde{H}' \in \scQ^{\scE}(\scN)$ constructed above can then be rewritten as
		\begin{equation*}
			\tilde{H}': = \tilde{H}' \backslash \{h_{\iota}': \iota \in I \}
		\end{equation*}
		and $(\tilde{h}_I', \tilde{H}') \in \scQ^{\scE}(\scN)[I]$. Further, it satisfies \eqref{eq:proposition:Partitions2} and we conclude.  
	\end{proof}
	}{}

	\subsubsection{Lions forests as the completion of tree operations}
	
	Lions forests (and tagged Lions forests) have a natural product structure:
	\begin{definition}
		\label{definition:Lionsproduct}
		Let $I$ be an index set. Then we define $\circledast: \scF_{0,d}[I] \times \scF_{0,d}[I] \to \scF_{0,d}[I]$ so that for two Lions forests 
		\begin{equation*}
			T_1=\big( \scN^1, \scE^1, (h_\iota^1)_{\iota\in I}, H^1, \scL^1 \big)
			\quad \mbox{and} \quad 
			T_2 = \big( \scN^2, \scE^2, (h_\iota^2)_{\iota\in I}, H^2, \scL^2 \big), 
		\end{equation*}
		we have $T_1 \circledast T_2 = \big( \tilde{\scN}, \tilde{\scE}, (\tilde{h_\iota})_{\iota\in I}, \tilde{H}, \tilde{\scL} \big)$ such that 
		\begin{align*}
			&\tilde{\scN} = \scN^1 \cup \scN^2,
			\quad
			\tilde{\scE} = \scE^1 \cup \scE^2, 
			\quad
			\tilde{h}_\iota = h_\iota^1 \cup h_\iota^2 
			\quad
			\tilde{H} = H^1 \cup H^2,  
			\\
			&\tilde{\scL}: \tilde{\scN} \to \{1, ..., d\} \quad \mbox{ such that } \quad \tilde{\scL}|_{\scN^i} = \scL^i. 
		\end{align*}
	\end{definition}
	Simple computations allow us to verify that the $\circledast$ operation is associative and commutative product on the space of (tagged) Lions forests with unit $\rId$. 
	
	\begin{definition}
		\label{definition:Coupled-E}
		Let $\cE: \scF_{0,d}[0] \to \scF_{0,d}[0]$ be the operator defined for $T=(\scN, \scE, h_0, H, \scL) \in \scF_{0,d}[0]$ by
		\begin{equation*}
			\cE[T] = (\scN, \scE, \emptyset, H', \scL).
		\end{equation*}
		We refer to $\cE$ as the decoupling operation. 
		
		Let $a \in A_{n}[0]$. We define the operator $\cE^a: \big(\scF_{0, d}[0]\big)^{\times n} \to \scF_{0, d}[0]$ by
		\begin{equation}
			\label{eq:definition:E^a-notation}
			\cE^a\Big[ T_1, ..., T_n\Big] = \Big[ \underset{\substack{i:\\ a_i=0}}{\circledast} T_i \Big] \circledast \cE\Big[ \underset{\substack{i:\\ a_i=1}}{\circledast} T_i \Big] \circledast ... \circledast \cE\Big[ \underset{\substack{i:\\ a_i=m[a]}}{\circledast} T_i \Big] . 
		\end{equation}
		Intuitively, the decoupling by the partition sequence $\cE^a$ groups the sequence of trees into $m[a]+1$ partitions with a single group tagged and all others detagged. 
		
		Similarly, let $I$ be an index set. Then we extend the decoupling operation to a collection of coupling operations $(\cE_\iota)_{\iota\in I}$ such that for each $\iota\in I$, $\cE_i: \scF_{0,d}[0] \to \scF_{0,d}[I]$ be the operator defined for $T = \big( \scN, \scE, h_0, H, \scL \big) \in \scF_{0,d}[0]$ by
		\begin{equation*}
			\cE_\iota[T] = \big( \scN, \scE, (\tilde{h}_{\hat{\iota}})_{\hat{\iota} \in I}, H, \scL \big), \quad \tilde{h}_{\iota} = h_0, \quad \tilde{h}_{\hat{\iota}} = \emptyset \quad \mbox{for $\hat{\iota} \in I \backslash \{\iota\}$}. 
		\end{equation*}
		
		Let $a \in A_{n}[I]$. We define the operator $\cE^a: \big( \scF_{0, d}[0] \big)^{\times n} \to \scF_{0, d}[I]$ by
		\begin{equation}
			\label{eq:definition:E^a-notation-}
			\cE^a\Big[ T_1, ..., T_n\Big] = \underset{\iota \in I}{\scalebox{1.5}{$\circledast$}} \cE_\iota \Big[ \underset{\substack{i:\\ a_i=\iota }}{\scalebox{1.5}{$\circledast$}} T_\iota \Big]  \circledast \cE\Big[ \underset{\substack{i:\\ a_i=1}}{\scalebox{1.5}{$\circledast$}} T_\iota \Big] \circledast ... \circledast \cE\Big[ \underset{\substack{i:\\ a_i=m[a]}}{\scalebox{1.5}{$\circledast$}} T_\iota \Big] 
		\end{equation}
		Intuitively, the decoupling by the partition sequence $\cE^a$ groups the sequence of trees into $m[a]+|I|$ partitions.	
	\end{definition}

	\begin{definition}
		\label{definition:Rooting}
		We denote $\scT_{0, d}[0]$ to be the collection of all Lions trees with labellings taking their value in the set $\{1, ..., d\}$ such that the tagged hyperedge $h_0$ is non-empty. 
		
		Let $i\in \{1, ..., d\}$. The rooting operator 
		\begin{equation*}
			\lfloor \cdot \rfloor: \scF_{0, d}[0] \to \scT_{0, d}[0]
		\end{equation*}
		is defined so that for $T=(\scN, \scE, h_0, H, \scL)$, $\lfloor T\rfloor_i = (\tilde{\scN}, \tilde{\scE}, \tilde{h_0}, H, \tilde{\scL})$ where
		\begin{align*}
			\mbox{For} \quad& x_0 \notin \scN,
			\quad
			\tilde{\scN} = \scN \cup \{ x_0\}, 
			\quad
			\tilde{h_0} = h_0 \cup\{x_0\}, 
			\\
			\mbox{For} \quad& \{x_1, ..., x_n\}=\fr[T] \subseteq \scN, 
			\quad \tilde{\scE} = \scE \cup \big\{ (x_1, x_0), ..., (x_n, x_0) \big\}
			\\
			& \tilde{\scL}:\tilde{\scN} \to \{1, ..., d\}, 
			\quad
			\tilde{\scL}\big|_{\scN} = \scL, \quad \tilde{\scL}[x_0] = i.
		\end{align*}
	\end{definition}
	
	The next proposition shows that the product, decoupling and rooting operations suffice to generate all the forests:
	\begin{proposition}
		\label{proposition:CompletionOfTrees}
		Let $d\in \bN$. The completion of the set $\{ \rId \}$ with respect to the operations $\circledast$, $\cE$ and $\lfloor \cdot\rfloor_i$ such that $i\in\{1, ..., d\}$ is the complete set of Lions forests $\scF_{0, d}[0]$. 
	\end{proposition}
	
	\begin{proof}
		Firstly, note that $\rId \circledast \rId = \rId$ and $\cE[\rId] = \rId$ so the first step is to consider $T_i = \lfloor \rId \rfloor_i$: each elements $T_i$ is a Lions tree (and therefore forests with a single root) and the operations $\circledast$, $\cE$ and $\lfloor \cdot \rfloor_\cdot$ all map Lions forests to Lions forests. Hence the completion of the set $\{\rId\}$ must be contained in $\scF_{0,d}[0]$. We prove that for any forest $T\in \scF_{0,d}[0]$, it must be expressible in terms of some combination of the operators $\circledast$, $\cE$ and $\lfloor \cdot \rfloor_\cdot$. 
		\vskip 4pt
		
		\textit{Step 1.}
		Let $T=(\scN, \scE, h_0, H, \scL)\in \scF_d[0]$. Let $\fr[T]$ be the collection of roots of $T$ and let $P:=\{ h \in H\cup \{h_0\}: h\cap \fr[T] \neq \emptyset\}$ be the collection of hyperedges that contain a root. The set $\fr[T]$ cannot be empty since $\scN$ must contain at least one element that is minimal with respect to the partial ordering $<$ and $P$ cannot be empty since every node of the hypergraph $\big(\scN, H \cup \{h_0\}\big)$ is 1-regular. Further, since $\scN$ is finite we know that $\fr[T]$ and $P$ are also both finite. Let $m= |P|$. 
		
		By \ref{definition:Forests:2.1}, either $h_0 \in P$ or $h_0 = \emptyset$. Suppose that $h_0$ is the empty set and let $h\in P$. Then for any choice of $h\in P$, by assumption $(\scN, \scE)$ is a directed forest, $(\scN, H')$ is a 1-regular hypergraph, $\scL:\scN \to \{1, ..., d\}$ is a labelling and $T = \cE[ (\scN, \scE, h, H \backslash \{h\}, \scL)]$. Therefore, $(\scN, \scE, h, H\backslash \{h\}, \scL)$ is a Lions forest with non-empty tagged hyperedge so that we can assume without loss of generality that $h_0$ is non-empty. 
		
		Suppose first that $|\fr[T]| = 1$ so that necessarily $m=1$ too and write $\fr[T] = \{x_0\}$. Then we have that $\exists h_0\in H'$ such that $x_0 \in h_0$. Therefore, we can write either 
		\begin{equation*}
			T \in \scT_{0, d}[0] \quad \mbox{or} \quad T = \cE\big[ (\scN, \scE, h_0, H, \scL) \quad \mbox{where} \quad (\scN, \scE, h_0, H, \scL) \in \scT_{0,d}[0]
		\end{equation*}
		
		Now we consider the case when $|\fr[T]|>1$. Denote $n = |\fr[T]|$ and note that $m, n \in \bN$ and $m\leq n$. For each $x\in \fr[T]$, let 
		\begin{equation*}
			\scN_x = \{x\} \cup \big\{ y\in \scN: \exists! (\hat{y}_i)_{i=1, ..., \hat{n}}, \hat{y}_1= y, \hat{y}_{\hat{n}} = x \mbox{ and } \forall i=1, ..., \hat{n}-1, (\hat{y}_i, \hat{y}_{i+1}) \in \scE \big\}. 
		\end{equation*}
		That is, the collection $(\scN_x)_{x\in \fr[T]}$ represents the connected components of the directed forest $(\scN, \scE)$ and $\cup_{x\in \fr[T]} \scN_x = \scN$. Further, let
		\begin{equation*}
			\scE_x = \big\{ (y, z) \in \scE: y, z \in \scN_x \big\}. 
		\end{equation*}
		Then $\cup_{x\in \fr[T]} \scE_x = \scE$ since any edge between two nodes in different connected components necessarily doesn't exist. 
		
		For each $p\in P$, we define
		\begin{equation*}
			\scN_p = \bigcup_{x\in p \cap \fr[T]} \scN_x, \quad
			\scE_p = \bigcup_{x\in p \cap \fr[T]} \scE_x \quad\mbox{and}\quad
			\scL_p: \scN_p \to \{1, ..., d\}, \quad  \scL_p = \scL\big|_{\scN_p}. 
		\end{equation*}
		We can easily verify that $(\scN_p, \scE_p, \scL_p)$ is a labelled directed forest. Next denote
		\begin{equation*}
			H'_p = \big\{ h\in H': h\cap \scN_p \neq \emptyset \big\}. 
		\end{equation*}
		We want to show that the set $H_p' \in \scP(\scN_p)$: certainly every element $h\in H_p'$ contains at least one element of $\scN_p$ and every element of $\scN_p$ is contained in one element of $H_p'$ but we have not established that every $h\in H_p'$ contains only elements of $\scN_p$. Suppose for a contradiction that there exists a $\hat{h}\in H_p'$ such that there exists $y\in \hat{h}$ and $y \notin \scN_p$. Further, let $x \in \hat{h}$ be a minimal element (with respect to the partial ordering $\leq$ on $(\scN, \scE)$) of $\hat{h}$ such that $x\in \scN_p$. Suppose first that $y>x$. As the set $h$ satisfies \ref{definition:Forests:2.2}, we conclude that there exists $y'<y$ such that $(y, y') \in \scE$ and $y' \in h$ too. As $y$ and $y'$ are connected, $y'$ is not in $\scN_p$ and we can keep repeating this argument until we obtain an element $\tilde{y} \lesseqqgtr x$ such that $\tilde{y} \in h$ and $\tilde{y} \notin \scN_p$. Next, suppose that $\tilde{y}, x \notin \fr[T]$, so that $\exists \tilde{y}', x' \in \scN$ such that $(\tilde{y}, \tilde{y}), (x, x') \in \scE$. Then $x' \in \scN_p$ and $\tilde{y}' \notin \scN_p$. However, as $h$ satisfies \ref{definition:Forests:2.3} we have that $x' \in h$ and this is a contradiction since we assumed that $x$ was the minimal element of $h$ also contained in $\scN_p$. Now, suppose that $x, \tilde{y} \in \fr[T]$ so that $x\in p$ and $\tilde{y} \notin p$. This ensures that there exists no element of $h\in H'$ such that $x, \tilde{y} \in h$ which is a contradiction. Thus we conclude that $H_p' \in \scP(\scN_p)$. 
		
		Finally, the partial ordering determined by $(\scN_p, \scE_p)$ is the same as the partial ordering determined by $(\scN, \scE)$ restricted to the set $\scN_p$ so that $\scQ^{\scE_p}(\scN_p) = \scQ^{\scE}(\scN_p)$ and thus $H' \in \scQ^{\scE}(\scN)$ implies that $H_p' \in \scQ^{\scE_p}(\scN_p)$. Therefore, the partition $H_p'$ satisfies \ref{definition:Forests:2.2} and \ref{definition:Forests:2.3} so that 
		\begin{equation*}
			\big( \scN_p, \scE_p, p, H_p'\backslash\{p\}, \scL_p \big) \in \scF_{d}[0]
		\end{equation*} 
		and we obtain
		\begin{equation*}
			T = \big( \scN_{h_0}, \scE_{h_0}, h_0, H_{h_0}', \scL_{h_0} \big) \circledast \underset{p \in P\backslash \{h_0\}}{\scalebox{1.5}{$\circledast$}} \cE\Big[ \big( \scN_p, \scE_p, p, H_p', \scL_p \big) \Big]
		\end{equation*}
		where $\fr\big[ ( \scN_p, \scE_p, p, H_p', \scL_p ) \big] = p \cap \fr[T]$. 
		
		Next, for each $p \in P$ and $x\in p$ we denote 
		\begin{equation*}
			H_x' = \big\{ h\in H': h\cap \scN_x \neq \emptyset \big\}. 
		\end{equation*}
		Arguing in the same fashion as before, we conclude that $(\scN_x, H_x')$ is a 1-regular hypergraph with hyperedges that satisfy \ref{definition:Forests:2.2} and \ref{definition:Forests:2.3}. Finally, the set $p \cap \scN_x$ is non-empty and contains the root $x$ and no other elements of $\fr[T]$ (so that $p \cap \scN_x$ satisfies \ref{definition:Forests:2.1}). Therefore, we obtain that 
		\begin{equation*}
			\big( \scN_x, \scE_x, p\cap \scN_x, H_x' \backslash\{ p\cap \scN_x\}, \scL_x \big) \in \scF_d[0]. 
		\end{equation*}
		Further, since $(\scN_x, \scE_x)$ is a sub-forest of $(\scN, \scE)$ that represents a single connected component and $p\cap \scN_x$ is non-empty, we obtain that
		\begin{equation*}
			\big( \scN_x, \scE_x, p\cap \scN_x, H_x' \backslash\{ p\cap \scN_x\}, \scL_x \big) \in \scT_{0,d}[0]. 
		\end{equation*}
		Therefore, we conclude that any Lions forest $T\in \scF_{0, d}[0]$ can be represented as
		\begin{equation}
			\label{eq:proposition:CompletionOfTrees-1}
			T = \bigg( \underset{x \in h_0}{\scalebox{1.5}{$\circledast$}} T_x \bigg) \circledast \underset{p \in P\backslash \{h_0\}}{\scalebox{1.5}{$\circledast$}} \cE\bigg[ \underset{x \in p}{\scalebox{1.5}{$\circledast$}} T_x \bigg]
		\end{equation}
		where for each $x\in \fr[T]$, $T_x \in \scT_{0, d}[0]$. 
		\vskip 4pt
		
		\textit{Step 2.}
		Our next step is to consider $T = (\scN, \scE, h_0, H, \scL)\in \scT_{0, d}[0]$ and let $\fr[T] = \{x_0\}$. We denote
		\begin{equation*}
			\tilde{\scN} = \scN \backslash \{x_0\} \quad
			\tilde{\scE} = \scE \backslash \big\{ (y, x_0): y\in \scN \big\}
			\quad \mbox{and}\quad
			\tilde{\scL}: \tilde{\scN} \to \{1,..., d\}, \quad \tilde{\scL} = \scL\Big|_{\tilde{\scN}}. 
		\end{equation*}
		The first goal is to show that $(\tilde{\scN}, \tilde{\scE})$ is a directed forest. First, note that $\tilde{\scE} \subseteq \tilde{\scN} \times \tilde{\scN}$ and $(x, y) \in \scE \implies (y, x) \notin \tilde{\scE}$. Consider the set $\tilde{R} = \{x \in \scN: (x, x_0) \in \scE \} \subseteq \tilde{\scN}$. As $(\scN, \scE)$ is a directed forest, for every $y \in \scN$ there exists a unique sequence $(y_i)_{i=1, ..., \hat{n}}$ such that $y_1 = y$, $y_{\hat{n}} = x_0$ and for $i=1, ..., \hat{n}-1$, $(y_{i}, y_{i+1}) \in \scE$. As $\tilde{\scN} \subseteq \scN$, for every $y\in \tilde{\scN}$ there exists a unique sequence $(y_i)_{i=1, ..., \hat{n}-1}$ such that $y_1 = y$, $y_{\hat{n}-1} \in \tilde{R}$ and for every $i=1, ... \hat{n}-2$ we have that $(y_i, y_{i+1}) \tilde{\scE}$. As such, we conclude that $(\tilde{\scN}, \tilde{\scE})$ is a directed forest and $\fr\big[ (\tilde{\scN}, \tilde{\scE}) \big] = R$. 
		
		Next, consider the set $\tilde{h}_0 = h_0 \cap \tilde{\scN}$. As $h_0 \subseteq \scN$, we have that $\tilde{h}_0 \subseteq \tilde{\scN}$. Suppose for the moment that $\tilde{h}_0$ is non-empty and for a contradiction additionally suppose that $(h_0 \cap \scN) \cap \tilde{R} = \emptyset$. Then $\exists y \in h_0$ such that $y\in \tilde{\scN}$ and $y \notin R$. As there is a unique sequence $(\hat{y}_{i})_{i=1, ..., \hat{n}}$ such that $\hat{y}_1 = y$, $\hat{y}_{\hat{n}} = x_0$ and for all $i=1, ..., \hat{n}-1$ we have that $(\hat{y}_i, \hat{y}_{i+1}) \in \scE$. Therefore, $(\hat{y}_{\hat{n}-1}, \hat{y}_{\hat{n}}) \in \scE$ so that $\hat{y}_{\hat{n}-1} \in \tilde{R}$. By comparing $y$ and $x_0$ and using that $h_0$ satisfies \ref{definition:Forests:2.2}, we conclude that $\hat{y}_{\hat{n}-1} \in h_0$ which is a contradiction. 
		
		Now consider $\tilde{H}' = \big( H \cup \{ \tilde{h}_0 \} \big)\backslash \{\emptyset\} \in \scP(\tilde{\scN})$. By assumption, any $h \in H' = \big( H \cup \{h_0\} \big)\backslash \{\emptyset\}$ will satisfy \ref{definition:Forests:2.2} and \ref{definition:Forests:2.3} with respect to the partial ordering on $(\scN, \scE)$ and the partial ordering on $(\tilde{\scN}, \tilde{\scE})$ agrees with the partial ordering on $(\scN, \scE)$ when restricted to $(\tilde{\scN}, \tilde{\scE})$ so that the tagged partition $(\tilde{h}_0, H) \in \scQ^{\tilde{\scE}}(\tilde{\scN})[0]$. Thus we conclude that
		\begin{equation}
			\label{eq:proposition:CompletionOfTrees-2}
			\big( \tilde{\scN}, \tilde{\scE}, \tilde{h}_0, H, \tilde{\scL} \big) \in \scF_{0, d}[0]
			\quad \mbox{and}\quad
			T = \Big\lfloor \big( \tilde{\scN}, \tilde{\scE}, \tilde{h}_0, H, \tilde{\scL} \big) \Big\rfloor_{\scL[x_0]}
		\end{equation}

		\vskip 4pt
		\textit{Step 3.} The final step is to observe that for any Lions forest $T\in \scF_{0, d}[0]$ the collection of Lions trees with non-empty tagged hyperedge from representation \eqref{eq:proposition:CompletionOfTrees-1} each have fewer nodes that $T$. Similarly, for any Lions tree $T$ with non-empty tagged hyperedge, the Lions forest from representation \eqref{eq:proposition:CompletionOfTrees-2} has one fewer nodes that $T$. 
		
		Thus, repeating these arguments iteratively, we eventually obtain a representation for any choice of $T\in \scF_{0, d}[0]$ in terms of the operations $\circledast$, $\cE$, $\lfloor \cdot \rfloor_i$ and the empty forest $\rId$. This concludes the proof. 
	\end{proof}
	
	\begin{corollary}
		\label{corollary:CompletionOfTrees}
		Let $d\in \bN$ and let $I$ be an index set. Then 
		\begin{equation*}
			\scF_{0, d}[I] = \Big\{ \cE^a\big[ T_1, ..., T_{|a|} \big] : T_1, ..., T_{|a|} \in \scT_{0, d}[0], \quad a\in A[I] \Big\}
		\end{equation*}
	\end{corollary}
	
	\begin{proof}
		The proof is very similar to that of Proposition \ref{proposition:CompletionOfTrees}, so we only highlight the key points: Firstly by Definition
		\begin{equation*}
			\forall (a, T_1, ..., T_{|a|}) \in \bigsqcup_{a\in A[I]} \big( \scT_{d}[0] \big)^{\times |a|}, 
			\qquad
			\cE^a\big[ T_1, ..., T_{|a|} \big] \in \scF_{0, d}[I]
		\end{equation*}
		so the focus is on proving the reverse implication. 
		
		Now take $T = (\scN, \scE, h_I, H, \scL) \in \scF_{0, d}[I]$ and as before divide it into $|\fr[T]|$ sub-trees. Take the restrictions of the hyperedges to each of these trees. Next, observe that each sub-tree contains only one root so there is at most one of the tagged hyperedges that is non-empty. We partition up the sub-trees that have some non-empty tagged hyperedge and all other sub-trees are partitioned up according to whether their roots are in a common hyperedge of $H$. 
		
		By construction, each of these sub-trees is still an element of $\scT_d[I]$. For each of the sub-trees that have some non-empty $I$-tagged hyperedge, we can transform them to an element of $\scT_{0, d}[0]$ by discarding all of the empty hyperedges. On the other hand, for every one of the sub-trees for which every $I$-tagged hyperedge is empty, we can transform these to an element of $\scT_{0, d}[0]$ by discarding all of the empty $I$-tagged hyperedges and moving the hyperedge of $H$ that contains a root to be the tagged hyperegde. 
		
		We denote the collection of transformed sub-trees with non-empty $h_\iota$ for $\iota \in I$ by the set $\hat{\scT}_\iota$ and the collection of transformed sub-trees that have all empty $I$-tagged hyperedges and all have root in the hyperedge $h\in H$ by $\hat{\scT}_h$. We can then verify that
		\begin{equation*}
			\big( \scN, \scE, h_I, H, \scL \big) = \underset{\iota \in I}{\scalebox{1.5}{$\circledast$}} \cE_{\iota}\bigg[ \underset{\hat{T} \in \hat{\scT}_{\iota}}{\scalebox{1.5}{$\circledast$}} \hat{T} \bigg]
			\circledast 
			\underset{h \in H}{\scalebox{1.5}{$\circledast$}} \cE\bigg[ \underset{\hat{T} \in \hat{\scT}_{h} }{\scalebox{1.5}{$\circledast$}} \hat{T} \bigg]. 
		\end{equation*}
		To conclude, we choose
		\begin{equation*}
			\hat{a} = \big( \underbrace{\underbrace{\iota, ....}_{\times |\hat{\scT}_\iota|}, ..., }_{\iota \in I} \underbrace{\underbrace{h, ...}_{\times |\hat{\scT}_h|}, ... }_{h\in H} \big)
			\quad \mbox{and}\quad
			a = \big\llbracket \hat{a} \rrbracket_{I} \in A[I]
		\end{equation*}
		and rearrange to obtain
		\begin{equation*}
			\big( \scN, \scE, h_I, H \big) = \cE^a\Big[ \underbrace{\underbrace{\hat{T}, ...}_{ \in \hat{\scT}_{\iota}}, ...}_{\iota \in I}, \underbrace{\underbrace{\hat{T}, ...}_{ \in \hat{\scT}_{h}}, ...}_{h \in H} \Big]. 
		\end{equation*}
	\end{proof}
	
	\subsubsection{A module spanned by Lions forests}
	
	The first step in describing a coupled bialgebra is to introduce the algebra structure. We are able to extent the product first introduced on Definition \ref{definition:Lionsproduct}:
	\begin{definition}
		\label{definition:product-span}
		Let $d\in \bN$, let $I$ be an index set and let $(\cR, +, \centerdot)$ be a ring. Let $\spn_\cR\big( \scF_{0, d}[I] \big)$ be the free $\cR$-module generated by $\scF_{0, d}[I]$. We define
		\begin{equation*}
			\circledast: \spn_{\cR}\big( \scF_{0, d}[I] \big) \times \spn_{\cR}\big( \scF_{0, d}[I] \big) \to \spn_{\cR}\big( \scF_{0, d}[I] \big)
		\end{equation*}
		to be the linear extension of the operator of $\circledast$ as defined in Definition \ref{definition:Lionsproduct}. 
	\end{definition}

	\begin{proposition}
		\label{proposition:Prod=Assoc}
		Let $d\in \bN$, let $I$ be an index set and let $(\cR, +, \centerdot)$ be a ring. Then the $\cR$-module
		\begin{equation*}
			\Big( \spn_{\cR} \big( \scF_{0, d}[I] \big), \circledast, \rId \Big)
			\quad \mbox{is \emph{associative} and \emph{unitary}.}
		\end{equation*}
		
		Further, if $\cR$ is a commutative ring then $\big( \spn_{\cR} \big( \scF_{0, d}[I] \big), \circledast, \rId \big)$ is commutative. 
	\end{proposition}

	\begin{proof}
		This follows from Definition \ref{definition:Lionsproduct} and the associativity of the binary set operation union. 
	\end{proof}

	\subsection{Couplings for Lions forests and the coupled coproduct}
	
	Building on the ideas of Remark \ref{remark:couplings=partitions}, a curiosity relating to Definition \ref{definition:Quotient-Partition} is that in this section we are not interested in the complete set of tagged partitions $\scP(\scN)[I]$ but rather a sub-poset $\scQ^{\scE}(\scN)[I]$. This restriction on the number of partitions also reduces the number of couplings that we shall consider. However, the coupled tensor product will still constructed in the same fashion. 

	Building on Equations \eqref{eq:Tag-Part-Seq[a]} and \eqref{eq:Tag-Part-Seq[a]2}, for any $T=(\scN, \scE, h_I, H, \scL) \in \scF_{0, d}[I]$ we define
	\begin{equation*}
		\scF_{0, d}[T]:= \scF_{0, d}\big[ I \cup H \big]. 
	\end{equation*}
	
	\begin{definition}
		\label{definition:coupled_pair-F}
		Let $d\in \bN$ and let $I$ be an index set. We define the set
		\begin{equation*}
			\scF_{0, d}\tilde{\times} \scF_{0, d}[I] 
			:=
			\bigsqcup_{T \in \scF_{0, d}[I]} \scF_{0, d}[T]
		\end{equation*}
		The set $\scF_{0, d}\tilde{\times} \scF_{0, d}[I]$ is a poset with partial ordering $(\Upsilon, Y) \subseteq (\Upsilon', Y')$ if and only if
		\begin{align*}
			&\big( (\scN^\Upsilon, \scE^\Upsilon, \scL^{\Upsilon}), (\scN^Y, \scE^Y, \scL^Y) \big)= \big( (\scN^{\Upsilon'}, \scE^{\Upsilon'}, \scL^{\Upsilon'}), (\scN^{Y'}, \scE^{Y'}, \scL^{Y'}) \big), 
			\\
			&\mbox{and both}\quad
			(h_I^{Y}, H^Y) \subseteq (h_I^{Y'}, H^{Y'})
			\quad \mbox{and}\quad
			(h_{Y}^{\Upsilon}, H^{\Upsilon}) \subseteq (h_{Y'}^{\Upsilon'}, H^{\Upsilon'}). 
		\end{align*}
		We define $\fm: \scF_{0, d} \tilde{\times} \scF_{0, d}[I] \to \bN_0$ by $\fm [ \Upsilon, Y ] = \fm[Y] + \fm[\Upsilon]$. 
		
		Further, we inductively define
		\begin{align*}
			&\scalebox{1.5}{$\tilde{\times}$}_{1}^n \scF_{0, d}[I] = \scF_{0, d} \tilde{\times} \Big( \scalebox{1.5}{$\tilde{\times}$}_{1}^{n-1} \scF_{0, d} \Big)[I]
			\\
			&:=
			\Big\{ \big( \hat{T}^n, ..., \hat{T}^1 \big): \quad \hat{T}^n \in \scF_{0, d}[\hat{T}^{n-1}], \hat{T}^{n-1} \in \scF_{0, d}[\hat{T}^{n-2}], ...,  \hat{T}^1 \in \scF_{0, d}[I] \Big\}
		\end{align*}
		and extend the partial ordering $\subseteq$ and $\fm$ appropriately. 
	\end{definition}
	Following the same ideas as the proof of Lemma \ref{lemma:associativity-coupledproduct}, we conclude that for any $m, n\in \bN$
	\begin{equation*}
		\Big( \scalebox{1.5}{$\tilde{\times}$}_{1}^m \scF_{0,d} \Big) \tilde{\times} \Big( \scalebox{1.5}{$\tilde{\times}$}_{1}^m \scF_{0,d} \Big)[I] = \Big( \scalebox{1.5}{$\tilde{\times}$}_{1}^{m+n} \scF_{0,d} \Big)[I]. 
	\end{equation*}

	\subsubsection{Coupled coproduct on Lions forests}
	
	An admissible cut is a way of dividing a directed tree into two subtrees, one a directed tree (referred to as the root) and one a directed forest (referred to as the prune), see for instance \cite{connes1999hopf}. In the context of Lions forests, a cut is a subset of the edges of a Lions forest that can be removed for combinatorial purposes but this operation should not alter the hypergraphic structure. Thus, a cut takes a Lions tree to a coupled pair of Lions forests contained in $\scF_0 \tilde{\times} \scF_0$ rather than simply $\scF_0 \times \scF_0$ as one might naively expect in the Connes-Kreimer setting. 
	
	\begin{definition}
		\label{definition:CutsCoproduct}
		Let $I$ be an index set and let $T=(\scN, \scE, h_I, H, \scL) \in \scT_d[I]$ be a non-empty labelled Lions tree. A subset $c\subseteq \scE$ is called an admissible cut if $\forall y\in \scN$, the unique path $(e_i)_{i=1, ..., n}$ from $y$ to the root $x$ satisfies that if $e_i\in c$ then $\forall j\neq i, e_j\notin c$. Further, we denote the two empty cuts $(\emptyset, +)$ and $(\emptyset, -)$ (which correspond to the cut that passes over and under respectively the Lions tree). The set of admissible cuts for the Lions tree $T$ is denoted $C(T)$ and we emphasise that $C(T)$ also contains the two empty cuts. 
		
		For a non-empty cut $c\in C(T)$, we call the tuple $T_c^R$ the root of the cut $c$ of the Lions tree $T$ 
		\begin{align*}
			T_c^R :&= (\scN_c^R, \scE_c^R, h_I^{c,R}, H_c^R, \scL_c^R), \quad \mbox{where}
			\\
			\scN_c^R:&= \big\{ y\in \scN: \exists (y_i)_{i=1, ..., n}\in \scN, y_1 = y, y_n \in \fr(T), (y_i, y_{i+1})\in \scE \backslash c \big\}, 
			\\
			\scE_c^R:&= \big\{ (y,z)\in \scE: y,z\in \scN_c^R \big\}, 
			\qquad 
			\scL_c^R: \scN_c^R \to \{1, ..., d\}, \quad \scL_c^R = \scL|_{\scN_c^R}
			\\
			h_I^{c,R}:&= (h_\iota \cap \scN_c^R)_{\iota \in I}, 
			\qquad
			H_c^R:= \big\{ h\cap \scN_c^R: h\in H \big\}\backslash \{\emptyset\}. 
		\end{align*} 
		We call $T_c^P$ the prune of the cut $c$ of $T$ is the tagged Lions forest 
		\begin{align*}
			T_c^P:&= (\scN_c^P, \scE_c^P, h_I^{c,P}, h_{H_c^R}^{c,P}, H_c^P, \scL_c^P) \quad \mbox{where}
			\\
			\scN_c^P :&= \scN\backslash \scN_c^R, 
			\qquad
			\scE_c^P = \scE\backslash (\scE_c^R \cup c), 
			\qquad 
			\scL_c^P = \scL|_{\scN_c^R}, 
			\\
			h_I^{c, P}:&= ( h_{\iota} \cap \scN_c^{P})_{\iota \in I}, 
			\qquad
			H_c^P := \big\{ h \cap \scN_c^P: h \cap \scN_c^R \neq \emptyset \big\} \backslash \{\emptyset\}, 
			\\
			\mbox{and for}\quad  h^{c, P} \in H_c^R, 
			\quad
			h_{h^{c, R}}^{c, P} :&= h \cap \scN_c^P
			\quad \mbox{where}\quad
			h \in H , h^{c, R} = h \cap \scN^{R} \neq \emptyset. 
		\end{align*}
		On the other hand, the root and prune for the empty cuts are denoted
		\begin{align*}
			T_{(\emptyset, +)}^R = ( \scN, \scE, h_I, H, \scL)
			\quad\mbox{and}\quad&
			T_{(\emptyset, -)}^R = (\emptyset, \emptyset, (\emptyset)_I, \emptyset, \scL), 
			\\
			T_{(\emptyset, +)}^P = ( \emptyset, \emptyset, (\emptyset)_I, \emptyset, \scL)
			\quad\mbox{and}\quad&
			T_{(\emptyset, -)}^P = (\scN, \scE, h_I, H, \scL). 
		\end{align*}
		
		Now suppose that $T \in \scF_{0, d}[I]$ is a non-empty labelled Lions forest. By Corollary \ref{corollary:CompletionOfTrees}, we have that
		\begin{equation*}
			T = \cE^a\Big[ T_1, ..., T_{|a|} \Big] 
			\quad \mbox{where}\quad
			(a, T_1, ..., T_{|a|})\in \bigsqcup_{a\in A[I]} \big(\scT_{0, d}[0]\big)^{\times |a|}. 
		\end{equation*}
		We define $C(T) = \prod_{i=1}^{|a|} C(T_i)$. Further, for each $c=(c_1, ..., c_{|a|}) \in C(T)$ where $c_i \in C(T_i)$, we denote
		\begin{align*}
			T_c^R =& \bigg( \bigcup_{i=1}^{|a|} (\scN^{T_i})_{c_i}^R, \quad \bigcup_{i=1}^{|a|} (\scE^{T_i})_{c_i}^R, \quad \Big( \bigcup_{i=1}^{|a|} (h_{\iota}^{T_i} \cap (\scN^{T_i})_{c_i}^R \Big)_{\iota \in I}, \quad \bigcup_{i=1}^{|a|} (H^{T_i})_{c_i}^R, \quad \scL_c^R \bigg), 
			\\
			T_c^P =& \bigg( \bigcup_{i=1}^{|a|} (\scN^{T_i})_{c_i}^P, \quad \bigcup_{i=1}^{|a|} (\scE^{T_i})_{c_i}^P, \quad \Big( \bigcup_{i=1}^{|a|} (h_{\iota}^{T_i} \cap (\scN^{T_i})_{c_i}^P \Big)_{\iota \in I}, 
			\\
			&\quad \Big( h \cap(\scN^{T_i})_{c_i}^P: h \cap (\scN^{T_i})_{c_i}^R \in (H^{T_i})_{c_i}^{R} \Big)_{i =1, ..., |a|}, 
			\quad 
			\bigcup_{i=1}^{|a|} (H^{T_i})_{c_i}^P, \quad \scL_c^P \bigg). 
		\end{align*}
	\end{definition}
	
	\begin{lemma}
		\label{lemma:Lions-coupling}
		Let $T = (\scN, \scE, h_I, H, \scL) \in \scF_d[I]$ and let $c\in C(T)$. Then $(T_c^R, T_c^P) \in \scF_d \tilde{\times} \scF_{0, d}[I]$. 
	\end{lemma}
	
	\begin{proof}
		By construction, we have that $(\scN_c^R, \scE_c^R, \scL_c^R), (\scN_c^P, \scE_c^P, \scL_c^P) \in \fF_{0, d}$,
		\begin{align*}
			&(h_I^{c, R}, H_c^R) \in \scP(\scN_c^R)[I]
			\quad \mbox{and}\quad
			(h_I^{c, P}, h_{H_c^R}^{c, P}, H_c^P) \in \scP(\scN_c^P)[ I \cup H_c^R ]. 
		\end{align*}
		Therefore, we conclude if we can show that
		\begin{equation}
			\label{eq:lemma:Lions-coupling}
			(h_I^{c, R}, H_c^R) \in \scQ^{\scE_c^R}(\scN_c^R)[I]
			\quad \mbox{and}\quad
			(h_I^{c, P}, h_{H_c^R}^{c, P}, H_c^P) \in \scQ^{\scE_c^P}(\scN_c^P)[ I \cup H_c^R ]. 
		\end{equation}
		To avoid lengthy sub/superscripts, we denote 
		\begin{align*}
			&(\scN^Y, \scE^Y, h_I^{Y}, H^{Y}) = (\scN_c^R, \scE_c^R, h_I^{c, R}, H_c^R)
			\quad \mbox{and}
			\\
			&(\scN^\Upsilon, \scE^\Upsilon, h_I^{\Upsilon}, h_{H^Y}^{\Upsilon}, H^{\Upsilon}) = (\scN_c^P, \scE_c^P, h_I^{c, P}, h_{H_c^R}^{c, P}, H_c^P). 
		\end{align*}
		
		\emph{Step 1.1} First consider when $h\in (H^{Y})':=\big( H^Y\cup\{h_{\iota}^Y:\iota \in I \} \big)\backslash \{\emptyset\}$ and suppose that $x, y\in h$ and $x <_{Y} y$. Since $x<_Y y$ implies that $x<_T y$ and there exists $h'\in H^T$ such that $x, y\in h'$. Then thanks to $h'$ satisfying \ref{definition:Forests:2.2} for the preorder $\leq_T$, there exists $z\in h'$ such that $(y, z)\in \scE^T$. Hence, $(y, z) \in \scE^Y$ and we conclude that $z\in h$ too so that $(H^Y)'$ satisfies \ref{definition:Forests:2.2}. 
		
		\emph{Step 1.2} Let $h\in (H^{\Upsilon})':= \big( H^{\Upsilon}\cup \{h_\iota^{\Upsilon}: \iota \in I\} \cup \{h_{h^Y}\Upsilon: h^Y \in H^Y \} \big)\backslash\{\emptyset\}$ and suppose that $x, y\in h$ and $x <_{\Upsilon} y$. First of all, suppose that $x$ and $y$ have a common unique root in $\fr[\Upsilon]$. If this is the case, then $x<_T y$ and $\exists h'\in (H^T)'$ such that $x, y\in h'$. Since $T\in \scF_d[I]$, this implies that $\exists z\in h'$ such that $(y, z) \in\scE^T$. Since $x<_{\Upsilon} y$, we know that $y \notin \fr[\Upsilon]$ so that $(y, z) \in \scE^{\Upsilon}$ and $z\in \scN^{\Upsilon}$. As such, $z\in h$. 
		
		On the other hand, suppose that $x$ and $y$ do not have a common unique root in $\fr[\Upsilon]$. Thus $\exists x', y' \in\fr[\Upsilon]$ such that there exists unique sequence $(x_i)_{i=1, ..., m}$ and $(y_i)_{i=1, ..., n} \in \scN^{\Upsilon}$ such that $x_1 = x$, $y_1 = y$, $x_m = x'$, $y_n = y'$ and $(x_i, x_{i+1}), (y_i, y_{i+1}) \in \scE^{\Upsilon}$. Further, these two sequences do not intersect. From the argument above we know that if $x<_T y$ then we can conclude that the element $y_2 \in h$ so let us suppose that instead we have that $x\geq_T y$. The unique sequence $(w_i)_{i=1, ..., l} \in \scN^T$ such $w_1 = x$, $w_l \in \fr[T]$ and $(w_i, w_{i+1}) \in \scE^T$ will satisfy that $l\geq m$ and $x_i= w_i$ for $i=1, ..., m$. Now, $\exists h'\in H^T$ such that $x, y\in h'$ and $x>_T y$ so that $x_2 = w_2 \in h'$ as well. If we still have that $w_2>_T y$, then we can repeat this argument until we obtain a $w_j \in \scN^T$ such that $w_j \leq\geq_T y$ such that $w_j, y\in h$. We also have that $(y, y_2) \in \scE^{T}$ and $y_2$ is not included in the set $\{x_i: i=1, ..., m\}$ so that $y_2 \neq w_j$. Hence, thanks to \ref{definition:Forests:2.3}, we have that $y_2, w_{j+1}\in h'$. Since $h = h' \cap \scN^{\Upsilon}$ and $y_2 \in \scN^{\Upsilon}$, we can conclude that $y_2 \in h$. Hence $(H^\Upsilon)'$ satisfies \ref{definition:Forests:2.2}
		
		\emph{Step 2.1} Let $h\in (H^{Y})'$ and suppose that we have $x_1, y_1 \in h$ such that $x \lesseqqgtr_\Upsilon y_1$, $(x_1, x_2), (y_1, y_2) \in \scE^{Y}$ and $x_2 \neq y_2$. Then $x_1 \lesseqqgtr_{T} y_1$ so we conclude that $x_2, y_2 \in h$ and $(H^Y)'$ satisfies \ref{definition:Forests:2.3}. 
		
		\emph{Step 2.2} Let $h\in (H^{\Upsilon})'$ and suppose that we have $x_1, y_1 \in h$ such that $x \lesseqqgtr_\Upsilon y_1$, $(x_1, x_2), (y_1, y_2)$ $\in \scE^{\Upsilon}$ and $x_2 \neq y_2$. First let us suppose that $x_1 \lesseqqgtr_T y_1$. Then we are done since there exists $h' \in H^T$ such that $x_1, y_1 \in h'$ and $h'$ satisfies \ref{definition:Forests:2.3} so that we can conclude that $x_2, y_2 \in h'$. As $h = h' \cap \scN^{\Upsilon}$ and $x_2, y_2 \in \scN^{\Upsilon}$, we conclude $x_2, y_2 \in h$ and we are done. 
		
		Therefore, we assume that without loss of generality that $x_1 >_T y_1$. Let $(x_i)_{i=1, ..., n}$ be the unique sequence such that $x_n \in \fr[T]$ and $(x_i, x_{i+1}) \in \scE^T$. However, since $h'$ satisfies \ref{definition:Forests:2.2}, we conclude that also $x_2\in h'$ and we can repeat this argument until we obtain that $x_j \in h'$ and $x_j \lesseqqgtr_T y$. Since $x_1 \neq y_1$ and $x_2 \neq y_2$, we have that $y_2$ is not included in the sequence $\{x_i: i=1, ..., n\}$. We also have that $(x_j, x_{j+1}), (y_1, y_2) \in \scE^T$ and that $h'$ satisfies \ref{definition:Forests:2.3} and conclude that $x_{j+1}, y_2 \in h'$. Finally, by assumption $y_2 \in \scN^{\Upsilon}$ so that $x_2, y_2 \in h$ and we conclude that $(H^\Upsilon)'$ satisfies \ref{definition:Forests:2.3}
		
		\emph{Step 3.1} Fix $\iota \in I$ and suppose $h_\iota^{Y} \neq \emptyset$. Then $h_\iota^T \neq \emptyset$ so that $\fr[T] \cap h_\iota^T \neq \emptyset$. Suppose that $h_\iota^{Y} \cap \fr[Y] = \emptyset$. Then $\exists x_0 \in \fr[T] \backslash \fr[Y]$ and $\exists! y \in \fr[Y]$ such that $x_1, ..., x_m \in \scN^Y$ such that $x_1 = x$, $x_m = y$ and $(x_i, x_{i+1}) \in \scE^Y$. Then $x_{i+1} <_Y x_i$, which implies $x_{i+1} <_T x_i$. As $x_1, x_0 \in h_{\iota}^T$, we obtain from \ref{definition:Forests:2.2} that $x_2 \in h_{\iota}^T$ and $x_2 \in \scN^Y$ so we also have $x_2 \in h_{\iota}^Y$. Repeating this argument, we obtain that $x_m = y \in h_{\iota}^Y$. However, this is a contradiction since $y \in \fr[Y]$. Then \ref{definition:Forests:2.1} is satisfied. 
		
		\emph{Step 3.2} For $\iota \in I$ let $h_\iota^{\Upsilon} \neq \emptyset$. Then $h_\iota^T \neq \emptyset$ so that $\fr[T] \cap h_\iota^T \neq \emptyset$. Suppose that $h_\iota^{\Upsilon} \cap \fr[\Upsilon] = \emptyset$. Then there exists $x\in h_\iota^{\Upsilon}$ and $\exists! x'\in \fr[T]$ such that there is a sequence $(x_i)_{i=1,..., m}$ such that $x_1 = x$, $x_m = x'$ and $(x_i, x_{i+1}) \in \scE^{T}$. For any $i \in \{1, ..., m\}$, if $x_i \in \scN^{\Upsilon}$ then $x_1, ..., x_i \in \scN^{\Upsilon}$, $(x_j, x_{j+1}) \in \scE^{\Upsilon}$ and $x_j \in h_{\iota}^{\Upsilon}$ for every $j \in \{1, ..., i\}$. Suppose that $x_{i+1} \notin \scN^{\Upsilon}$. then $x_i \in \fr[\Upsilon]$ so that $x_i \in h_{\iota}^{\Upsilon} \cap \fr[\Upsilon]$. As this would contradict our hypothesis, we conclude that $x_{i+1} \in \scN^{\Upsilon}$. This implies that $x_m \in \scN^{\Upsilon}$ and $x_m \in \fr[T]$. Therefore,  $\forall z \in \scN^T$, $y \leq_T z$ which implies that $\forall z \in \scN^{\Upsilon}$, $y \leq_{\Upsilon} z$ which implies $y \in \fr[\Upsilon]$. This would contradict our hypothesis so we conclude that \ref{definition:Forests:2.1} is satisfied. 
		
		\emph{Step 3.3} For $h^Y \in H^Y$, let $h_{h^Y}^\Upsilon \in (H^{\Upsilon})'$ and suppose that $h_{h^Y}^\Upsilon \neq \emptyset$. Then $\exists x\in h_{h^Y}^\Upsilon$ and $\exists h' \in H^T$ such that $h' \cap \scN^{\Upsilon} = h_{h^Y}^\Upsilon$ and $h' \cap \scN^Y = h^Y \neq \emptyset$. There exists $x_1, ..., x_m \in \scN^T$ such that $x_1 = x$, $x_m \in \fr[T]$ and $(x_i, x_{i+1}) \in \scE^T$. Arguing as before with an induction argument on $i$ we conclude that for the hypothesis to be true we require that $x_m \in \scN^{\Upsilon}$ and we conclude that $x_m \in \fr[\Upsilon]$. However, this contradicts the hypothesis and we conclude that \ref{definition:Forests:2.1} is satisfied. 
		
		As such \eqref{eq:lemma:Lions-coupling} is proved. 		
	\end{proof}
	
	Following on from Remark \ref{remark:Whatis-CoupledTensor}, we want to interpret the coupled tensor product $\cR$-module as
	\begin{equation*}
		\spn_{\cR}\big(\scF_{0, d}[I] \big)\tilde{\otimes} \spn_{\cR}\big( \scF_{0, d}[I] \big) := \spn_{\cR} \big(\scF_{0, d} \tilde{\times} \scF_{0, d}[I] \big)
	\end{equation*}
	
	Motivated by the Connes-Kreimer coproduct (see for example \cite{connes1999hopf}), we define the following:
	\begin{definition}
		\label{definition:coproduct}
		Let $d\in \bN$, let $I$ be an index set and let $(\cR, +, \centerdot)$ be a ring. We define 
		\begin{equation*}
			\Delta: \spn_{\cR}\Big( \scF_{0,d}[I] \Big) \to \spn_{\cR}\Big( \scF_{0,d} \tilde{\times} \scF_{0,d}[I] \Big)
		\end{equation*}
		to be the linear operator such that $\Delta[\rId] = \rId\times \rId$ and for $T=\big( \scN,\scE, h_I, H, \scL \big) \in \scF_{d}[I]$, 
		\begin{equation}
			\label{eq:proposition:Coproduct2-Equivalence}
			\Delta\Big[ T \Big] = \sum_{c\in C(T)} (T_c^P,T_c^R)
		\end{equation}		
		We pair the operator $\Delta$ with the linear functional $\epsilon: \spn_{\cR}\big( \scF_{0,d}[I] \big) \to \cR$ which satisfies $\epsilon(\rId) = 1_\cR$, and $\forall T\in \scF_{0,d}[I]$, $\epsilon(T) = 0$. We refer to $\epsilon$ as the counit of $\Delta$. 
	\end{definition}
	
	Following in the footsteps of the Definition of $\fm$-diagonal operators in \ref{eq:lemma:Delta-diagonal}, we define
	\begin{equation*}
		\Big( \scalebox{1.5}{$\tilde{\times}$}_{1}^n \scF_{0,d} \Big)^{\{k\}}[I]:= \Big\{ (\Upsilon^n, ..., \Upsilon^1) \in \Big( \scalebox{1.5}{$\tilde{\times}$}_{1}^n \scF_{0,d} \Big)[I]: \fm\big[ (\Upsilon^n, ..., \Upsilon^1) \big] = k \Big\}. 
	\end{equation*}
	\begin{lemma}
		For every index set $I$, we have that
		\begin{equation*}
			\forall k \in \bN_0, 
			\quad 
			\Delta\bigg|_{\spn_{\cR} \big( \scF_{0, d}^{\{k\}}[I] \big)} \subseteq \spn_{\cR} \Big( (\scF_{0, d} \tilde{\times} \scF_{0, d})^{\{k\}}[I] \Big). 
		\end{equation*}
		We say that the operator $\Delta$ is $\fm$-diagonal. 
	\end{lemma}
	
	\begin{proof}
		For every $c \in C(T)$, we have that $\fm\big[ (T_c^P, T_c^R) \big] = \fm[T_c^R] + \fm[T_c^P] = \fm[T]$ and we conclude. 
	\end{proof}
	
	In the same line as Proposition \ref{proposition:M-coassociativity*}, we want to show that the coproduct is \emph{coassociative}:
	\begin{proposition}
		\label{proposition:H-coassociativity*}
		Let $d\in \bN$, let $I$ be an index set and let $(\cR, +, \centerdot)$ be a ring. Let 
		\begin{align*}
			\fI:& \spn_{\cR}\big( \scF_{0, d}[I] \big) \to \spn_{\cR}\big( \scF_{0, d}[I] \big)
		\end{align*}
		be the identity operator and let $\Delta$ be the linear operator defined in Definition \ref{definition:coproduct}. 

		Then
		\begin{equation}
			\label{eq:proposition:H-coassociativity-}
			\Big( \Delta \tilde{\otimes} \fI \Big) \circ \Delta = \Big( \fI \tilde{\otimes} \Delta \Big) \circ \Delta
			\quad \mbox{and}\quad
			\centerdot \circ \Big( \epsilon \tilde{\otimes} \fI \Big) \circ \Delta = \centerdot \circ \Big( \fI \tilde{\otimes} \epsilon \Big) \circ \Delta = \fI. 
		\end{equation}
	\end{proposition}
	
	\iftoggle{figure}{In particular, this tells us that the coupled coproduct $\Delta$ is \emph{coupled-coassociative}, or equivalently satisfies the two commutative diagrams described in Figure \ref{fig:coupled-associativity-H}
	\begin{figure}[htb]
		\centering
		\begin{tikzpicture}
			\node at (0,0) {$\spn_{\cR}(\scF_{0,d} \tilde{\times} \scF_{0,d}[I])$};
			\node at (6,3) {$\spn_{\cR}(\scF_{0,d}\tilde{\times} \scF_{0,d}[I])$};
			\node at (0,3) {$\spn_{\cR}(\scF_{0,d}[I])$};
			\node at (6,0) {$\spn_{\cR}\Big(\scalebox{1.5}{$\tilde{\times}$}_{1}^3 \scF_{0, d}[I] \Big)$};
			\draw[-to](2.25,0) to (4,0);
			\draw[-to](6,2.5) to (6,0.5);
			\draw[-to](0,2.5) to (0,0.5);
			\draw[-to](1.5,3) to (3.75,3);
			\node at (0.5,1.5) {$\Delta$};
			\node at (3,2.5) {$\Delta$};
			\node at (3,0.5) {$\Delta \tilde{\otimes} \fI$};
			\node at (5.5,1.5) {$\fI \tilde{\otimes} \Delta$};
		\end{tikzpicture}
		\begin{tikzpicture}
			\node at (0,0) {$\spn_{\cR}(\scF_{0,d}[I])$};
			\node at (3,1) {$\spn_{\cR}(\scF_{0,d} \tilde{\times} \scF_{0,d}[I])$};
			\node at (9,1) {$\spn_{\cR}(\scF_{0,d}[I]) \otimes \cR$};
			\node at (3,-1) {$\spn_{\cR}(\scF_{0,d} \tilde{\times} \scF_{0,d}[I])$};
			\node at (9,-1) {$\cR \otimes \spn_{\cR}( \scF_{0,d}[I])$};
			\node at (12,0) {$\spn_{\cR}(\scF_{0,d}[I])$};
			\draw[-to](0,-0.25) to (0.9,-1);
			\draw[-to](0, 0.25) to (0.9, 1);
			\draw[-to](5.25, -1) to (7, -1);
			\draw[-to](5.25, 1) to (7, 1);
			\draw[-to] (11.1,-1) to (12,-0.25);
			\draw[-to] (11.1,1) to (12, 0.25);
			\node at (0.25,0.75) {$\Delta$};
			\node at (0.25,-0.75) {$\Delta$};
			\node at (6,-1.5) {$\epsilon \tilde{\otimes} \fI$};
			\node at (6,1.5) {$\fI \tilde{\otimes} \epsilon$};
			\node at (11.75,0.75) {$\centerdot$};
			\node at (11.75,-0.75) {$\centerdot$};
		\end{tikzpicture}
		\caption{Coupled coassociativity}
		\label{fig:coupled-associativity-H}
	\end{figure}
	}
	
	\begin{proof}
		We proceed with the same ideas as Proposition \ref{proposition:M-coassociativity*}. For any $T\in \scF_{0, d}[I]$ and $c\in C(T)$, 
		\begin{align*}
			\Delta\big[ T_c^P \big] = \sum_{\bar{c} \in C(T_c^P)}  \big( (T_c^P)_{\bar{c}}^{P}, (T_c^P)_{\bar{c}}^{R} \big)
			\quad\mbox{and}\quad
			\Delta\big[ T_c^R \big] = \sum_{\hat{c} \in C(T_c^R)}  \big( (T_c^R)_{\hat{c}}^P, (T_c^R)_{\hat{c}}^R \big). 
		\end{align*}
		Following the same arguments as the proof of coassociativity of the Connes-Kreimer coproduct
		\begin{align*}
			(\fI \tilde{\otimes} \Delta) \circ \Delta\Big[ T \Big] =& \sum_{c\in C(T)} (\fI \tilde{\otimes} \Delta) \Big[ (T_c^P, T_c^R) \Big]
			\\
			=&\sum_{c\in C(T)} \sum_{\hat{c} \in C(T_c^R)} \Big( T_c^P, (T_c^R)_{\hat{c}}^P, (T_c^R)_{\hat{c}}^R \Big)
			=\sum_{c'\in C(T)} \sum_{\bar{c} \in C(T_{c'}^P)} \Big( (T_{c'}^P)_{\bar{c}}^P, (T_{c'}^P)_{\bar{c}}^R, T_{c'}^R \Big)
			\\
			=& \sum_{c' \in C(T)} (\Delta \tilde{\otimes} \fI)\Big[ (T_{c'}^P, T_{c'}^R ) \Big] = (\Delta \tilde{\otimes} \fI) \circ \Delta\Big[ T \Big]. 
		\end{align*}
	\end{proof}
	
	\subsubsection{Bialgebra identity on Lions Forests}
	
	Following the ideas of Lemma \ref{lemma:Shuffle_diagonal}:
	\begin{lemma}
		For any choice of index set $I$ the linear map
		\begin{equation*}
			\circledast: \spn_{\cR} \big( \scF_{0, d}[I] \times \scF_{0, d}[I] \big) \to \spn_{\cR}\big( \scF_{0, d}[I] \big)
			\qquad
			\mbox{is $\fm$-diagonal. }
		\end{equation*}
	\end{lemma}
	
	\begin{proof}
		For $(T_1, T_2) \in (\scF_{0, d} \times \scF_{0, d})[I]$, we have that
		\begin{align*}
			\fm\big[ T_1 \circledast T_2 \big] = \big| H^{T_1 \circledast T_2} \big| = \big| H^{T_1} \big| + \big| H^{T_2} \big| = \fm[T_1] + \fm[T_2] = \fm\big[ (T_1, T_2) \big]. 
		\end{align*}
	\end{proof}
	
	Having established in the proof of Proposition \ref{proposition:H-coassociativity*} the ways in which the product $\circledast$ and the coupled coproduct $\Delta$ can be combined, our next step is to establish the coupled bialgebra identities equivalent to those proved in Theorem \ref{theorem:M-coupledBialgebra-}: 
	\begin{proposition}
		\label{proposition:Coproduct-Equivalence}
		Let $d\in \bN$, let $I$ be an index set and let $(\cR, +, \centerdot)$ be a ring. 
		
		The operation $\Delta: \spn_{\cR}\big( \scF_{0,d}[I] \big) \to \spn_{\cR}\big( \scF_{0,d} \tilde{\times} \scF_{0,d}[I] \big)$ defined in Definition \ref{definition:coproduct} satisfies that
		\begin{equation}
			\label{eq:definition:coproduct}
			\begin{split}
				\forall T \in \scF_{0, d}[0],
				\qquad 
				&\Delta\Big[ \lfloor T\rfloor \Big] = \big( \lfloor T\rfloor, \rId\big) + \Big( \fI \tilde{\otimes} \lfloor \cdot \rfloor\Big) \circ \Delta \big[ T \big], 
				\\
				\forall (T, \iota) \in \scF_{0, d}[0] \times I, 
				\qquad
				&\Delta\Big[ \cE_{\iota}[T] \Big] = (\cE_{\iota} \tilde{\otimes} \cE_{\iota}) \Big[ \Delta [T] \Big], 
				\\
				\forall T, T' \in \scF_{0, d}[I],
				\qquad
				&\Delta\Big[ T \circledast T' \Big] = \Big(\circledast \tilde{\otimes} \circledast \Big) \circ \overline{\mbox{Twist}} \circ \Big( \Delta \otimes \Delta\Big) \Big[ T \otimes T' \Big]. 
			\end{split}
		\end{equation}
		
		Further, this is the unique linear operator from $\spn_{\cR}\big( \scF_{0,d}[I] \big)$ to $\spn_{\cR}\big( \scF_{0,d} \tilde{\times} \scF_{0,d}[I] \big)$ that satisfies Equation \eqref{eq:definition:coproduct}. 
	\end{proposition}
		
	\begin{proof}
		\textit{Step 1.} Let $\Delta: \spn_{\cR}\big( \scF_{0,d}[I] \big) \to \spn_{\cR}\big( \scF_{0,d} \tilde{\times} \scF_{0,d}[I] \big)$ be the linear operator defined in Definition \ref{definition:coproduct}. Then the set of cuts of the empty forest $C(\rId) = \{ \emptyset \}$ so that $\Delta[\rId] = \rId \times \rId$. 
		
		For $i\in \{1, ..., d\}$ the Lions tree with a single node $\lfloor \rId \rfloor_i$, the set of cuts $C\big( \lfloor \rId \rfloor_i \big) = \big\{ (\emptyset,+), (\emptyset, -) \big\}$ so that
		\begin{equation*}
			\Delta\big[ \lfloor \rId \rfloor_i \big] = \big( \rId, \lfloor \rId \rfloor_i \big) + \big( \lfloor \rId \rfloor_i, \rId \big). 
		\end{equation*}
		Similarly, $\cE_{\iota}[\rId] = \rId$ and $\rId \circledast \rId = \rId$ so that Equation \eqref{eq:definition:coproduct} is satisfied in the case where $T = \rId$.  
	
		\text{Step 1.1} Suppose that $T\in \scF_{0,d}[I]$ and without loss of generality assume that $h_{\iota}^T\neq \emptyset$. Then
		\begin{equation*}
			\Delta\big[ T \big] =  \sum_{c\in C(T)} \big( T_c^P, T_c^R \big). 
		\end{equation*}
		The sets $C(\cE_{\iota}[T])$ and $C(T)$ are same and $H^{\cE_{\iota}[T]} = H^T \cup \{h_\iota^T\}$, so that
		\begin{equation*}
			\Delta\Big[ \cE_{\iota}[T] \Big] = \sum_{c\in C(T)} \Big( \cE_{\iota}[T]_c^P, \cE_{\iota}[T]_c^R \Big). 
		\end{equation*}
		From Definition \ref{definition:CutsCoproduct} we get
		\begin{equation*}
			\cE_{\iota}[T]_c^P = \cE_{\iota}[T_c^P], \quad \cE_{\iota}[T]_c^R = \cE_{\iota}[T_c^R]
		\end{equation*}
		so that $\Delta\big[ \cE_{\iota}[T] \big] = \cE_{\iota} \tilde{\otimes} \cE_{\iota} \circ \Delta \big[ T\big]$. 
		
		\textit{Step 1.2.} Let $T_1,  T_2 \in \scF_{0, d}[0]$ so that
		\begin{align*}
			\Delta\Big[ T_1 \Big] = \sum_{c\in C(T_1)} \Big( (T_1)_c^P, (T_1)_c^R \Big)
			\quad\mbox{and}\quad
			\Delta\Big[ T_2 \Big] = \sum_{c\in C(T_2)} \Big( (T_2)_c^P, (T_2)_c^R \Big). 
		\end{align*}
		From Definition \ref{definition:CutsCoproduct} we get $C(T_1 \circledast T_2) = C(T_1) \times C( T_2)$ so that
		\begin{align*}
			\Delta\Big[ T_1 \circledast T_2 \Big] 
			&= 
			\sum_{c \in C(T_1 \circledast T_2)} \Big( (T_1 \circledast T_2)_c^P, (T_1 \circledast T_2)_c^R \Big)
			\\
			&= \sum_{c_1 \in C(T_1)} \sum_{c_2 \in C(T_2)} \Big( \big( (T_1)_{c_1}^P \circledast (T_2)_{c_2}^P \big), \big( (T_1)_{c_1}^R \circledast (T_2)_{c_2}^R \big) \Big)
			\\
			&= \Big( \circledast \otimes \circledast \Big) \circ \overline{\mbox{Twist}} \circ \Delta \otimes \Delta \Big[ (T_1,T_2) \Big]
		\end{align*}
		
		\textit{Step 1.3.} Let $T\in \scT_{0, d}[I]$. Then there exists $a\in A[I]$ and $\boldsymbol{T} = (T_1, ..., T_{|a|}) \in (\scT_{0, d}[0])^{\times |a|}$ such that
		\begin{equation*}
			T = \big\lfloor \cE^a[ \boldsymbol{T}] \big\rfloor. 
		\end{equation*} 		
		Let $c\in C(T)$ be a non-empty cut. Following on from Definition \ref{definition:CutsCoproduct}, we have that for each leaf of $T$ the cut $c$ must pass through at most 1 edge between the leaf and the root. Thus, for each $T_i$ with root $x_i \in \fr[T_i]$ and $x_0 \in \fr[T]$, we have that the cut $c$ restricted to the set of vertices $\scN^{T_i} \cup \{x_0\}$ is either $(x_i, x_0)$, $(\emptyset,+)$ or $c \in C(T_i)$. In particular, this means that
		\begin{align*}
			C\big( \cE^a[\boldsymbol{T}] \big) =& \prod_{i=1}^{|a|} C(T_i)
			\quad\mbox{and} \quad
			C\Big( \big\lfloor \cE^a[\boldsymbol{T}] \big\rfloor \Big) = \big\{ (\emptyset, -) \big\} \cup \prod_{i=1}^{|a|} \bigg( \Big( C(T_i) \backslash\{ (\emptyset,-)\} \Big) \cup \big\{ (x_i, x_0) \big\} \bigg). 
		\end{align*}
		Replacing the element $(\emptyset, -) \in C(T_i)$ with the element $(x_i, x_0)$ corresponds to applying the rooting operation to the root of any cut. Therefore
		\begin{equation*}
			\Delta \Big[ \big\lfloor \cE^a[\boldsymbol{T}] \big\rfloor \Big] 
			= 
			\Big( \big\lfloor \cE^a[\boldsymbol{T}] \big\rfloor,  \rId \Big) + \Big( \fI \tilde{\otimes} \lfloor \cdot \rfloor \Big) \circ \Delta \Big[ \cE^a[\boldsymbol{T}] \Big]. 
		\end{equation*}
		Hence the linear operator defined in Definition \ref{definition:coproduct} satisfies \eqref{eq:definition:coproduct}. 
		
		\textit{Step 2.} Now suppose we have a second linear operator 
		\begin{equation*}
			\Delta': \spn_{\cR}\big( \scF_{0,d}[I] \big) \to \spn_{\cR}\big( \scF_{0,d} \tilde{\times} \scF_{0,d}[I] \big)
		\end{equation*}
		that satisfies Equation \eqref{eq:definition:coproduct} and that $\Delta'\big[ \rId \big] = \rId \times^\emptyset \rId$. Then $(\Delta - \Delta')[\rId] = (\Delta - \Delta')\big[ \lfloor \rId\rfloor \big] = 0$. 
		
		\textit{Step 2.1.} Suppose for $T \in \scF_{0, d}[I]$ that $(\Delta - \Delta')[T] = 0$ and let $\iota \in I$. Then
		\begin{align*}
			(\Delta - \Delta')\Big[ \cE_\iota[T] \Big] =& \Delta\Big[ \cE_\iota[T] \Big] - \Delta'\Big[ \cE_\iota[T] \Big] 
			= \cE_\iota \tilde{\otimes} \cE_{\iota} \circ (\Delta - \Delta')[ T ] = 0. 
		\end{align*}
		\textit{Step 2.2.} Now suppose that for $T_1, T_2 \in\scF_{0, d}[I]$ that $(\Delta - \Delta')[T_1] = (\Delta - \Delta')[T_2] = 0$. Then
		\begin{align*}
			&(\Delta - \Delta')\Big[ T_1 \circledast T_2 \Big] = \Delta\Big[ T_1 \circledast T_2 \Big] - \Delta'\Big[ T_1 \circledast T_2 \Big] 
			\\
			&= \Big(\circledast \tilde{\otimes} \circledast \Big) \circ \overline{\mbox{Twist}} \circ \Big( \Delta \otimes \Delta\Big)\Big[ T_1 \otimes T_2 \Big] - \Big(\circledast \tilde{\otimes} \circledast \Big) \circ \overline{\mbox{Twist}} \circ \Big( \Delta' \otimes \Delta' \Big)\Big[ T_1 \otimes T_2 \Big]
			\\
			&= \Big(\circledast \tilde{\otimes} \circledast \Big) \circ \overline{\mbox{Twist}} \circ \Big( (\Delta \otimes \Delta) - (\Delta' \otimes \Delta') \Big) \Big[ T_1 \otimes T_2 \Big]
			\\
			&= \Big(\circledast \tilde{\otimes} \circledast \Big) \circ \overline{\mbox{Twist}} \circ \Big( (\Delta - \Delta')\big[T_1\big] \otimes (\Delta - \Delta')\big[T_2\big] \Big) = 0. 
		\end{align*}
		
		\textit{Step 2.3.} To conclude, let $T\in \scF_{0, d}[0]$ so that $\lfloor T \rfloor \in \scT_{0, d}[0]$. Then
		\begin{align*}
			(\Delta - \Delta')\Big[ \big\lfloor T \big\rfloor \Big] =& \Delta \Big[ \big\lfloor T \big\rfloor \Big] - \Delta'\Big[ \big\lfloor T \big\rfloor \Big]
			\\
			=& \Big( \big\lfloor T\big\rfloor, \rId\Big) + \Big( \fI \tilde{\otimes} \big\lfloor \cdot \big\rfloor \Big) \circ \Delta\Big[ T \Big] - \Big( \big\lfloor T\big\rfloor, \rId\Big) - \Big( \fI \tilde{\otimes} \big\lfloor \cdot \big\rfloor \Big) \circ \Delta'\Big[ T \Big]
			\\
			=& \Big( \fI \tilde{\otimes} \big\lfloor \cdot \big\rfloor \Big) \circ (\Delta - \Delta')\Big[ T \Big] = 0. 
		\end{align*}
		Arguing via induction, we conclude that $(\Delta - \Delta') [T]= 0$ for every $T\in \scF_{0, d}[I]$. 
	\end{proof}
	
	\begin{theorem}
		\label{theorem:H-coupledBialgebra-}
		Let $d\in \bN$, let $I$ be an index set and let $(\cR, +, \centerdot)$ be a ring. Then the product of Lions forests and the coupled coproduct $\Delta$ satisfy
		\begin{equation}
			\label{eq:theorem:H-coupledBialgebra-}
			\begin{aligned}
				\Delta \circ \circledast =& \Big( \circledast \tilde{\otimes} \circledast \Big) \circ \overline{\mbox{Twist}} \circ \Big( \Delta \otimes \Delta \Big)
				\\
				\epsilon \circ \circledast =& \centerdot \circ \epsilon \otimes \epsilon \quad \Delta \circ \rId \circ \epsilon = \rId \otimes \rId \quad \epsilon \otimes \rId = \rId_{\cR}. 
			\end{aligned}
		\end{equation}
	\end{theorem}
	
	\iftoggle{figure}{Theorem \ref{theorem:H-coupledBialgebra-} is equivalent to the operators $(\circledast, \rId, \Delta, \epsilon)$ satisfying the commutative relationship described in Figure \ref{fig:coupled-bialgebra-H}. 
	\begin{figure}[htb]
		\centering
		\begin{tikzpicture}
			\node at (0,2) {$\spn_{\cR}\big( \scF_{0,d}[I] \times \scF_{0,d}[I] \big)$};
			\node at (0,0) {$\spn_{\cR}\big( \scF_{0,d}[I] \big)$};
			\node at (0,-2) {$\spn_{\cR} \big( \scF_{0,d} \tilde{\times}\scF_{0,d}[I] \big)$};
			\node at (9,2) {$\spn_{\cR}\Big( \big(\scF_{0,d} \tilde{\times}\scF_{0,d}[I]\big) \times \big(\scF_{0,d}\tilde{\times}\scF_{0,d}[I]\big) \Big)$};
			\node at (9,-2) {$\spn_{\cR}\Big( \big(\scF_{0,d}\times\scF_{0,d}\big) \tilde{\times} \big(\scF_{0,d} \times\scF_{0,d}\big)\big[I\big] \Big)$};
			\draw[-to](0,1.6) to (0,0.4);
			\draw[-to](0,-0.4) to (0,-1.6);
			\draw[-to](9.5,1.6) to (9.5,-1.6);
			\draw[-to](2.4,2) to (4.6,2);
			\draw[-to](4.6,-2) to (2.4,-2);
			\node at (0.5,1) {$\circledast$};
			\node at (0.5,-1) {$\Delta$};
			\node at (3.5,1.5) {$\Delta\otimes \Delta$};
			\node at (3.5,-1.5) {$\circledast \tilde{\otimes} \circledast$};
			\node at (8.5,0) {$\overline{\mbox{Twist}}$};
		\end{tikzpicture}	
		\\
		\begin{tikzpicture}
			\node at (0,0) {$\spn_{\cR}\big( \scF_{0,d}[I] \times \scF_{0,d}[I] \big)$};
			\node at (6,0) {$\spn_{\cR}\big( \scF_{0,d}[I] \big)$};
			\node at (3,-2) {$\cR \otimes \cR \equiv \cR$};
			\draw[-to](1,-0.4) to (2,-1.6);
			\draw[-to](5,-0.4) to (4,-1.6);
			\draw[-to](2.25,0) to (4.25,0);
			\node at (3.25,0.5) {$\circledast$};
			\node at (1,-1) {$\epsilon \otimes \epsilon$};
			\node at (5,-1) {$\epsilon$};
		\end{tikzpicture}
		\\
		\begin{tikzpicture}
			\node at (0,0) {$\spn_{\cR}\big( \scF_{0,d} \tilde{\times} \scF_{0,d}[I] \big)$};
			\node at (6,0) {$\spn_{\cR}\big( \scF_{0,d}[I] \big)$};
			\node at (3,2) {$\cR \otimes \cR \equiv \cR$};
			\draw[-to](2,1.6) to (1,0.4);
			\draw[-to](4,1.6) to (5,0.4);
			\draw[-to](4.25,0) to (2.25,0);
			\node at (3.25,0.5) {$\Delta$};
			\node at (1,1) {$\rId \tilde{\otimes} \rId$};
			\node at (5,1) {$\rId$};
		\end{tikzpicture}
		\\
		\begin{tikzpicture}
			\node at (0,0) {$\spn_{\cR}\big( \scF_{0,d}[I] \big)$};
			\node at (4,0) {$\cR$};
			\draw[->] (1,0.4) to[bend left] (3.75,0.4);
			\draw[->] (3.75,-0.4) to[bend left] (1,-0.4);
			\node at (2.375,1) {$\epsilon$};
			\node at (2.375,-1) {$\rId$};
		\end{tikzpicture}
		\caption{Coupled bialgebra}
		\label{fig:coupled-bialgebra-H}
	\end{figure}
	}
	
	\begin{proof}
		Follows immediately from Proposition \ref{proposition:Coproduct-Equivalence}. 
	\end{proof}
	
	\subsubsection{Grading on Lions forests}
	
	The $\cR$-module spanned by labelled forests has a grading determined by the number of vertices of the forest. As with Subsection \ref{subsubsection:Grading-Lionswords}, we need a grading that additionally captures the hypergraphic structure of a Lions forests:	
	\begin{definition}
		\label{lemma:grading(2)-}
		Let $d\in \bN$ and let $I$ be an index set. We define $\scG^I:\scF_{0,d}[I] \to \bN_0^{\times 2}$ by
		\begin{equation}
			\label{eq:lemma:grading(2)-}
			\scG^I\big[ (\scN, \scE, h_I, H, \scL) \big]:= \big( |\scN|, |H| \big). 
		\end{equation}	
		For $(k,n) \in \bN_0^{\times 2}$, we define
		\begin{equation*}
			\scF_{0, d}^{(k, n)}[I] := \Big\{ T\in \scF_{0, d}[I]: \scG^I[T] = (k, n) \Big\}
		\end{equation*}
		and denote
		\begin{equation*}
			\big( \scF_{0, d}^{(k_2, n_2)} \tilde{\times} \scF_{0, d}^{(k_1, n_1)} \big) [I]:= \bigsqcup_{T \in \scF_{0, d}^{(k_1, n_1)}[I]} \scF_{0, d}^{(k_2, n_2)}[ T]. 
		\end{equation*}
	\end{definition}
	
	In the next result, we demonstrate that this decomposition is natural to the underlying structure of the algebra operation $\circledast$ and the coupled coproduct operation $\Delta$:
	\begin{proposition}
		\label{proposition:H-Grading-}
		Let $d\in \bN$, let $I$ be an index set and let $(\cR, +, \centerdot)$ be a ring. Then
		\begin{equation}
			\label{eq:proposition:H-Grading-}
			\left.
			\begin{aligned}
				&\spn_{\cR}\big( \scF_{0,d}^{(0, 0)}[I]\big) = \cR, 
				\\
				&\spn_{\cR}\Big( \scF_{0, d}^{(k_1, n_1)}[I] \Big) \circledast \spn_{\cR}\Big( \scF_{0, d}^{(k_2, n_2)}[I] \Big) \subseteq \spn_{\cR}\Big( \scF_{0, d}^{(k_1+k_2, n_1+n_2)} \Big),  
				\\
				&\Delta\Big[ \spn_{\cR}\big( \scF_{0, d}^{(k, n)}[I] \big) \Big] \subseteq  \bigoplus_{k'=0}^{k} \bigoplus_{n'=0}^n \spn_{\cR}\Big( \scF_{0, d}^{(k-k', n-n')} \tilde{\times} \scF_{0, d}^{(k', n')}[I] \Big). 
			\end{aligned}
			\right\}
		\end{equation}
		Finally, for any choice of index set $I$ and $(k, n) \in \bN_0^{\times 2}$ we have that the set $\big| \scF_{0, d}^{(k, n)}[I] \big|$ is finite. 
	\end{proposition}
	
	\begin{proof}
		\textit{Step 1.} Firstly, for any $T\in \scF_{0, d}[I]$
		\begin{equation*}
			\scG^I[T] = (0, 0) \quad \iff \quad T = \rId
		\end{equation*}
		so that
		\begin{equation*}
			\scF_{0, d}^{(0, 0)}[I] = \big\{\rId \big\}
			\quad \mbox{and} \quad 
			\spn_{\cR}\Big( \scF_{0, d}^{(0_I, 0)}[I] \Big) = \cR. 
		\end{equation*}
		
		\textit{Step 2.} Let $k_1, k_2, n_1, n_2 \in \bN_0^{\times 2}$ and suppose that 
		\begin{align*}
			T^1 = (\scN^1, \scE^1, h_I^1, H^1, \scL^1) \in \scF_{0,d}[I]& \quad \mbox{such that} \quad \scG^I[T^1] = (k_1, n_1) \quad \mbox{and}
			\\
			T^2 = (\scN^2, \scE^2, h_I^2, H^2, \scL^2) \in \scF_{0,d}[I]& \quad \mbox{such that} \quad \scG^I[T^2] = (k_2, n_2). 
		\end{align*}
		Then
		\begin{align*}
			&\scG^I\big[ T^1 \circledast T^2 \big] = \big( k_1 + k_2, n_1 + n_2 \big)
			\quad \mbox{so that}
			\\
			&\circledast: \spn_{\cR} \Big( \scF_{0, d}^{(k_I^1, n^1)}[I] \Big) \times \spn_{\cR} \Big( \scF_{0, d}^{(k_I^2, n^2)}[I] \Big) \to \spn_{\cR} \Big( \scF_{0, d}^{(k_I^1+k_I^2, n^1+n^2)}[I] \Big). 
		\end{align*}
		
		\textit{Step 3.} For any $T \in \scF_{0, d}[I]$ and $(\Upsilon, Y) \in \scF_{0,d} \tilde{\times} \scF_{0, d}[I]$ such that
		\begin{equation*}
			\Big\langle \Delta[T], (\Upsilon, Y) \Big\rangle >0 
			\quad \implies \quad 
			\scG^I[\Upsilon] + \scG^I[Y] = \scG^I[T]. 
		\end{equation*}
		Therefore
		\begin{equation*}
			\Delta\Big[ \spn_{\cR}\big( \scF_{0, d}^{(k, n)}[I] \big) \Big] \subseteq \bigoplus_{k' \in 0}^n \bigoplus_{n'=0}^n \spn_{\cR}\Big( \scF_{0, d}^{(k - k', n-n')} \tilde{\times} \scF_{0, d}^{(k', n')}[I] \Big)
		\end{equation*}
		and we conclude that Equation \eqref{eq:proposition:H-Grading-} holds. 
	\end{proof}

	Let $g: \bN_0^{\times 2} \to \bR$ be a monotone increasing function such that $g(0_I, 0) \leq 0$. Following \eqref{eq:truncatedWord}, we denote
	\begin{equation*}
		\scF_{0, d}^{g,+}[I]:=\Big\{ T \in \scF_{0, d}[I]: g\big( \scG^I[T] \big) > 0 \Big\}, 
		\quad \mbox{and}\quad
		\scF_{0, d}^{g,-}[I]:=\Big\{ T \in \scF_{0, d}[I]: g\big( \scG^I[T] \big) \leq 0 \Big\}. 
	\end{equation*}
	
	\begin{corollary}
		\label{lemma:H-Finite-grading-}
		Let $d\in \bN$, let $I$ be an index set and let $(\cR, +, \centerdot)$ be a ring. Let $g: \bN_0^{\times 2} \to \bR$ be a monotone increasing function such that $g(0_I, 0) \leq 0$. 
		
		Then 
		\begin{enumerate}[label=(\ref*{lemma:H-Finite-grading-}.\roman*)]
			\item
			\label{enum:lemma:H-Finite-grading-1}
			The $\cR$-module
			\begin{equation*}
				\spn_{\cR}\Big( \scF_{0, d}^{g,+} \Big)
				\quad \mbox{is an algebra ideals of } \quad
				\Big( \spn_{\cR}\big( \scF_{0, d}[I]\big), \circledast, \rId \Big)
			\end{equation*}
			and we can identify the algebra over the $\cR$-module quotient by
			\begin{align*}
				\spn_{\cR}\Big( \scF_{0, d}^{g,-}[I] \Big) =& \spn_{\cR} \Big( \scF_{0, d}[I] \Big) / \spn_{\cR}\Big( \scF_{0, d}^{g,+}[I] \Big) 
			\end{align*}
			\item 
			\label{enum:lemma:H-Finite-grading-2}
			The coupled coproduct and counit $(\Delta, \epsilon)$ restricted to the sub-module
			\begin{equation*}
				\Big( \spn_{\cR}\big( \scF_{0, d}^{g,-}[I] \big), \Delta, \epsilon \Big)
			\end{equation*}
			is a co-associative sub-coupled coalgebras of $\big( \spn_{\cR}( \scF_{0, d}[I] ), \Delta, \epsilon \big)$. 
			\item 
			\label{enum:lemma:H-Finite-grading-3}
			By pairing the quotient algebra and unit $(\circledast, \rId)$ to the restriction of the coupled coproduct and counit $(\Delta, \epsilon)$, we obtain
			\begin{equation*}
				\Big( \spn_{\cR}\big( \scF_{0, d}^{g,-}[I] \big), \circledast, \rId, \Delta, \epsilon \Big)
			\end{equation*}
			satisfies
			\begin{equation}
				\label{eq:theorem:H-coupledBialgebra--}
				\begin{aligned}
					\Delta \circ \circledast =& \Big( \circledast \tilde{\otimes} \circledast \Big) \circ \overline{\mbox{Twist}} \circ \Big( \Delta \otimes \Delta \Big)
					\\
					\epsilon \circ \circledast =& \centerdot \circ \epsilon \otimes \epsilon \quad \Delta \circ \rId \circ \epsilon = \rId \otimes \rId \quad \epsilon \otimes \rId = \rId_{\cR}. 
				\end{aligned}
			\end{equation}
		\end{enumerate}
	\end{corollary}
	
	\begin{proof}
		\ref{enum:lemma:H-Finite-grading-1}: Thanks to Proposition \ref{proposition:H-Grading-} and the monotonicity of $f$, we have for any $T^1 \in \scF_{0, d}^{g,+}[I]$ and $T^2 \in \scF_{0, d}[I]$ that
		\begin{equation*}
			T^1 \circledast T^2 \in \spn_{\cR}\Big( \scF_{0, d}^{g,+} \Big). 
		\end{equation*}
		Thus the sub $\cR$-module $\spn_{\cR}\big( \scF_{0, d}^{g,+}[I] \big)$ is an algebra ideal over $\cR$ and the conclusion follows. 
		
		\ref{enum:lemma:H-Finite-grading-2}: In the same fashion, Proposition \ref{proposition:H-Grading-} and the monotinicity of $g$ implies that for any $T\in \scF_{0, d}^{g,-}[I]$ and any $(\Upsilon, Y) \in \scF_{0, d}\tilde{\times} \scF_{0, d}[I]$ such that $c_I(T, \Upsilon, Y)>0$ we have that
		\begin{equation*}
			g\Big( \scG^I\big[ Y \big] \Big) , g\Big( \scG^{Y}\big[ \Upsilon \big] \Big) \leq g \Big( \scG^I[T] \Big)
		\end{equation*}
		so that 
		\begin{equation*}
			\Delta\Big[ \spn_{\cR}\big( \scF_{0, d}^{g, -}[I] \big) \Big] \subseteq \spn_{\cR}\big( \scF_{0, d}^{g, -}[I] \big) \tilde{\otimes} \spn_{\cR}\big( \scF_{0, d}^{g, -}[I] \big)
		\end{equation*}
		and thus $\big( \spn_{\cR}( \scF_{0, d}^{g, -}[I] ), \Delta, \epsilon \big)$ is a sub-coupled coalgebra. 
		
		\ref{enum:lemma:H-Finite-grading-3}: This follows easily from Proposition \ref{proposition:H-Grading-} and Proposition \ref{proposition:Coproduct-Equivalence}, where the unit and counit identities are immediate by direct computation. 
	\end{proof}
	
	\iftoggle{Plus}{
	\subsection{Modules of measurable functions indexed by Lions forests}
	
	To streamline notation, we define the coproduct counting function 
	\begin{equation}
		\label{eq:definition:deconcatenation-coproduct-counting+}
		c_I: \scF_{0, d}[I] \times \big( \scF_{0, d}[I] \tilde{\times} \scF_{0, d}[I] \big) \to \bN_0
		\quad \mbox{by}\quad
		c_I\Big( T, (\Upsilon, Y) \Big) = \Big\langle \Delta\big[ T \big] , (\Upsilon, Y) \Big\rangle. 
	\end{equation}
	
	Following in the footsteps of Assumption \ref{notation:V-(2)}, we now wish to work with a collection of modules that satisfy:
	\begin{assumption}
		\label{notation:V-(3)}
		Let $(\Omega, \cF, \bP)$ and $(\Omega', \cF', \bP')$ be probability spaces. Let $d \in \bN$ and let $(\cR, +, \centerdot)$ be a separable normed ring which we associate with the Borel $\sigma$-algebra $\cB(\cR)$. 
		
		For any index set $I$, let
		\begin{equation*}
			\Big( \bV^T\big( \Omega; \cR \big) \Big)_{T\in \scF_{0, d}[I]}
		\end{equation*}
		be a collection of unital $\cR$-modules that satisfy:
		\begin{enumerate}[label=(\ref*{notation:V-(3)}.\roman*)]
			\item 
			\label{enum:notation:V-(3)-1}
			For every $T\in \scF_{0,d}[I]$, we have
			\begin{equation*}
				\bV^T\big( \Omega; \cR \big) \subseteq L^{\boldsymbol{0}}\Big( \Omega^{\times \fm[T]}; \cR \Big)
			\end{equation*}
			\item \label{enum:notation:V-(3)-2}
			For any $T\in \scF_{0, d}[I]$ such that $\fm[T] = 0$, we have that
			\begin{equation*}
				\bV^{T}\big( \Omega; \cR \big) = \cR. 
			\end{equation*}
			\item 
			\label{enum:notation:V-(3)-3}
			For any $T^1, T^2 \in \scF_{0, d}[I]$, we have that
			\begin{equation*}
				\bV^{T^1}\big( \Omega; \cR \big) \otimes \bV^{T^2}\big( \Omega; \cR \big) \subseteq \bV^{T^1\circledast T^2} \big(\Omega; \cR \big)
			\end{equation*}
			\item 
			\label{enum:notation:V-(3)-4}
			For all $T\in \scF_{0, d}[I]$ and for any
			\begin{equation*}
				(\Upsilon,Y) \in \scF_{0, d} \tilde{\times} \scF[I] 
				\quad \mbox{such that} \quad 
				c_I\Big( T, (\Upsilon, Y) \Big) > 0,
			\end{equation*}
			we have that
			\begin{equation*}
				\bV^{T}\Big( \Omega; \cR \Big) \subseteq \bV^{Y}\Big( \Omega; \bV^{\Upsilon}\big( \Omega; \cR \big) \Big). 
			\end{equation*}
		\end{enumerate}				
	\end{assumption}

	\begin{definition}
		\label{definition:AlgebraRVs}
		Let $(\Omega, \cF, \bP)$ and $(\Omega', \cF', \bP')$ be probability spaces, let $I$ be an index set and let $(\cR, +, \centerdot)$ be a unital normed ring. 
		
		Suppose that $\bU(\Omega; \cR)$ satisfies Assumption \ref{notation:U-(1)} and that $\big( \bV^T \big)_{T\in \scF_{0, d}[I]}$ are a collection of modules that satisfies Assumption \ref{notation:V-(3)}.
				
		We define the $\bU(\Omega; \cR)$-module 
		\begin{align}
			\label{eq:definition:AlgebraRVs}
			\scH_I(\Omega, \Omega') 
			= 
			\bU\bigg( \Omega; \bigoplus_{T\in \scF_{0,d}[I]} \bV^{T} \Big( \Omega'; \cR \Big) \bigg).  
		\end{align}
		
		We use the convention that for any $X\in \scH_I(\Omega, \Omega')$, we write
		\begin{equation*}
			X(\omega_I) = \sum_{T\in \scF_{0, d}[I]} \Big\langle X, T \Big\rangle(\omega_I, \cdot) 
		\end{equation*}
		where
		\begin{equation}
			\label{eq:definition:AlgebraRVs:<X,T>}
			\begin{aligned}
				&\Big\langle X, T \Big\rangle \in \bU\bigg( \Omega; \bV^{T} \Big( \Omega'; \cR \Big) \bigg) 
				\quad\mbox{or}\quad
				\Big\langle X, T \Big\rangle(\omega_I, \cdot) \in \bV^T\Big( \Omega'; \cR \Big) 
			\end{aligned}
		\end{equation}
		
		The ring operators $+: \scH_I(\Omega, \Omega') \times \scH_I(\Omega, \Omega') \to \scH_I(\Omega, \Omega')$ is defined for $X, Y \in \scH_I(\Omega, \Omega')$ by
		\begin{equation}
			\label{definition:AlgebraRVs+}
			(X+Y)(\omega_I) 
			= \sum_{T \in \scF_{0, d}[I]} \Big\langle (X+Y), T \Big\rangle(\omega_I, \cdot)
			= \sum_{T \in \scF_{0, d}[I]} \bigg( \Big\langle X, T \Big\rangle(\omega, \cdot) + \Big\langle Y, T \Big\rangle(\omega_I, \cdot) \bigg)
		\end{equation}
		and $\centerdot: \bU(\Omega; \cR) \times \scH_I(\Omega, \Omega') \to \scH_I(\Omega, \Omega')$ is defined for $R \in \bU(\Omega; \cR)$ and $X \in \scH_I(\Omega, \Omega')$ by
		\begin{equation}
			\label{definition:AlgebraRVsx}
			(R \centerdot X)(\omega) 
			= \sum_{T \in \scF_{0, d}[I]} \Big\langle (R \centerdot X), T \Big\rangle(\omega_I, \cdot)
			=\sum_{T \in \scF_{0, d}[I]} R(\omega_I) \centerdot \Big\langle X, T \Big\rangle(\omega_I, \cdot). 
		\end{equation}
	\end{definition}
	
	We use the convention that $(\omega_{H^{T}}) \in \Omega^{\times \fm[T]}$ where $H$ is the set of hyperedges of the tree $T$ written as $H^T = \{h^T, ...\}$. 
	
	\subsubsection{Algebra over $\scH_I$}
	Following on from Definition \ref{definition:product-span}, we extend $\scH_I(\Omega, \Omega')$ to be an algebra over the ring $\bU(\Omega; \cR)$:
	\begin{definition}
		Let $(\Omega, \cF, \bP)$ and $(\Omega', \cF, \bP')$ be probability spaces. Let $(\cR, +, \centerdot)$ be a normed ring and let $I$ be an index set. 
		
		Suppose that $\bU\big( \Omega; \cR \big)$ satisfies Assumption \ref{notation:U-(1)} and $\big( \bV^W \big)_{W\in \scW_{0, d}[I]}$ is a collection of modules that satisfies Assumption \ref{notation:V-(3)}. 
		
		We define $\circledast: \scH_I(\Omega, \Omega') \times \scH_I(\Omega, \Omega') \to \scH_I(\Omega, \Omega')$ be the bilinear mapping that satisfies the identity
		\begin{align}
			\nonumber
			&X \circledast Y(\omega_I) = \sum_{T\in \scF_{0, d}[I]} \Big\langle X \circledast Y, T \Big\rangle(\omega_I, \cdot)
			\\
			\label{eq:definition:AlgebraRVs:Product}
			&\Big\langle X \circledast Y, T \Big\rangle(\omega_I) = \sum_{\substack{T^1, T^2 \in \scF_{0, d}[I] \\ T^1\circledast T^2 = T}} \Big\langle X, T^1 \Big\rangle(\omega_I, \cdot) \otimes \Big\langle Y, T^2 \Big\rangle(\omega_I, \cdot)
		\end{align}
	\end{definition}
	
	\begin{proposition}
		\label{proposition:associativity(2)}
		Let $(\Omega, \cF, \bP)$ and $(\Omega', \cF, \bP')$ be probability spaces. Let $(\cR, +, \centerdot)$ be a normed ring and let $I$ be an index set. 
		
		Suppose that $\bU\big( \Omega; \cR \big)$ satisfies Assumption \ref{notation:U-(1)} and $\big( \bV^W \big)_{W\in \scW_{0, d}[I]}$ is a collection of modules that satisfies Assumption \ref{notation:V-(3)}. Then
		\begin{enumerate}
			\item $\scH_I(\Omega, \Omega')$ as defined in Equation \eqref{eq:definition:coupled-shuffle-1} is a $\Big( \bU\big( \Omega; \cR \big), +, \centerdot \Big)$ module. 
			\item $\big( \scH_I(\Omega, \Omega'), \circledast , \rId \big)$ is a \emph{associative} \emph{unital} algebra over the ring $\bU(\Omega; \cR)$. 
			\item If $\bU(\Omega; \cR)$ is commutative, then $\big( \scH_I(\Omega, \Omega'), \circledast , \rId \big)$ is a commutative algebra over the ring. 
		\end{enumerate}
	\end{proposition}

	\begin{proof}
		The proof is similar to that of Proposition \ref{proposition:associativity(1)} where we replace references to Assumption \ref{notation:V-(2)} with Assumption \ref{notation:V-(3)} and references to Proposition \ref{proposition:Shuffle=Assoc} with Proposition \ref{proposition:Prod=Assoc}. 
	\end{proof}
	
	\subsubsection{The coupled tensor product of $\scH_I$}
	
	Following in the footsteps of Definition \ref{definition:M-coupledTensorModule}, we define the following:
	\begin{definition}
		\label{definition:H-coupledTensorModule}
		Let $(\Omega, \cF, \bP)$ and $(\Omega', \cF, \bP')$ be probability spaces. Let $(\cR, +, \centerdot)$ be a unital normed ring and let $I$ be an index set. 
		
		Suppose that $\bU\big( \Omega; \cR \big)$ satisfies Assumption \ref{notation:U-(1)} and $\big( \bV^T \big)_{T\in \scF_{0, d}[I]}$ is a collection of modules that satisfies Assumption \ref{notation:V-(3)}. We define
		\begin{equation}
			\label{eq:definition:H-coupledTensorModule}
			\scH_I \tilde{\otimes} \scH_I(\Omega, \Omega') = \bU\Bigg( \Omega; \bigoplus_{\hat{\Upsilon} \in \scF_{0,d}[I]} \bV^{\hat{\Upsilon}} \bigg( \Omega'; \bigoplus_{\substack{\bar{\Upsilon} \in \scF_{0,d}[\hat{\Upsilon}] }} \bV^{\bar{\Upsilon}} \Big( \Omega'; \cR \Big) \bigg) \Bigg). 
		\end{equation}
		We use the representation that for $X \in \scH_I\tilde{\otimes}\scH_I(\Omega, \Omega')$, we write
		\begin{align*}
			X(\omega_I) =& \sum_{(\bar{\Upsilon}, \hat{\Upsilon}) \in \scF_{0, d} \tilde{\times} \scF_{0, d}[I] } \Big\langle X, (\bar{\Upsilon}, \hat{\Upsilon}) \Big\rangle(\omega_I, \cdot)
		\end{align*}
		where 
		\begin{align*}
			&\Big\langle X, (\bar{\Upsilon}, \hat{\Upsilon}) \Big\rangle(\omega_I, \omega_{H^{\hat{\Upsilon}}}', \cdot ) 
			\in 
			\bV^{\bar{\Upsilon}}\Big( \Omega'; \cR \Big),
			\\
			\sum_{\bar{\Upsilon} \in \scF_{0,d}[\hat{\Upsilon}]} &\Big\langle X, (\bar{\Upsilon}, \hat{\Upsilon}) \Big\rangle(\omega_I, \omega_{H^{\hat{\Upsilon}}}', \cdot ) 
			\in 
			\bigoplus_{\bar{\Upsilon} \in \scF_{0,d}[\hat{\Upsilon}]} \bV^{\bar{\Upsilon}}\Big( \Omega'; \cR \Big),
		\end{align*}
		and
		\begin{align*}
			\sum_{(\bar{\Upsilon}, \hat{\Upsilon}) \in \scF_{0,d}\tilde{\times} \scF_{0, d}[I] } & \Big\langle X, (\bar{\Upsilon}, \hat{\Upsilon}) \Big\rangle (\omega_I, \cdot, \cdot)
			\in 
			\bigoplus_{\hat{\Upsilon} \in \scF_{0,d}[I]} \bV^{\hat{\Upsilon}} \bigg( \Omega'; \bigoplus_{\bar{\Upsilon} \in \scF_{0,d}[\hat{\Upsilon}]} \bV^{\bar{\Upsilon}} \Big( \Omega'; \cR \Big) \bigg),
			\\
			\sum_{(\bar{\Upsilon}, \hat{\Upsilon}) \in \scF_{0,d}\tilde{\times} \scF_{0, d}[I] } & \Big\langle X, (\bar{\Upsilon}, \hat{\Upsilon}) \Big\rangle \in \bU\Bigg( \Omega; \bigoplus_{\hat{\Upsilon}\in \scF_{0,d}[I]} \bV^{\hat{\Upsilon}} \bigg( \Omega'; \bigoplus_{\bar{\Upsilon} \in \scF_{0,d}[\hat{\Upsilon}]} \bV^{\bar{\Upsilon}} \Big( \Omega'; \cR \Big) \bigg) \Bigg). 
		\end{align*}
	\end{definition}
	
	Further, we also define for any $n\in \bN$ 
	\begin{align*}
		\scH_I^{\tilde{\otimes}n}(\Omega, \Omega') =& \bU\Bigg( \Omega; \bigoplus_{\hat{\Upsilon}_n \in \scF_{0, d}[I]} \bV^{\hat{\Upsilon}_n}\bigg( \Omega'; \bigoplus_{\hat{\Upsilon}_{n-1} \in \scF_{0, d}[\hat{\Upsilon}_n]} \bV^{\hat{\Upsilon}_{n-1}}\Big( \Omega';... \bigoplus_{\hat{\Upsilon}_1 \in \scF_{0, d}[\hat{\Upsilon}_2]} \bV^{\hat{\Upsilon}_1}\big( \Omega'; \cR \big) \Big) \bigg) \Bigg). 
	\end{align*}
	
	\subsubsection{The coupled coproduct on $\scH_I$}
	
	The next step of our agenda is to introduce the coupled Connes-Kreimer coproduct:
	\begin{definition}
		Let $(\Omega, \cF, \bP)$ and $(\Omega', \cF, \bP')$ be probability spaces. Let $(\cR, +, \centerdot)$ be a commutative unital normed ring and let $I$ be an index set. 
		
		Suppose that $\bU\big( \Omega; \cR \big)$ satisfies Assumption \ref{notation:U-(1)} and $\big( \bV^W \big)_{W\in \scW_{0, d}[I]}$ is a collection of modules that satisfies Assumption \ref{notation:V-(3)}. We define
		\begin{equation*}
			\Delta: \scH_I(\Omega, \Omega') \to \scH_I \tilde{\otimes} \scH_I(\Omega, \Omega')
			\quad \mbox{and}\quad
			\epsilon: \scH_I(\Omega, \Omega') \to \bU(\Omega; \cR)
		\end{equation*}
		for $X\in \scH_I(\Omega, \Omega')$ and $(\Upsilon, Y) \in \scF_{0, d} \tilde{\times} \scF_{0, d}[I]$ by
		\begin{equation*}
			\Big\langle \Delta\big[ X \big], (\Upsilon,Y) \Big\rangle(\omega_I,  \omega'_{H^{Y}},\omega'_{H^{\Upsilon}}) = \sum_{T \in \scF_{0,d}[I]} c_I\Big( T, \Upsilon, Y \Big) \cdot \Big\langle X, T\Big\rangle(\omega_I, \omega'_{\phi^{T, \Upsilon, Y}[T]})
		\end{equation*}
	\end{definition}
	
	Following in the same ideas as Proposition \ref{proposition:M-coassociativity}, we have the following:
	\begin{proposition}
		\label{proposition:H-coassociativity}
		Let $(\Omega, \cF, \bP)$ and $(\Omega', \cF', \bP')$ be probability spaces and let $(\cR, +, \centerdot)$ be a commutative, unital normed ring. 
		
		Suppose that $\bU\big( \Omega; \cR \big)$ satisfies Assumption \ref{notation:U-(1)} and $\big( \bV^T \big)_{T\in \scF_{0, d}[I]}$ is a collection of modules that satisfies Assumption \ref{notation:V-(3)}. Then
		\begin{equation}
			\label{eq:proposition:H-coassociativity}
			\fI \tilde{\otimes} \Delta \circ \Delta = \Delta \tilde{\otimes} \fI \circ \Delta
			\quad \mbox{and} \quad 
			\centerdot \circ \epsilon \tilde{\otimes} \fI \circ \Delta = \centerdot \circ \fI \tilde{\otimes} \epsilon \circ \Delta = \fI. 
		\end{equation}
		More specifically, $\big( \scH_I(\Omega, \Omega'), \Delta, \epsilon \big)$ is a coassociative coupled coalgebra over the ring $\big( \bU(\Omega; \cR), +, \centerdot \big)$. 
	\end{proposition}
	
	\begin{proof}
		The proof follows in the same fashion as Proposition \ref{proposition:M-coassociativity} where Proposition \ref{proposition:M-coassociativity*} is replaced by Proposition \ref{proposition:H-coassociativity*} and Assumption \ref{notation:V-(2)} is replaced by Assumption \ref{notation:V-(3)}. 
	\end{proof}
	
	\subsubsection{The coupled bialgebra over $\scH_I$}
	
	Recalling Definition \ref{definition:Twist} and Theorem \ref{theorem:M-coupledBialgebra}, we continue with the following:
	\begin{theorem}
		\label{theorem:H-coupledBialgebra}
		Let $(\Omega, \cF, \bP)$ and $(\Omega', \cF', \bP')$ be probability spaces. Let $(\cR, +, \centerdot)$ be a commutative unital normed ring and let $I$ be an index set. 
		
		Suppose that $\bU\big( \Omega; \cR \big)$ satisfies Assumption \ref{notation:U-(1)} and $\big( \bV^W \big)_{W\in \scW_{0, d}[I]}$ is a collection of modules that satisfies Assumption \ref{notation:V-(3)}. Then 
		\begin{equation}
			\label{eq:proposition:H-coupledBialgebra}
			\begin{aligned}
				\Delta \circ \circledast =& \Big( \circledast \tilde{\otimes} \circledast \Big) \circ \overline{\mbox{Twist}} \circ \Big( \Delta \otimes \Delta\Big),
				\\
				\epsilon \circ \circledast =& \centerdot \circ \epsilon \otimes \epsilon,
				\quad 
				\Delta \circ \rId \circ \centerdot = \rId \otimes \rId,
				\quad
				\epsilon \circ \rId = \fI_{\cR}. 
			\end{aligned}
		\end{equation}
		We say that $\big( \scH_I(\Omega, \Omega'), \circledast, \rId, \Delta, \epsilon \big)$ is a coupled bialgebra over the ring $\big( \bU(\Omega; \cR), +, \centerdot \big)$. More specifically, $\big( \scH_I(\Omega, \Omega'), \circledast, \rId \big)$ are associative algebras and $\big( \scH_I(\Omega, \Omega'), \Delta, \epsilon \big)$ is a co-associative coupled coalgebras.
	\end{theorem}

	\begin{proof}
		The proof follows in the same fashion as Theorem \ref{theorem:M-coupledBialgebra} where Theorem \ref{theorem:M-coupledBialgebra-} is replaced by Proposition \ref{proposition:Coproduct-Equivalence} and Assumption \ref{notation:V-(2)} is replaced by Assumption \ref{notation:V-(3)}. 
	\end{proof}
	
	\subsubsection{The grading over $\scH_I$}
	
	In the same style as Section \ref{subsubsection:grading-M}, we demonstrate a grading on our coupled bialgebra:
	\begin{definition}		
		Let $(\Omega, \cF, \bP)$ and $(\Omega', \cF', \bP')$ be probability spaces. Let $(\cR, +, \centerdot)$ be a commutative unital normed ring and let $I$ be an index set. 
		
		Suppose that $\bU\big( \Omega; \cR \big)$ satisfies Assumption \ref{notation:U-(1)} and $\big( \bV^T \big)_{T\in \scF_{0, d}[I]}$ is a collection of modules that satisfies Assumption \ref{notation:V-(3)}.
		
		For $(k, n) \in \bN_0^{\times 2}$, we define the $\bU(\Omega; \cR)$-module
		\begin{align*}
			\scH_I^{(k, n)}\big( \Omega, \Omega'\big) 
			:=& 
			\bU\bigg( \Omega; \bigoplus_{\substack{T\in \scF_{0,d}[I] \\ \scG^I[T] = (k,n)}}
			\bV^T\Big( \Omega'; \cR \Big) \bigg). 
		\end{align*}
		while for $(k_1, n_1), (k_2, n_2) \in \bN_0^{\times 2}$, we define the $\bU(\Omega; \cR)$-module
		\begin{equation*}
			\scH_{I}^{(k_1, n_1)} \tilde{\otimes} \scH_{I}^{(k_2, n_2)} (\Omega, \Omega'):= \bU \bigg( \Omega; \bigoplus_{Y \in \scF_{0, d}^{(k_2, n_2)}} \bV^{Y}\Big( \Omega'; \bigoplus_{\Upsilon \in \scF_{0, d}^{(k_1, n_1)}[Y]} \bV^{\Upsilon}\big( \Omega'; \cR \big) \Big) \bigg). 
		\end{equation*}
	\end{definition}
	
	\begin{proposition}
		\label{proposition:H-Grading}
		Let $(\Omega, \cF, \bP)$ and $(\Omega', \cF', \bP')$ be probability spaces. Let $(\cR, +, \centerdot)$ be a unital normed ring and let $I$ be an index set. 
		
		Suppose that $\bU\big( \Omega; \cR \big)$ satisfies Assumption \ref{notation:U-(1)} and $\big( \bV^T \big)_{T\in \scF_{0, d}[I]}$ is a collection of modules that satisfies Assumption \ref{notation:V-(3)}. Then 
		\begin{equation}
			\label{eq:proposition:H-Grading}
			\begin{split}
				&\scH_I^{(0_I,0)}(\Omega, \Omega') = \bU\big( \Omega; \cR \big), 
				\\
				&\scH_I^{(k_I^1, n^1)}(\Omega, \Omega') \circledast \scH_I^{(k_I^2, n^2)}(\Omega, \Omega') \subseteq  \scH_I^{(k_I^1+k_I^2, n^1+n^2)}(\Omega, \Omega'),  
				\\
				&\Delta\Big[ \scH_I^{(k_I, n)}(\Omega, \Omega') \Big] \subseteq  \bigoplus_{\substack{(k_I', n') \in \bN_0^{\times |I|} \times \bN_0}} \Big(  \scH_I^{(k_I - k_I', n - n')}(\Omega, \Omega') \Big) \tilde{\otimes} \scH_I^{(k_I', n')}(\Omega, \Omega'). 
			\end{split}
		\end{equation}
		That is, the function $\scG^I:\scT_{0,d}[I] \to \bN_0^{\times |I|} \times \bN_0$ as defined in Equation \eqref{eq:lemma:grading(1)} describes a $\bN_0^{\times |I|} \times \bN_0$-grading on the connected coupled bialgebras $\big( \scH_I(\Omega, \Omega'), \circledast, \rId, \Delta, \epsilon \big)$. 
	\end{proposition}
	
	\begin{proof}
		The proof follows in the same fashion as Proposition \ref{proposition:M-Grading} where Proposition \ref{proposition:M-Grading-} is replaced by Proposition \ref{proposition:H-Grading-} and Assumption \ref{notation:V-(2)} is replaced by Assumption \ref{notation:V-(3)}. 
	\end{proof}
	
	Following on from Equation \eqref{eq:truncatedWord}, we define
	\begin{equation}
		\label{eq:truncatedForest-Mod}
		\begin{aligned}
			\scH_I^{g, -}(\Omega, \Omega'):=& \bU\bigg( \Omega; \bigoplus_{T \in \scF_{0, d}^{g,-}[I]} \bV^{T}\Big(\Omega'; \cR \Big) \bigg)
			\\
			\scH_I^{g, +}(\Omega, \Omega'):=& \bU\bigg( \Omega; \bigoplus_{T \in \scF_{0, d}^{g,+}[I]} \bV^{T}\Big(\Omega'; \cR \Big) \bigg)
		\end{aligned}
	\end{equation}
	
	\begin{corollary}
		\label{lemma:H-Finite-grading}
		Let $(\Omega, \cF, \bP)$ and $(\Omega', \cF', \bP')$ be probability spaces. Let $(\cR, +, \centerdot)$ be a unital normed ring and let $I$ be an index set. 
		
		Suppose that $\bU\big( \Omega; \cR \big)$ satisfies Assumption \ref{notation:U-(1)} and $\big( \bV^T \big)_{T\in \scF_{0, d}[I]}$ is a collection of modules that satisfies Assumption \ref{notation:V-(3)}.
		
		Let $g: \bN_0^{\times |I|} \times \bN_0 \to \bR$ be a monotone increasing function such that $g(0_I, 0) \leq 0$. Then 
		\begin{enumerate}[label=(\ref*{lemma:H-Finite-grading}.\roman*)]
			\item
			\label{enum:lemma:H-Finitegrading-1}
			The $\bU(\Omega; \cR)$-module
			\begin{equation*}
				\scH_I^{g, +}(\Omega, \Omega')
				\quad \mbox{is an algebra ideals of } \quad
				\Big( \scH_I(\Omega, \Omega'), \circledast, \rId \Big)
			\end{equation*}
			and we can identify the algebra over the $\bU(\Omega; \cR)$-module quotient by
			\begin{align*}
				\scH_I^{g,-}(\Omega, \Omega') =& \scH_I(\Omega,\Omega') / \scH_I^{g,+}(\Omega,\Omega') 
			\end{align*}
			\item 
			\label{enum:lemma:H-Finitegrading-2}
			The coupled coproduct and counit $(\Delta, \epsilon)$ restricted to the sub-module
			\begin{equation*}
				\Big( \scH_I^{g,-}(\Omega,\Omega') , \Delta, \epsilon \Big)
			\end{equation*}
			is a co-associative sub-coupled coalgebras of $\big( \scH_I(\Omega,\Omega'), \Delta, \epsilon \big)$. 
			\item 
			\label{enum:lemma:H-Finitegrading-3}
			By pairing the quotient algebra and unit $(\circledast, \rId)$ to the restriction of the coupled coproduct and counit $(\Delta, \rId)$, we obtain
			\begin{equation*}
				\Big( \scH_I^{g,-}(\Omega,\Omega'), \circledast, \rId, \Delta, \epsilon \Big)
			\end{equation*}
			satisfies the commutative identities of Equation \eqref{eq:theorem:H-coupledBialgebra-}. 
		\end{enumerate}
	\end{corollary}
	
	\begin{proof}
		The proof follows in the same fashion as Corollary \ref{lemma:M-Finite-grading} where Proposition \ref{proposition:M-Grading} is replaced by Proposition \ref{proposition:H-Grading} and Assumption \ref{notation:V-(2)} is replaced by Assumption \ref{notation:V-(3)}. 
	\end{proof}
	}
	


\end{document}